\documentclass[opre]{informs3}
\renewcommand{\theARTICLETOP}{}
\renewcommand{\theRRHSecondLine}{}
\renewcommand{\theLRHSecondLine}{}
\usepackage{amsmath}

\usepackage{hyperref}
\usepackage{cleveref}
\crefname{equation}{}{}
\usepackage{mathtools}

\OneAndAHalfSpacedXI %

\usepackage{endnotes}
\usepackage{bm}

\usepackage{color-edits}
\addauthor[Chiwei]{cy}{cyan}
\addauthor[Louis]{lc}{brown}
\addauthor[Roberto]{rs}{red}

\usepackage{algorithm}
\usepackage{algpseudocode}
\usepackage{tikz}
\usetikzlibrary{tikzmark}
\usetikzlibrary{arrows.meta,positioning,calc, fit, backgrounds}
\usepackage{subcaption}
\usepackage{diagbox}
\usetikzlibrary{arrows}
\usepackage{booktabs}
\usepackage{multirow}
\usepackage{array}
\usepackage{comment}
\usepackage{enumitem}
\usepackage{bbm}
\usepackage{pgfplots}
\usepackage{pgfplotstable}
\usepackage{threeparttable}
\usepgfplotslibrary{groupplots}
\usepackage{fix-cm}
\usetikzlibrary{matrix,calc}
\pgfplotsset{compat=1.18}
\usepackage{soul}

\usepackage{natbib}
 \bibpunct[, ]{(}{)}{,}{a}{}{,}%
 \def\bibfont{\small}%

\EquationsNumberedThrough    %

\TheoremsNumberedThrough     %
\ECRepeatTheorems  %

\begin{document}

\RUNAUTHOR{Chen, Xu, et al.}

\RUNTITLE{Meeting Uncertain Threats with Feedback}

\TITLE{Meeting Uncertain Threats with Feedback}

\ARTICLEAUTHORS{%
\AUTHOR{Louis L. Chen,\textsuperscript{*,a} Ang Xu,\textsuperscript{*,b} Roberto Szechtman,\textsuperscript{a} Chiwei Yan,\textsuperscript{b} Vince Vanterpool\textsuperscript{a}}
\AFF{\textsuperscript{a}Department of Operations Research, Naval Postgraduate School, Monterey, California, \\
\EMAIL{louis.chen@nps.edu; rszechtm@nps.edu; vince.vanterpool@nps.edu}\\
\textsuperscript{b}Department of Industrial Engineering \& Operations Research, University of California, Berkeley,  California, 
\EMAIL{angxu@berkeley.edu; chiwei@berkeley.edu}
}
} %

\ABSTRACT{%

Air and missile defense, naval force protection, and critical-infrastructure security require rapid allocation of scarce effectors against multiple incoming threats when neutralization is uncertain and observed only after a firing round. We formulate this defensive-allocation problem as a Markov decision process where a commander assigns a fixed number of effector capacity each round to heterogeneous threats under three operational objectives: minimizing expected threat-clearance time, maximizing probability of clearance by a deadline, and maximizing effective assignments before a deadline. Rather than expensive full dynamic optimization, we study simple time-oblivious policies suited to fast implementation. Remarkably, fair allocation—which ignores threat difficulty and spreads fire evenly—is highly effective: it is optimal for all three objectives under homogeneous threats or low-capacity engagements; as well, when effector capacity scales at least linearly with threat counts, its threat-clearance time trails that of the optimal policy by at most a constant number of rounds. We further develop threat-difficulty-aware greedy policies for each specific objective, including a constant-factor guarantee for effective assignment maximization. Numerically, greedy policies are near-optimal across heterogeneous instances, while fair allocation remains a principled choice when threat difficulties are unknown.
}

\KEYWORDS{Dynamic weapon-target assignment, Markov decision process} 
\HISTORY{First version, 07/2026.}

\maketitle

\newcommand{\marker}[1]{}

\section{Introduction}\label{sec:Intro}

A fundamental problem in modern defense operations is how to dynamically allocate scarce effectors over time to neutralize an evolving set of incoming threats. The problem arises in air and missile defense, naval force protection, and defense of critical infrastructure, where a commander must repeatedly decide how a small firing capacity should be split across heterogeneous threats under severe time pressure \citep{brown2005two, smith2024airmissiledefense, karako2017missiledefense2020}. Two features make the problem especially difficult. First, different threats respond differently to fire, so the value of an additional effector depends on the target. Second, kill assessment is available only at the end of each round: within a round, the defender must commit all shots before learning which threats have already been neutralized. The central challenge is therefore dynamic allocation under uncertain neutralization and round-end feedback.

The classical weapon-target assignment literature provides a strong foundation for one-shot allocation problems \citep{manne1958target, lloyd1986weapons, ahuja2007exact, bertsimas2025solving}. Modern defensive engagements, however, are sequential. After each round, some threats survive, the engagement state changes, and the defender reallocates the next round's limited capacity. This temporal structure creates intertemporal trade-offs that are absent in static formulations: when should the defender concentrate fire on hard threats, when should the defender spread fire to hedge against deadline risk, and how should the answer depend on the operational objective? %

Motivated by the above challenges, we study a stylized but operationally relevant defensive allocation model with a constant per-round capacity, threat-specific kill probabilities, and round-end feedback. We formulate the problem as a Markov decision process (MDP) and analyze three objectives---defusing time minimization (DTM), survival likelihood maximization (SLM), and effective assignment maximization (EAM)---corresponding respectively to expected time until all threats are cleared, the probability that all threats are neutralized by a deadline $\tau$ rounds from now, and the number of effective assignments completed before that deadline. A main theme of the paper is that these objectives are linked rather than isolated. In particular, we show that expected defusing time can be written in terms of deadline-indexed survival probabilities and that minimizing defusing time is equivalent to minimizing cumulative waste up to clearance. This perspective lets us compare policies across objectives and ask how much of the value of optimal dynamic control can be captured by simple, scalable rules.

While our setting is related to the dynamic weapon-target assignment (DWTA) problem, it differs in two fundamental ways. First, our perspective is defensive rather than offensive. More precisely, we attempt to optimize the employment of effectors to protect from incoming missiles, while the more classical DWTA aims to optimize the employment of assets to destroy targets. Second, the classic objective is the maximization of the expected value of the destroyed targets. In contrast, we propose and analyze a number of different objectives, as already mentioned, to better capture the operational priorities of a defensive engagement.

\smallskip

\textbf{Our results.} Our contributions are multifold. 

\begin{itemize}
    \item We begin by characterizing the structure of optimal policies in the two-threat setting. We show that, as the deadline becomes more distant, for both the SLM and EAM objectives, the optimal first-round allocation becomes increasingly concentrated on the harder threat (\Cref{theorem:properties_n=2}).

    \item We analyze fair allocation $\pi_{\textsf{FA}}$, a time-oblivious policy that ignores threat heterogeneity and allocates capacity fairly among all remaining threats each period. We show that $\pi_{\textsf{FA}}$ is exactly optimal under homogeneous threats or two-unit per-round capacity (\Cref{theorem:fair_is_optimal_when_qi=q}). We establish the performance ratio and the worst-case performance of $\pi_{\textsf{FA}}$ under each objective (Theorems \ref{theorem:DTM_fair_ratio}, \ref{thm:tau1_ratio_cfree} and \ref{thm:fair_alloc_min_waste}). We also characterize the asymptotic behavior of $\pi_{\textsf{FA}}$ when the number of threats and the capacity both approach infinity, showing that in this case, $\pi_{\textsf{FA}}$ matches the optimal policy in order of growth of expected defusing time (\Cref{thm:spread-upper}). We further strengthen this asymptotic comparison by showing that, whenever per-round capacity is at least proportional to the number of threats, the DTM gap between fair allocation and the optimal policy remains uniformly bounded in $n$ (\Cref{thm:dtm_fair_additive_gap}).

    \item We develop two objective-aware greedy policies, $\pi_{\textsf{SL}}$ and $\pi_{\textsf{EA}}$, that remain computationally light while exploiting heterogeneity in threat difficulty. For EAM, we prove a constant-factor approximation guarantee for $\pi_{\textsf{EA}}$ (\Cref{thm:adaptive_greedy}). In the two-threat setting, we also show that the greedy policies dominate fair allocation under their aligned objectives (Theorems \ref{thm:sl_vs_fair_n2} and \ref{thm:ea_vs_fair_n2}).

    \item In numerical experiments, we demonstrate that our proposed greedy approaches achieve near-optimal performance, while fair allocation remains surprisingly robust when resource capacity scales alongside the threat environment.
\end{itemize}

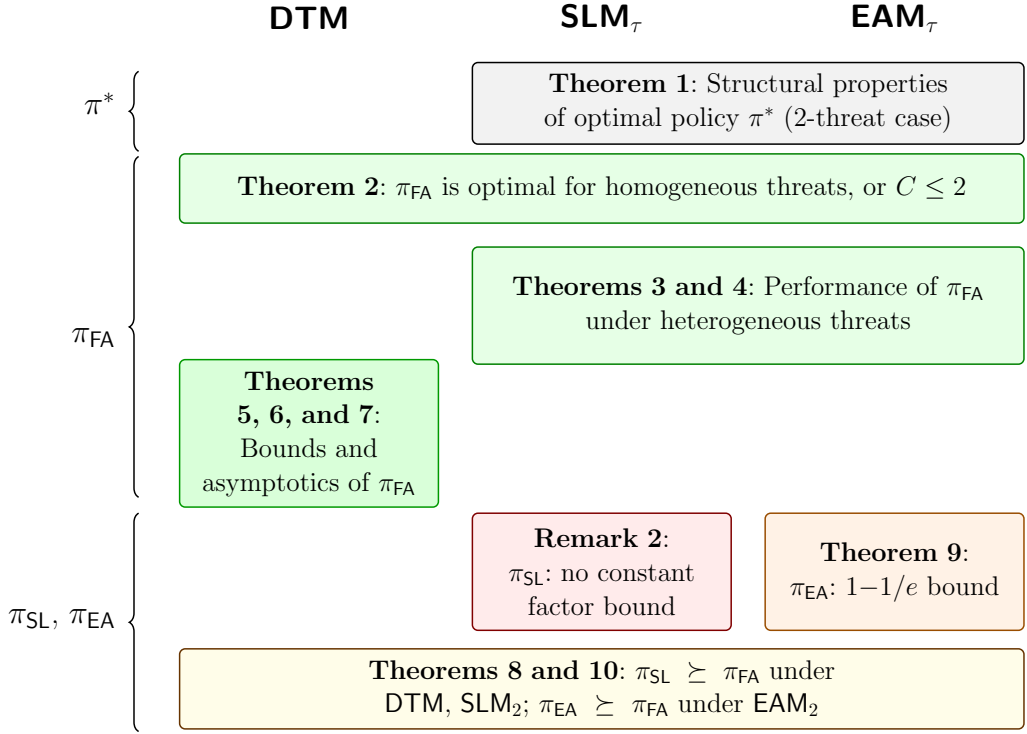
\begin{figure}[htbp]
\centering
\resizebox{0.82\textwidth}{!}{%
\begin{tikzpicture}[
    >=Latex,
    every node/.style={font=\large},
    header/.style={font=\Large\bfseries, minimum height=0.8cm},
    thm/.style={draw, rounded corners=3pt, align=center, inner sep=6pt, line width=0.8pt},
    label/.style={font=\Large, anchor=east},
    smallbox/.style={thm, text width=4.44cm, minimum height=2.2cm},
]
\def\singleW{4.44cm}
\def\doubleW{9.94cm}
\def\tripleW{15.44cm}
\def\dtmC{0}
\def\slmC{5.5}
\def\eamC{11.0}
\def\allC{5.5}
\def\slmeamC{8.25}
\def\rowH{0}
\def\rowOne{-1.6}
\def\rowTwo{-3.2}
\def\rowTF{-5.4}
\def\rowTS{-7.8}
\def\rowProp{-10.4}
\def\rowSeven{-10.4}
\def\rowEight{-12.6}
\def\htSingle{1.3cm}
\def\htDouble{2.2cm}
\node[header] at (\dtmC,\rowH) {\textsf{DTM}};
\node[header] at (\slmC,\rowH) {\textsf{SLM}$_\tau$};
\node[header] at (\eamC,\rowH) {\textsf{EAM}$_\tau$};
\node[label] at (-3.5, \rowOne) {$\pi^*$};
\draw[thick, decorate, decoration = {brace, amplitude = 5pt, mirror}]
(-3.2,-1) -- (-3.2,-2.5);
\node[label] at (-3.5, -6.0) {$\pi_{\textsf{FA}}$};
\draw[thick, decorate, decoration={brace, amplitude=5pt, mirror}]
  (-3.2,-2.6) -- (-3.2,-9.0);
\draw[thick, decorate, decoration={brace, amplitude=5pt, mirror}]
  (-3.2,-9.3) -- (-3.2,-13.4);
\node[font=\Large, anchor=east, align=right] at (-3.5, -11.3)
  {$\pi_{\textsf{SL}}$, $\pi_{\textsf{EA}}$};
\node[thm, fill=gray!10, text width=\doubleW, minimum height=\htSingle] at (\slmeamC,\rowOne)
{\textbf{\Cref{theorem:properties_n=2}}: Structural properties of optimal policy $\pi^*$ (2-threat case)};
\node[thm, fill=green!10, draw=green!50!black, text width=\tripleW, minimum height=\htSingle] at (\allC,\rowTwo)
  {\textbf{\Cref{theorem:fair_is_optimal_when_qi=q}}: $\pi_{\textsf{FA}}$ is optimal for homogeneous threats, or $C \leq 2$};
\node[thm, fill=green!10, draw=green!50!black, text width=\doubleW, minimum height=\htDouble] at (\slmeamC,\rowTF)
  {\textbf{Theorems \ref{theorem:DTM_fair_ratio} and \ref{thm:spread-upper}}: Performance of $\pi_{\textsf{FA}}$\\under heterogeneous threats};
\node[smallbox, fill=green!15, draw=green!60!black] at (\dtmC,\rowTS)
  {\textbf{Theorems \ref{thm:dtm_fair_additive_gap}, \ref{thm:tau1_ratio_cfree}, and \ref{thm:fair_alloc_min_waste}}:\\Bounds and asymptotics of $\pi_{\textsf{FA}}$};
\node[smallbox, fill=red!8, draw=red!50!black] at (\slmC,\rowProp)
  {\textbf{\Cref{example:ml_greedy_no_const_ratio}}:\\$\pi_{\textsf{SL}}$: no constant factor bound};
\node[smallbox, fill=orange!10, draw=orange!60!black] at (\eamC,\rowSeven)
  {\textbf{\Cref{thm:adaptive_greedy}}:\\$\pi_{\textsf{EA}}$: $1{-}1/e$ bound};
\node[thm, fill=yellow!10, draw=orange!40!black, text width=\tripleW, minimum height=\htSingle] at (\allC,\rowEight)
  {\textbf{Theorems \ref{thm:sl_vs_fair_n2} and \ref{thm:ea_vs_fair_n2}}: $\pi_{\textsf{SL}} \succeq \pi_{\textsf{FA}}$ under \textsf{DTM}, \textsf{SLM}$_2$; $\pi_{\textsf{EA}} \succeq \pi_{\textsf{FA}}$ under \textsf{EAM}$_2$};
\end{tikzpicture}%
}
\caption{Logical map of the paper's main results. \textsf{SLM}$_\tau$ and \textsf{EAM}$_\tau$ refer to objectives under a deadline $\tau$.}
\label{fig:theorem-map}
\end{figure}

\textbf{Decision prescriptions.}
\Cref{tab:decision_insights} summarizes the main operational prescriptions under different objectives. In particular, $\pi_{\textsf{SL}}$ is an effective choice for DTM and missions with fixed deadlines, where quickly neutralizing all threats or meeting a deadline (SLM) is important, because it directly targets the kill-all probability. When the objective rewards partial progress toward threat neutralization and emphasizes magazine efficiency over a fixed engagement horizon (EAM), $\pi_{\textsf{EA}}$ is preferred, since it maximizes the expected number of effective assignments and captures how efficiently the defense converts scarce shots into useful pressure on the attack. This is especially important when preserving resources for subsequent threats matters. Although greedy policies such as $\pi_{\textsf{SL}}$ and $\pi_{\textsf{EA}}$ tend to perform well across a range of scenarios, the fair allocation policy $\pi_{\textsf{FA}}$ can also be a practical and effective choice, particularly when threat difficulty levels are not known precisely. We show that $\pi_{\textsf{FA}}$ is most effective under the DTM objective, followed by EAM and then SLM.

\begin{table}[t]
\centering
\small
\caption{Decision prescriptions for different objectives. \textsf{SLM}$_\tau$ and \textsf{EAM}$_\tau$ refer to objectives under a deadline $\tau$.}
\label{tab:decision_insights}
\begin{tabular}{p{0.23\linewidth}p{0.42\linewidth}>{\centering\arraybackslash}p{0.23\linewidth}}
\toprule
Objective & Operational Context & Prescribed Policy \\
\midrule
DTM 
& Clear the attack as quickly as possible
& $\pi_{\textsf{SL}}$ \\
SLM$_\tau$
& Neutralize all threats by a deadline
& $\pi_{\textsf{SL}}$ \\
EAM$_\tau$ 
& Maximize magazine efficiency 
& $\pi_{\textsf{EA}}$ \\
\bottomrule
\end{tabular}
\end{table}

\section{Literature Review}\label{sec:LitReview}

Our paper is related to three streams of literature: weapon-target assignment, sequential shoot-look-shoot defense models, and tractable policies for stochastic dynamic resource allocation.

The first stream is the classical weapon-target assignment (WTA) literature. Foundational work formulates WTA as a nonlinear combinatorial optimization problem and establishes its computational difficulty \citep{manne1958target, lloyd1986weapons}. Subsequent research develops exact and large-scale methods for static instances \citep{ahuja2007exact, bertsimas2025solving}. Related defense-oriented OR models also study missile-defense planning at the theater level \citep{brown2005two}. These papers are essential predecessors, but they mostly consider one-shot or open-loop allocation. Our setting is instead a closed-loop defensive engagement in which the set of surviving threats is observed after each round and the allocation is revised accordingly.

The second stream is the DWTA and shoot-look-shoot literature. Early treatments emphasize sequential engagement logic and kill assessment in missile allocation problems \citep{eckler1972mathematical, hosein1990dynamic, glazebrook2004shoot}. More recent work considers richer engagement sequences and aerial-threat settings, including shoot-look-shoot algorithms and hard/soft defense against sequential aerial threats \citep{atkinson2025hard, merkulov2024virtual}. A closely related body of work is the theory of search and detection, in which a searcher allocates limited effort to locate one or more objects under an imperfect, often exponential, detection law \citep{koopman1957theory, stone1989theory}. The sequential, multi-object version of this problem---finding several objects that may be missed even when present and that differ in difficulty---is the natural search analogue of our defensive setting \citep{assaf1987continuous, alpern2013mining, lin2013graph}: DTM minimizes the expected time to clear every threat and SLM maximizes the probability of clearing all threats by a deadline, each the defensive counterpart of a find-all criterion. Our model belongs to this family, but differs in two ways that are central to the paper. First, we focus on a constant per-round defensive capacity with delayed within-round feedback. Second, we study how the preferred allocation rule changes across three operationally distinct objectives rather than under a single criterion.

The third stream concerns tractable control for stochastic dynamic resource allocation. Exact MDP formulations are natural but quickly become computationally difficult as the state and action spaces expand \citep{bertsekas2012dynamic}. To address this, prior work related to our setting studies index policies for shooting problems \citep{glazebrook2007index}, approximate dynamic programming for WTA \citep{ahner2015approximate}, and generalized stochastic task-resource allocation with retry opportunities \citep{gulpinar2018heuristics}. More generally, this literature often exploits weakly coupled structure through decomposable MDP relaxations and fluid (re)optimization policies \citep{bertsimas2016decomposable, brown2022dynamic,brown2025fluid}. 
A key distinction from our work is structural. Weakly coupled MDP and fluid-reoptimization approaches usually obtain tractability by relaxing shared resource constraints to achieve an additively separable structure. Outside of EAM, our objectives are \emph{not} additively separable: DTM depends on the first time at which the entire threat set has been cleared, and SLM depends on the joint event that all threats are neutralized by a deadline. 
Beyond this structural distinction, our approach is also different. For the additively separable EAM objective, we exploit the problem-specific submodularity of the allocation reward to develop much simpler policies, avoiding the need to solve a linear program in every round as in typical fluid (re)optimization frameworks \citep{brown2022dynamic,brown2025fluid}. The resulting policy comes with non-asymptotic performance guarantees, and substantially outperform state-of-the-art fluid (re)optimization baselines in our numerical experiments.

In summary, relative to the aforementioned streams of literature, our paper focuses on analytical questions that have received less attention: how multiple defensive objectives are related, when simple time-oblivious policies are optimal, and what performance guarantees can be proved for these scalable policies. In that sense, the paper complements existing algorithmic DWTA work by providing structural insight and approximation guarantees for policies that are simple enough to be implementable in real time.

\section{Model and Preliminaries} \label{sec:model}
In this work, we study three related, discrete-time stochastic dynamic resource-allocation problems with decisions made sequentially over time. All three can be formalized as a Markov Decision Process (MDP), with all three sharing the same dynamics; however, they will be distinguished by three different objectives (i.e., cost/reward). We begin by describing these dynamics before outlining the three objectives. But first we introduce important parameters for the setting and discussion to follow.

\smallskip

\textbf{Instance parameters $(n, C, q, \tau)$.}
There are $n \in \mathbb{Z}_{\geq 1}$ incoming threats indexed by $i \in [n] := \{1,\dots,n\}$. In each round, a total of $C \in \mathbb{Z}_{\geq 1}$ defensive effectors (e.g., interceptors) is available for allocation across the threats, and we refer to the ratio $C/n$ as the \emph{resource tightness}. The set $[n]$ will be accompanied by a vector $q := (q_i)_{i=1}^n \in [0,1)^n$ of \emph{difficulties} for each threat, where $q_i$ is the probability that a single effector fails to neutralize threat $i$. We refer to an instance involving homogeneous threats, i.e., $q_1 = q_2 = \ldots = q_n$ as a \emph{salvo} instance. Finally, $\tau \in \mathbbm{Z}_{\geq 1}$ will denote a \emph{deadline} that is featured in some objectives to be discussed (Sections \ref{sec:: SLM_Objective} and \ref{sec:: EAM_Objective}).  %

\subsection{States, Actions, and Policies} \label{sec:: MDP_Outline}
In the following, we regard the instance parameters $(n, C, q, \tau)$ as given. Here we formalize the discrete-time dynamics to the MDP, in which time will be indexed via $t\in \mathbb{Z}_{\geq 1}.$ 

\smallskip

\noindent\textbf{State Space $\mathcal{S}$:} We define the state space as the power set $\mathcal{S} := \{\mathcal{M} :\mathcal{M} \subseteq [n]\}.$
At $t = 1,$ the beginning of the time horizon, the initial state is $\mathcal{M}_1\coloneqq [n].$ In words, $\mathcal{M}_t \subseteq [n]$ denotes the set of \emph{active} threats at the beginning of round $t$, and we let $M_t := |\mathcal{M}_t| \le n$ denote its cardinality.

\smallskip

\noindent\textbf{Action Space $\mathcal{A}$:} The action space is $\mathcal{A} := \left\{ (x_i)_{i \in [n]} \in \mathbb{Z}_{\ge 0}^n : \sum_{i=1}^n x_i \le C \right\},$ the set of allocations of $C$ effectors to the set of threats $[n]$.

\smallskip

\noindent\textbf{State Transitions:} At any time $t \in \mathbb{Z}_{\geq 1},$ given $\mathcal{M}_t \in \mathcal{S}$ and action $x \in \mathcal{A},$ the transition to the state $\mathcal{M}_{t+1} \subseteq \mathcal{M}_t$ occurs with probability $\Pi_{i \in \mathcal{M}_{t+1}} q_i^{x_i} \cdot \Pi_{i \in \mathcal{M}_t \setminus \mathcal{M}_{t+1}}(1-q_i^{x_i}).$ 

\smallskip

\noindent\textbf{Policies:}
A (Markovian) deterministic policy is a collection of mappings $\pi:=(\pi_t)_{t\in\mathbb{Z}_{\ge1}}$, wherein $\pi_t: \mathcal{S} \to \mathcal{A}$ for all $t$\footnote{A random policy $\pi$ will by contrast have $\pi_t: \mathcal{S} \rightarrow \Delta^\mathcal{A}$, where $\Delta^\mathcal{A}\coloneqq \{\xi \in [0,1]^{|\mathcal{A}|} : \sum_{a \in \mathcal{A}} \xi_a = 1\}$ is the set of distributions over the action space. \label{footnote: random policy}} 
 and consideration of parameters $(n, C, q, \tau)$ is implied but notationally suppressed for the sake of brevity. Formally, we define the set of feasible polices as 
 \[
\Pi := \left\{ \pi : \sum_{i \in \mathcal{M}} \left[\pi_t(\mathcal{M})\right]_i = C,~ \forall\mathcal{M} \in \mathcal{S}\ \text{with}\ \mathcal{M} \neq \emptyset, ~\forall t\in\mathbb{Z}_{\geq 1} \right\},
\]
where we emphasize that exhaustion of all $C$ effectors is required in all rounds. We say a policy $\pi$ is \emph{time-oblivious} (stationary) if $\pi_t = \pi_{t+1}$ for all $t \in \mathbbm{Z}_{\geq 1}$. Given a policy $\pi$, a random nested collection is yielded, i.e., $[n] = \mathcal{M}_1 \supseteq \mathcal{M}_2 \supseteq \ldots \supseteq \emptyset$, that describes the random evolution of the set of active threats from round to round.  
Hence, in the remainder of the paper, %
we will often find the shorthand $x^{t}:= \pi_t(\mathcal{M}_t)$ convenient in denoting the (random) allocation in round $t$ when the policy $\pi$ is clear from context.\footnote{If $\pi$ is a random policy, then at time $t$, given $\mathcal{M}_t$, we instead understand the notation $x^t$ to denote a random action drawn according to $\pi_t(\mathcal{M}_t),$ i.e., $x^t \sim \pi_t(\mathcal{M}_t)$.} Further, unless otherwise specified, the shorthand $x_i := x_i^1$ will denote the allocation in the first round.  
\subsubsection{Uncertain Demands and Feedback} \label{sec:: Non_MDP Language}The formalism of the MDP aside, we will often find the following alternative, less formal, description more convenient, as we will see when discussing the optimization objectives. Each effector assigned to threat $i$ either neutralizes with probability $1-q_i$ or fails to neutralize with probability $q_i$, independently of all other effector applications. 
It follows that the number of effectors required to neutralize threat $i$ is geometrically distributed with parameter $1-q_i$, and we will denote this random quantity with $L_i \sim \operatorname{Geom}(1-q_i)$, and write $L \coloneqq (L_i)_{i \in [n]}$ for the mutually independent collection of these random integers that is drawn at the outset before time $t=1$. Then $\mathcal{M}_{t+1} = \{i \in \mathcal{M}_t: \sum_{\ell = 1}^t x^{\ell}_i < L_i\}$. In other words, although $L$ is never explicitly revealed, at the conclusion of each time period $t,$ the decision maker receives \emph{satisficing} feedback, that is, for any active threat $i$, it is revealed whether or not the sum $\sum_{\ell = 1}^t x^{\ell}_i$ of allocations up to that point in time has met the requisite uncertain demand $L_i$ for neutralization. This feedback is provided after every assignment of $C$ effectors; this feature to the model has the following physical interpretation: given a fixed firing rate of effectors towards incoming threats, battle damage assessment (e.g., live/neutralized status) of the set of active threats can be attained at the frequency at which $C$ effectors are fired, which will be a function of factors including radar/imaging technology and the firing rate. In this way, the parameter $C$ models the rate at which satisficing feedback is attained; smaller $C$ means quicker, whereas larger $C$ means slower, feedback. %

\subsection{Objectives}
To complete the description of the MDPs, we will now formally define three different objectives, and with them, three optimization problems of interest, each corresponding to a different operational priority in a defensive engagement: clearing the attack quickly \cref{eq::MinDefuse}, surviving a hard deadline \cref{eq::SurvivalObjective}, or conserving scarce shots by avoiding overkill \cref{eq::EffectiveObjective}. We will treat each of these three in order. 
At a glance, %
\cref{eq::MinDefuse} is a minimization problem, and from it, there are two related maximization problems \cref{eq::SurvivalObjective} and \cref{eq::EffectiveObjective}, each parametrized by the deadline $\tau$.
We note that although all three admit Bellman equation formulations aligned with the MDP framework of \Cref{sec:: MDP_Outline}, we will eschew such presentations here in favor of forms utilizing the shorthands and notations of \Cref{sec:: Non_MDP Language}.

\subsubsection{Defusing Time Minimization (\textsf{DTM})}
For this optimization problem, we minimize the expected time required to neutralize all threats. To this end, for a policy $\pi$, we define its \emph{defusing time} %
\begin{equation*}
    T^\pi := \min\left\{ t : \sum_{\ell=1}^{t} x_i^\ell \ge L_i,\quad \forall i = 1, \ldots, n \right\},
\end{equation*}
which is the number of rounds required for every threat $i$ to receive at least its required number $L_i$ of effectors under policy $\pi$. Although the dependence on $(n,C,q)$ is implicit here, we may occasionally make this explicit via the notational conventions such as  %
$T^{\pi}(q)$ or $T^{\pi}(q; C)$, so as to facilitate clarity.%

The \emph{defusing time minimization} problem is then to choose a policy that minimizes the expected defusing time:
\begin{equation} \label{eq::MinDefuse} \tag{\textsf{DTM}}
    \mathbb{E}[T^*] := \min_{\pi \in \Pi}~\mathbb{E}[T^\pi],
\end{equation}
where $T^* := T^{\pi^*}$ denotes the defusing time under an optimal policy $\pi^*.$
The \cref{eq::MinDefuse} objective, $\mathbb{E}[T^\pi]$, is most natural when the defender's primary goal is to stop the attack as quickly as possible. In naval or air defense, a shorter defusing time reduces the duration over which the defended asset remains exposed, frees resources for follow-on salvos, and helps restore freedom of maneuver. For example, a ship confronting a mixed missile-UAV attack may prefer a policy that clears the current inbound set rapidly so that assets (e.g., sensors and interceptors) can be reallocated to the next wave.

\smallskip

\textbf{Deadline-Parametrized Objectives.}
Towards discussing the next two optimization formulations, we note that in many military applications, the defender faces an explicit deadline $\tau \in \mathbb{Z}_{>0}$, induced for example by time-to-impact or by the last feasible intercept opportunity. The introduction of a deadline $\tau$ effectively constrains the defender's supply of total effectors to $C\tau$, which leads to two related maximization problems that examine the balancing act between satisfying and over-satisfying the uncertain demands $L$. 

\subsubsection{Survival Likelihood Maximization (\textsf{SLM$_\tau$})} \label{sec:: SLM_Objective}

 The first is \emph{survival likelihood maximization}, where the defender seeks to maximize the probability of neutralizing all threats by the deadline $\tau$, i.e., $\operatorname{Pr}(T^\pi \le \tau)$. Specifically, given $\tau$, this objective can be written as
\begin{equation} \label{eq::SurvivalObjective} \tag{\textsf{SLM$_\tau$}}
   V_{\textsf{SLM}_\tau}^* \coloneqq \max_{\pi \in \Pi} \left(V_{\textsf{SLM}_\tau}^\pi := 
   \operatorname{Pr}(T^\pi \le \tau)\right). 
\end{equation}
We note that unless the resource tightness $C/n$ is at least $1/\tau,$ we trivially find $V_{\textsf{SLM}_\tau}^* = 0.$
The objective $V_{\textsf{SLM}_\tau}^\pi$ is appropriate for point-defense missions with a hard engagement window. If even one inbound missile survives past round $\tau$ and reaches the protected ship, air base, or critical asset, the defensive mission can fail catastrophically; in that setting, partial attrition has limited value unless the entire salvo is defeated in time. Thus, \cref{eq::SurvivalObjective} captures the all-or-nothing nature of survival against uncertain threats. 

\smallskip

\textbf{On the connection between \cref{eq::SurvivalObjective} and \cref{eq::MinDefuse}.} 
By the tail-sum formula, the expected defusing time under a policy $\pi$ can be written as
\begin{align}\label{eq:dtm_slm}
\mathbb{E}[T^\pi] = \sum_{\tau=0}^{\infty} \operatorname{Pr}(T^\pi > \tau)
= \sum_{\tau=0}^{\infty} \bigl(1-\operatorname{Pr}(T^\pi \le \tau)\bigr)=1+\sum_{\tau=1}^{\infty} \bigl(1-V_{\textsf{SLM}_\tau}^\pi\bigr).
\end{align}
Thus, the expected defusing time under policy $\pi$ is determined by its survival likelihood across all deadlines. In particular, \cref{eq::MinDefuse} and \cref{eq::SurvivalObjective} are aligned but not identical: \cref{eq::MinDefuse} considers the collective performance across all deadlines by a policy, 
whereas \cref{eq::SurvivalObjective}
considers the performance for the single deadline $\tau$. Consequently,
\begin{equation} \label{eq:: DTM_SLM Connection}
\mathbb{E}[T^*] \ge 1 + \sum_{\tau = 1}^\infty \left(1 - V_{\textsf{SLM}_\tau}^*\right).
\end{equation}
This reveals that if a policy $\pi^*$ exists such that $T^{\pi^*}$ is smallest with respect to stochastic ordering ($\preceq_{st}$), i.e., $T^{\pi^*} \preceq_{st} T^\pi$ for all $\pi \in \Pi,$ then $\pi^*$ simultaneously solves \eqref{eq::SurvivalObjective} for all $\tau$ as well as \eqref{eq::MinDefuse}. Such a policy can in fact exist under a parameter regime ($q_i = q_j$ for all $i,j \in [n]$), as will be examined in \Cref{sec:fair_alloc}.

\subsubsection{Effective Assignment Maximization (\textsf{EAM$_\tau$})}  \label{sec:: EAM_Objective}
The second deadline-based criterion is \emph{effective assignment maximization}, which aims to maximize the expected number of effective assignments completed within $\tau$ rounds. For a given horizon $\tau$ and policy $\pi$, the total number of assignments allocated to threat $i$ can be decomposed as
\[
\sum_{t=1}^{\tau} x_i^t
=
\underbrace{\min\left\{\sum_{t=1}^{\tau} x_i^t,\; L_i\right\}}_{\text{effective assignments}}
+
\underbrace{\max\left\{\sum_{t=1}^{\tau} x_i^t - L_i,\; 0\right\}}_{\text{wasted assignments}}.
\]
Here, $\min\left\{\sum_{t=1}^{\tau} x_i^t,\; L_i\right\}$ is the number of \emph{effective assignments} to threat $i$, while the total number of \emph{wasted assignments} under policy $\pi$ is
\[
W_\tau^\pi \coloneqq \sum_{i=1}^n \max\left\{\sum_{t=1}^{\tau} x_i^t - L_i,\; 0\right\}.
\]
In words, for any active threat, all prior effectors assigned to it are counted as effective. On the other hand, for any neutralized threat $i$, all effectors assigned to it in excess of $L_i$ would constitute waste. In shoot-look-shoot terms, such waste is overkill generated by batched kill assessment within a round.

Under this criterion, the problem is to find a policy that maximizes the expected total number of effective assignments across all threats:
\begin{equation}
\label{eq::EffectiveObjective} \tag{\textsf{EAM$_\tau$}}
V_{\textsf{EAM}_\tau}^*
\coloneqq
\max_{\pi \in \Pi}~ \left(V_{\textsf{EAM}_\tau}^\pi
:=
\mathbb{E}\left[\sum_{i=1}^n \min\!\left\{\sum_{t=1}^{\tau} x_i^t,\; L_i\right\}\right]\right).
\end{equation}
The objective, $V_{\textsf{EAM}_\tau}^\pi$, captures magazine efficiency over a fixed engagement horizon. In layered air defense or counter-UAS operations, commanders want as many interceptors as possible to be committed to still-live threats rather than wasted on incoming missiles that were already neutralized earlier in the same salvo. Maximizing effective assignments over the next $\tau$ rounds therefore measures how efficiently the defense converts scarce shots into useful pressure on the attack, which is especially important when preserving inventory for follow-on salvos matters.

Under \cref{eq::EffectiveObjective}, our model can be interpreted as a special case of a weakly coupled dynamic program, a class of problems that has been studied extensively in the literature (e.g., \citealt{brown2022dynamic,daeth2023optimal,brown2025fluid}). Following the terminology of \cite{brown2025fluid}, each threat corresponds to a subproblem in this dynamic program, and the objective in \cref{eq::EffectiveObjective} represents the total reward obtained by summing effective assignments across threats. These subproblems are coupled through a shared resource constraint $\sum_{i=1}^n x_i^t \le C$ in each round $t$. Unlike the more general models considered in this literature, our setting has a submodular reward structure in the allocation decisions, and we will exploit such a setting in \Cref{sec:greedy_policies} to derive a simple and effective policy.

\smallskip

\textbf{On the connection between \cref{eq::EffectiveObjective} and \cref{eq::MinDefuse}.}
With regards to the number of effective assignments made by a policy $\pi$, its complement, $W^\pi_\tau$ presents a connection to \eqref{eq::MinDefuse}. Indeed, for any policy $\pi$, it clearly holds that $C \cdot T^\pi = W_{T^\pi}^\pi + \sum_{i=1}^n L_i$, with probability 1; hence, 
\begin{equation} \label{eq:: EAM_DTM Connection}
\mathbb{E}[W_{T^\pi}^\pi] = C \cdot \mathbb{E}[T^\pi] - \sum_{i=1}^n \frac{1}{1-q_i}.
\end{equation}
\noindent In other words, \cref{eq::MinDefuse} is equivalent to minimizing $\mathbb{E}[W_{T^\pi}^\pi]$, the expected total number of wasted assignments made by $\pi$ throughout the process of neutralizing all threats.
Any policy $\pi^*$ for which $W^{\pi^*}_\tau \preceq W^\pi_\tau$ for all $\pi \in \Pi$ and $\tau$ will clearly solve \eqref{eq::EffectiveObjective} and \cref{eq::MinDefuse} simultaneously. In fact, as we will see in \Cref{sec:fair_alloc}, there exists a parameter regime in which a time-oblivious policy will simultaneously solve  \eqref{eq::EffectiveObjective} and \eqref{eq::SurvivalObjective} for all $\tau$ as well as \eqref{eq::MinDefuse}.

\subsection{On the Structure of Optimal Policies}\label{sec:optimality}
As stated previously, each of the problems \eqref{eq::MinDefuse}, \eqref{eq::SurvivalObjective}, and \eqref{eq::EffectiveObjective} admits a Bellman equation formulation consistent with the MDP framework in \Cref{sec:: MDP_Outline}. These equations can be derived explicitly, and we defer their presentation to Online Appendix \ref{sec:bellman_optimality}, wherein we also show that efficient computation of all problems is possible when $n=2$. They provide a computational approach for finding optimal policies, which we use as benchmarks in the numerical experiments in \Cref{sec:numerical_experiment}. 
In the meantime, we highlight a property that can arise in optimal policies for the deadline-parametrized problems \eqref{eq::SurvivalObjective} and \eqref{eq::EffectiveObjective}, and that is, they may critically rely on $\tau.$ This \emph{time-awareness} is illustrated in the following result.  
\begin{theorem}[Structure of optimal policies when $n=2$] \label{theorem:properties_n=2}
    Let there be $n=2$ threats and $q_1> q_2>0$. When $C\geq 2$, for any positive integer $\tau$, \cref{eq::SurvivalObjective} and \cref{eq::EffectiveObjective} admit an optimal policy $\pi^*$, for which its first round allocation $(x_1^*,x_2^*)$ follows:
    \begin{enumerate}
        \item $(x_1^*,x_2^*)=(C-1,1)$ for large enough $\tau$. Moreover, $(x_1^*,x_2^*)\neq (C,0)$ for any $\tau$;
        \item $x_1^*$ is non-decreasing in $\tau$ and $x_2^*$ is non-increasing in $\tau$;
        \item $x_1^*\geq x_2^*$ for any $\tau$.
    \end{enumerate}
\end{theorem}

\Cref{theorem:properties_n=2} shows that, under both \cref{eq::SurvivalObjective} and \cref{eq::EffectiveObjective}, there exists an optimal policy whose first-round allocation becomes increasingly concentrated on the harder threat as the deadline $\tau$ increases, eventually taking the form $(C-1,1)$. This result is nontrivial because its proof requires a careful comparison of the multi-round outcomes induced by different policies. The argument hinges on a key lemma, \Cref{lemma:switch_dominate} in Online Appendix~\ref{sec:proof}. Roughly speaking, this lemma establishes a monotonicity property of $\pi^*$ for any $\tau$: if the first-round allocation fails to neutralize any threat, then the number of resources assigned to the harder threat, threat~1, in the next round is no larger than the number assigned to it in the first round. %

For general $n$, one might expect optimal policies to always prioritize harder threats by assigning them more effectors than easier threats in each round. Although this intuition is often correct, it does not hold universally. Indeed, under both \cref{eq::SurvivalObjective} and \cref{eq::EffectiveObjective}, allocating effectors across several moderately hard threats can sometimes outperform concentrating effectors on the single hardest threat, as the following example illustrates.

\begin{example}
      Let $n=4$, $C=3$ and $(q_1,q_2,q_3,q_4)=(0.8,0.4,0.4,0.4)$. Then the unique optimal allocation in the first round under (\textsf{SLM}$_2$), (\textsf{EAM}$_2$), and \cref{eq::MinDefuse} is $(x_1^*,x_2^*,x_3^*,x_4^*)=(0,1,1,1)$.
\end{example}

Although $\pi^*$ can in principle be derived via dynamic programming, the computation quickly becomes intractable as $n$ and $C$ grow. Both the state space and the action space scale rapidly. The state space consists of all possible subsets of active threats, $\mathcal{M} \subseteq [n]$, and therefore has size $2^n$. The action space consists of all integer allocations satisfying $\sum_{i \in [n]} x_i = C$, whose cardinality is $\binom{C+n-1}{n-1}$. This quantity is polynomial in $C$ for fixed $n$ and polynomial in $n$ for fixed $C$, but grows exponentially when $n$ and $C$ grow jointly. For both \cref{eq::SurvivalObjective} and \cref{eq::EffectiveObjective}, solving the Bellman equation has the same growth rate in $n$ and $C$ as under \cref{eq::MinDefuse}, with an additional linear factor in $\tau$. This motivates our development of simple time-oblivious policies in the following sections.

\section{Fair Allocation Policy}\label{sec:fair_alloc}
The first policy we will study is not only simple to describe but also requires minimal state information. Put simply, it distributes the $C$ effectors of each round as evenly/fairly as possible across the active threats. This policy requires only the set of live threats, $\mathcal{M}_t$, for its decision at time $t.$ It operates without knowledge of $\tau$, nor even the collection of difficulties $\{q_i\}_{i \in \mathcal{M}_t}$. We formalize fair allocation as follows.

\begin{definition}[Fair Allocation]
Let $\emptyset \neq \mathcal{M} \subseteq [n]$ be a nonempty set of threats. An allocation $x \in \mathbbm{Z}_{\geq 0}^{n}$ is called \textit{fair} for $\mathcal{M}$ if $\sum_{i \in \mathcal{M}} x_i = C$, $x_i = 0$ for $i \notin \mathcal{M}$, and $|x_i- x_j|\in\{0,1\}$ for all $i,j \in \mathcal{M}$.
\end{definition}
A set $\mathcal{M}$ may admit multiple fair allocations, in which case tie-breaking rules can be implemented. We define the class of \emph{fair allocation policies} via
\[
\Pi_{\textsf{FA}}\coloneqq \left\{\pi \in \Pi: \pi_t(\mathcal{M}) \text{ is a fair allocation for $\mathcal{M}$,  ~$\forall \mathcal{M}\in\mathcal{S}$ with $\mathcal{M}\neq\emptyset$},~\forall t\in\mathbb{Z}_{\ge1}\right\}.
\]
In words, a fair allocation policy is one that, at all times $t$, distributes all $C$ effectors as evenly as possible across the active threats $\mathcal{M}_t.$  
Throughout, the notation $\pi_{\textsf{FA}}$ will denote an arbitrary member of $\Pi_{\textsf{FA}}$, and any statement in which $\pi_{\textsf{FA}}$ is referenced as if it were a particular policy, shall be disambiguated by understanding it as a statement holding for all members of the class $\Pi_{\textsf{FA}}$ (unless otherwise noted).\footnote {Alternatively, we can define $\pi_{\textsf{FA}}$ to be the unique random policy that, at each time $t,$ selects a fair allocation for the set $\mathcal{M}_t$ uniformly at random. See \cref{footnote: random policy} for the formalization of random policies.}

Although $\pi_{\textsf{FA}}$ utilizes minimal information (even ignoring $q$), under a particular parameter regime, it can in fact be uniformly optimal for \eqref{eq::MinDefuse}, \eqref{eq::SurvivalObjective}, and \eqref{eq::EffectiveObjective} for all $\tau$, as we will discuss next.   

\subsection{Homogeneous Threats or $C\leq 2$: Fair Allocation Optimality}
In this section, we identify and examine a class of instances for which $\pi_{\textsf{FA}}$ uniformly optimizes \cref{eq::MinDefuse}, \cref{eq::SurvivalObjective}, and \cref{eq::EffectiveObjective}. This class is comprised of all salvo instances as well as instances in which $C\leq 2$. In the former, all active threats are statistically indistinguishable, so there is no benefit to favoring one active threat over another; this is precisely the allocation principle implemented by $\pi_{\textsf{FA}}$, and it explains the policy's optimality.
In the latter, with $C\leq 2$ effectors in each round, assigning them to two different active threats (i.e., a fair allocation), for as long as $M_t \geq 2$, guarantees waste can only be incurred by $\pi_{\textsf{FA}}$ in the last round that an active threat exists (and is neutralized). These two insights are formalized in the following theorem.

\begin{theorem}[Optimality of fair allocation]\label{theorem:fair_is_optimal_when_qi=q}
If $(n, C, q, \tau)$ satisfies: (Salvo) $q_1=q_2=\cdots=q_n$; or (Two-unit capacity) $C\leq 2$,
then any $\pi_{\textsf{FA}} \in \Pi_{\textsf{FA}}$ solves \cref{eq::MinDefuse}, \cref{eq::SurvivalObjective}, and \cref{eq::EffectiveObjective}.  
\end{theorem}
That fair allocations solve all three dynamic programs in the symmetric case of a salvo is not unlike how the majorization-minimal solution optimizes Schur convex functions \citep{marshall1979inequalities}; indeed, our proof similarly proceeds by examining the benefit of a Pigou--Dalton, or ``Robin Hood,'' transfer. 

\subsection{Heterogeneous Threats and $C\geq 3$ : Fair Allocation Suboptimality}

When threat difficulties are heterogeneous and $C \geq 3$, $\pi_{\textsf{FA}}$ may no longer be optimal for one or more of the three problems. In this section, we discuss how suboptimality can in fact be found even in instances with $C\leq n$. Indeed, for such instances, a fair allocation made in the first round (also greedily optimal as discussed in \Cref{sec:greedy_policies}), which albeit incurs no waste ($W^{\pi_{\textsf{FA}}}_1 = 0$), can still be a misstep for \emph{all} three objectives. We demonstrate with the following counterexamples.

\begin{example}[Counterexample for \cref{eq::SurvivalObjective} and \cref{eq::EffectiveObjective}]\label{ex:fair_poor1}
Consider $n=3$, $C=3$, $\tau=3$, and $(q_1,q_2,q_3)=(0.9,0.8,0.1)$. 
Fair allocation at $t=1$ is $(1,1,1)$. Solving the Bellman recursion for either (\textsf{SLM}$_3$) or (\textsf{EAM}$_3$) shows that the unique optimal initial allocation is $(x_1^*,x_2^*,x_3^*)=(2,1,0)$, which is not fair. %
\end{example}

It is perhaps particularly surprising in the above example that $(x_1^*,x_2^*,x_3^*)=(1,1,1)$ is not optimal for (\textsf{EAM}$_3$). After all, in the first period, this allocation guarantees three effective shots, whereas under $(x_1^*,x_2^*,x_3^*)=(2,1,0)$, the second shot assigned to threat $1$ may be wasted if the first shot already neutralizes it. Nevertheless, this sacrifice in immediate reward is outweighed by the gain in future rewards. %

\begin{example}[Counterexample for \cref{eq::MinDefuse}]\label{ex:fair_poor2}
For \cref{eq::MinDefuse}, heterogeneous threats can also make fairness suboptimal. 
For example, consider $n=4$, $C=4$, and $(q_1,q_2,q_3,q_4)=(0.99,0.9,0.3,0.1)$. 
Fair allocation at $t=1$ is $(1,1,1,1)$. The unique optimal allocation in the first round is $(x_1^*,x_2^*,x_3^*,x_4^*)=(2,1,1,0)$, again not fair. 
\end{example}
These counterexamples highlight the importance of taking into account the difficulties $q$, wherein more difficult threats can sometimes require prioritization.

\subsection{Performance Analysis of Fair Allocation %
}\label{sec:fair_policy}

Having examined instances in which fair allocation $\pi_{\textsf{FA}}$ is optimal as well as instances in which it is suboptimal, we now turn to a general characterization of its performance. 
Specifically, we characterize its performance in terms of $n$, $C$, and $q$ under each objective. 
We also identify and explain worst-case scenarios for $\pi_{\textsf{FA}}$'s performance.

\smallskip

\textbf{Fair allocation for \cref{eq::MinDefuse}.} 
For any realization $L \in \mathbb{Z}_{\geq 1}^n$, we denote by $T^{\pi}(L)$ the defusing time incurred by policy $\pi$. 
Clearly, if a decision maker were to have knowledge of $L,$ i.e., were omniscient, then the defusing time would be $\lceil (\sum_{i=1}^n L_i)/C\rceil$. 
The following theorem characterizes both $\pi_{\textsf{FA}}$'s performance against this omniscient benchmark and any optimal policy.

\begin{theorem}[Performance of $\pi_{\textsf{FA}}$ for \eqref{eq::MinDefuse}]
\label{theorem:DTM_fair_ratio}
    For any $C, n\in\mathbb{Z}_{\geq 1}$, $\pi_{\textsf{FA}}$ has a competitive ratio 
    \begin{align}\label{eq:DTM_competitive}
    \sup_{L\in \mathbb{Z}_{\geq 1}^n}\frac{T^{\pi_\textsf{FA}}(L)}{\lceil \frac{1}{C}\sum_{i=1}^n L_i
    \rceil }\leq 1+\frac{C\ln (n \wedge C) + 2C}{n}, %
    \end{align} and a performance ratio for any $q\in[0,1)^n$: 
    \begin{align}\label{eq:DTM_ratio1}
   \frac{\mathbb{E}[T^{\pi_{\textsf{FA}}}]}{\mathbb{E}[T^*]}\leq 1+ \frac{C\ln(n \wedge C)+2C}{\sum_{i=1}^n \frac{1}{1-q_i}}.
    \end{align}
    Moreover, %
    if $0 < \underline{q} < \bar{q} < 1$, then for any $C, n\in\mathbb{Z}_{\geq 1}$, $\pi_{\textsf{FA}}$ admits a performance ratio bound: %
    \begin{align}\label{eq:DTM_ratio2}
    \sup_{q \in [\underline{q}, \bar{q}]^n}\frac{\mathbb{E}[T^{\pi_{\textsf{FA}}}]}{\mathbb{E}[T^*]}\leq \left\lceil\frac{\ln \underline{q}}{\ln \bar{q}}\right\rceil.
    \end{align}
\end{theorem}

The performance-ratio bound \eqref{eq:DTM_ratio1} in Theorem~\ref{theorem:DTM_fair_ratio} follows by taking expectations over $L$ in the derivation of the competitive-ratio bound \eqref{eq:DTM_competitive}. While this is not a constant bound, it remains useful in low-capacity regimes. For any fixed $C$, the bound converges to $1$ as $n\to\infty$. This is because when $n$ is large, $\pi_{\textsf{FA}}$ can fully utilize the available capacity and incurs no waste in every round until the system is close to completion, that is, until $M_t<C$. As $n\to\infty$, the cumulative waste under rounds with $M_t<C$ is negligible relative to the total defusing rounds. Therefore, fair allocation is asymptotically optimal. The second bound is a constant independent of $C$ and $n$, so it is qualitatively stronger than the first one. Intuitively, bounding $q_i$ away from 0 and 1 prevents extreme heterogeneity, so the loss from equal splitting cannot accumulate without limit.

We now move from relative ratios to absolute scaling and ask a complementary question: as both $n$ and $C$ increase, how does the defusing time under $\pi_{\textsf{FA}}$ (and $\pi^\ast$) scale?  The following theorem characterizes this, where we note that by \eqref{eq:: EAM_DTM Connection}, the expected defusing time of any policy is lower bounded by $\frac{1}{C}\sum_{i=1}^n\mathbb E[L_i]=\frac{1}{C} \sum_{i=1}^n\frac{1}{(1-q_i)}$.

\begin{theorem}[Growth rate of \texorpdfstring{$\mathbb{E}[T^{\pi_{\textsf{FA}}}]$}{E[T^{pi_FA}]}]
\label{thm:spread-upper}
Let $C\coloneqq C(n)$ and $q\coloneqq q^{(n)} \in [0,1)^n$, and consider the asymptotic regime in which $n\to\infty$.
Suppose there exist constants $0 < \underline{q} < \bar{q} < 1,$ independent of $n,$ such that  $ q^{(n)} \in [\underline{q}, \bar{q}]^n$ for all $n.$
Then the resource tightness $C(n)/n$ determines the rate of growth for $\mathbb{E}[T^{\pi_{\textsf{FA}}}]$:
\begin{enumerate}
    \item If $C(n)= n$, then $\mathbb{E}[T^{\pi_{\textsf{FA}}}]
        =
        \Theta(\ln^* n).$
    \item If $C(n)>n$, then $\mathbb{E}[T^{\pi_{\textsf{FA}}}]
        =
        O(\ln^* n).$
    \item If $C(n)<n$, then $\mathbb{E}[T^{\pi_{\textsf{FA}}}]
        =
        \frac{1}{C(n)}\sum_{i=1}^n \frac{1}{1-q_i^{(n)}}
        +
        O(\ln^* C(n)).$
\end{enumerate}
Here, $\ln^*(m)
    \coloneqq
    \min\{k\in\mathbb{Z}_{\ge 0}: \ln^{(k)}(m)\le 1\},$ for $m\in\mathbb{Z}_{\ge 1}$
denotes the iterated logarithm.
\end{theorem}

We note that the lower bound $\frac{1}{C(n)}\sum_{i=1}^n \frac{1}{1-q_i}=O(1)$ when $C(n)\ge n$, and hence is absorbed into the stated $O(\ln^* n)$ and $\Theta(\ln^* n)$ bounds. In the balanced case $C(n)=n$, the same order also holds for the optimal policy: by \Cref{theorem:DTM_fair_ratio}, $\mathbb E[T^{\pi_{\textsf{FA}}}]\le K\mathbb E[T^*]$ for a constant $K$ depending only on $\underline q$ and $\bar q$, while optimality gives $\mathbb E[T^*]\le \mathbb E[T^{\pi_{\textsf{FA}}}]$. Thus item 1 implies $\mathbb E[T^*]=\Theta(\ln^* n)$ when $C(n)=n$. When $C(n)<n$, no waste is incurred until the number of active threats falls below $C$; consequently, the expected number of rounds needed to reach $M_t<C(n)$ is at most $\frac{1}{C(n)}\sum_{i=1}^n \frac{1}{1-q_i^{(n)}}$, at which point the expected defusing time of the remaining active threats is $\Theta(\ln^* C(n))$. Since $\ln^*(\cdot)$ grows extremely slowly, these bounds suggest that fair allocation remains highly scalable and practically effective even in large systems.

Theorem~\ref{thm:spread-upper} shows that fair allocation and the optimal policy share the same iterated-logarithmic growth rate in the balanced regime, where $C=n$. An order comparison can be, however, coarse at this scale: two quantities of order $\ln^* n$ may differ by an amount comparable to the quantities themselves, and in an engagement even a single additional round of exposure is operationally significant. Our next result asserts that no such difference arises. Under at-least-proportional capacity, the expected defusing time under fair allocation exceeds the optimal value by at most an additive constant, uniformly in the number of threats $n$.

\begin{theorem}[Bounded additive DTM gap under at-least-proportional capacity]
\label{thm:dtm_fair_additive_gap}
Let $\pi_{\textsf{FA}}$ be a fair allocation policy incorporating uniformly-at-random tie-breaking\footnote{The proof provided in the e-companion Section \ref{sec:AppendixProofAdditiveGap} shows that any fixed priority tie-breaking rule  also suffices.}.
Let $0<\underline q<\bar q<1$, $\underline{\alpha} > 0$, and let $(n, C(n), q^{(n)})$ denote a sequence of instances such that: \textsf{(1)} $q^{(n)}=(q_1^{(n)},\ldots,q_n^{(n)})\in[\underline q,\bar q]^n$; \textsf{(2)} 
$C(n)\geq \underline{\alpha} n$ for sufficiently large $n$.
Then there exists a constant  $K=K(\underline{\alpha},\underline q,\bar q)<\infty$, such that for all sufficiently large $n,$
\[
0\leq \mathbb E\left[T^{\pi_{\textsf{FA}}}\left(q^{(n)}; C(n)\right)\right] - \mathbb E\left[T^*\left(q^{(n)}; C(n)\right)\right] \leq K.
\]
\end{theorem}

\smallskip

\textbf{Fair allocation for (\textsf{SLM}$_1$).} We now evaluate the performance of $\pi_{\textsf{FA}}$ under \cref{eq::SurvivalObjective}. Fixing all other parameters, it is intuitive that fair allocation should perform well when $\tau$ is sufficiently large, since the survival likelihood approaches one as $\tau\to\infty$. This intuition is also supported by the constant-factor performance ratio in \eqref{eq:DTM_ratio2}, together with the relationship in \eqref{eq:dtm_slm}, which shows that the expected defusing time is determined by the survival likelihood across all possible deadlines. Our next result therefore focuses on the case $\tau=1$ and establishes a worst-case performance guarantee for fair allocation relative to the optimal policy for $(\textsf{SLM}_1)$. Again, we focus on the case $C \ge n$, since otherwise the survival likelihood is identically zero.

\begin{theorem}[Performance of $\pi_{\textsf{FA}}$ for (\textsf{SLM}$_1$)]
\label{thm:tau1_ratio_cfree}
Let $\tau=1$, $C\geq n$, and $\pi_{\textsf{FA}} \in \Pi_{\textsf{FA}}.$ Then
$$
\frac{V_{\textsf{SLM}_1}^{\pi_\textsf{FA}}}{V_{\textsf{SLM}_1}^*}
\;\ge\;\max\left(\exp \left(-\frac{C}{e\lfloor C / n\rfloor}\right),\left(\frac{\lfloor C/n\rfloor}{\,C-n+1\,}\right)^n\right),  \quad\quad\quad \forall q \in [0,1)^n.
$$
In particular, the above bound implies $V_{\textsf{SLM}_1}^{\pi_\textsf{FA}}/V_{\textsf{SLM}_1}^* = 1$\; for $n=1$ or $C=n$. More generally, when $C/n\in  \mathbb{Z}$, we have $V_{\textsf{SLM}_1}^{\pi_\textsf{FA}}/V_{\textsf{SLM}_1}^*\geq \exp(-n/e)$; further, this bound is asymptotically tight as $n\to\infty$. %

\end{theorem}

When $C/n\notin\mathbb{Z}$, fair allocation cannot split the effectors evenly across threats, and the remaining $C-n\lfloor C/n\rfloor$ surplus effectors must be assigned according to tie-breaking rules. To obtain a performance lower bound that is independent of the particular tie-breaking rule, our analysis conservatively ignores the contribution of these surplus assignments to the overall survival likelihood; this explains why the lower bound depends on $\lfloor C/n\rfloor$. The bound in \Cref{thm:tau1_ratio_cfree} reveals two regimes governing the single-round performance of fair allocation. The first term inside the maximum decays exponentially in $n$, namely as $\exp(-\Theta(n))$. The second term can decay faster, on the order of $\exp(-\Theta(n\ln n))$, for example, when $C(n) - n = \Omega(n)$; 
on the other hand, when $C$ is close to $n$ or $n$ is small, it can be the dominating lower bound, and even assume the value 1 when $C = n$, or $n = 1.$

To see the asymptotic tightness, %
consider we have $m$ hard threats with $q_H=e^{-\epsilon}$ and $n-m$ easy threats with $q_E=\epsilon>0$. Here, $m\in\{1,2,\dots,n-1\}$ is an arbitrary integer less than $n$. We choose $\epsilon \coloneqq \epsilon(n)$ and $C \coloneqq C(n)$ such that $\epsilon C\rightarrow0$ and $C/n\rightarrow\infty$ as $n\rightarrow\infty$ (e.g., $\epsilon=1/n^4$ and $C=n^2$). When $\epsilon\rightarrow 0$, the probability that a hard threat is neutralized under $x$ effectors follows $1-q_H^x = 1-e^{-\epsilon x}=\epsilon x+O(\epsilon^2)$. Meanwhile, with high probability, each easy threat can be neutralized by a single effector. Thus, the survival probability is asymptotically determined by the number of effectors allocated to the hard threats, 
given that each easy threat receives at least one effector. On one hand, $\pi^*$ allocates exactly one effector to every easy threat and allocates the remaining $C-(n-m)$ effectors as evenly as possible across the hard threats. Thus, $V_{\textsf{SLM}_1}^*/\epsilon^m $ approaches $((C-(n-m)) / m)^m\sim(C / m)^m$ as $n\rightarrow\infty$. On the other hand, $\pi_{\textsf{FA}}$ allocates at most $\lceil C / n \rceil$ units to each threat, so $V_{\textsf{SLM}_1}^{\pi_{\textsf{FA}}}/\epsilon ^m$ approaches $(C / n)^m$ as $n\rightarrow\infty$. Consequently, the ratio $V_{\textsf{SLM}_1}^{\pi_{\textsf{FA}}} / V_{\textsf{SLM}_1}^*$ approaches $(m / n)^m$ as $n\rightarrow\infty$. Since $m$ can be chosen arbitrarily, optimizing over $m$ gives $\inf_{m\in \mathbb{N}}(m / n)^m=\exp (-(1-o(1)) n / e)$ as $n \rightarrow \infty$. This leads to the worst-case exponential decay rate in \Cref{thm:tau1_ratio_cfree}.

\smallskip

\textbf{Fair allocation for \cref{eq::EffectiveObjective}.}
We now turn to \cref{eq::EffectiveObjective}. Unlike the all-or-nothing and coupled structure of \cref{eq::SurvivalObjective}, \cref{eq::EffectiveObjective} admits both a round-by-round and a threat-wise decomposition, since its objective can be expressed as the total number of effective assignments accumulated across all rounds and all threats. This separable structure allows us to extend the performance analysis to arbitrary $\tau$. For any realization $L$, let $V_{\textsf{EAM}_\tau}^{\pi}(L)$ denote the total number of effective assignments with deadline $\tau$ under policy $\pi$. Since at most $C\tau$ assignments can be made over the horizon and each threat $i$ can contribute at most $L_i$ effective assignments, the maximum achievable value under an omniscient benchmark with access to the realization of $L$ is $\min\left\{C\tau,\sum_{i=1}^n L_i\right\}$. Similar to \Cref{theorem:DTM_fair_ratio}, the following theorem characterizes both $\pi_{\textsf{FA}}$'s performance against this omniscient benchmark and also any optimal policy.

\begin{theorem}[Performance of Fair for \eqref{eq::EffectiveObjective}]
\label{thm:fair_alloc_min_waste}
    For any $\tau, n\in\mathbb{Z}_{\ge1}$, fair allocation $\pi_{\textsf{FA}}$ has a competitive ratio $$%
    \inf_{L\in \mathbb{Z}_{\geq 1}^n, C\in\mathbb{Z}_{\geq 1}}\frac{V_{\textsf{EAM}_\tau}^{\pi_\textsf{FA}}(L)}{\min\{C\tau ,\sum_{i=1}^n L_i\} }\geq n^{-1/\tau}.$$ %
    In particular, when $\tau=1$, %
    $V_{\textsf{EAM}_1}^{\pi_\textsf{FA}}/V_{\textsf{EAM}_1}^{*}\geq 1/n$ for any $C \in \mathbbm{Z}_{\geq 1},q \in [0,1)^n.$ 
    Moreover, for any fixed $\tau, n\in\mathbb{Z}_{\ge1}$ such that $n^{1/\tau}\in \mathbb{Z}_{\ge1}$, these bounds are asymptotically tight as $C/n \to_{\mathbb{Z}}
    \infty$.
\end{theorem}

Theorem~\ref{thm:fair_alloc_min_waste} shows that the worst-case performance ratio of $\pi_{\textsf{FA}}$ under \cref{eq::EffectiveObjective} has an exponential decay in $1/\tau$. For any fixed $n$, the ratio approaches $1$ and thus $\pi_{\textsf{FA}}$ is asymptotically optimal as $\tau$ increases. 
The worst‑case construction for the $n^{-1/\tau}$ ratio uses a layered hierarchy of $L_i$. Under such construction, $\pi_{\textsf{FA}}$ neutralizes a constant fraction of alive threats in each round, but for each neutralized threat, only one effector is effectively assigned and the rest are wasted. As a consequence, $\pi_{\textsf{FA}}$ repeatedly wastes effectors on threats that are about to be neutralized. At the same time, we construct one ``impossible'' threat so that the omniscient policy can put all effectors to it and make every assignment effective. A detailed construction can be found in Online Appendix \ref{sec:proof}.

\begin{remark}
Comparing \Cref{thm:tau1_ratio_cfree} and \Cref{thm:fair_alloc_min_waste} shows that fair allocation can be more fragile under \cref{eq::SurvivalObjective} than under \cref{eq::EffectiveObjective}. When $\tau=1$, \Cref{thm:fair_alloc_min_waste} guarantees
$V_{\textsf{EAM}_1}^{\pi_{\textsf{FA}}}/V_{\textsf{EAM}_1}^{*}\geq 1/n$. In contrast, when $C/n\in\mathbb{Z}$, \Cref{thm:tau1_ratio_cfree} gives worst-case instances with
$V_{\textsf{SLM}_1}^{\pi_{\textsf{FA}}}/V_{\textsf{SLM}_1}^{*}=\exp(-(1-o(1))n/e)$, exponentially smaller in $n$. The reason is that the all-or-nothing nature of SLM amplifies misallocation whereas EAM still credits effective assignments on threats that are served. Numerical comparisons corroborating this observation will be provided in \Cref{sec:numerical_experiment}. %
\end{remark}

\section{Greedy Policies}\label{sec:greedy_policies}
In \Cref{sec:fair_alloc}, %
we show that fair allocation can underperform when threats are highly heterogeneous in difficulty. This is partly because $\pi_{\textsf{FA}}$ %
does not take into account the difficulty $q = (q_1,\ldots,q_n)$. One could anticipate improved performance by a time-oblivious policy that actually does consider $q$. In this section, we present two such policies: $\pi_{\textsf{SL}}$ and $\pi_{\textsf{EA}}$. 

We consider \emph{greedy} policies, which at each round execute any allocation that makes the largest immediate progress for an objective, without regard for later rounds. More formally, in each round $t,$ the policy effectively re-instantiates the parameters $(n,C,q,\tau)$ by setting $n = |\mathcal{M}_t|$, $q = (q_i)_{i \in \mathcal{M}_t}$, and $\tau = 1$. The policies $\pi_{\textsf{SL}}$ and $\pi_{\textsf{EA}}$ are fashioned in this way to repeatedly solve the resulting instances of $(\textsf{SLM}_1)$ and $(\textsf{EAM}_1)$, respectively, at each round $t.$ As we will see, both problems $(\textsf{SLM}_1)$ and $(\textsf{EAM}_1)$ are efficiently solvable, making both policies' decisions at each round efficiently implementable. In fact, their respective 
objective functions $V^\pi_{\textsf{SLM}_1}$ and $V^\pi_{\textsf{EAM}_1}$ both amount to $M^\natural$- concave functions, which can be efficiently maximized via the discrete steepest ascent method \citep{MurotaDiscreteConvex}. The following lemma is an adaptation and consequence of this insight for our context. 
\begin{lemma}[\cite{MurotaDiscreteConvex}, Theorem 6.24]
\label{lem:separable_concave_greedy}
Let $f_1, \ldots, f_n$ be a collection of functions in which for every $i,$
$f_i:\mathbb{Z}_{\ge 0}\to\mathbb{R}$ is non-decreasing and submodular, i.e., $\Delta_i(x_i) \coloneqq f_i(x_i+1)-f_i(x_i)$ is non-negative and non-increasing in $x\in\mathbb{Z}_{\ge 0}$. Then a vector 
\[
\bar{x} \in \argmax_{x\in\mathbb{Z}_{\ge 0}^n:\ \sum_{i=1}^n x_i=C}\ \sum_{i=1}^n f_i(x_i)
\]
can be constructed via the following steepest ascent method: first initialize $\bar{x}_i=0$ for all $i\in[n]$; then iteratively assign the $C$ effectors, wherein for each iteration we increment $\bar{x}_{i^*}$ by 1 for some $i^*\in\argmax_{i\in[n]} \Delta_i(\bar{x}_i)$.
\end{lemma}

For both \cref{eq::SurvivalObjective} and \cref{eq::EffectiveObjective}, greedy policies can be suboptimal when $\tau\geq 2$. As $\tau$ increases, the optimal policy $\pi^*$ tends to allocate more and more effectors to hard threats in early rounds. This is reflected in property 2 of \Cref{theorem:properties_n=2}, where the optimal first-round allocation for the harder threat increases in $\tau$ when $n=2$. By contrast, the allocation under both greedy policies does not depend on $\tau$. %

We now proceed to apply \Cref{lem:separable_concave_greedy} towards the study of $\textsf{SLM}_1$ and $\textsf{EAM}_1$, which will reveal the explicit forms of $\pi_{\textsf{SL}}$ and $\pi_{\textsf{EA}}$.

\subsection{Greedy Policy $\pi_{\textsf{SL}}$ for Survival Likelihood Maximization}
Under the parameter regime $C \geq n$, any instance of problem (\textsf{SLM}$_1$) seeks to find an allocation from among the set
\[
\argmax_{x\in \mathbb{Z}_{\ge 0}^n:\ \sum_i x_i=C}\ \prod_{i=1}^n \bigl(1-q_i^{x_i}\bigr) = \mathbbm{1}_{[n]} + \argmax_{x\in \mathbb{Z}_{\ge 0}^n:\ \sum_i x_i= C - n}\ \sum_{i=1}^n \ln\bigl(1-q_i^{x_i+1}\bigr),
\]
where $\mathbbm{1}_{[n]}$ denotes the $n$-vector of all ones. 
Noting that for all $i,$ the summand $\ln\bigl(1-q_i^{x_i+1}\bigr)$ is non-decreasing and 
\[
\Delta_{\textsf{SL},i}(x_i) := \ln\bigl(1-q_i^{x_i+1}\bigr)-\ln\bigl(1-q_i^{x_i}\bigr)
= \ln\!\left(1+\frac{q_i^{x_i}(1-q_i)}{1-q_i^{x_i}}\right) = \ln\!\left(1+\frac{1}{\sum_{k=1}^{x_i}q_i^{-k}}\right)
\]
is non-increasing, the steepest ascent method of \Cref{lem:separable_concave_greedy} solves such instances to optimality. Combined with the fact that cases of $C < n$ are solved trivially (by any $x \in \mathcal{A}$: $\sum_i x_i = C$), the greedy policy $\pi_{\textsf{SL}}$ now follows in a straightforward manner along the lines discussed previously, and formalized in \Cref{alg::ML_Estim}.

\begin{algorithm}
\small
\caption{Greedy Survival Likelihood Maximization $\pi_{\textsf{SL}}$}
\begin{algorithmic}
\Require Set of active threats $\mathcal{M}_t$, number of effectors $C \in \mathbbm{Z}_{\geq n}$.\\
\hspace{-0.38cm}\textbf{Output:} Allocation policy $x^t\in\mathbbm{Z}_{\ge 0}^{n}$ with $\sum_{i\in\mathcal{M}_t} x_i^t = C$ in round $t$.
\State Initialize $x^t_i\leftarrow 0$ and $s_i\leftarrow 0$ for $i\in [n]$. 
\State \textbf{if} $\mathcal{M}_t = \varnothing$ \textbf{then return} $x^t$.
\For{$k = 1,\ldots, C$}
    \State $i^* \gets \argmin_{i\in \mathcal{M}_t}\ s_i$ \Comment{tie-breaking rule}.
    \State $x^t_{i^*} \gets x^t_{i^*} + 1$,\;\;\;\;\; $s_{i^*} \gets s_{i^*} + q_{i^*}^{-x^t_{i^*}}$.
\EndFor
\State \Return allocation $x^t$ in round $t$.
\end{algorithmic}
\label{alg::ML_Estim}
\end{algorithm}
As Algorithm \ref{alg::ML_Estim} admits degrees of freedom in its execution by way of different tie-breaking rules (e.g., break ties between threats in favor of the one with harder difficulty), we define the class of \emph{greedy survival likelihood policies} via
\[
\Pi_{\textsf{SL}}\coloneqq \left\{\pi \in \Pi: \pi_t(\mathcal{M}) \text{ is an output of Algorithm \ref{alg::ML_Estim},~$\forall \mathcal{M}\in\mathcal{S}$ with $\mathcal{M}\neq\emptyset$},~\forall t\in\mathbb{Z}_{\ge1}\right\}.
\]

Throughout, the notation $\pi_{\textsf{SL}}$ will denote an arbitrary member of $\Pi_{\textsf{SL}}$, and any statement in which $\pi_{\textsf{SL}}$ is referenced as if it were a particular policy, shall be disambiguated by understanding it as a statement holding for all members of the class $\Pi_{\textsf{SL}}$ (unless otherwise noted).

A natural question is whether there exists a constant worst-case performance guarantee for $\pi_{\textsf{SL}}$ under \cref{eq::SurvivalObjective} for all $\tau$. Unfortunately, the following remark shows that no such constant factor guarantee exists for $\pi_{\textsf{SL}}$ even when $\tau = 2$. 

\begin{remark}\label{example:ml_greedy_no_const_ratio}
    There exists a sequence of instances indexed by $m \in \mathbbm{Z}_{\geq 1}$, i.e., 
    $(n(m), C(m), q^{(m)})$ such that 
    $ V_{\textsf{SLM}_2}^{\pi_{\textsf{SL}}}/V_{\textsf{SLM}_2}^*
    \;\longrightarrow\; 0     
    $ as $m\to\infty$.
\end{remark}

In the worst-case sequence of instances, there is one group of easy threats and another group of extremely hard threats. Greedy \(\pi_{\textsf{SL}}\) allocates two effectors to each easy threat and splits the remaining effectors evenly among the hard threats in the first round. By contrast, the benchmark policy allocates only one effector to each easy threat and devotes more early effort to the hard threats. Although \(\pi_{\textsf{SL}}\) achieves a higher first-round survival probability, this allocation becomes inefficient in the second round: since the hard threats are extremely difficult to neutralize, meeting the deadline requires concentrating more effort on them early. As a result, \(\pi_{\textsf{SL}}\) becomes arbitrarily worse than the designed policy as the number of (both easy and hard) threats grows.

Although $\pi_{\textsf{SL}}$ has no constant-factor guarantee against the optimal policy, it still performs well for small $n$. %
The following theorem establishes that when there are two threats, greedy $\pi_\textsf{SL}$ provably outperforms fair allocation $\pi_\textsf{FA}$ in survival likelihood maximization as well as defusing time minimization. 

\begin{theorem}[Greedy policy $\pi_{\textsf{SL}}$ dominates fair allocation $\pi_\textsf{FA}$ when $n=2$]\label{thm:sl_vs_fair_n2}
{\color{red}  } When $n=2$, for all $\tau\geq 1$, we have $V_{\textsf{SLM}_\tau}^{\pi_\textsf{SL}}\;\ge\;V_{\textsf{SLM}_\tau}^{\pi_\textsf{FA}}$ and $\mathbb{E}[T^{\pi_{\textsf{SL}}}]\;\le\;\mathbb{E}[T^{\pi_{\textsf{FA}}}]$.\footnote{For this theorem, we require $\pi_{\textsf{FA}}$ and $\pi_{\textsf{SL}}$ assign the first $\min\{C,|\mathcal M_t|\}$ effectors to (distinct) active threats in each round $t$ following the same index order for threats $i\in \mathcal{M}_t$.}
\end{theorem}

The dominance of $\pi_{\textsf{SL}}$ can be intuitively explained by the structural characterization of an optimal policy of \cref{eq::SurvivalObjective} in Theorem~\ref{theorem:properties_n=2}. When $n=2$, $\pi_{\textsf{SL}}$ allocates more effectors to the harder threat, which is aligned with the optimal structure. By contrast, $\pi_{\textsf{FA}}$ allocates effectors uniformly across threats. The dominance of $\pi_{\textsf{SL}}$ under \cref{eq::MinDefuse} follows directly from the fact that the expected defusing time under any policy $\pi$ is determined by its survival likelihood over all deadlines; see \cref{eq:dtm_slm}. Since $\pi_{\textsf{SL}}$ outperforms $\pi_{\textsf{FA}}$ for every $\tau$ under \cref{eq::SurvivalObjective}, it also outperforms $\pi_{\textsf{FA}}$ under \cref{eq::MinDefuse}.

\subsection{Greedy Policy $\pi_{\textsf{EA}}$ for Effective Assignment Maximization}
Any instance of problem $(\textsf{EAM}_1)$ seeks to find an allocation from the set
$$
\argmax _{x \in \mathbb{Z}_{\geq 0}^n: \sum_{i=1}^n x_i=C} ~ \sum_{i=1}^n \mathbb{E}\left[\min \left\{x_i, L_i\right\}\right]=\argmax _{x \in \mathbb{Z}_{\geq 0}^n: \sum_{i=1}^n x_i=C} \sum_{i=1}^n \frac{1-q_i^{x_i}}{1-q_i}.
$$
Noting that for all $i,$ the summand $(1-q_i^{x_i})/(1-q_i)$ is non-decreasing and 
\[
\Delta_{\textsf{EA},i}(x_i) := (1-q_i^{x_i+1})/(1-q_i)-(1-q_i^{x_i})/(1-q_i)=q_i^{x_i}
\]
is non-increasing, the steepest ascent method of \Cref{lem:separable_concave_greedy} solves to optimality. We hence formalize the greedy policy $\pi_{\textsf{EA}}$ in Algorithm \ref{alg:: WM}.

\begin{algorithm} 
\caption{Greedy Effective Assignment Maximization $\pi_{\textsf{EA}}$}\label{alg:: WM}
\small
\begin{algorithmic} 
\Require Set of active threats $\mathcal{M}_t$, number of effectors $C \in \mathbbm{Z}_{\geq 1}$. \\
\hspace{-0.38cm}\textbf{Output:} Allocation policy $x^t\in\mathbbm{Z}_{\ge 0}^{n}$ with $\sum_{i\in\mathcal{M}_t} x_i^t = C$ in round $t$.
\State Initialize $x_i^t\leftarrow 0$ for $i\in [n]$. 
\State \textbf{if} $\mathcal{M}_t = \varnothing$ \textbf{then return} $x^t$.
\For{$k = 1,\ldots, C$}
    \State $i^* \gets \argmax_{i\in \mathcal{M}_t}\ q_i^{x_i^t}$. \Comment{tie-breaking rule}
    \State $x_{i^*}^t \gets x_{i^*}^t + 1$.
\EndFor
\State \Return allocation $x^t$ in round $t$.
\end{algorithmic}
\end{algorithm}
Similar to Algorithm \ref{alg::ML_Estim}, as Algorithm \ref{alg:: WM} admits degrees of freedom in its execution by way of different tie-breaking rules, %
we define the class of \emph{greedy effective assignment policies} via
\[
\Pi_{\textsf{EA}}\coloneqq \left\{\pi \in \Pi: \pi_t(\mathcal{M}_t) \text{ is an output of Algorithm \ref{alg:: WM},~$\forall \mathcal{M}\in\mathcal{S}$ with $\mathcal{M}\neq\emptyset$},~\forall t\in\mathbb{Z}_{\ge1}\right\}.
\]

Again, throughout, the notation $\pi_{\textsf{EA}}$ will denote an arbitrary member of $\Pi_{\textsf{EA}}$, and any statement in which $\pi_{\textsf{EA}}$ is referenced as if it were a particular policy, shall be disambiguated by understanding it as a statement holding for all members of the class $\Pi_{\textsf{EA}}$ (unless otherwise noted).

Recall that $V^\pi_{\textsf{EAM}_\tau}$ is the sum of the expected effective assignments obtained over $\tau$ rounds under policy $\pi$. In fact, it also exhibits a kind of \emph{adaptive} submodularity \citep{golovin2011adaptive}, whereby any allocation's contribution to the accumulation of effective assignments is weaker as time goes on. These characteristics facilitate analysis of greedy allocations in each round. The next theorem shows that, under \cref{eq::EffectiveObjective}, the policy $\pi_{\textsf{EA}}$ has a $\tau$-dependent performance guarantee with respect to an optimal policy $\pi^*$.

\begin{theorem}[Performance guarantee for $\pi_{\textsf{EA}}$] \label{thm:adaptive_greedy} For $\tau \in \mathbbm{Z}_{\geq 1}$,  
    \[
     V_{\textsf{EAM}_\tau}^{\pi_{\textsf{EA}}} \geq \left(1 - (1 - \frac{1}{\tau})^\tau\right)V_{\textsf{EAM}_\tau}^* \geq \left(1 - \frac{1}{e} \right) V_{\textsf{EAM}_\tau}^*.
    \]
\end{theorem}

\Cref{thm:adaptive_greedy} shows that in each round, $\pi_{\textsf{EA}}$ captures at least a $1 / \tau$ fraction of the remaining achievable expected effective assignments, so the residual gap to optimality contracts by a factor $1-1 / \tau$ per round. Iterating it over $\tau$ rounds yields the factor $1-(1-1 / \tau)^\tau$.  As $\tau$ increases, this factor (monotonically) converges to $1-1/e$ as $\tau\to\infty$. Thus, under \cref{eq::EffectiveObjective}, $\pi_{\textsf{EA}}$ admits a constant-factor performance ratio. 

Moreover, similar to \Cref{thm:sl_vs_fair_n2}, when there are two threats, we can show that $\pi_{\textsf{EA}}$ provably outperforms $\pi_{\textsf{FA}}$ in effective assignment maximization as $\pi_{\textsf{EA}}$ prioritizes the harder threat, which aligns with the optimal structural characterization in \Cref{theorem:properties_n=2}.

\begin{theorem}[Greedy policy $\pi_\textsf{EA}$ dominates fair allocation $\pi_{\textsf{FA}}$ when $n=2$]\label{thm:ea_vs_fair_n2}
{\color{red}  } When $n=2$, for all $\tau\geq 1$, we have %
$V_{\textsf{EAM}_\tau}^{\pi_\textsf{EA}}\;\ge\;V_{\textsf{EAM}_\tau}^{\pi_\textsf{FA}}$.%
\footnote{Similar to Theorem \ref{thm:sl_vs_fair_n2}, we require $\pi_{\textsf{FA}}$ and $\pi_{\textsf{EA}}$ assign the first $\min\{C,|\mathcal M_t|\}$ effectors to (distinct) active threats in each round $t$ following the same index order for threats $i\in \mathcal{M}_t$.}
\end{theorem}

However, unlike \Cref{thm:sl_vs_fair_n2}, the dominance result for defusing time minimization does not hold for $\pi_{\textsf{EA}}$. That is, the dominance of $\pi_{\textsf{EA}}$ under \cref{eq::EffectiveObjective} does not imply its dominance under \cref{eq::MinDefuse}. The following example illustrates this gap.
\begin{example}
When $n=2$, $C=4$, $(q_1,q_2)=(0.7,0.4)$, we have $\mathbb{E}[T^{\pi_{\textsf{EA}}}]\;>\;\mathbb{E}[T^{\pi_{\textsf{FA}}}]$. In fact, $\pi_{\textsf{FA}}$ is optimal for this instance.
\end{example}

\subsection{Further Remarks on $\pi_{\textsf{SL}}$ and $\pi_{\textsf{EA}}$}\label{sec:comp_pi_sl_and_pi_ea}

Under both greedy policies, priority is first given to active threats that have not yet received any effector in the current round. As a result, both policies begin by assigning one effector to each active threat and they differ only in how they allocate the remaining effectors once every active threat has received one. %
Moreover, $\pi_{\textsf{EA}}$ tends to allocate \emph{more aggressively} to harder threats compared to $\pi_{\textsf{SL}}$. To see why, consider a two-threat case $n=2$ with $q_1>q_2$. Suppose both policies are at the same intermediate allocation, with $x_1$ effectors already assigned to threat $1$ and $x_2$ effectors already assigned to threat $2$, and there are still some effectors remaining. If $\pi_{\textsf{EA}}$ allocates the next effector to threat $2$, this implies that $q_1^{x_1}\leq q_2^{x_2}$ according to \Cref{alg:: WM}, or equivalently, $q_1^{-x_1}\geq q_2^{-x_2}$. Since $q_1>q_2$, the allocation of next effector under $\pi_{\textsf{SL}}$, according to \Cref{alg::ML_Estim}, follows
$$s_1=\sum_{k=1}^{x_1} q_1^{-k}=
\frac{q_1^{-x_1}-1}{1-q_1}
>
\frac{q_1^{-x_1}-1}{1-q_2}
\geq 
\frac{q_2^{-x_2}-1}{1-q_2}=\sum_{k=1}^{x_2} q_2^{-k}=s_2.
$$
Therefore, $\pi_{\textsf{SL}}$ allocates next effector to threat $2$ as well. This also implies that whenever $\pi_{\textsf{SL}}$ allocates the next effector to threat $1$, $\pi_{\textsf{EA}}$ also allocates it to threat $1$.  %
Therefore, given both threats are active, $\pi_{\textsf{EA}}$ can never allocate fewer effectors to the harder threat than $\pi_{\textsf{SL}}$ in each round. %

\section{Numerical Experiments}\label{sec:numerical_experiment}
In this section, we conduct numerical experiments to compare and contrast the performance of the time-oblivious policies $\pi_{\textsf{FA}}$, $\pi_{\textsf{SL}}$, and $\pi_{\textsf{EA}}$. In doing so, we will confirm theoretical findings, highlight best-performers, and note performance trends. 
We begin with a high-level summary of the setup and benchmarks. 

\smallskip

\noindent\textbf{Instance Generation:} The parameter space $\mathcal{P}\coloneqq \{(n,C,q,\tau): n, C, \tau \in \mathbbm{Z}_{\geq 1}, q \in [0,1)^n\}$ 
that governs the instantiation of the three related problems $\eqref{eq::MinDefuse}, \eqref{eq::SurvivalObjective},$ and $\eqref{eq::EffectiveObjective}$ is
explored via: (1) a small-scale, factorial design; (2) restriction to 2-salvo instances; (3) larger-scale, simulation-based design.

\smallskip

\noindent\textbf{Policies:} 
\begin{itemize}
    \item \underline{Deterministic (Bellman-)optimal policy $\pi^*$}. Computed via recursive Bellman equations, $\pi^*$ will generically refer to an optimal policy for either $\eqref{eq::MinDefuse},\eqref{eq::SurvivalObjective},$ or $\eqref{eq::EffectiveObjective}$, but is (computationally) expensive, so implemented only for the smaller ($\tau = 2$) instances of Sections \ref{sec:exp_hypercube} and \ref{sec:diff_var_Li}.
    
    \item Randomized Policy Benchmarks
    \begin{itemize}
        \item \underline{Random allocation $\pi_{\textsf{RA}}$}. Policy $\pi_{\textsf{RA}}$ denotes a naive randomization: at time $t$, $x^t \sim \text{Multinomial}(C; (1/M_t)_{i \in \mathcal{M}_t})$. Although (computationally) inexpensive, its performance suffered at larger scales.
    \item \underline{Reoptimized feasible fluid policy $\pi_{\textsf{R}}$.}
    Specialized for problem \cref{eq::EffectiveObjective}, $\pi_{\textsf{R}}$ denotes the reoptimized feasible fluid policy introduced in Section~4.2 of \citet{brown2025fluid}. At each time $t$, the allocation $x^t$ is randomized according to the solution of a linear program (see Section \ref{appendix:additional_experiments} in the Online Appendix for details); hence this policy's expected performance is 
    not amenable to exact recursive computation, and we evaluate via simulation. Specifically, on each instance, we report the sample average performance over 500 independent runs.
    \end{itemize}

    \item \underline{Time-oblivious policies: $\pi_{\textsf{FA}}$, $\pi_{\textsf{SL}}$, and $\pi_{\textsf{EA}}$}. Although they require tie-breaking rules for deterministic execution, for the experiments, we had all three select uniformly at random from among all candidate allocations that satisfy their respective selection criterion.\footnote{We can potentially improve performance with more sophisticated tie-breaking rules. For example, %
we found that breaking ties by always assigning each effector to the largest $q_i$ improves performance for both the fair allocation and greedy policies.}
\end{itemize}
Note that all policies outside of $\pi^*$ are random policies. 
In Sections \ref{sec:exp_hypercube} and \ref{sec:diff_var_Li}, the performance of every policy outside of $\pi_\textsf{R}$ is computed exactly. In Section \ref{sec:simulation}, the performance of every random policy is simulated.

\subsection{Small-Scale, Factorial Design}\label{sec:exp_hypercube}
In this first set of experiments, we focus on smaller scale parameters, and examine the effect of resource tightness on policy performance. With smaller scale, the computational burden of calculating Bellman optimal policies is alleviated, and we present exact performance comparisons against optimality. 
We begin by defining a factorial design $\mathcal{D} \subseteq \mathcal{P}$ via
\[
\mathcal{D} \coloneqq 
\Bigl\{
(n,C,q,\tau): 
n \in \{2,3,\ldots, 7\},\ 
C/n \in \{1,2,3\}, 
q \in \mathcal{I}_n, \tau = 2
\Bigr\},
\]
where for any $n,$ the set $\mathcal I_n:=\left\{(q_1,\ldots,q_n): q_1\ge q_2\ge \cdots \ge q_n,\ q_i\in Q_n,\ \forall i\in[n]\right\}$ is the collection of (nonincreasing-) ordered members of $Q_n:=\left\{0.05+\frac{0.9k}{10-n}: k=0,1,\ldots,10-n\right\}$, a collection of $11-n$ equally-spaced difficulties in $[0.05,0.95]$. The number of candidate difficulties $|Q_n|$ is decreasing in $n$ to control the computational burden as $n$ increases.

\smallskip

\noindent\textbf{Metrics:} For any $n,C$ in the projection $\{
(n',C'): \exists q \;\; \text{such that }(n',C',q,2) \in \mathcal{D}\}$, we can evaluate any policy $\pi \neq \pi^*$ under objective $\diamond \in \{\textsf{SLM}_2, \textsf{EAM}_2\}$ by averaging its expected performance across $q \in \mathcal{I}_n$ via
\[
        \bar V_{\diamond}^{\pi}(n,C)
        :=
        \frac{1}{|\mathcal I_n|}
        \sum_{(q_1,\ldots,q_n)\in\mathcal I_n}
        V_{\diamond}^{\pi}(q_1,\ldots,q_n; C), 
\]
where $V_{\diamond}^{\pi}(q_1,\ldots,q_n; C)$ is the quantity $V_{\diamond}^{\pi}$, just with parameter instantiations of $n,C,$ and $q$ highlighted. Analogously, with $T^{\pi}(q_1,\ldots,q_n; C)$ denoting the (random) defusing time by $\pi$ on the parameter instantiation of $n,C,$ and $q,$ we can average this over $q \in \mathcal{I}_n$ via
\[
\mathbb{E}\left[\bar{T}^\pi(n,C) \coloneqq  \frac{1}{|\mathcal I_n|}
        \sum_{(q_1,\ldots,q_n)\in\mathcal I_n}
        T^\pi(q_1, \ldots, q_n; C) \right],
\]
where $\mathbb{E}\left[\cdot \right]$ is w.r.t. the mutually independent collection of $L_i \sim Geom(1-q_i)$ and any randomization in $\pi.$
Then we say the \emph{average optimality gap} incurred by a policy $\pi \neq \pi^*$ under objective $\diamond$ is 
\begin{equation*}
\operatorname{OptGap}^{\pi}_{\diamond}(n, C) \coloneqq
\begin{cases}
\dfrac{
\bar V_{\diamond}^{\pi^*}(n,C)-\bar V_{\diamond}^{\pi}(n,C)
}{
\bar V_{\diamond}^{\pi^*}(n,C)
},
& \diamond \in \{\textsf{SLM}_2, \textsf{EAM}_2\}, \\[1.25em]
\dfrac{
\mathbb{E}\!\left[\bar T^\pi(n,C)\right]
-
\mathbb{E}\!\left[\bar T^{\pi^*}(n,C)\right]
}{
\mathbb{E}\!\left[\bar T^{\pi^*}(n,C)\right]
},
& \diamond = \textsf{DTM},
\end{cases}
\end{equation*}
where $\pi^*$ is any arbitrary optimal policy for objective $\diamond$. 

\smallskip

\noindent \textbf{Results:} 
Figure \ref{fig:hypercube_avg_gap} displays how the various policies perform (w.r.t. optimal) as a function of $n$, across different resource tightness ratios $C/n$. 
We discuss several observations to this small-scale experiment. 
\begin{enumerate}
    \item \underline{$\pi_{\textsf{FA}}$ is the policy most adversely affected by increased $C/n$.} In Figure \ref{fig:hypercube_avg_gap}, $\pi_{\textsf{FA}}$ curves translated upward as $C/n$ increased. 
    \item \underline{Greedy policies were minimally affected by change in $C/n$.} In Figure \ref{fig:hypercube_avg_gap}, the greedy policy curves translated upward minimally as $C/n$ increased. 
    
    \item \underline{Greedy policies exhibited near-optimal performance in \cref{eq::MinDefuse}.} Both $\pi_{\textsf{SL}}$ and $\pi_{\textsf{EA}}$ are nearly optimal under \cref{eq::MinDefuse}, with $\pi_{\textsf{SL}}$ performing particularly well: its average optimality gap remains below $0.01$ for all $n$, which corroborates \Cref{thm:sl_vs_fair_n2}.

    \item \underline{$\pi_{\textsf{SL}}$ was the best policy for DTM and $\textsf{SLM}_2$; $\pi_{\textsf{EA}}$ was the best policy for $\textsf{EAM}_2$.}
    These objective-aligned performance patterns extend the theoretical findings of \Cref{thm:sl_vs_fair_n2} and \Cref{thm:ea_vs_fair_n2} beyond the case of $n=2$. Moreover, under $(\textsf{EAM}_2)$, $\pi_{\textsf{EA}}$ outperforms $\pi_{\textsf{R}}$, an asymptotically optimal (as $C, n\to\infty$) policy for generic weakly coupled MDPs, despite the latter's much higher computational cost from solving a linear program in each round; see Section~\ref{appendix:additional_experiments} of the Online Appendix for additional comparisons between $\pi_{\textsf{EA}}$ and $\pi_{\textsf{R}}$ regarding their performances and computational times.

    \item \underline{$\pi_{\textsf{FA}}$ was effective across all objectives, particularly for \cref{eq::MinDefuse}.} $\pi_{\textsf{FA}}$ had an average optimality gap no larger than 15\% across all objectives, $n$, and resource tightness ratios tested. In particular, it is most effective under \cref{eq::MinDefuse}, followed by ($\textsf{EAM}_2$) and next ($\textsf{SLM}_2$), corroborating theoretical results in \Cref{sec:fair_policy}. Finally, we note that this policy is independent of $q.$
    
    \item \underline{$\pi_{\textsf{RA}}$ was the worst performing policy}. As seen in \Cref{fig:hypercube_avg_gap}, $\pi_{\textsf{RA}}$ was consistently the worst-performing policy, with its most severe average optimality gaps found under $(\textsf{SLM}_2)$.
\end{enumerate}  

\smallskip

\noindent \textbf{Conclusions:} Operationally speaking, our empirical results suggest using $\pi_{\textsf{SL}}$ for problem \cref{eq::SurvivalObjective} and \cref{eq::MinDefuse}, and  $\pi_{\textsf{EA}}$
for problem \cref{eq::EffectiveObjective}. 
In Section \ref{appendix:additional_plots} of the Online Appendix, we also report the corresponding average objective values and worst-case ratios under each policy, which exhibit a similar pattern as the average optimality gap.

\begin{figure}[htbp]
\centering
\begin{tikzpicture}

\begin{groupplot}[
    group style={
        group size=3 by 3,
        horizontal sep=1.0cm,
        vertical sep=1.1cm
    },
    width=0.35\textwidth,
    height=0.30\textwidth,
    grid=both,
    tick align=outside,
    xtick pos=bottom,
    ytick pos=left,
    tick label style={font=\scriptsize},
    label style={font=\small},
    title style={font=\small},
    scaled y ticks=false,
    xtick=data,
    xticklabel style={/pgf/number format/fixed},
    every axis plot/.append style={
        line width=1pt,
        mark=*,
        mark size=0.8pt
    },
]

\nextgroupplot[
    title = {$C = n$},
    ylabel = {$\operatorname{OptGap}^\pi_{\textsf{DTM}}(n, C)$},
    ylabel style={align=center},
    ymax=0.14,
    ytick={0,0.03,0.06,0.09,0.12},
    yticklabel style={
        /pgf/number format/fixed,
        /pgf/number format/precision=2
    }
]
\addplot[blue]
table[x=n,y=uniform_random,col sep=comma] {tikzdata/hypercube_new_cn/cn1_avg_gap_min_diffuse.csv};
\addplot[orange]
table[x=n,y=fair,col sep=comma] {tikzdata/hypercube_new_cn/cn1_avg_gap_min_diffuse.csv};
\addplot[green!60!black]
table[x=n,y=ml_greedy,col sep=comma] {tikzdata/hypercube_new_cn/cn1_avg_gap_min_diffuse.csv};
\addplot[red]
table[x=n,y=ea_greedy,col sep=comma] {tikzdata/hypercube_new_cn/cn1_avg_gap_min_diffuse.csv};

\nextgroupplot[
    title = {$C = 2n$},
    ymax=0.14,
    ytick={0,0.03,0.06,0.09,0.12},
    yticklabel style={
        /pgf/number format/fixed,
        /pgf/number format/precision=2
    }
]
\addplot[blue]
table[x=n,y=uniform_random,col sep=comma] {tikzdata/hypercube_new_cn/cn2_avg_gap_min_diffuse.csv};
\addplot[orange]
table[x=n,y=fair,col sep=comma] {tikzdata/hypercube_new_cn/cn2_avg_gap_min_diffuse.csv};
\addplot[green!60!black]
table[x=n,y=ml_greedy,col sep=comma] {tikzdata/hypercube_new_cn/cn2_avg_gap_min_diffuse.csv};
\addplot[red]
table[x=n,y=ea_greedy,col sep=comma] {tikzdata/hypercube_new_cn/cn2_avg_gap_min_diffuse.csv};

\nextgroupplot[
    title = {$C = 3n$},
    ymax=0.14,
    ytick={0,0.03,0.06,0.09,0.12},
    yticklabel style={
        /pgf/number format/fixed,
        /pgf/number format/precision=2
    }
]
\addplot[blue]
table[x=n,y=uniform_random,col sep=comma] {tikzdata/hypercube_new_cn/cn3_avg_gap_min_diffuse.csv};
\addplot[orange]
table[x=n,y=fair,col sep=comma] {tikzdata/hypercube_new_cn/cn3_avg_gap_min_diffuse.csv};
\addplot[green!60!black]
table[x=n,y=ml_greedy,col sep=comma] {tikzdata/hypercube_new_cn/cn3_avg_gap_min_diffuse.csv};
\addplot[red]
table[x=n,y=ea_greedy,col sep=comma] {tikzdata/hypercube_new_cn/cn3_avg_gap_min_diffuse.csv};

\nextgroupplot[
    ylabel = {$\operatorname{OptGap}^\pi_{\textsf{SLM}_2}(n, C)$},
    ylabel style={align=center},
    ymax=0.6,
    ytick={0,0.15,0.3,0.45},
    yticklabel style={
        /pgf/number format/fixed,
        /pgf/number format/precision=2
    }
]
\addplot[blue]
table[x=n,y=uniform_random,col sep=comma] {tikzdata/hypercube_new_cn/cn1_avg_gap_max_ml.csv};
\addplot[orange]
table[x=n,y=fair,col sep=comma] {tikzdata/hypercube_new_cn/cn1_avg_gap_max_ml.csv};
\addplot[green!60!black]
table[x=n,y=ml_greedy,col sep=comma] {tikzdata/hypercube_new_cn/cn1_avg_gap_max_ml.csv};
\addplot[red]
table[x=n,y=ea_greedy,col sep=comma] {tikzdata/hypercube_new_cn/cn1_avg_gap_max_ml.csv};

\nextgroupplot[
    ymax=0.6,
    ytick={0,0.15,0.3,0.45},
    yticklabel style={
        /pgf/number format/fixed,
        /pgf/number format/precision=2
    }
]
\addplot[blue]
table[x=n,y=uniform_random,col sep=comma] {tikzdata/hypercube_new_cn/cn2_avg_gap_max_ml.csv};
\addplot[orange]
table[x=n,y=fair,col sep=comma] {tikzdata/hypercube_new_cn/cn2_avg_gap_max_ml.csv};
\addplot[green!60!black]
table[x=n,y=ml_greedy,col sep=comma] {tikzdata/hypercube_new_cn/cn2_avg_gap_max_ml.csv};
\addplot[red]
table[x=n,y=ea_greedy,col sep=comma] {tikzdata/hypercube_new_cn/cn2_avg_gap_max_ml.csv};

\nextgroupplot[
    ymax=0.6,
    ytick={0,0.15,0.3,0.45},
    yticklabel style={
        /pgf/number format/fixed,
        /pgf/number format/precision=2
    }
]
\addplot[blue]
table[x=n,y=uniform_random,col sep=comma] {tikzdata/hypercube_new_cn/cn3_avg_gap_max_ml.csv};
\addplot[orange]
table[x=n,y=fair,col sep=comma] {tikzdata/hypercube_new_cn/cn3_avg_gap_max_ml.csv};
\addplot[green!60!black]
table[x=n,y=ml_greedy,col sep=comma] {tikzdata/hypercube_new_cn/cn3_avg_gap_max_ml.csv};
\addplot[red]
table[x=n,y=ea_greedy,col sep=comma] {tikzdata/hypercube_new_cn/cn3_avg_gap_max_ml.csv};

\nextgroupplot[
    xlabel={$n$},
    ylabel={$\operatorname{OptGap}^\pi_{\textsf{EAM}_2}(n, C)$},
    ylabel style={align=center},
    ymax=0.18,
    ytick={0,0.04,0.08,0.12,0.16},
    yticklabel style={
        /pgf/number format/fixed,
        /pgf/number format/precision=3
    }
]
\addplot[blue]
table[x=n,y=uniform_random,col sep=comma] {tikzdata/hypercube_new_cn/cn1_avg_gap_max_effective.csv};
\addplot[orange]
table[x=n,y=fair,col sep=comma] {tikzdata/hypercube_new_cn/cn1_avg_gap_max_effective.csv};
\addplot[green!60!black]
table[x=n,y=ml_greedy,col sep=comma] {tikzdata/hypercube_new_cn/cn1_avg_gap_max_effective.csv};
\addplot[red]
table[x=n,y=ea_greedy,col sep=comma] {tikzdata/hypercube_new_cn/cn1_avg_gap_max_effective.csv};
\addplot[purple]
table[x=n,y=reoptimized_fluid,col sep=comma] {tikzdata/hypercube_new_cn/cn1_avg_gap_max_effective.csv};

\nextgroupplot[
    xlabel={$n$},
    ymax=0.18,
    ytick={0,0.04,0.08,0.12,0.16},
    yticklabel style={
        /pgf/number format/fixed,
        /pgf/number format/precision=3
    }
]
\addplot[blue]
table[x=n,y=uniform_random,col sep=comma] {tikzdata/hypercube_new_cn/cn2_avg_gap_max_effective.csv};
\addplot[orange]
table[x=n,y=fair,col sep=comma] {tikzdata/hypercube_new_cn/cn2_avg_gap_max_effective.csv};
\addplot[green!60!black]
table[x=n,y=ml_greedy,col sep=comma] {tikzdata/hypercube_new_cn/cn2_avg_gap_max_effective.csv};
\addplot[red]
table[x=n,y=ea_greedy,col sep=comma] {tikzdata/hypercube_new_cn/cn2_avg_gap_max_effective.csv};
\addplot[purple]
table[x=n,y=reoptimized_fluid,col sep=comma] {tikzdata/hypercube_new_cn/cn2_avg_gap_max_effective.csv};

\nextgroupplot[
    xlabel={$n$},
    ymax=0.18,
    ytick={0,0.04,0.08,0.12,0.16},
    yticklabel style={
        /pgf/number format/fixed,
        /pgf/number format/precision=3
    }
]
\addplot[blue]
table[x=n,y=uniform_random,col sep=comma] {tikzdata/hypercube_new_cn/cn3_avg_gap_max_effective.csv};
\addplot[orange]
table[x=n,y=fair,col sep=comma] {tikzdata/hypercube_new_cn/cn3_avg_gap_max_effective.csv};
\addplot[green!60!black]
table[x=n,y=ml_greedy,col sep=comma] {tikzdata/hypercube_new_cn/cn3_avg_gap_max_effective.csv};
\addplot[red]
table[x=n,y=ea_greedy,col sep=comma] {tikzdata/hypercube_new_cn/cn3_avg_gap_max_effective.csv};
\addplot[purple]
table[x=n,y=reoptimized_fluid,col sep=comma] {tikzdata/hypercube_new_cn/cn3_avg_gap_max_effective1.csv};

\end{groupplot}

\path (group c1r3.south west) -- (group c3r3.south east)
    coordinate[midway] (legendmid);

\matrix[
    matrix of nodes,
    anchor=north,
    row sep=0pt,
    column sep=4pt,
    nodes={anchor=west, inner sep=0pt, outer sep=0pt, font=\footnotesize}
] at ($(legendmid)+(0,-0.9cm)$) {
\tikz{\draw[blue,line width=1pt] (0,0)--(0.30,0); \fill[blue] (0.15,0) circle (0.8pt);} &
$\pi_{\textsf{RA}}$ &
\tikz{\draw[orange,line width=1pt] (0,0)--(0.30,0); \fill[orange] (0.15,0) circle (0.8pt);} &
$\pi_{\textsf{FA}}$ &
\tikz{\draw[green!60!black,line width=1pt] (0,0)--(0.30,0); \fill[green!60!black] (0.15,0) circle (0.8pt);} &
$\pi_{\textsf{SL}}$ &
\tikz{\draw[red,line width=1pt] (0,0)--(0.30,0); \fill[red] (0.15,0) circle (0.8pt);} &
$\pi_{\textsf{EA}}$ &
\tikz{\draw[purple,line width=1pt] (0,0)--(0.30,0); \fill[purple] (0.15,0) circle (0.8pt);} &
$\pi_{\textsf{R}}$ \\
};

\end{tikzpicture}
\caption{
For every combination of resource tightness $\kappa=C/n \in \{1,2,3\}$ and objective $\diamond \in \{\textsf{DTM}, (\textsf{SLM}_2), (\textsf{EAM}_2)\}$, OptGap$^\pi_\diamond(n, \kappa n)$-versus-$n$ is plotted for $\pi \in \{\pi_{\textsf{FA}}$, $\pi_{\textsf{SL}}$, $\pi_{\textsf{EA}}$, $\pi_{\textsf{R}}\}$. Note: the specialized $\pi_{\textsf{R}}$ was only tested for the $(\textsf{EAM}_2)$ objective. 
}
\label{fig:hypercube_avg_gap}
\end{figure}
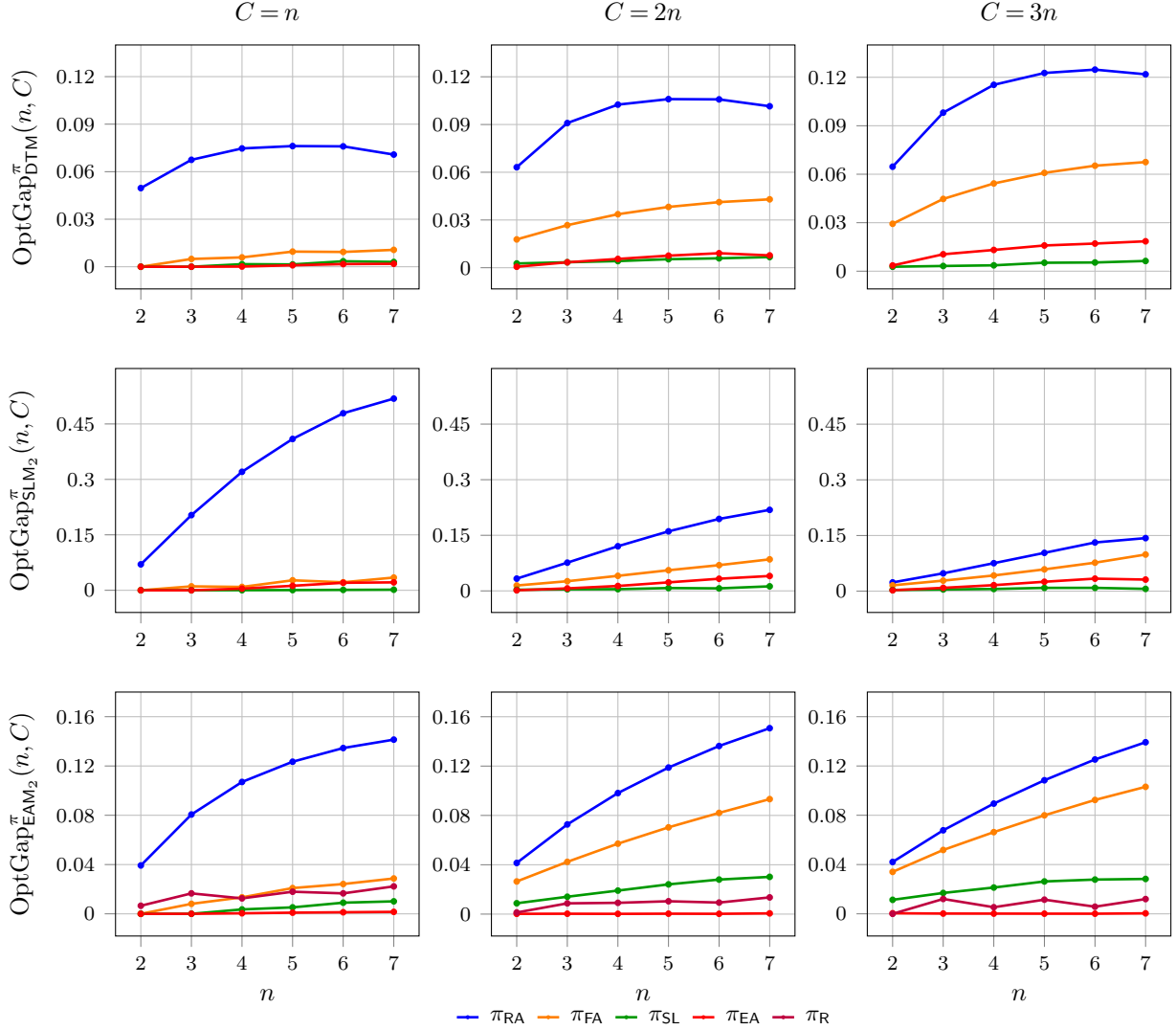

\subsection{2-Salvo Instances}\label{sec:diff_var_Li}
In this second set of experiments, we remain focused on smaller scale parameters, but we now examine the effect of difficulty heterogeneity on policy performance. 
We restrict attention to the subset $\mathcal{D}_2 \subseteq \mathcal{D}$ defined via
\[
\mathcal{D}_2 \coloneqq 
\Bigl\{
(n,C,q,\tau) \in \mathcal{D}: |\{q_i: i = 1, \ldots, n\}| = 2
\Bigr\},
\]
that is, we consider only those instances from the small-scale, factorial design $\mathcal{D}$ that express precisely 2 salvos, and we proceed to vary the difficulty discrepancy between these two salvos.

Noting that for any instance $(n, C, q, \tau)$, the expected demand for threat $i$ can be written $\mathbb{E}[L_i]=1/(1-q_i)$, we let $H(n,C,q,\tau) \coloneqq n^{-1}\sum_{i=1}^n\left(\mathbb{E}[L_i]-n^{-1}\sum_{j=1}^n\mathbb{E}[L_j]\right)^2$ denote the \emph{level of heterogeneity} for any instance. For each $n \in \{2,3,\ldots, 7\}$, this metric then yields the (roughly equal) partitioning 
\[
\{(n', C, q, \tau) \in \mathcal{D}_2: n' = n\} = \textsf{SmallVar}(n)\; \cup \textsf{MediumVar}(n) \cup \textsf{LargeVar}(n),
\]
wherein $\textsf{SmallVar}(n), \textsf{MediumVar}(n),$ and $\textsf{LargeVar}(n)$ are comprised, respectively, of the instances forming the smallest-third, middle-third, and largest-third of the set $H\left(\{(n',C,q,\tau): (n,C,q,\tau) \in \mathcal{D}_2, n' = n\}\right)$. These three $n$-parametrized sets will henceforth be referred to as the small, medium, and large \emph{variance regimes}. 

\noindent\textbf{Metrics:} For any $n \in \{2,\ldots, 7\}$ and $r \in \{\textsf{SmallVar}(\cdot), \textsf{MediumVar}(\cdot),$ and $\textsf{LargeVar}(\cdot)\},$ we evaluate any policy $\pi \neq \pi^*$ under objective $\diamond \in \{(\textsf{SLM}_2), (\textsf{EAM}_2)\}$ by averaging its expected performance across instances in $r(n)$ via 
\[
        \bar V_{\diamond}^{\pi}(n,r)
        :=
        \frac{1}{|r(n)|}
        \sum_{(n, C, q, \tau)\in r(n)}
        V_{\diamond}^{\pi}(q_1,\ldots,q_n; C).
\]
Analogously, we can average the expected defusing times via
\[
\mathbb{E}\left[\bar{T}^\pi(n,r) \coloneqq  \frac{1}{|r(n)|}
        \sum_{(n, C, q, \tau) \in r(n)}
        T^\pi(q_1, \ldots, q_n; C)\right],
\]
where $\mathbb{E}\left[\cdot \right]$ is w.r.t. the mutually independent collection of $L_i \sim Geom(1-q_i)$ and any randomization in $\pi.$
Then we say the \emph{optimality gap} exhibited by a policy $\pi \neq \pi^*$ for objective $\diamond$ is 

\begin{equation*}
\operatorname{OptGap}_\diamond^\pi(n, r) \coloneqq
\begin{cases}
\dfrac{
\bar V_{\diamond}^{\pi^*}(n,r)-\bar V_{\diamond}^{\pi}(n,r)
}{
\bar V_{\diamond}^{\pi^*}(n,r)
},
& \diamond \in \{(\textsf{SLM}_2),(\textsf{EAM}_2)\}, \\[1.25em]
\dfrac{
\mathbb{E}\!\left[\bar T^\pi(n,r)\right]
-
\mathbb{E}\!\left[\bar T^{\pi^*}(n,r)\right]
}{
\mathbb{E}\!\left[\bar T^{\pi^*}(n,r)\right]
},
& \diamond = \textsf{DTM}.
\end{cases}
\end{equation*}
where $\pi^*$ is any arbitrary optimal policy for objective $\diamond$. 

\noindent \textbf{Results:}
Figure \ref{fig:two_slices_avg_gap} displays how the various policies perform (w.r.t. optimal) as a function of $n$, across the different variance regimes \textsf{SmallVar}, \textsf{MediumVar}, and \textsf{LargeVar}. The performance of $\pi_{\textsf{RA}}$ was particularly poor, hence omitted. We discuss several observations. 
\begin{enumerate}
    \item \underline{$\pi_{\textsf{FA}}$ is the policy most adversely affected by increased level of heterogeneity.} 
    \item  \underline{Greedy policy advantage over $\pi_{\textsf{FA}}$ grows steadily with level of heterogeneity, as well as $n$.}
    \item \underline{$\pi_{\textsf{FA}}$ is near optimal under \textsf{SmallVar}.} When the 2 salvos are nearly identical in difficulty,  $\pi_{\textsf{FA}}$ performs well in spite of ignoring difficulty differences across threats -- Theorems \ref{theorem:properties_n=2}, \ref{theorem:DTM_fair_ratio}.
\end{enumerate}  

\noindent \textbf{Conclusions:} Just as in the small-scale experiments, our empirical results on 2-salvos continue to suggest using $\pi_{\textsf{SL}}$ for problem \cref{eq::SurvivalObjective} annd \cref{eq::MinDefuse}, and  $\pi_{\textsf{EA}}$
for problem \cref{eq::EffectiveObjective}. Increasing the level of heterogeneity between two salvos had a similar effect as that of increasing the resource tightness ratio-- the advantage of the greedy policies grows.

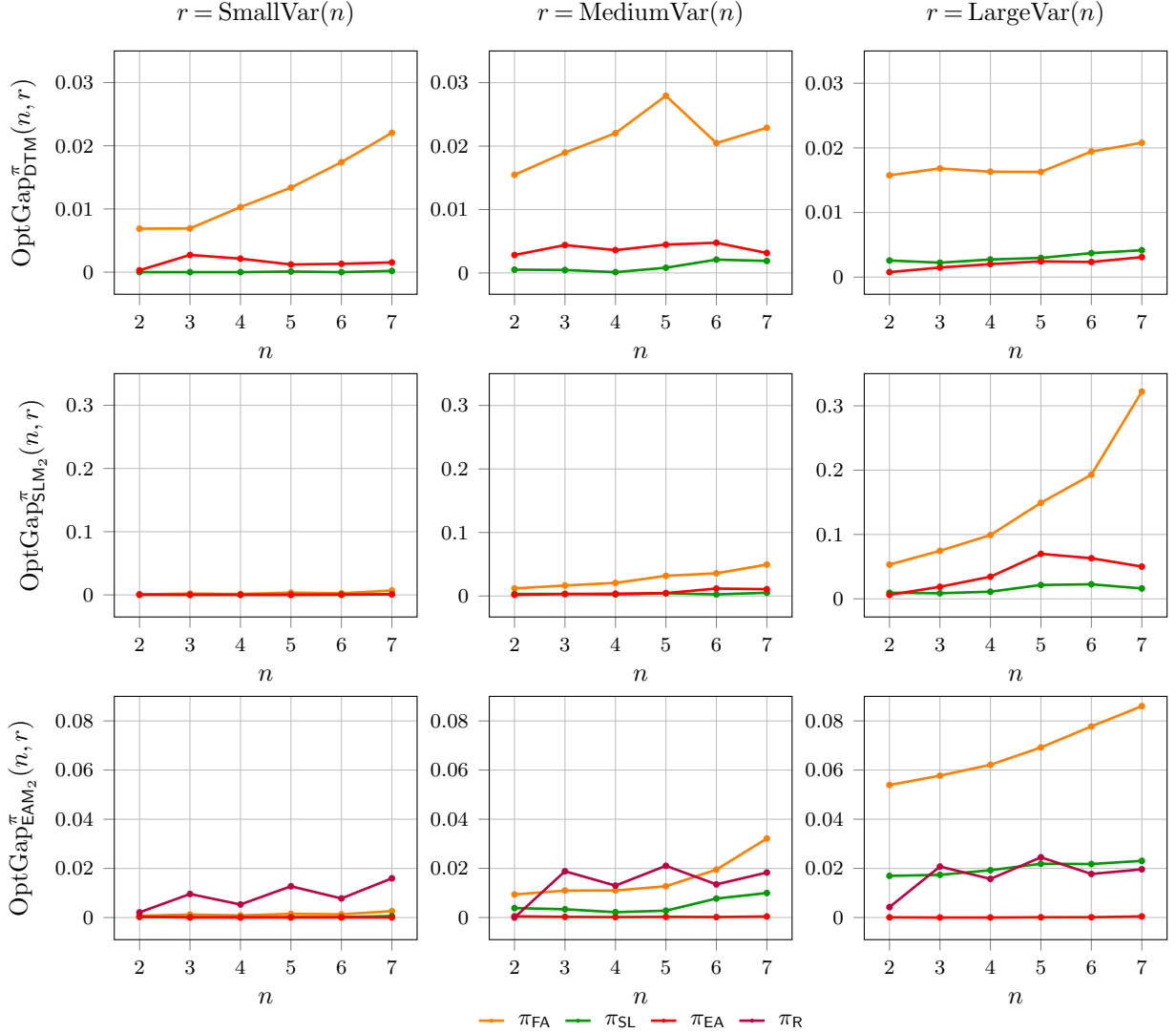
\begin{figure}[htbp]
\centering
\begin{tikzpicture}

\begin{groupplot}[
    group style={
        group size=3 by 3,
        horizontal sep=1.0cm,
        vertical sep=1.1cm
    },
    width=0.35\textwidth,
    height=0.30\textwidth,
    xlabel={$n$},
    grid=both,
    tick align=outside,
    xtick pos=bottom,
    ytick pos=left,
    tick label style={font=\scriptsize},
    label style={font=\small},
    title style={font=\small},
    scaled y ticks=false,
    scaled x ticks=false,
    xtick=data,
    every axis plot/.append style={
        line width=1pt,
        mark=*,
        mark size=0.8pt
    },
]

\nextgroupplot[
    title = {$r = \operatorname{SmallVar}(n)$},
    ylabel = {$\operatorname{OptGap}_{\textsf{DTM}}^\pi(n,r)$},
    ylabel style={align=center},
    ymax=0.035,
    ytick={0,0.01,0.02,0.03},
    yticklabels={0,0.01,0.02,0.03}
]
\addplot[orange]
table[x=n,y=piFA] {tikzdata/salvo_variance/dtm_small_gap.dat};
\addplot[green!60!black]
table[x=n,y=piSL] {tikzdata/salvo_variance/dtm_small_gap.dat};
\addplot[red]
table[x=n,y=piEA] {tikzdata/salvo_variance/dtm_small_gap.dat};

\nextgroupplot[
     title = {$r = \operatorname{MediumVar}(n)$},
    ymax=0.035,
    ytick={0,0.01,0.02,0.03},
    yticklabels={0,0.01,0.02,0.03}
]
\addplot[orange]
table[x=n,y=piFA] {tikzdata/salvo_variance/dtm_medium_gap.dat};
\addplot[green!60!black]
table[x=n,y=piSL] {tikzdata/salvo_variance/dtm_medium_gap.dat};
\addplot[red]
table[x=n,y=piEA] {tikzdata/salvo_variance/dtm_medium_gap.dat};

\nextgroupplot[
     title = {$r = \operatorname{LargeVar}(n)$},
    ymax=0.035,
    ytick={0,0.01,0.02,0.03},
    yticklabels={0,0.01,0.02,0.03}
]
\addplot[orange]
table[x=n,y=piFA] {tikzdata/salvo_variance/dtm_large_gap.dat};
\addplot[green!60!black]
table[x=n,y=piSL] {tikzdata/salvo_variance/dtm_large_gap.dat};
\addplot[red]
table[x=n,y=piEA] {tikzdata/salvo_variance/dtm_large_gap.dat};

\nextgroupplot[
     ylabel = {$\operatorname{OptGap}^\pi_{\textsf{SLM}_2}(n,r)$},
    ylabel style={align=center},
    ymax=0.35,
    ytick={0,0.1,0.2,0.3},
    ytick={0,0.1,0.2,0.3}
]
\addplot[orange]
table[x=n,y=piFA] {tikzdata/salvo_variance/slm_small_gap.dat};
\addplot[green!60!black]
table[x=n,y=piSL] {tikzdata/salvo_variance/slm_small_gap.dat};
\addplot[red]
table[x=n,y=piEA] {tikzdata/salvo_variance/slm_small_gap.dat};

\nextgroupplot[
    ymax=0.35,
    ytick={0,0.1,0.2,0.3},
    yticklabels={0,0.1,0.2,0.3}
]
\addplot[orange]
table[x=n,y=piFA] {tikzdata/salvo_variance/slm_medium_gap.dat};
\addplot[green!60!black]
table[x=n,y=piSL] {tikzdata/salvo_variance/slm_medium_gap.dat};
\addplot[red]
table[x=n,y=piEA] {tikzdata/salvo_variance/slm_medium_gap.dat};

\nextgroupplot[
    ymax=0.35,
    ytick={0,0.1,0.2,0.3},
    ytick={0,0.1,0.2,0.3}
]
\addplot[orange]
table[x=n,y=piFA] {tikzdata/salvo_variance/slm_large_gap.dat};
\addplot[green!60!black]
table[x=n,y=piSL] {tikzdata/salvo_variance/slm_large_gap.dat};
\addplot[red]
table[x=n,y=piEA] {tikzdata/salvo_variance/slm_large_gap.dat};

\nextgroupplot[
     ylabel = {$\operatorname{OptGap}^\pi_{\textsf{EAM}_2}(n,r)$},
    ylabel style={align=center},
    ymax=0.09,
    ytick={0,0.02,0.04,0.06,0.08},
    yticklabels={0,0.02,0.04,0.06,0.08}
]
\addplot[orange]
table[x=n,y=piFA] {tikzdata/salvo_variance/eam_small_gap.dat};
\addplot[green!60!black]
table[x=n,y=piSL] {tikzdata/salvo_variance/eam_small_gap.dat};
\addplot[red]
table[x=n,y=piEA] {tikzdata/salvo_variance/eam_small_gap.dat};
\addplot[purple]
table[x=n,y=piR] {tikzdata/salvo_variance/eam_small_gap.dat};

\nextgroupplot[
    ymax=0.09,
    ytick={0,0.02,0.04,0.06,0.08},
    yticklabels={0,0.02,0.04,0.06,0.08}
]
\addplot[orange]
table[x=n,y=piFA] {tikzdata/salvo_variance/eam_medium_gap1.dat};
\addplot[green!60!black]
table[x=n,y=piSL] {tikzdata/salvo_variance/eam_medium_gap1.dat};
\addplot[red]
table[x=n,y=piEA] {tikzdata/salvo_variance/eam_medium_gap1.dat};
\addplot[purple]
table[x=n,y=piR] {tikzdata/salvo_variance/eam_medium_gap1.dat};

\nextgroupplot[
    ymax=0.09,
    ytick={0,0.02,0.04,0.06,0.08},
    yticklabels={0,0.02,0.04,0.06,0.08}
]
\addplot[orange]
table[x=n,y=piFA] {tikzdata/salvo_variance/eam_large_gap.dat};
\addplot[green!60!black]
table[x=n,y=piSL] {tikzdata/salvo_variance/eam_large_gap.dat};
\addplot[red]
table[x=n,y=piEA] {tikzdata/salvo_variance/eam_large_gap.dat};
\addplot[purple]
table[x=n,y=piR] {tikzdata/salvo_variance/eam_large_gap.dat};

\end{groupplot}

\path (group c1r3.south west) -- (group c3r3.south east)
    coordinate[midway] (legendmid);

\matrix[
    matrix of nodes,
    anchor=north,
    row sep=0pt,
    column sep=5pt,
    nodes={
        anchor=west,
        inner sep=0pt,
        outer sep=0pt,
        font=\footnotesize
    }
] at ($(legendmid)+(0,-0.9cm)$) {
\tikz{\draw[orange,line width=1pt] (0,0)--(0.35,0); \fill[orange] (0.175,0) circle (0.8pt);} &
$\pi_{\textsf{FA}}$ &
\tikz{\draw[green!60!black,line width=1pt] (0,0)--(0.35,0); \fill[green!60!black] (0.175,0) circle (0.8pt);} &
$\pi_{\textsf{SL}}$ &
\tikz{\draw[red,line width=1pt] (0,0)--(0.35,0); \fill[red] (0.175,0) circle (0.8pt);} &
$\pi_{\textsf{EA}}$ &
\tikz{\draw[purple,line width=1pt] (0,0)--(0.35,0); \fill[purple] (0.175,0) circle (0.8pt);} &
$\pi_{\textsf{R}}$ \\
};

\end{tikzpicture}
\caption{
For every combination of variance regime $r \in \{\textsf{SmallVar}(\cdot), \textsf{MediumVar}(\cdot),$ and $\textsf{LargeVar}(\cdot)\}$ and objective $\diamond \in \{\textsf{DTM}, (\textsf{SLM}_2), (\textsf{EAM}_2)\}$, OptGap$_{\diamond}^\pi(n,r)$-versus-$n$ is plotted for $\pi \in \{\pi_{\textsf{FA}}$, $\pi_{\textsf{SL}}$, $\pi_{\textsf{EA}}$, $\pi_{\textsf{R}}\}$. Note: the specialized $\pi_{R}$ was only tested for the $(\textsf{EAM}_2)$ objective.
}
\label{fig:two_slices_avg_gap}
\end{figure}

\subsection{Large-Scale Settings}\label{sec:simulation}
We next evaluate the performance of our policies on larger parameter instances. As the scale now makes exact dynamic programming prohibitively (computationally) expensive, we turn to simulation, i.e., reporting average (and not exact) performances on any instance.
For every combination of $n$ and $C,$ let $\tilde{Q}_{n,C}$ be a collection of $10^4$ vectors $\{q^{(j)}\}_{j=1}^{10^4} \subseteq [0,1]^n$, wherein for each $j$, $q^{(j)}_1, \ldots, q^{(j)}_n \overset{iid}{\sim} \operatorname{Unif}(0.05,0.95)$. With these collections of difficulties drawn, we construct the subset of instances $\tilde{\mathcal{P}}_{\mathcal{L}} \subseteq \mathcal{P}$ via
\[
\tilde{\mathcal{P}}_{\mathcal{L}} \coloneqq 
\Bigl\{
(n,C,q,\tau): n \in \{10,14,18,22,26,30\}, C/n \in \{1,2,3\}, q \in \tilde{Q}_{n,C}, \tau = 3
\Bigr\}.
\]

\noindent
\textbf{Metrics:} 
For any $n,C$ in the projection $\{
(n',C'): \exists q \;\; \text{such that }(n',C',q,3) \in \mathcal{P}_{\mathcal{L}}\}$, we can evaluate any policy $\pi$ under objective $\diamond \in \{\textsf{SLM}_3, \textsf{EAM}_3\}$ by averaging its sample average performances across $q \in \tilde{Q}_{n,C}$ via
\[
        \bar V_{\diamond}^{\pi}(n,C)
        :=
        \frac{1}{|\tilde{Q}_{n,C}|}
        \sum_{(q_1,\ldots,q_n)\in \tilde{Q}_{n,C}}
        \tilde{V}_{\diamond}^{\pi}(q_1,\ldots,q_n; C), 
\]
where $\tilde{V}_{\diamond}^{\pi}(q_1,\ldots,q_n; C)$ is the reward (under objective $\diamond$) that $\pi$ attains under a single draw of $L_1, \ldots, L_n$ via $(L_i \sim \operatorname{Geom}(1 - q_i))_{i=1}^n$
Analogously,
\[
V^\pi_{\textsf{DTM}}(n,C) \coloneqq  \frac{1}{|\tilde{Q}_{n,C}|}
        \sum_{(q_1,\ldots,q_n)\in \tilde{Q}_{n,C}}
        \tilde{T}^\pi(q_1, \ldots, q_n; C), 
\]
where $\tilde{T}^\pi(q_1, \ldots, q_n; C)$ is the time to defuse by $\pi$ under a single draw of $L_1, \ldots, L_n$ via $(L_i \sim \operatorname{Geom}(1 - q_i))_{i=1}^n$.

\smallskip

\noindent
\textbf{Results:} 
In Figure \ref{fig:sim_obj_by_cn_tau3}, we simulate the performance of the various policies as a function of $n$ across different resource tightness ratios $C/n$. The error bars were negligible, so for the sake of visualization they were omitted. 

In \Cref{tab:uniform_q_tau3_winfreq_summary}, Appendix \ref{appendix:additional_plots}, we report the average objective value and the percentage of simulation replications (out of $10^4$) in which each policy attains the best objective value under $n\in\{10,20,30\}$. The results are broadly consistent with the findings in \Cref{sec:exp_hypercube}, with two additional observations in this large-scale setting. 
\begin{itemize}
    \item \underline{Under $\textsf{SLM}_3$, $\pi_{\textsf{EA}}$ slightly outperforms $\pi_{\textsf{SL}}$ when $C/n \in \{2,3\}$.} As discussed in Section \ref{sec:comp_pi_sl_and_pi_ea}, $\pi_{\textsf{EA}}$ allocates more aggressively to harder threats than $\pi_{\textsf{SL}}$. This suggests that, with sufficient capacity, the optimal policy $\pi^*$ similarly tends to allocate very aggressively to the most challenging threats. 
    \item \underline{Under $\textsf{EAM}_3$, $\pi_{\textsf{R}}$ slightly outperforms $\pi_{\textsf{EA}}$ when $C/n=1.$} This indicates that the reoptimized relaxation can provide a modest advantage in scarce-resource regimes.
\end{itemize}

\pgfplotsset{
    onlyselectedn/.style={
        restrict expr to domain={mod(\thisrow{n}-10,4)}{-0.001:0.001}
    }
}

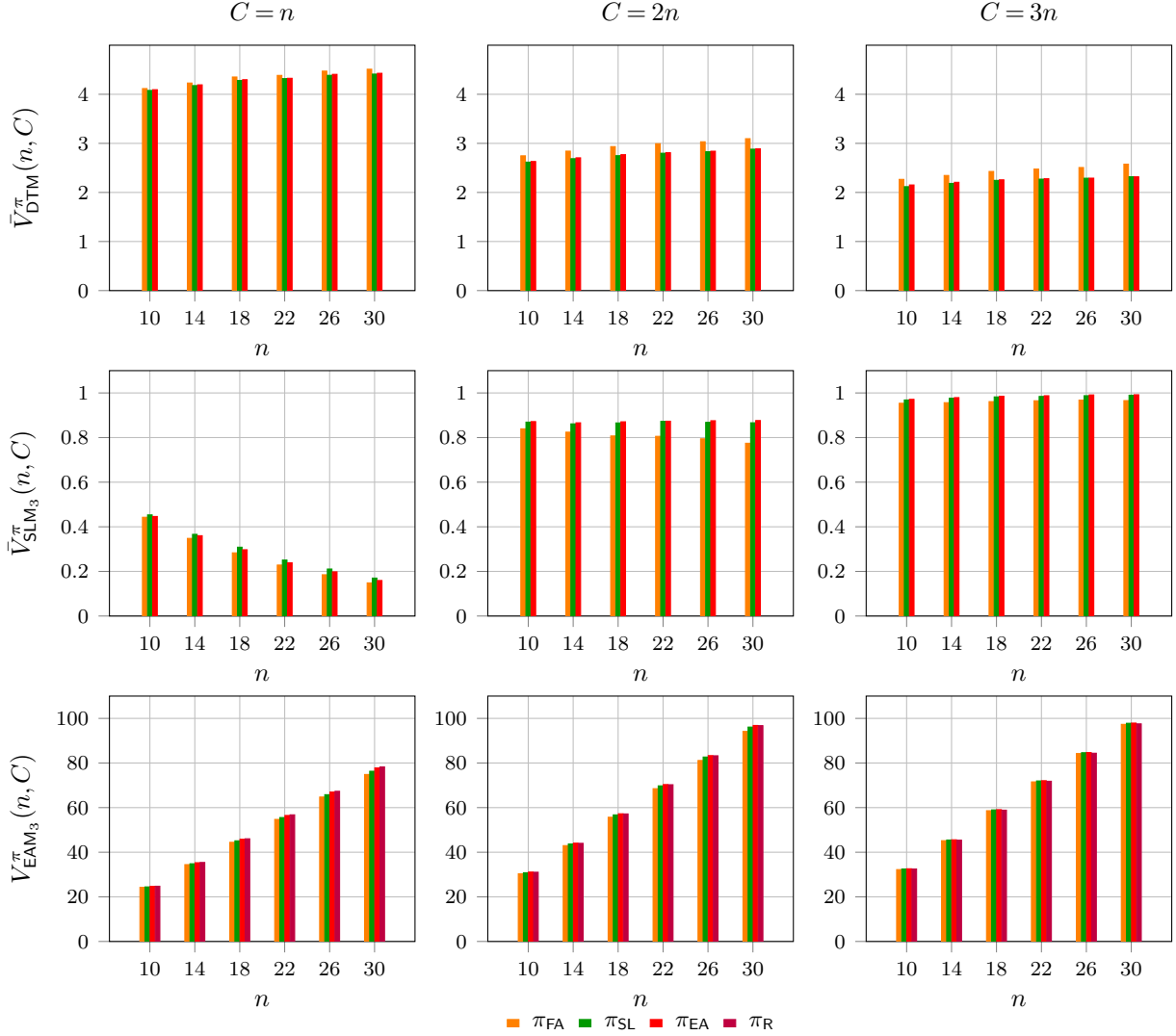
\begin{figure}[htbp]
\centering
\begin{tikzpicture}

\def\DTMymax{5}
\def\SLMymax{1.1}
\def\EAMymax{110}

\begin{groupplot}[
    group style={
        group size=3 by 3,
        horizontal sep=1.0cm,
        vertical sep=1.1cm
    },
    width=0.35\textwidth,
    height=0.30\textwidth,
    xlabel={$n$},
    grid=both,
    tick align=outside,
    xtick pos=bottom,
    ytick pos=left,
    tick label style={font=\scriptsize},
    label style={font=\small},
    title style={font=\small},
    ylabel style={
        align=center,
        at={(axis description cs:-0.20,0.5)},
        anchor=south
    },
    scaled y ticks=false,
    scaled x ticks=false,
    xtick={10,14,18,22,26,30},
    xticklabel style={/pgf/number format/fixed},
    ybar,
    /pgf/bar width=1.8pt,
    enlarge x limits=0.18,
    unbounded coords=discard,
]

\nextgroupplot[
    title={$C = n$},
    ylabel={  $\bar{V}_{\textsf{DTM}}^\pi(n,C)$},
    ymin=0,
    ymax=\DTMymax,
    ytick={0,1,2,3,4},
    yticklabel style={
        text width=2.2em,
        align=right,
        /pgf/number format/fixed,
        /pgf/number format/precision=0
    }
]
\addplot+[onlyselectedn, bar shift=-2pt, fill=orange, draw=orange]
table[x=n,y=piFA] {tikzdata/simulation_uniform_q_005_095_tau3/sim_uniform_q005_095_tau3_dtm_cn1_obj.dat};
\addplot+[onlyselectedn, bar shift=0pt, fill=green!60!black, draw=green!60!black]
table[x=n,y=piSL] {tikzdata/simulation_uniform_q_005_095_tau3/sim_uniform_q005_095_tau3_dtm_cn1_obj.dat};
\addplot+[onlyselectedn, bar shift=2pt, fill=red, draw=red]
table[x=n,y=piEA] {tikzdata/simulation_uniform_q_005_095_tau3/sim_uniform_q005_095_tau3_dtm_cn1_obj.dat};

\nextgroupplot[
    title={$C = 2n$},
    ymin=0,
    ymax=\DTMymax,
    ytick={0,1,2,3,4},
    yticklabel style={
        text width=2.2em,
        align=right,
        /pgf/number format/fixed,
        /pgf/number format/precision=0
    }
]
\addplot+[onlyselectedn, bar shift=-2pt, fill=orange, draw=orange]
table[x=n,y=piFA] {tikzdata/simulation_uniform_q_005_095_tau3/sim_uniform_q005_095_tau3_dtm_cn2_obj.dat};
\addplot+[onlyselectedn, bar shift=0pt, fill=green!60!black, draw=green!60!black]
table[x=n,y=piSL] {tikzdata/simulation_uniform_q_005_095_tau3/sim_uniform_q005_095_tau3_dtm_cn2_obj.dat};
\addplot+[onlyselectedn, bar shift=2pt, fill=red, draw=red]
table[x=n,y=piEA] {tikzdata/simulation_uniform_q_005_095_tau3/sim_uniform_q005_095_tau3_dtm_cn2_obj.dat};

\nextgroupplot[
    title={$C = 3n$},
    ymin=0,
    ymax=\DTMymax,
    ytick={0,1,2,3,4},
    yticklabel style={
        text width=2.2em,
        align=right,
        /pgf/number format/fixed,
        /pgf/number format/precision=0
    }
]
\addplot+[onlyselectedn, bar shift=-2pt, fill=orange, draw=orange]
table[x=n,y=piFA] {tikzdata/simulation_uniform_q_005_095_tau3/sim_uniform_q005_095_tau3_dtm_cn3_obj.dat};
\addplot+[onlyselectedn, bar shift=0pt, fill=green!60!black, draw=green!60!black]
table[x=n,y=piSL] {tikzdata/simulation_uniform_q_005_095_tau3/sim_uniform_q005_095_tau3_dtm_cn3_obj.dat};
\addplot+[onlyselectedn, bar shift=2pt, fill=red, draw=red]
table[x=n,y=piEA] {tikzdata/simulation_uniform_q_005_095_tau3/sim_uniform_q005_095_tau3_dtm_cn3_obj.dat};

\nextgroupplot[
    ylabel={ $\bar{V}^\pi_{\textsf{SLM}_3}(n, C)$},
    ymin=0,
    ymax=\SLMymax,
    ytick={0,0.2,0.4,0.6,0.8,1.0},
    yticklabel style={
        text width=2.2em,
        align=right,
        /pgf/number format/fixed,
        /pgf/number format/precision=1
    }
]
\addplot+[onlyselectedn, bar shift=-2pt, fill=orange, draw=orange]
table[x=n,y=piFA] {tikzdata/simulation_uniform_q_005_095_tau3/sim_uniform_q005_095_tau3_slm_cn1_obj.dat};
\addplot+[onlyselectedn, bar shift=0pt, fill=green!60!black, draw=green!60!black]
table[x=n,y=piSL] {tikzdata/simulation_uniform_q_005_095_tau3/sim_uniform_q005_095_tau3_slm_cn1_obj.dat};
\addplot+[onlyselectedn, bar shift=2pt, fill=red, draw=red]
table[x=n,y=piEA] {tikzdata/simulation_uniform_q_005_095_tau3/sim_uniform_q005_095_tau3_slm_cn1_obj.dat};

\nextgroupplot[
    ymin=0,
    ymax=\SLMymax,
    ytick={0,0.2,0.4,0.6,0.8,1.0},
    yticklabel style={
        text width=2.2em,
        align=right,
        /pgf/number format/fixed,
        /pgf/number format/precision=1
    }
]
\addplot+[onlyselectedn, bar shift=-2pt, fill=orange, draw=orange]
table[x=n,y=piFA] {tikzdata/simulation_uniform_q_005_095_tau3/sim_uniform_q005_095_tau3_slm_cn2_obj.dat};
\addplot+[onlyselectedn, bar shift=0pt, fill=green!60!black, draw=green!60!black]
table[x=n,y=piSL] {tikzdata/simulation_uniform_q_005_095_tau3/sim_uniform_q005_095_tau3_slm_cn2_obj.dat};
\addplot+[onlyselectedn, bar shift=2pt, fill=red, draw=red]
table[x=n,y=piEA] {tikzdata/simulation_uniform_q_005_095_tau3/sim_uniform_q005_095_tau3_slm_cn2_obj.dat};

\nextgroupplot[
    ymin=0,
    ymax=\SLMymax,
    ytick={0,0.2,0.4,0.6,0.8,1.0},
    yticklabel style={
        text width=2.2em,
        align=right,
        /pgf/number format/fixed,
        /pgf/number format/precision=1
    }
]
\addplot+[onlyselectedn, bar shift=-2pt, fill=orange, draw=orange]
table[x=n,y=piFA] {tikzdata/simulation_uniform_q_005_095_tau3/sim_uniform_q005_095_tau3_slm_cn3_obj.dat};
\addplot+[onlyselectedn, bar shift=0pt, fill=green!60!black, draw=green!60!black]
table[x=n,y=piSL] {tikzdata/simulation_uniform_q_005_095_tau3/sim_uniform_q005_095_tau3_slm_cn3_obj.dat};
\addplot+[onlyselectedn, bar shift=2pt, fill=red, draw=red]
table[x=n,y=piEA] {tikzdata/simulation_uniform_q_005_095_tau3/sim_uniform_q005_095_tau3_slm_cn3_obj.dat};

\nextgroupplot[
    ylabel={$V^\pi_{\textsf{EAM}_3}(n,C)$},
    ymin=0,
    ymax=\EAMymax,
    ytick={0,20,40,60,80,100},
    yticklabel style={
        text width=2.2em,
        align=right,
        /pgf/number format/fixed,
        /pgf/number format/precision=0
    }
]
\addplot+[onlyselectedn, bar shift=-3pt, fill=orange, draw=orange]
table[x=n,y=piFA] {tikzdata/simulation_uniform_q_005_095_tau3/sim_uniform_q005_095_tau3_eam_cn1_obj.dat};
\addplot+[onlyselectedn, bar shift=-1pt, fill=green!60!black, draw=green!60!black]
table[x=n,y=piSL] {tikzdata/simulation_uniform_q_005_095_tau3/sim_uniform_q005_095_tau3_eam_cn1_obj.dat};
\addplot+[onlyselectedn, bar shift=1pt, fill=red, draw=red]
table[x=n,y=piEA] {tikzdata/simulation_uniform_q_005_095_tau3/sim_uniform_q005_095_tau3_eam_cn1_obj.dat};
\addplot+[onlyselectedn, bar shift=3pt, fill=purple, draw=purple]
table[x=n,y=piR] {tikzdata/simulation_uniform_q_005_095_tau3/sim_uniform_q005_095_tau3_eam_cn1_obj.dat};

\nextgroupplot[
    ymin=0,
    ymax=\EAMymax,
    ytick={0,20,40,60,80,100},
    yticklabel style={
        text width=2.2em,
        align=right,
        /pgf/number format/fixed,
        /pgf/number format/precision=0
    }
]
\addplot+[onlyselectedn, bar shift=-3pt, fill=orange, draw=orange]
table[x=n,y=piFA] {tikzdata/simulation_uniform_q_005_095_tau3/sim_uniform_q005_095_tau3_eam_cn2_obj.dat};
\addplot+[onlyselectedn, bar shift=-1pt, fill=green!60!black, draw=green!60!black]
table[x=n,y=piSL] {tikzdata/simulation_uniform_q_005_095_tau3/sim_uniform_q005_095_tau3_eam_cn2_obj.dat};
\addplot+[onlyselectedn, bar shift=1pt, fill=red, draw=red]
table[x=n,y=piEA] {tikzdata/simulation_uniform_q_005_095_tau3/sim_uniform_q005_095_tau3_eam_cn2_obj.dat};
\addplot+[onlyselectedn, bar shift=3pt, fill=purple, draw=purple]
table[x=n,y=piR] {tikzdata/simulation_uniform_q_005_095_tau3/sim_uniform_q005_095_tau3_eam_cn2_obj.dat};

\nextgroupplot[
    ymin=0,
    ymax=\EAMymax,
    ytick={0,20,40,60,80,100},
    yticklabel style={
        text width=2.2em,
        align=right,
        /pgf/number format/fixed,
        /pgf/number format/precision=0
    }
]
\addplot+[onlyselectedn, bar shift=-3pt, fill=orange, draw=orange]
table[x=n,y=piFA] {tikzdata/simulation_uniform_q_005_095_tau3/sim_uniform_q005_095_tau3_eam_cn3_obj.dat};
\addplot+[onlyselectedn, bar shift=-1pt, fill=green!60!black, draw=green!60!black]
table[x=n,y=piSL] {tikzdata/simulation_uniform_q_005_095_tau3/sim_uniform_q005_095_tau3_eam_cn3_obj.dat};
\addplot+[onlyselectedn, bar shift=1pt, fill=red, draw=red]
table[x=n,y=piEA] {tikzdata/simulation_uniform_q_005_095_tau3/sim_uniform_q005_095_tau3_eam_cn3_obj.dat};
\addplot+[onlyselectedn, bar shift=3pt, fill=purple, draw=purple]
table[x=n,y=piR] {tikzdata/simulation_uniform_q_005_095_tau3/sim_uniform_q005_095_tau3_eam_cn3_obj.dat};

\end{groupplot}

\path (group c1r3.south west) -- (group c3r3.south east)
    coordinate[midway] (legendmid);

\matrix[
    matrix of nodes,
    anchor=north,
    row sep=0pt,
    column sep=5pt,
    nodes={
        anchor=west,
        inner sep=0pt,
        outer sep=0pt,
        font=\footnotesize
    }
] at ($(legendmid)+(0,-0.9cm)$) {
\tikz{\draw[fill=orange,draw=orange] (0,0) rectangle (0.16,0.10);} &
$\pi_{\textsf{FA}}$ &
\tikz{\draw[fill=green!60!black,draw=green!60!black] (0,0) rectangle (0.16,0.10);} &
$\pi_{\textsf{SL}}$ &
\tikz{\draw[fill=red,draw=red] (0,0) rectangle (0.16,0.10);} &
$\pi_{\textsf{EA}}$ &
\tikz{\draw[fill=purple,draw=purple] (0,0) rectangle (0.16,0.10);} &
$\pi_{\textsf{R}}$ \\
};

\end{tikzpicture}
\caption{
For every combination of resource tightness $\kappa=C/n \in \{1,2,3\}$ and objective $\diamond \in \{\textsf{DTM}, (\textsf{SLM}_3), (\textsf{EAM}_3)\}$, V$^\pi_\diamond(n, \kappa n)$-versus-$n$ is plotted for $\pi \in \{\pi_{\textsf{FA}}$, $\pi_{\textsf{SL}}$, $\pi_{\textsf{EA}}$, $\pi_{\textsf{R}}\}$. Note: the specialized $\pi_{\textsf{R}}$ was only tested for the $(\textsf{EAM}_3)$ objective. 
}
\label{fig:sim_obj_by_cn_tau3}
\end{figure}

\section{Conclusion}
In this paper, we study a dynamic allocation problem in which the decision maker has a limited number of effectors in each round. We formulate the problem as an MDP and analyze the \cref{eq::MinDefuse}, \cref{eq::SurvivalObjective}, and \cref{eq::EffectiveObjective} objectives. We focus on time-oblivious policies $\pi_{\textsf{FA}}$, $\pi_{\textsf{SL}}$, and $\pi_{\textsf{EA}}$ that are simple, scalable, and easy to implement. We show that $\pi_{\textsf{FA}}$ is optimal under homogeneous threats. At the same time, our analysis also identifies that it can perform poorly when threat difficulties are highly heterogeneous. For greedy policies $\pi_{\textsf{SL}}$ and $\pi_{\textsf{EA}}$, we establish a worst-case constant-factor guarantee for $\pi_{\textsf{EA}}$ under \cref{eq::EffectiveObjective}, and we show that no analogous constant-factor bound holds for $\pi_{\textsf{SL}}$ under \cref{eq::SurvivalObjective}. We also show that each greedy policy outperforms $\pi_{\textsf{FA}}$ under its corresponding objective. Our numerical experiments indicate that fair allocation performs generally well and that greedy policies achieve near-optimal performance. Across a broad range of heterogeneous instances, each greedy policy achieves performance within a 1\% relative gap of the optimal policy under its corresponding objective. %

Our model can be extended to the case in which threats have distinct deadlines under the SLM objective. This extension broadens the applicability of the model by allowing threats to differ not only in difficulty, but also in urgency. While the common-deadline setting is the main focus of this paper, the distinct-deadline formulation preserves the same all-or-nothing nature of mission success and continues to yield meaningful structural insights. We provide the formulation and some additional discussion of this extension in Appendix \ref{sec:distinct_deadlines}.

\ACKNOWLEDGMENT{Louis L. Chen and Ang Xu contributed equally to this manuscript.}

\bibliographystyle{informs2014} %
\bibliography{sample} %

\ECSwitch
\ECHead{\centering{Online Appendix}}

\section{\Cref{sec:model} Proofs }\label{sec:proof}

\begin{proof}{Proof of \Cref{theorem:properties_n=2}.}

We first prove property 1, that is, $(x_1^*,x_2^*)=(C-1,1)$ for large enough $\tau$ and $(x_1^*,x_2^*)\neq (C,0)$ for any $\tau$ under both \cref{eq::SurvivalObjective} and \cref{eq::EffectiveObjective}.

\smallskip

\noindent\underline{Property 1 under \cref{eq::SurvivalObjective}.}
We first show that in the first round, allocation $(C-1,1)$ is better than $(C,0)$ under any $\tau$. We consider the following two policies:
    \begin{enumerate}
        \item $\pi_1$: at the first round $t=1$, set $(x_1,x_2)=(C,0)$. At round $t\geq 2$, $\pi_1$ uses the optimal policy.
        \item $\pi_2$: at the first round $t=1$, set $(x_1,x_2)=(C-1,1)$. At round $t\geq 2$, $\pi_2$ uses the optimal policy.
    \end{enumerate}
    Instead of using a single draw $(L_1,L_2)$ to determine the demand of each threat, we sequentially draw ($L_1^{(t)}, L_2^{(t)}$) in each round $t=1,2,\cdots,\tau$, where $L_1^{(t)}\sim \text{Geom}(1-q_1)$ and $L_2^{(t)}\sim \text{Geom}(1-q_2)$ are independent geometric random variables. Then by the memoryless property, for both $\pi_1$ and $\pi_2$, threat $i$ is neutralized in round $t$ if $x_i^{(t)}\geq L_i^{(t)}$. Let $\Delta_\tau:=V_{\textsf{SLM}_\tau}^{\pi_2}-V_{\textsf{SLM}_\tau}^{\pi_1}$ be the difference of survival probability between $\pi_2$ and $\pi_1$ under deadline $\tau$. We want to show $\Delta_\tau \geq 0$. There are six possible cases based on the realization of $(L_1^{(1)},L_2^{(1)})$ at $t=1$. For each case, we identify its contribution to $\Delta_\tau$:
\begin{enumerate}
    \item $L_1^{(1)} \le C-1$ and $L_2^{(1)}=1$: $\pi_2$ neutralizes both threats and $\pi_1$ only neutralizes threat 1. Thus, the probability that $\pi_2$ survives and $\pi_1$ fails to survive is $q_2^{(\tau-1) C}$. Thus, its contribution to $\Delta_\tau$ is $\operatorname{Pr}(L_1^{(1)} \le C-1, L_2^{(1)}=1) q_2^{(\tau-1) C}$. 
    \item $L_1^{(1)} \le C-1$ and $L_2^{(1)}>1$: both neutralize threat 1 only. We know both policies have the same probability to survive in this case. Thus, there is no contribution to $\Delta_\tau$ in this case.
    \item $L_1^{(1)}=C$ and $L_2^{(1)}=1$: $\pi_2$ neutralizes threat 2 only and $\pi_1$ neutralizes threat 1 only. The probability that $\pi_2$ survives and $\pi_1$ fails to survive in this case is $(1-q_1^{(\tau-1) C})q_2^{(\tau-1) C}$. The probability that $\pi_1$ survives and $\pi_2$ fails to survive in this case is $(1-q_2^{(\tau-1) C}) q_1^{(\tau-1) C}$. Thus, the contribution is $\operatorname{Pr}(L_1^{(1)}=C, L_2^{(1)}=1)\left(q_2^{(\tau-1) C}-q_1^{(\tau-1) C}\right)$.
    \item $L_1^{(1)}=C$ and $L_2^{(1)}>1$: $\pi_2$ fails to neutralize any threat and $\pi_1$ neutralizes threat 1 only. The probability that $\pi_1$ survives and $\pi_2$ fails to survive in this case is $V_{\textsf{SLM}_{\tau-1}}^*-(1-q_2^{(\tau-1)C})$. Thus, the contribution is $\operatorname{Pr}(L_1^{(1)}=C, L_2^{(1)}>1)\left(V^*_{\textsf{SLM}_{\tau-1}}-(1-q_2^{(\tau-1)C})\right)$.
    \item $L_1^{(1)} \ge C+1$ and $L_2^{(1)}=1$: $\pi_2$ only neutralizes threat 2 and $\pi_1$ fails to neutralize any threat. The contribution is $\operatorname{Pr}(L_1^{(1)}\geq C+1, L_2^{(1)}=1)\left(1-q_1^{(\tau-1) C}-V_{\textsf{SLM}_{\tau-1}}^*\right)$.
    \item $L_1^{(1)} \ge C+1$ and $L_2^{(1)}>1$: both fail to neutralize any threat. Thus, there is no contribution.
\end{enumerate}
Then we can compute $\Delta_\tau$ by summing over all contributions under each case. We can write the probability for each $(L_1^{(1)},L_2^{(1)})$ case above using the fact $\operatorname{Pr}(L_1^{(1)} \le C-1)=1-q_1^{C-1}$, $\operatorname{Pr}(L_1^{(1)}=C)=\left(1-q_1\right) q_1^{C-1}$, $\operatorname{Pr}(L_1^{(1)} \ge C+1)=q_1^C$, $\operatorname{Pr}(L_2^{(1)}=1)=1-q_2$, and $\operatorname{Pr}(L_2^{(1)}>1)=q_2$. Therefore, $\Delta_\tau$ can be computed as \begin{equation}\label{eq:thm2_slm_delta_tau}
    \begin{aligned}
    \Delta_\tau= & \left(1-q_1^{C-1}\right)\left(1-q_2\right)q_2^{(\tau-1) C} +q_1^{C-1}\left(1-q_1\right)\left(1-q_2\right)(q_2^{(\tau-1) C}-q_1^{(\tau-1) C}) \\
    &+q_1^{C-1}\left(1-q_1\right) q_2\left(V_{\textsf{SLM}_{\tau-1}}^*-(1-q_2^{(\tau-1) C})\right)+q_1^C\left(1-q_2\right)\left((1-q_1^{(\tau-1) C})-V_{\textsf{SLM}_{\tau-1}}^*\right) \\
    \geq & \ q_1^{C-1}\left(1-q_1\right)\left(1-q_2\right)(q_2^{(\tau-1) C}-q_1^{(\tau-1) C}) +q_1^{C-1}\left(1-q_1\right) q_2\left(V_{\textsf{SLM}_{\tau-1}}^*-(1-q_2^{(\tau-1) C})\right)\\
    &+q_1^C\left(1-q_2\right)\left((1-q_1^{(\tau-1) C})-V_{\textsf{SLM}_{\tau-1}}^*\right), \\
    \end{aligned}
    \end{equation}
    where the inequality in equation \eqref{eq:thm2_slm_delta_tau} is derived by removing the first term in the equation, which is non-negative. By moving the terms containing $V_{\textsf{SLM}_{\tau-1}}^*$ to one side, we know that showing the lower bound in equation \eqref{eq:thm2_slm_delta_tau} greater than or equal to 0 is equivalent to proving the following inequality:

    \begin{align}
        V_{\textsf{SLM}_{\tau-1}}^*&\leq 
        \frac{1}{q_1-q_2}\bigg( (1-q_1)(1-q_2) (q_2^{(\tau-1) C}-q_1^{(\tau-1) C})  -q_2\left(1-q_1\right)(1-q_2^{(\tau-1) C}) +q_1\left(1-q_2\right)(1-q_1^{(\tau-1) C}) \bigg) \notag
        \\
        &=\frac{(1-q_1)(1-q_2)}{q_1-q_2}\sum_{i=1}^{C(\tau-1)-1}(q_1^i-q_2^i)=(1-q_1)(1-q_2)\sum_{i=1}^{C(\tau-1)-1}\sum_{j=0}^{i-1}q_1^jq_2^{i-1-j} \notag \\
        &=(1-q_1)(1-q_2)\sum_{i=1}^{C(\tau-1)-1}\sum_{j=1}^{C(\tau-1)-i}q_1^{i-1}q_2^{j-1}, \label{eq:thm2_slm_v_star_ineq}
    \end{align}
where the inequality is derived by moving terms containing $V_{\textsf{SLM}_{\tau-1}}^*$ in equation \eqref{eq:thm2_slm_delta_tau} to one side. The final expression on the right-hand side of equation \eqref{eq:thm2_slm_v_star_ineq}is exactly the survival probability that we receive feedback each time we allocate an effector (i.e., capacity 1 with $C(\tau-1)$ rounds), which is an upper bound of $V_{\textsf{SLM}_{\tau-1}}^*$. To see why, notice that for $i,j=1,\cdots,C(\tau-1)$ and $i+j\leq (\tau-1) C$, the probability $L_1=i$ and $L_2=j$ is exactly $\text{Pr}(L_1=i,L_2=j)= (1-q_1)q_1^{i-1}\ (1-q_2)q_2^{j-1}$. Thus, the overall survival probability can be derived by summing over $i=1,\cdots,C(\tau-1)-1$ and $j=1,\cdots, C(\tau-1)-i$, which is exactly the last term in equation \eqref{eq:thm2_slm_v_star_ineq}. This implies that $\Delta_\tau> 0$ and $\pi_2$ is strictly better than $\pi_1$. 

\medskip

We have proved that $(C-1,1)$ is better than $(C,0)$ under any $\tau$. Now we show that for large enough $\tau$, the first allocation $(x_1,x_2)=(C-1,1)$ is better than $(u,C-u)$ for all $u\leq C-2$. For any split $(u,C-u)$ at $t=1$, we use the optimal continuation policy for $t\ge2$. First, for $u\in\{0,1,\cdots,C\}$, the survival probability with the first allocation $(x_1,x_2)=(u,C-u)$ and the optimal policy for $t\ge 2$ can be written as the following recursive formula:
\begin{align}
V_{\textsf{SLM}_\tau}^*(u):&= (1-q_1^u)(1-q_2^{C-u})+(1-q_1^u)q_2^{C-u}(1-q_2^{(\tau-1)C})+(1-q_2^{C-u})q_1^u(1-q_1^{(\tau-1)C})+q_1^uq_2^{C-u}V_{\textsf{SLM}_{\tau-1}}^* \notag
\\
&=1-q_1^{(\tau-1)C} q_1^u-q_2^{(\tau-1)C} q_2^{C-u}
+\left(q_1^{(\tau-1)C}+q_2^{(\tau-1)C}-1+V_{\textsf{SLM}_{\tau-1}}^*\right) q_1^u q_2^{C-u}. \label{eq:succ_prob_n=2}
\end{align}

We similarly define the $\Delta_\tau(u):=V_{\textsf{SLM}_\tau}^*(C-1)-V_{\textsf{SLM}_\tau}^*(u)$ as the difference between the survival probability under allocation $(C-1,1)$ and $(u,C-u)$. Using the expression of $V_{\textsf{SLM}_\tau}^*(u)$ in \ref{eq:succ_prob_n=2}, we can write $\Delta_\tau(u)$ as 
$$\Delta_\tau(u)=q_1^{(\tau-1)C} \left(q_1^u-q_1^{C-1}+T(u)\right)+q_2^{(\tau-1)C} \left(-q_2^u+q_2^{C-u}+T(u)\right)-(1-V_{\textsf{SLM}_{\tau-1}})T(u),$$
where $T(u):=q_1^{C-1} q_2-q_1^u q_2^{C-u}$. Let $k:=C-u-1\in\{1,2,\ldots,C-1\}$. A direct algebraic simplification gives
\begin{equation}\label{eq:thm2_slm_property2_1}
\begin{aligned}
    \frac{q_1^u-q_1^{C-1}}{T(u)}
&=\frac{1-q_1^{C-u-1}}{q_2\,(q_1^{C-u-1}-q_2^{C-u-1})}=\frac{(1-q_1)\sum_{j=0}^{C-u-2}q_1^j}{q_2(q_1-q_2)\sum_{j=0}^{C-u-1}q_1^{C-u-2-j}q_2^j}\\
&\;\ge\;
\frac{(1-q_1)\cdot (C-u-1) q_1^{C-u-2}}{q_2(q_1-q_2)\cdot (C-u-1) q_1^{C-u-2}}
=\frac{1-q_1}{q_2(q_1-q_2)}.
\end{aligned}
\end{equation}
where the inequality is from $\sum_{j=0}^{k-1}q_1^j\ge k q_1^{k-1}$ and
$\sum_{j=0}^{k-1}q_1^{k-1-j}q_2^j \le \sum_{j=0}^{k-1}q_1^{k-1}=k q_1^{k-1}$.
Then from the definition of $\Delta_\tau(u)$ and equation \eqref{eq:thm2_slm_property2_1}, we can lower bound $\Delta_\tau(u)$ using the following inequality:
\begin{equation}\label{eq:d_tau_u}
    \begin{aligned}
\Delta_\tau(u)
&= q_1^{(\tau-1)C} T(u)\frac{q_1^u-q_1^{C-1}+T(u)}{T(u)}+q_2^{(\tau-1)C} \left(-q_2^u+q_2^{C-u}+T(u)\right)-(1-V_{\textsf{SLM}_{\tau-1}})T(u) \\
&\ge T(u)\Big(q_1^{(\tau-1)C}(1+\frac{1-q_1}{q_2\left(q_1-q_2\right)})-(1-V_{\textsf{SLM}_{\tau-1}})\Big)+q_2^{(\tau-1)C}\left(-q_2^u+q_2^{C-u}+T(u)\right)\\
&\ge T(u)q_1^{(\tau-1)C}\Big(1+\frac{1-q_1}{q_2\left(q_1-q_2\right)}-\frac{1-V_{\textsf{SLM}_{\tau-1}}}{q_1^{(\tau-1)C}}\Big)
 -q_2^{(\tau-1)C}\,\left(1-q_1^{C-1}q_2\right),      
\end{aligned}
\end{equation}
where the last inequality is due to the fact that $-q_2^u+q_2^{C-u}+T(u)=q_1^{C-1} q_2-q_1^u q_2^{C-u}-q_2^u+q_2^{C-u}$ is increasing in $u$, so we can get a lower bound by setting $u=0$. It remains to bound $V_{\textsf{SLM}_{\tau-1}}^*$. We evaluate equation \eqref{eq:succ_prob_n=2} at $u=C-1$ and use the fact $V_{\textsf{SLM}_\tau}^* \ge V_{\textsf{SLM}_\tau}^*(C-1)$ to get the following inequality
\begin{equation}\label{eq:thm2_slm_property_2_1-V}
1-V_{\textsf{SLM}_\tau}^* \le q_1^{(\tau-1)C} q_1^{C-1}\left(1-q_2\right)
+q_2^{(\tau-1)C} \left(1-q_1^{C-1}q_2\right)
+(1-V_{\textsf{SLM}_{\tau-1}}^*) q_1^{C-1} q_2 .
\end{equation}
Now we define $r_\tau:=(1-V_{\textsf{SLM}_{\tau-1}}^*)/q_1^{(\tau-1)C}$ for $\tau=1,2,\cdots$. Then for $\tau\ge1$, divided boths sides of equation \eqref{eq:thm2_slm_property_2_1-V} by $q_1^{\tau C}$, we have
\begin{equation}\label{eq:recursive_comp}
    r_{\tau+1}\le \frac{1-q_2}{q_1}
+\frac{q_2^{(\tau-1)C}}{q_1^{\tau C}}\,\left(1-q_1^{C-1}q_2\right)
+\frac{q_2}{q_1}\, r_\tau.
\end{equation}
On the other hand, note that $q_2^{(\tau-1)C}/q_1^{\tau C}=\left(q_2/q_1\right)^{(\tau-1)C}/q_1^C \to 0$ as $\tau \rightarrow\infty$. Then by applying equation \eqref{eq:recursive_comp} recursively, for any $\epsilon>0$, there exists $\tau_1$ such that for all $\tau\ge\tau_1$,
\begin{equation}\label{eq:delta_bound}
r_\tau=\frac{1-V_{\textsf{SLM}_{\tau-1}}^*}{q_1^{(\tau-1)C}}
\le \frac{1-q_2}{q_1}\frac{1}{1-q_2/q_1}+\epsilon=
\frac{1-q_2}{q_1-q_2}+\epsilon.
\end{equation}
We now evaluate the first term in the lower bound of $\Delta_\tau(u)$ in equation \eqref{eq:d_tau_u}. Note that 
\[
\left(1+\frac{1-q_1}{q_2(q_1-q_2)}\right)-r_\tau+\epsilon\geq
\left(1+\frac{1-q_1}{q_2(q_1-q_2)}\right)-\frac{1-q_2}{q_1-q_2}
=
\frac{(1-q_1)(1-q_2)}{q_2(q_1-q_2)}
>0.
\]
Thus, we can choose $\epsilon\in(0,\frac{(1-q_1)(1-q_2)}{q_2(q_1-q_2)})$ and there exists $\tau_1$ such that for all $\tau\ge\tau_1$, we have
\[
q_1^{(\tau-1)C}\left(1+\frac{1-q_1}{q_2(q_1-q_2)}\right)-(1-V_{\textsf{SLM}_{\tau-1}}^*)
\ge
\eta\,q_1^{(\tau-1)C},
\quad \text{where }
\eta:=\frac{(1-q_1)(1-q_2)}{q_2(q_1-q_2)}-\epsilon>0.
\]
Finally, since $T(u)=q_1^{C-1} q_2-q_1^u q_2^{C-u}$ is minimized over $u\in\{0,\ldots,C-2\}$ at $u=C-2$,
we have $T(u)\ge T(C-2)=q_1^{C-2}q_2(q_1-q_2)>0$, hence for all $\tau\ge\tau_1$, by equation \eqref{eq:d_tau_u}, $\Delta_\tau(u)$ satisfies
\[
\Delta_\tau(u)\ge
\eta\,q_1^{(\tau-1)C}\,q_1^{C-2}q_2(q_1-q_2)
-q_2^{(\tau-1)C}\,q_2\left(1-q_1^{C-1}\right).
\]
For the second term in the lower bound of $\Delta_\tau(u)$ in equation \eqref{eq:d_tau_u}, because $\left(q_2/q_1\right)^{(\tau-1)C}\to 0$ as $\tau\rightarrow\infty$, there exists $\tau_2$ such that for all $\tau\ge\tau_2$, we have
\[
q_2^{(\tau-1)C}\,q_2\left(1-q_1^{C-1}\right)
\le
\frac{1}{2}\eta\,q_1^{(\tau-1)C}\,q_1^{C-2}q_2(q_1-q_2).
\]
Therefore, for all $\tau\ge \max\{\tau_1,\tau_2\}$ and all $u\in\{0,1,\ldots,C-2\}$, we have
$
\Delta_\tau(u)\ge
\frac{1}{2}\eta\,q_1^{(\tau-1)C}\,q_1^{C-2}q_2(q_1-q_2)>0$. Hence, for all $\tau\ge \max\{\tau_1,\tau_2\}$, we have $V_{\textsf{SLM}_\tau}^*(C-1)>V_{\textsf{SLM}_\tau}^*(u)$ for every $u\le C-2$. This proves property 1 under \cref{eq::SurvivalObjective}.

\smallskip

\noindent\underline{Property 1 under \cref{eq::EffectiveObjective}.} 
The proof for \cref{eq::EffectiveObjective} is similar to one for \cref{eq::SurvivalObjective}, but is slightly more complicated in computation. For notation simplicity, define %
$D_\tau^*:=1/(1-q_1)+1/(1-q_2)-V_{\textsf{EAM}_\tau}^*\geq 0$ as the difference between the total expected demand and expected number of effective assignments under $\pi^*$. We also define $W_\tau^*(u):=\tau C-V_{\textsf{EAM}_\tau}^*(u)$, $u\in\{0,1,\dots,C\}$ as the expected total waste incurred by the policy that uses $(u,C-u)$ in round $1$ and then follows an optimal policy from round $2$ onward. The following lemma shows that $W_\tau^*(u)$ has a recursive form based on $D_{\tau-1}^*$.

\begin{lemma}[Recursion form for $W_\tau^*(u)$]
\label{lem:F_closed_form}
For every integer $\tau\ge1$ and every $u\in\{0,\dots,C\}$, $W_\tau^*(u)$ has the following recursive expression:
\[
W_\tau^*(u)
=\tau C-\frac{1}{1-q_1}-\frac{1}{1-q_2}
+q_1^{u}q_2^{C-u}\,D_{\tau-1}^*
+\frac{1}{1-q_1} q_1^{u}(1-q_2^{C-u})\,q_1^{C(\tau-1)}
+\frac{1}{1-q_2} (1-q_1^{u})q_2^{C-u}\,q_2^{C(\tau-1)}.
\]
\end{lemma}

\begin{proof}{Proof of Lemma \ref{lem:F_closed_form}.}
Let $W^{(i)}_\tau:=\tau C-\mathbb E[\min(L_i,\tau C)]$ be the expected number of wastes within $\tau$ rounds given that only threat $i$ is alive at the beginning. Since the expectation of $\min(L_i,\tau C)$ can be computed as $\mathbb E[\min(L_i,n)]=\sum_{k=1}^{n}\Pr(L_i\ge k)=\sum_{k=0}^{n-1}q_i^k=(1-q_i^n)/(1-q_i)$, we can further simplify this quantity as
$W^{(i)}_\tau=\tau C-(1-q_i^{\tau C})/(1-q_i)=\tau C-1/(1-q_i)+q_i^{\tau C}/(1-q_i).$ %
After round $1$, we know that both threats survive with probability $q_1^uq_2^{C-u}$, only threat $1$ survives with probability $q_1^u(1-q_2^{C-u})$, only threat $2$ survives with probability $(1-q_1^u)q_2^{C-u}$, and both threats are neutralized with probability $(1-q_1^u)(1-q_2^{C-u})$.
By summing up the expected waste conditional on each case above, we can write $W_\tau^*(u)$ as 
\[
\begin{aligned}
W_\tau^*(u)
=\,&\mathbb{E}[\text{wastes in round 1}]
+q_1^uq_2^{C-u}\,W_{\tau-1}^*
+q_1^u(1-q_2^{C-u})\,W^{(1)}_{\tau-1}\\
&+(1-q_1^u)q_2^{C-u}\,W^{(2)}_{\tau-1}
+(1-q_1^u)(1-q_2^{C-u})\,C(\tau-1)\\
=\,&C-\frac{1}{1-q_1}(1-q_1^u)-\frac{1}{1-q_2}(1-q_2^{C-u})
+q_1^uq_2^{C-u}\,W_{\tau-1}^*
+q_1^u(1-q_2^{C-u})\,W^{(1)}_{\tau-1}\\
&+(1-q_1^u)q_2^{C-u}\,W^{(2)}_{\tau-1}
+(1-q_1^u)(1-q_2^{C-u})\,C(\tau-1)\\
=\,&C-\frac{1}{1-q_1}(1-q_1^u)-\frac{1}{1-q_2}(1-q_2^{C-u})
+q_1^uq_2^{C-u}\,\left((\tau-1)C-\frac{1}{1-q_1}-\frac{1}{1-q_2}+D_{\tau-1}^*\right)
\\
&+q_1^u(1-q_2^{C-u})\left((\tau-1) C-\frac{1}{1-q_1}+\frac{1}{1-q_1} q_1^{(\tau-1) C}\right)\\
&+(1-q_1^u)q_2^{C-u}\,\left((\tau-1) C-\frac{1}{1-q_2}+\frac{1}{1-q_2} q_2^{(\tau-1) C}\right)
+(1-q_1^u)(1-q_2^{C-u})\,C(\tau-1)\\
=\,&\tau C-\frac{1}{1-q_1}-\frac{1}{1-q_2}+q_1^uq_2^{C-u} D_{\tau-1}^*+\frac{1}{1-q_1}q_1^u(1-q_2^{C-u})q_1^{C(\tau-1)}\\
&+\frac{1}{1-q_2}(1-q_1^u)q_2^{C-u} q_2^{C(\tau-1)},
\end{aligned}
\]
which proves the result. \halmos
\end{proof}

\medskip

Now we prove that in the first round, allocation $(C-1,1)$ is better than $(C,0)$ under \cref{eq::EffectiveObjective} as well. Clearly, for every sample path, w.p.1, the number of effective assignments can not exceed the offline optimal, that is,
$\min \left(\sum_{t=1}^\tau x_1^t, L_1\right)+\min \left(\sum_{t=1}^\tau x_2^t, L_2\right)\;\le\;\min(L_1+L_2,C\tau)$.
Taking expectations yields $V_{\textsf{EAM}_\tau}^*\le \mathbb E[\min(L_1+L_2,C\tau)]$. Thus, we have 
$$        V_{\textsf{EAM}_\tau}^*\le\mathbb E[\min(L_1+L_2,C\tau)]=\mathbb E[L_1+L_2]-\mathbb E[(L_1+L_2-C\tau)^+]=\frac{1}{1-q_1}+\frac{1}{1-q_2}-\mathbb E[(L_1+L_2-C\tau)^+].
$$
Then using the above inequality and
and by the definition of $D_\tau^*$, we can lower bound $D_\tau^*$ as follows:
\begin{align}
D_\tau^*&=\frac{1}{1-q_1}+\frac{1}{1-q_2}-V_{\textsf{EAM}_\tau}^* \;\ge\;\mathbb E[(L_1+L_2-C\tau)^+] \notag \geq \sum_{m=C\tau}^{\infty}\Pr(L_1+L_2>m) \notag\\
&=\sum_{m=C\tau}^{\infty}\left(\Pr(L_2>m)+\sum_{\ell=1}^{m}\Pr(L_2=\ell)\Pr(L_1>m-\ell)\right) \notag\\
&=\sum_{m=C\tau}^{\infty}\left(q_2^{m}+(1-q_2)\sum_{\ell=1}^{m}q_2^{\ell-1}q_1^{m-\ell}\right)=\sum_{m=C\tau}^{\infty}\left(q_2^{m}+\frac{(1-q_2)(q_1^{m}-q_2^{m})}{q_1-q_2}\right) \notag\\
&=\sum_{m=C\tau}^{\infty}\frac{(1-q_2)q_1^{m}-(1-q_1)q_2^{m}}{q_1-q_2}=\frac{\frac{1}{1-q_1}(1-q_2)\,q_1^{C\tau}-\frac{1}{1-q_2}(1-q_1)\,q_2^{C\tau}}{q_1-q_2}. \label{eq:lower_bound_on_D_tau}
\end{align}

To show $(C,0)$ is never a unique optimal allocation in the first round, we apply Lemma~\ref{lem:F_closed_form} with $u=C$ and at $u=C-1$, and then subtract $W_\tau^*(C-1)$ from $W_\tau^*(C)$. After some simplification, we can get the following equation:
\begin{equation}\label{eq:thm2_eam_property1_1}
    \begin{aligned}
W_\tau^*(C)-W_\tau^*(C-1)
= \ & q_1^{C-1}(q_1-q_2)D_{\tau-1}^*
-\frac{1}{1-q_1} q_1^{C-1}(1-q_2)\,q_1^{C(\tau-1)}\\
&+\frac{1}{1-q_2}\Big((1-q_2)-q_1^{C-1}(q_1-q_2)\Big)\,q_2^{C(\tau-1)}.
\end{aligned}
\end{equation}
In equation \eqref{eq:thm2_eam_property1_1}, the only unknown quantity is $D_{\tau-1}^*$. By equation \eqref{eq:lower_bound_on_D_tau}, we multiply $D_{\tau-1}^*$ by $q_1^{C-1}(q_1-q_2)$ and this yields
\begin{equation}\label{eq:thm2_eam_property1_2}
  q_1^{C-1}(q_1-q_2)D_{\tau-1}^*
\;\ge\;
\frac{1}{1-q_1} q_1^{C-1}(1-q_2)\,q_1^{C(\tau-1)}
-\;q_1^{C-1}\frac{1-q_1}{1-q_2}\,q_2^{C(\tau-1)}.  
\end{equation}
Apply equation \eqref{eq:thm2_eam_property1_2} to equation \eqref{eq:thm2_eam_property1_1} and we have
\begin{equation}
\begin{aligned}
W_\tau^*(C)-W_\tau^*(C-1) & \ge
\left[
-\;q_1^{C-1}\frac{1-q_1}{1-q_2}
+\frac{1}{1-q_2}\Big((1-q_2)-q_1^{C-1}(q_1-q_2)\Big)
\right]q_2^{C(\tau-1)}=(1-q_1^{C-1})q_2^{C(\tau-1)}>0.   
\end{aligned}
\end{equation}
Therefore, $(C,0)$ yields a larger waste compared to $(C-1,1)$ for every $\tau$, so there must exist an optimal solution such that $(x_1^*,x_2^*)\neq (C,0)$.
\medskip

Next we prove that $(C-1,1)$ is optimal for large enough $\tau$. We fix $u\in\{0,\dots,C-2\}$. Under \cref{eq::EffectiveObjective}, we similarly define the gap
$\Delta_\tau(u):=W_\tau^*(u)-W_\tau^*(C-1)$. Then by Lemma~\ref{lem:F_closed_form}, $\Delta_\tau(u)$ can be written as

\begin{equation}\label{eq:thm2_eam_property1_delta}
    \begin{aligned}
        \Delta_\tau(u)&=-T(u) D_{\tau-1}^*+\frac{1}{1-q_1}\left(q_1^u(1-q_2^{C-u})-q_1^{C-1}(1-q_2)\right)q_1^{C(\tau-1)}\\
        &\quad +\frac{1}{1-q_2}\left((1-q_1^u)q_2^{C-u} -(1-q_1^{C-1})q_2\right)q_2^{C(\tau-1)},
    \end{aligned}
\end{equation}
where $T(u)=q_1^{C-1}q_2-q_1^uq_2^{C-u}$ has the same expression as previous. By \eqref{eq:thm2_slm_property2_1}, we know the second term in RHS of equation \eqref{eq:thm2_eam_property1_delta} follows
\begin{equation}\label{eq:thm2_eam_property1_delta_2}
   \frac{q_1^{u}(1-q_2^{C-u})-q_1^{C-1}(1-q_2)}{1-q_1}=\frac{1}{1-q_1}T(u)+\frac{q_1^u-q_1^{C-1}}{1-q_1} \ge \left( \frac{1}{1-q_1}+\frac{1}{q_2(q_1-q_2)}\right)T(u). 
\end{equation}
With abuse of notation, we now write $r_\tau:=D_{\tau}^*/q_1^{C\tau}$. Then plug equation \eqref{eq:thm2_eam_property1_delta_2} into equation \eqref{eq:thm2_eam_property1_delta} and we obtain
\begin{equation}\label{eq:thm2_eam_delta_tau}
  \frac{\Delta_\tau(u)}{q_1^{C(\tau-1)}}
\ge T(u)(\frac{1}{1-q_1}+\frac{1}{q_2(q_1-q_2)}-r_{\tau-1})
+\frac{(1-q_1^{u})q_2^{C-u}-(1-q_1^{C-1})q_2}{1-q_2}\left(\frac{q_2}{q_1}\right)^{C(\tau-1)}.  
\end{equation}

Thus, to prove $\Delta_\tau(u)\geq 0$, it remains to show $r_{\tau-1}\leq 1/(1-q_1)+1/(q_2(q_1-q_2))$ for large enough $\tau$. Here, we similarly derive a recursive bound for $r_\tau$. Using the fact $\tau C-1/(1-q_1)+1/(1-q_2)+D_\tau^*=W_\tau^*\le W_\tau^*(C-1)$ and Lemma~\ref{lem:F_closed_form} at $u=C-1$, we have
\begin{equation}\label{eq:thm2_eam_d_tau}
    \begin{aligned}
        r_\tau=\frac{D_\tau^*}{q_1^{C\tau}} &\le \frac{1}{q_1^{C\tau}}\left(W_\tau^*(C-1)-\tau C+(\frac{1}{1-q_1}+\frac{1}{1-q_2})\right)
        \\
        &= q_1^{C-1}q_2\,D_{\tau-1}^*
+\frac{1}{1-q_1} q_1^{C-1}(1-q_2)\,q_1^{C(\tau-1)}
+\frac{1}{1-q_2}(1-q_1^{C-1})q_2\,q_2^{C(\tau-1)}\\
&=\frac{1-q_2}{q_1(1-q_1)} + \frac{q_2}{q_1} r_{\tau-1} + \frac{\frac{1}{1-q_2}(1-q_1^{C-1})q_2}{q_1^{C}}\left(\frac{q_2}{q_1}\right)^{C(\tau-1)}.
    \end{aligned}
\end{equation}
Note that the last term in equation \eqref{eq:thm2_eam_d_tau} geometrically approaches to 0 as $\tau\rightarrow\infty$ because $q_2/q_1<1$. By iterating the inequality \eqref{eq:thm2_eam_d_tau}, we know there exists $\tau_1$ such that for all $\tau\geq \tau_1$,
\begin{equation}\label{eq:eam_rtau_bound}
   r_\tau \le \frac{1-q_2}{q_1(1-q_1)}\frac{1}{1-\frac{q_2}{q_1}}+\epsilon=\frac{1-q_2}{(1-q_1)(q_1-q_2)}+\epsilon. 
\end{equation}

For the first term on right-hand side of equation \eqref{eq:thm2_eam_delta_tau},  it holds that
\[
\frac{1}{1-q_1}+\frac{1}{q_2(q_1-q_2)}-\frac{1-q_2}{(q_1-q_2)(1-q_1)}
=\frac{1-q_2}{q_2(q_1-q_2)}\;>\;0.
\]
Thus, we can similarly choose $\epsilon \in (0,(1-q_2)/\left(q_2(q_1-q_2)\right)$ in equation \eqref{eq:eam_rtau_bound} so that the first term in the right-hand side of equation \eqref{eq:thm2_eam_delta_tau} is non-negative. For the last term in equation \eqref{eq:thm2_eam_d_tau}, because $(q_2/q_1)^{C(\tau-1)}\to 0$ as $\tau\rightarrow\infty$, there exists some deadline $\tau_2$ such that for all $\tau\ge\tau_2$, the following inequality holds for any $u< C-1$:
\begin{equation}\label{eq:eam_rsec_bound}
    \frac{(1-q_1^{C-1})q_2}{1-q_2}\left(\frac{q_2}{q_1}\right)^{C(\tau-1)}
\le \frac12\,\left(\frac{1-q_2}{(1-q_1)(q_1-q_2)}-\epsilon\right)\,T(u).
\end{equation}

Then for all $\tau\ge\max\{\tau_1,\tau_2\}$ and all $u\in\{0,\dots,C-2\}$, by plugging equation \eqref{eq:eam_rsec_bound} into equation \eqref{eq:thm2_eam_delta_tau}, we have $\Delta_\tau(u)\geq0$, i.e., $
W_\tau^*(C-1)\leq W_\tau^*(u)$. Thus, we conclude that for all $\tau\ge\max\{\tau_1,\tau_2\}$,
$W_\tau^*(C-1)=\min_{u\in\{0,1,\dots,C\}}W_\tau^*(u)$. Hence, there exists an optimal policy with first-round allocation $(C-1,1)$ for all $\tau\ge\max\{\tau_1,\tau_2\}$.

\smallskip

\noindent\underline{Proof of properties 2 and 3: $x_1^*$ is increasing in $\tau$ and $x_1^*\geq x_2^*$ for any $\tau$.}
We first prove two lemmas.

\begin{lemma}\label{lemma:b(1-r_2)(1-a)<=a(1-b)(1-r_1)}
    Suppose $1>q_1> q_2>0$ and $C\geq v_1> v_2\geq 0$ for integers $v_1,v_2$. Let $d=v_1-v_2>0$. Then we have
    $q_2^{v_2}\left(1-q_2^d\right)\left(1-q_1^{v_2}\right) \leq q_1^{v_2}\left(1-q_2^{v_2}\right)\left(1-q_1^d\right)$. Moreover, the inequality becomes an equality if and only if $v_2=0$.
\end{lemma}
\begin{proof}{Proof of \Cref{lemma:b(1-r_2)(1-a)<=a(1-b)(1-r_1)}.}
    If $v_2=0$, then both sides are zero, so the claim holds trivially. Thus, we can assume $v_2>0$ without loss of generality. Remember $0<q_2<q_1<1$ implies $0<q_2^{v_2} \leq q_1^{v_2}<1$ and $0<q_2^d \leq q_1^d<1$. Use $1-q_i^d=\left(1-q_i\right) \sum_{k=0}^{d-1} q_i^k$, so
$$
    q_2^{v_2}(1-q_1^{v_2})\left(1-q_2\right) \sum_{k=0}^{d-1} q_2^k < q_1^{v_2}(1-q_2^{v_2})\left(1-q_1\right) \sum_{k=0}^{d-1} q_1^k.
$$
It suffices to prove the above inequality termwise for each $k=0, \ldots, d-1$, which is equivalent to showing
\begin{equation}\label{eq:lemma1_core1}
\frac{q_2^{v_2}}{q_1^{v_2}} \cdot \frac{1-q_1^{v_2}}{1-q_2^{v_2}} \cdot\left(\frac{q_2}{q_1}\right)^k < \frac{1-q_1}{1-q_2}.
\end{equation}
On the other hand, we note that
\begin{equation}\label{eq:q2_q1_ratio}
    \frac{q_2^{v_2}}{q_1^{v_2}} \cdot \frac{1-q_1^{v_2}}{1-q_2^{v_2}} \frac{1-q_2}{1-q_1}=\left(\frac{q_2}{q_1}\right)^{v_2} \cdot \frac{\sum_{i=0}^{v_2-1} q_1^i}{\sum_{i=0}^{v_2-1} q_2^i}=\frac{\sum_{i=0}^{v_2-1} q_1^{i-v_2}}{\sum_{i=0}^{v_2-1} q_2^{i-v_2}}.
\end{equation}
Since we know $q_1>q_2$ and the function $x \mapsto \sum_{j=1}^m x^{-j}$ is strictly decreasing on $(0,1)$, we get
$\sum_{i=0}^{m-1} q_1^{i-m}<\sum_{i=0}^{m-1} q_2^{i-m}$. This implies that the last ratio in equation \eqref{eq:q2_q1_ratio} is strictly less than $1$. Since $(q_2/q_1)^k\leq 1$, we know inequality \eqref{eq:lemma1_core1} holds. 
\halmos
\end{proof}

\begin{lemma}\label{lemma:switch_dominate}
    Given $n=2$, $C\geq 2$,  $1>q_1> q_2>0$ and $0\leq u_1< u_2\leq C$, consider the following two policies:
    \begin{enumerate}
        \item $\pi_1$: at $t=1$, set $(x_1^1,x_2^1)=(u_1,v_1)$ where $v_1=C-u_1$. If no threats are neutralized at $t=1$, then at $t=2$, set $(x_1^2,x_2^2)=(u_2,v_2)$ where $v_2=C-u_2$. If no threats are neutralized at $t\geq 2$, it uses an arbitrary policy (can be any policies, including the optimal one) at $t\geq 3$.
        \item $\pi_2$: at $t=1$, set $(x_1^1,x_2^1)=(u_2,v_2)$. If no threats are neutralized at $t=1$, then at $t=2$, set $(x_1^2,x_2^2)=(u_1,v_1)$. If no threats are neutralized at $t\geq 2$, it uses the same policy as $\pi_1$ at $t\geq 3$.
    \end{enumerate}
    Then we have $V_{\textsf{SLM}_\tau}^{\pi_1}\leq V_{\textsf{SLM}_\tau}^{\pi_2}$ and $V_{\textsf{EAM}_\tau}^{\pi_1}\leq V_{\textsf{EAM}_\tau}^{\pi_2}$ for any $\tau \geq 2$. Moreover, both $V_{\textsf{SLM}_\tau}^{\pi_1}= V_{\textsf{SLM}_\tau}^{\pi_2}$ and $V_{\textsf{EAM}_\tau}^{\pi_1}= V_{\textsf{EAM}_\tau}^{\pi_2}$ hold if and only if $u_1=0$ and $u_2=C$.%
\end{lemma}

\begin{proof}{Proof of \Cref{lemma:switch_dominate}.}
    We show $\pi_2$ is no worse than $\pi_1$ under each objective.
    
\noindent\underline{\cref{eq::SurvivalObjective}.} We compare $\pi_1$ and $\pi_2$ by coupling them on a single draw ($L_1, L_2$) and enumerating the disjoint cases where the outcome can differ. In any scenario where after $t=2$ they leave the same set of survived threats, they either both survive or both fail, because after two rounds the allocation policy is the same for $\pi_1$ and $\pi_2$. Thus the only ways they can differ are exactly the two cases below.
\begin{enumerate}
    \item  $\pi_1$ survives but $\pi_2$ fails to survive. There are two disjoint ways this can happen:
    \begin{enumerate}
        \item $\pi_1$ neutralizes threat 2 in $t=1$ but $\pi_2$ doesn't (since $v_1=C-u_1> C-u_2=v_2$). This is the event $v_2<L_2 \leq v_1$. Since $\pi_1$ survives but $\pi_2$ fails, threat 1 must remain for both $\pi_1$ and $\pi_2$ after $t=1$. Otherwise, both policies can neutralize all threats in two rounds.
        \begin{itemize}
            \item Under $\pi_1$ : after $t=1$, only threat 1 remains; the total additional resources it can get are $C+c_0$ where $c_0:=(\tau-2)C$, so $\pi_1$ survives iff $L_1 \leq u_1+C+c_0$.
            \item Under $\pi_2$ : after $t=1$, both threats remain; at $t=2$ it adds $\left(u_1, v_1\right)$ and then there is $c_0$ units of resource remaining for threat 2. $\pi_2$ fails iff $L_1>u_1+u_2+c_0$.
        \end{itemize}
        So the contribution to the probability that $\pi_1$ survives and $\pi_2$ fails to survive in this case is $\operatorname{Pr}\left(v_2<L_2 \leq v_1, u_1+u_2+c_0<L_1 \leq u_1+C+c_0\right)$.
        \item Both $\pi_1$ and $\pi_2$ neutralize threat 1 at $t=1$.
        This happens under the event $L_1 \leq u_1$.
        \begin{itemize}
            \item Under $\pi_1$ : only threat 2 remains. Thus, $\pi_1$ survives iff $L_2 \leq v_1+C+c_0$.
            \item Under $\pi_2$ : only threat 2 remains. Thus, $\pi_2$ survives iff $L_2 \leq v_2+C+c_0$.
        \end{itemize}
        Thus, to have $\pi_1$ survive and $\pi_2$ fail, we need $v_2+C+c_0<L_2 \leq v_1+C+c_0$. The contribution is $\operatorname{Pr}\left(L_1 \leq u_1, v_2+C+c_0<L_2 \leq v_1+C+c_0\right)$.
    \end{enumerate} 
    Putting the two cases together, we have
    \begin{align*}
       \operatorname{Pr}\left(\pi_1 \text { surv, } \pi_2 \text { fail}\right)=&\operatorname{Pr}\left(v_2<L_2 \leq v_1, u_1+u_2+c_0<L_1 \leq u_1+C+c_0\right)\\&+\operatorname{Pr}\left(L_1 \leq u_1, v_2+C+c_0<L_2 \leq v_1+C+c_0\right).
    \end{align*}
    \item $\pi_2$ survives but $\pi_1$ fails. Again, there are two disjoint ways this can happen:
    \begin{enumerate}
        \item  $\pi_2$ neutralizes threat 1 in round 1 but $\pi_1$ doesn't. This happens under the event $u_1<L_1 \leq u_2$.
        \begin{itemize}
            \item Under $\pi_2$  after $t=1$, only threat 2 remains. Thus, $\pi_2$ survives iff $L_2 \leq v_2+C+c_0$.
            \item Under $\pi_1$: after $t=1$, both threats remain. Thus, $\pi_1$ survives iff $L_2\leq v_1+v_2+c_0$.
        \end{itemize}
        So the contribution is $\operatorname{Pr}\left(u_1<L_1 \leq u_2, v_1+v_2+c_0<L_2 \leq v_2+C+c_0\right)$.
        \item Both only neutralize threat 2. This happens under the event $L_2 \leq v_2$. 
        \begin{itemize}
            \item Under $\pi_2$ : only threat 1 remains. Thus, $\pi_2$ survives iff $L_1 \leq u_2+C+c_0$.
            \item Under $\pi_1$ : only threat 1 remains. Thus, $\pi_1$ survives iff $L_1 \leq u_1+C+c_0$.
        \end{itemize}
        To have $\pi_2$ survive and $\pi_1$ fail we need $u_1+C+c_0<L_1 \leq u_2+C+c_0$. So the contribution is $\operatorname{Pr}\left(L_2 \leq v_2, u_1+C+c_0<L_1 \leq u_2+C+c_0\right)$.

    \end{enumerate}
    Putting  the two cases together, we have
    \begin{align*}
    \operatorname{Pr}\left(\pi_2 \text { surv., } \pi_1 \text{ fail}\right)=&\operatorname{Pr}\left(u_1<L_1 \leq u_2, v_1+v_2+c_0<L_2 \leq v_2+C+c_0\right)\\
    &+\operatorname{Pr}\left(L_2 \leq v_2, u_1+C+c_0<L_1 \leq u_2+C+c_0\right) .    
    \end{align*}
\end{enumerate}

Next, we show that $\operatorname{Pr}\left(\pi_2 \text { surv., } \pi_1 \text{ fail}\right)-\operatorname{Pr}\left(\pi_1 \text { surv., } \pi_2 \text { fail}\right)\geq 0$. That is, we want to show 
\begin{equation}\label{eq:lemma1_switch}
\begin{aligned}
& \operatorname{Pr}\left(v_2<L_2 \leq v_1, u_1+u_2+c_0<L_1 \leq u_1+C+c_0\right)+\operatorname{Pr}\left(L_1 \leq u_1, v_2+C+c_0<L_2 \leq v_1+C+c_0\right) \\
\leq & \operatorname{Pr}\left(u_1<L_1 \leq u_2, v_1+v_2+c_0<L_2 \leq v_2+C+c_0\right)+\operatorname{Pr}\left(L_2 \leq v_2, u_1+C+c_0<L_1 \leq u_2+C+c_0\right) .
\end{aligned}
\end{equation}

For simplicity of notation, let $d:=u_2-u_1=v_1-v_2 \geq 0$. Clearly, for any $a<b$ with $a,b\in\mathbb{Z}_{+}$, we know $\operatorname{Pr}(a<L_i\leq b)=q_i^a-q_i^b$. Then using the fact $v_1=C-u_1$ and $v_2=C-u_2$, the two sides of equation \eqref{eq:lemma1_switch} can be rewritten as
$$
\begin{aligned}
 \mathrm{LHS}&=\left(q_2^{v_2}-q_2^{v_1}\right)\left(q_1^{u_1+u_2+c_0}-q_1^{u_1+C+c_0}\right)+(1-q_1^{u_1})\left(q_2^{v_2+C+c_0}-q_2^{v_1+C+c_0}\right)
 \\
&=q_1^{c_0}\left(q_2^{v_2}-q_2^{v_1}\right)\left(q_1^{u_1+u_2}-q_1^{u_1+C}\right)+q_2^{c_0}\left(1-q_1^{u_1}\right)\left(q_2^{v_2+C}-q_2^{v_1+C}\right),\\
\mathrm{RHS}&=\left(q_1^{u_1}-q_1^{u_2}\right)\left(q_2^{v_1+v_2+c_0}-q_2^{v_2+C+c_0}\right)+(1-q_2^{v_2})\left(q_1^{u_1+C+c_0}-q_1^{u_2+C+c_0}\right)
\\
 &= q_2^{c_0}\left(q_1^{u_1}-q_1^{u_2}\right)\left(q_2^{v_1+v_2}-q_2^{v_2+C}\right)+q_1^{c_0}\left(1-q_2^{v_2}\right)\left(q_1^{u_1+C}-q_1^{u_2+C}\right).
\end{aligned}
$$
Since $q_1^{c_0}, q_2^{c_0}>0$, to show the left-hand side is smaller than or equal to the right-hand side in equation \eqref{eq:lemma1_switch}, it suffices to prove the following two component inequalities 
\begin{equation}\label{eq:lemma1_switch_ineq1}
\begin{aligned}
    \left(q_2^{v_2}-q_2^{v_1}\right)\left(q_1^{u_1+u_2}-q_1^{u_1+C}\right)\leq \left(1-q_2^{v_2}\right)\left(q_1^{u_1+C}-q_1^{u_2+C}\right).
\end{aligned}
\end{equation}
\begin{equation}\label{eq:lemma1_switch_ineq2}
    \left(1-q_1^{u_1}\right)\left(q_2^{v_2+C}-q_2^{v_1+C}\right)\leq \left(q_1^{u_1}-q_1^{u_2}\right)\left(q_2^{v_1+v_2}-q_2^{v_2+C}\right) .
\end{equation}
We first show inequality \eqref{eq:lemma1_switch_ineq1} holds. Using $v_1=v_2+d, u_2=u_1+d, C=u_1+v_2+d$, we can rewrite inequality \eqref{eq:lemma1_switch_ineq1} as $q_1^{2u_1+d}q_2^{v_2}\left(1-q_2^d\right)\left(1-q_1^{v_2}\right) \leq q_1^{2u_1+d}q_1^{v_2}\left(1-q_2^{v_2}\right)\left(1-q_1^d\right)$, which holds directly by lemma \ref{lemma:b(1-r_2)(1-a)<=a(1-b)(1-r_1)}. Similarly, we can rewrite inequality \eqref{eq:lemma1_switch_ineq2} as
$
q_2^{v_1+v_2}q_2^{u_1}\left(1-q_1^{u_1}\right)\left(1-q_2^{d}\right)\leq q_1^{u_1}q_2^{v_1+v_2}\left(1-q_1^{d}\right)\left(1-q_2^{u_1}\right)$. Since $C\geq u_2>u_1\geq 0$, we know inequality \eqref{eq:lemma1_switch_ineq2} also holds by replacing $v_1$ with $u_2$ and $v_2$ with $u_1$ in lemma \ref{lemma:b(1-r_2)(1-a)<=a(1-b)(1-r_1)}. Moreover, if both inequality hold as an equality, we must have $v_2=0$, $u_1=0$. This contradicts to property 1 that requires $(x_1^*,x_2^*)\neq (C,0)$. Therefore, $\pi_2$ is strictly better than $\pi_1$ under \cref{eq::SurvivalObjective}.

\smallskip

\noindent\underline{\cref{eq::EffectiveObjective}.} 
We similarly couple $\pi_1$ and $\pi_2$ on a single draw $(L_1,L_2)$.
Define $A_1$ as the difference of the total expected effective-assignment between $\pi_1$ and $\pi_2$ conditional on $\pi_1$ having strictly more effective assignments than $\pi_2$,
and define $A_2$ analogously for $\pi_2$ over $\pi_1$. In specific, let $x_i^{1,t}$ and $x_i^{2,t}$ be the number of effectors allocated to threat $i$ in round $t$ under $\pi_1$ and $\pi_2$, respectively. Then mathematically, $A_1$ and $A_2$ can be written as
$$A_1:=\mathbb{E} \left[\left(\sum_{i=1}^n\min(\sum_{t=1}^\tau x_i^{1,t},L_i)-\min(\sum_{t=1}^\tau x_i^{2,t},L_i)\right)^+\right], A_2:=\mathbb{E} \left[\left(\sum_{i=1}^n\min(\sum_{t=1}^\tau x_i^{2,t},L_i)-\min(\sum_{t=1}^\tau x_i^{1,t},L_i)\right)^+\right].
$$
Then it is clear that to prove $V_{\textsf{EAM}_\tau}^{\pi_2}-V_{\textsf{EAM}_\tau}^{\pi_1}=A_2-A_1\geq 0$, it is equivalent to prove $A_2-A_1\ge 0$. In the following, we show how we compute $A_1$ and $A_2$.

\begin{enumerate}
  \item Derivation of $A_1$. In this case, $\pi_1$ has more effective assignments than $\pi_2$. We split it into two cases.
  
  \noindent \underline{Case 1:} $\pi_1$ neutralizes threat $2$ at $t=1$ but $\pi_2$ does not. This corresponds to the event $\{v_2<L_2\le v_1,
L_1>u_1+u_2+c_0\}$.
On event $\{v_2<L_2\le v_1\}$, $\pi_1$ neutralizes threat $2$ in round $1$ while $\pi_2$ does not. Since $L_2\le v_1$, threat $2$ is neutralized by $\pi_2$ by the end of $t=2$. Under $\pi_1$, after $t=1$ only threat $1$ remains, so total number of effectors allocated to threat $1$ over all $\tau$ rounds is
$u_1+C(\tau-1)=u_1+C+c_0$. Under $\pi_2$, threat $2$ is neutralized by end of $t=2$, so total number of effectors allocated to threat $1$ is $u_2+u_1+c_0$. Hence $\pi_1$ allocates an extra number of $(u_1+C+c_0)-(u_1+u_2+c_0)=C-u_2=v_2$ effectors to threat $1$. Conditioned on $L_1>u_1+u_2+c_0$, the expected extra effective assignments is
$\mathbb{E}[\min(L_1,v_1+C+c_0)|L_1>u_1+u_2+c_0]=\mathbb{E}[\min(L_1^\prime,v_2)]=(1-q_1^{v_2})/(1-q_1)$ where $L_1^\prime\sim \text{Geom}(1-q_1)$. Therefore, the contribution towards $A_1$ is
\[
\operatorname{Pr}(v_2<L_2\le v_1,
L_1>u_1+u_2+c_0)\mathbb{E}[\min(L_1,v_2)]=(q_2^{v_2}-q_2^{v_1})\,q_1^{u_1+u_2+c_0}\cdot\frac{1-q_1^{v_2}}{1-q_1}.
\]

\smallskip
\noindent\underline{Case 2:} both policies neutralize threat $1$ at $t=1$. This corresponds to the event
$L_1\le u_1, L_2>v_2+C(\tau-1)$. For event $\{L_1\le u_1\}$, both policies neutralize threat $1$ at $t=1$.
Thus, only threat $2$ remains. Under $\pi_1$ the total number of effectors allocated to threat $2$ is $v_1+C(\tau-1)$,
and under $\pi_2$ it is $v_2+C(\tau-1)$, so $\pi_1$ has extra capacity $v_1-v_2=d$.
Conditioned on $L_2>v_2+C(\tau-1)$, the expected extra effective assignments is
$\mathbb{E}[\min(L_2^\prime,d)]=(1-q_2^d)/(1-q_2)$ where $L_2^\prime\sim \text{Geom}(1-q_2)$.
Hence the contribution is
\[
\operatorname{Pr}(L_1\le u_1,
L_2>v_2+C(\tau-1))\mathbb{E}[\min(L_2,d)]=(1-q_1^{u_1})\,q_2^{v_2+C(\tau-1)}\cdot\frac{1-q_2^d}{1-q_2}.
\]
Combine these two cases and we have 
$$
A_1=(q_2^{v_2}-q_2^{v_1})\,q_1^{u_1+u_2+c_0}\cdot\frac{1-q_1^{v_2}}{1-q_1}
+
(1-q_1^{u_1})\,q_2^{v_2+C(\tau-1)}\cdot\frac{1-q_2^d}{1-q_2}.
$$

\smallskip

\item Derivation of $A_2$. In this case, $\pi_2$ has more effective assignments than $\pi_1$. Simialrly, there are exactly two disjoint cases.

\noindent\underline{Case 1:} $\pi_2$ neutralizes threat $1$ at $t=1$ but $\pi_1$ does not. This corresponds to the event
$\{u_1<L_1\le u_2, L_2>v_1+v_2+c_0\}$. Here, the condition $L_2>v_1+v_2+c_0$ implies $L_2>v_1$, so threat $2$ is not neutralized at $t=1$ under both policies.
On the other hand, $u_1<L_1\le u_2$ means $\pi_2$ neutralizes threat $1$ at $t=1$ while $\pi_1$ does not. Under $\pi_1$, no threat is neutralized at $t=1$, so it uses $(u_2,v_2)$ at $t=2$, neutralizing threat $1$ by end of $t=2$. Thus, under $\pi_2$, only threat $2$ remains after $t=1$, so total number of effectors allocated to threat $2$ is $v_2+C(\tau-1)=v_2+C+c_0$. Under $\pi_1$, %
this quantity becomes $v_1+v_2+c_0$. Hence, $\pi_2$ has extra $(v_2+C+c_0)-(v_1+v_2+c_0)=C-v_1=u_1$ allocations on threat $2$. Then conditioned on $L_2>v_1+v_2+c_0$, the expected extra effective assignments equals
$\mathbb{E}[\min(L_2^\prime,u_1)]=(1-q_2^{u_1})/(1-q_2)$.
So the contribution towards $A_2$ is
\[
\operatorname{Pr}(u_1<L_1\le u_2,
L_2>v_1+v_2+c_0)\mathbb{E}[\min(L_2,u_1)]=(q_1^{u_1}-q_1^{u_2})\,q_2^{v_1+v_2+c_0}\cdot\frac{1-q_2^{u_1}}{1-q_2}.
\]

\smallskip
\noindent\underline{Case 2:} both policies neutralize threat $2$ at $t=1$. This corresponds to $\{L_2\le v_2,
L_1>u_1+C(\tau-1)\}$. For event $\{L_2\le v_2\}$, both policies neutralize threat $2$ at $t=1$, and only threat $1$ remains. Under $\pi_2$ total capacity for threat $1$ is $u_2+C(\tau-1)$ and under $\pi_1$ it is $u_1+C(\tau-1)$,
so $\pi_2$ has extra allocation $d=u_2-u_1>0$ towards threat $1$.
Conditioned on $L_1>u_1+C(\tau-1)$, expected extra effective assignments is
$\mathbb{E}[\min(L_1^\prime,d)]=(1-q_1^d)/(1-q_1)$.
Thus the contribution is
\[
\operatorname{Pr}(L_2\le v_2, L_1>u_1+C(\tau-1))\mathbb{E}[\min(L_1,d)]=(1-q_2^{v_2})\,q_1^{u_1+C(\tau-1)}\cdot\frac{1-q_1^d}{1-q_1}.
\]
Combine these two cases and we have %

$$A_2=
(q_1^{u_1}-q_1^{u_2})\,q_2^{v_1+v_2+c_0}\cdot\frac{1-q_2^{u_1}}{1-q_2}
+
(1-q_2^{v_2})\,q_1^{u_1+C(\tau-1)}\cdot\frac{1-q_1^d}{1-q_1}.$$

\end{enumerate}

\smallskip
So far, we have derived the expressions for $A_1$ and $A_2$. It remains to prove $A_2-A_1\ge 0$. For ease of expression, we write $A_2-A_1=q_1^{c_0}(A_{2,b}-A_{1,a})+q_2^{c_0}(A_{2,a}-A_{1,b})$, where
\[
A_{1,a}:=(q_2^{v_2}-q_2^{v_1})q_1^{u_1+u_2}\frac{1-q_1^{v_2}}{1-q_1},
\quad
A_{2,b}:=(1-q_2^{v_2})q_1^{u_1+C}\frac{1-q_1^{d}}{1-q_1},
\]
\[
A_{1,b}:=(1-q_1^{u_1})q_2^{v_2+C}\frac{1-q_2^{d}}{1-q_2},
\quad
A_{2,a}:=(q_1^{u_1}-q_1^{u_2})q_2^{v_1+v_2}\frac{1-q_2^{u_1}}{1-q_2}.
\]

Since $q_1^{c_0},q_2^{c_0}>0$, it suffices to prove $A_{2,b}\ge A_{1,a}$ and $A_{2,a}\ge A_{1,b}$. First, using $q_2^{v_2}-q_2^{v_1}=q_2^{v_2}(1-q_2^d)$ and $u_1+C=u_1+u_2+v_2$, we have
\begin{align*}
A_{2,b}\ge A_{1,a}
&\iff
(1-q_2^{v_2})q_1^{u_1+C}(1-q_1^d)
\ge
(q_2^{v_2}-q_2^{v_1})q_1^{u_1+u_2}(1-q_1^{v_2})\\
& \iff q_1^{v_2}(1-q_2^{v_2})(1-q_1^d)
\ge
q_2^{v_2}(1-q_2^d)(1-q_1^{v_2}),
\end{align*}
which holds from lemma \ref{lemma:b(1-r_2)(1-a)<=a(1-b)(1-r_1)}. Similarly, using $q_1^{u_1}-q_1^{u_2}=q_1^{u_1}(1-q_1^d)$ and $v_2+C=v_1+v_2+u_1$, we have
\begin{align*}
A_{2,a}\ge A_{1,b}
&\iff
(q_1^{u_1}-q_1^{u_2})q_2^{v_1+v_2}(1-q_2^{u_1})
\ge
(1-q_1^{u_1})q_2^{v_2+C}(1-q_2^d)\\
&\iff
q_1^{u_1}(1-q_1^d)(1-q_2^{u_1})
\ge
q_2^{u_1}(1-q_2^d)(1-q_1^{u_1}),
\end{align*}
which holds again from lemma \ref{lemma:b(1-r_2)(1-a)<=a(1-b)(1-r_1)} with $v_2$ replaced by $u_1$. Therefore, we know $A_{2,b}\ge A_{1,a}$ and $A_{2,a}\ge A_{1,b}$. Moreover, if both $A_{2,b}=A_{1,a}$ and $A_{2,a}=A_{1,b}$ hold, since $d=u_2-u_1>0$, we must have $v_2=0$ and $u_1=0$. Then we know $u_2=C-v_2=C$. This completes the proof.
\halmos
\end{proof}

\medskip

Now we go back to prove properties 2 and 3. We prove \cref{eq::SurvivalObjective} and \cref{eq::EffectiveObjective} individually in the following.
We first show the second property under \cref{eq::SurvivalObjective} and \cref{eq::EffectiveObjective}, that is, $x_1^*$ is increasing in $\tau$ and $x_2^*$ is decreasing in $\tau$. Let $x_1^*,x_2^*$ be the optimal allocation at $t=1$ under $\tau$, and $x_1^{\prime},x_2^{\prime}$ be the optimal allocation at $t=1$ under $\tau+1$. If $x_1^*>x_1^{\prime}$, then under (\textsf{SLM}$_{\tau+1}$), let $\pi_1$ be the optimal policy that allocates $(x_1^\prime,x_2^\prime)$ in the first round and $(x_1^*,x_2^*)$ in round 2 if no threats are neutralized in the first round. Let $\pi_2$ be the policy that allocates $x_1^*,x_2^*$ in the first round, and $x_1^{\prime},x_2^{\prime}$ at $t=2$ if no threats are neutralized after $t=1$. Moreover, let $\pi_2$ use the same continuation rule as $\pi_1$ from round 3 onward. Then by lemma \ref{lemma:switch_dominate}, we know $\pi_2$ is strictly better than $\pi_1$, which leads to contradiction. Thus, we know the second property holds.

Finally, we need to show the third property, $x_1^*\geq x_2^*$. By the second property, we know the optimal allocation $x_1^*$ is increasing in $\tau$, so we only need to show $x_1^*\geq x_2^*$ for $\tau=1$. This can be easily shown by Algorithm~\ref{alg::ML_Estim} and Algorithm~\ref{alg:: WM}.
\end{proof}

\medskip

\section{\Cref{sec:fair_alloc} Proofs}

\subsection{Theorem \ref{theorem:fair_is_optimal_when_qi=q}}

\begin{proof}{Proof of Theorem \ref{theorem:fair_is_optimal_when_qi=q}.}
    To prove the theorem, let us recall the relationship between \cref{eq::MinDefuse} and \cref{eq::SurvivalObjective}. Since the \cref{eq::MinDefuse} objective satisfies $\mathbb{E}[T^\pi]= \sum_{t=0}^{\infty} \bigl(1-\operatorname{Pr}(T^\pi \le t)\bigr)$, to show the optimality under \cref{eq::MinDefuse} objective, it is sufficient to show that $\pi_{\textsf{FA}}$ is optimal under \cref{eq::SurvivalObjective} for any $\tau$. 
    Now we split the proof into the following four cases.
\smallskip

\noindent\underline{Case 1: \cref{eq::SurvivalObjective}, $q_i=q$.}
    Under this homogeneous case, we note that the set of live threats $\mathcal{M}_t$ at any time $t$ matters only inasmuch as its cardinality, $|\mathcal{M}_t|$; hence, the following recursion will be useful. Given $m \in \mathbbm{Z}_{>0}$ and $x \in \mathbb{Z}_{\geq 0}^{[m]}$, define
    \[
    V_{\textsf{SLM}_\tau}(m;x) \coloneqq \mathbb{E}\left[V_{\textsf{SLM}_{\tau-1}}^*\left(\sum_{i = 1}^m \xi_i\right) \coloneqq \max_{y \in \mathbb{Z}_{\geq 0}^{[m]}: \sum_i y_i \leq C} V_{\textsf{SLM}_{\tau-1}}\left(\sum_{i = 1}^m \xi_i; y\right)\right],
    \]
    where all $\xi_i \sim Bern(q^{x_i})$ are independently drawn (indicators denoting threats that survive into the next round), and $V_{\textsf{SLM}_{0}}(m;x)$ is defined to be $0$ for all $x\in \mathcal{A}$ when $m \geq 1,$ otherwise it is $1.$ 
    
    We argue by contradiction. Suppose that at the start of the time horizon of length $\tau$ an optimal policy instructs an allocation $x^*$ to be taken that is not fair. This means there exist two threats $i,j \in [n]$ such that $x_j^*=x_i^*+\ell$, for some integer $\ell\geq 2$. We will show that in fact shifting one effector from $j\to i$ (i.e., $x_j=x_j^*-1$ and $x_i=x_i^*+1$) strictly improves the survival probability, i.e., that $ V_{\textsf{SLM}_{\tau}}(m;x) > V_{\textsf{SLM}_{\tau}}(m;x^*),$ presenting a contradiction. In the following, we will use the shorthand $\kappa \coloneqq x_i^* + x_j^* = x_i + x_j,$ and we note that 
    \[
     V_{\textsf{SLM}_\tau}(m;x) = \mathbb{E}\left[f(x_i) \cdot V^*_{\textsf{SLM}_{\tau-1}}(\sum_{k \neq i,j}^m \xi_i) + (1-f(x_i) - q^\kappa) \cdot V^*_{\textsf{SLM}_{\tau-1}}(\sum_{k \neq i,j}^m \xi_i + 1) + q^\kappa \cdot V^*_{\textsf{SLM}_{\tau-1}}(\sum_{k \neq i,j}^m \xi_i + 2)\right],
    \]
    where 
    $f(r)\coloneqq (1-q^{r})(1-q^{\kappa - r})$ defined over integer $r.$ It holds that $f(r) < f(r+1)$ whenever $r \in [0, \lfloor \frac{\kappa}{2}\rfloor - 1]_{\mathbbm{Z}}$. Noting $x_i^* < \lfloor \frac{\kappa}{2} \rfloor$
since by assumption, $x_i^* + 2 \leq x_j^*$, we conclude $f(x_i^*) < f(x_i)$. Since $ V^*_{\textsf{SLM}_{\tau-1}}(\sum_{k \neq i,j}^m \xi_i) > V^*_{\textsf{SLM}_{\tau-1}}(\sum_{k \neq i,j}^m \xi_i + 1)$ with probability 1, we conclude $V_{\textsf{SLM}_{\tau}}(m;x) > V_{\textsf{SLM}_{\tau}}(m;x^*),$ as desired. In summary, at any time, the allocation made should be fair, or else it is suboptimal.
    \smallskip

\noindent\underline{Case 2: \cref{eq::EffectiveObjective}, $q_i=q$.} It is clear that for $\tau = 1,$ any fair allocation is optimal. So we proceed under the assumption that $\tau  - 1 \geq 1,$ without loss of generality.  
The following recursion will be useful:
\[
V_{\textsf{EAM}_\tau}(m;x):= \mathbb{E}\left[\sum_{k=1}^m \min(x_k, L_k)\right]+\mathbb{E}\left[V^*_{\textsf{EAM}_{\tau-1}}(\sum_{k=1}^m \xi_{k})\coloneqq \max_{y \in \mathbb{Z}_{\geq 0}^{[m]}: \sum_{i} y_i \leq C}\!V_{\textsf{EAM}_{\tau-1}}(\sum_{k=1}^m \xi_{k}, y)\right],
\]
so then 
\begin{align*}
     V_{\textsf{EAM}_{\tau}}(m;x) &=  \sum_{k=1}^m\frac{1-q^{x_k}}{1-q} + \mathbb{E} 
     \begin{bmatrix}
         f(x_i) \cdot V^*_{\textsf{EAM}_{\tau-1}}(\sum_{k \neq i,j}^m \xi_i) + (1-f(x_i) - q^\kappa) \cdot V^*_{\textsf{EAM}_{\tau-1}}(\sum_{k \neq i,j}^m \xi_i + 1)\\
         + q^\kappa \cdot V^*_{\textsf{EAM}_{\tau-1}}(\sum_{k \neq i,j}^m \xi_i + 2)
     \end{bmatrix}
\end{align*}
Since 
$\frac{1 - q^{x_i}}{1-q} +  \frac{1 - q^{x_j}}{1-q} = \left(\frac{1-q^\kappa}{1-q} + \frac{f(x_i)}{1-q}\right)$, the desired inequality $V_{\textsf{EAM}_{\tau}}(m;x) > V_{\textsf{EAM}_{\tau}}(m;x^*)$ is equivalent to 
\[
\Big(f(x_i) - f(x_i^*)\Big)\left(\frac{1}{1-q} - \mathbb{E}\left[V^*_{\textsf{EAM}_{\tau-1}}(\sum_{k \neq i,j}^m \xi_i + 1) - V^*_{\textsf{EAM}_{\tau-1}}(\sum_{k \neq i,j}^m \xi_i)\right]\right) > 0. 
\]
As we have already established that $f(x_i) - f(x_i^*) > 0$, and because
\begin{align*}
    V^*_{\textsf{EAM}_{\tau-1}}\left(\sum_{k \neq i,j}^m \xi_i + 1\right) - V^*_{\textsf{EAM}_{\tau-1}}\left(\sum_{k \neq i,j}^m \xi_i\right) \leq \underset{{L \sim Geom(1-q)}}{\mathbb{E}}\left[\min\left(L, (\tau - 1)C\right)\right] = \frac{1-q^{(\tau - 1)C}}{1-q} < \frac{1}{1-q}, \;\;\; 
\end{align*}
we can conclude $V_{\textsf{EAM}_{\tau}}(m;x) > V_{\textsf{EAM}_{\tau}}(m;x^*).$ In summary, at any time, the allocation made should be fair, or else it is suboptimal.

\noindent\underline{Case 3: \cref{eq::SurvivalObjective} and \cref{eq::EffectiveObjective}, $C=2$.}
Let $\tau$ be given.
For $\pi_{\textsf{FA}}$, when $C=2$, no waste occurs in any round prior to the round in which the final threat is neutralized, if all threats are neutralized within $\tau$ periods; in other words, $W^{\pi_\textsf{FA}}_\tau \leq \max(C\tau - \sum_{i=1}^n L_i,0)$. This inequality means the event $\left[\sum_{i=1}^n L_i \leq C\tau\right]$ is equivalent to the event that the number of effective assignments made within $\tau$ periods under $\pi_{\textsf{FA}}$ is $\sum_{i=1}^n L_i$, yielding optimality for \eqref{eq::SurvivalObjective}. The inequality also clearly reveals optimality for \eqref{eq::EffectiveObjective}.

\end{proof}

\medskip

\subsection{\Cref{theorem:DTM_fair_ratio}}

\begin{definition}[Fair Allocation with Fixed Priority Tie-Breaking]
Let $\pi_{\textsf{FA}} \in \Pi_{\textsf{FA}}$. Then we say $\pi$ incorporates a \textit{fixed priority tie-breaking rule} if there exists a strict total order $\prec_{break}$ on $[n]$ such that at any time $t$: whenever $C/M_t \notin \mathbb{Z},$ the remaining $k \coloneqq C - \lfloor C/M_t \rfloor \cdot M_t$ effectors are distributed fairly among the subset $S \subseteq \mathcal{M}_t$ comprised of the $k$- highest priority members of $\mathcal{M}_t$ under $\prec_{break}$. 
\end{definition}

\begin{lemma}[$T^{\pi_{\textsf{FA}}} \text{ monotonicity}$] \label{lem:monotonicity}
    Let $n$ and $C$ be given. Suppose $q, \bar{q} \in [0,1]^n$ such that $q_i \leq \bar{q}_i$ for all $i.$ Let $\pi_{\textsf{FA}} \in \Pi_{\textsf{FA}}$ incorporate either a fixed priority or a uniformly-at-random tie-breaking rule. Then 
    \begin{equation} \label{eq: monotonicity}
        \mathbb{E}\left[ T^{\pi_\textsf{FA}}(q)\right]\;\le\;\mathbb{E}\left[ T^{\pi_\textsf{FA}}(\bar{q})\right] \;\; \text{ and } \;\; \mathbb{E}\left[T^{\pi_{\textsf{FA}}}\left((\min_i q_i) \cdot \mathbbm{1}_n \right)\right] \;\le\; \mathbb{E}\left[T^*(q)\right]
    \end{equation}
\end{lemma}
\begin{proof}{Proof:}
Let us use the shorthands $T^{\pi_\textsf{FA}} = T^{\pi_\textsf{FA}}(q)$, $\bar{T}^{\pi_\textsf{FA}} = T^{\pi_\textsf{FA}}(\bar{q})$, 
and $T^* = T^*(q)$. 

For the first inequality, because $q_i \leq \bar{q}_i$ for all $i$, 
we can couple the collection $\left(L_i \sim Geom(1-q_i)\right)_{i=1}^n$ with the collection $\left(\bar{L}_i \sim Geom(1-\bar{q}_i)\right)_{i=1}^n$ such that $L_i \leq \bar{L}_i$ almost surely, and $\{L_i\}_{i = 1}^n$ are mutually independent, as are $\{\bar{L}_i\}_{i = 1}^n$. Then $\pi_{\textsf{FA}}$ executed on the first collection will yield a nested sequence $\mathcal{M}_1 \supseteq \ldots \supseteq \mathcal{M}_{T^{\pi_{\textsf{FA}}}+1} = \emptyset$  of active threats, and execution on the second collection will yield a nested sequence $\bar{\mathcal{M}}_1 \supseteq \ldots \supseteq \bar{\mathcal{M}}_{\bar{T}^{\pi_{\textsf{FA}}}+1} = \emptyset$. For all $t < \bar{t} \coloneqq \min\{t: \mathcal{M}_t \neq \bar{\mathcal{M}}_t\}$, allocations made in both executions are the exact same. Hence, because $L_i \leq \bar{L}_i$ for all $i,$ it will hold that $\mathcal{M}_{\bar{t}} \subsetneq \bar{\mathcal{M}}_{\bar{t}}.$ We now proceed to argue that
for any $t \geq \bar{t}$, 
\[
i \in \mathcal{M}_t \implies [\pi_{\textsf{FA}}(\mathcal{M}_t)]_i \geq [\pi_{\textsf{FA}}(\bar{\mathcal{M}}_t)]_i.
\]
Let $t$ be such that $M_t \leq \bar{M}_t.$ Then define $k \coloneqq \lfloor C/M_t \rfloor$ and $\bar{k} \coloneqq \lfloor C/\bar{M}_t \rfloor$. We provide an argument for each of the two hypothesized tie-breaking variants. \\
\noindent
{\bf Case 1: $\pi_{\textsf{FA}}$ is assumed to have a fixed priority tie-breaking rule.}
If $k > \bar{k}$, then all threats in $\mathcal{M}_t$ receive at least $k \geq \bar{k} + 1 = \lceil C/\bar M_t \rceil$, where $\lceil C/\bar M_t \rceil$ is the most any threat in $\bar{\mathcal{M}}_t$ can receive, satisfying the claim. If $k = \bar{k}$, then the number of targets receiving $k+1$ effectors in $\mathcal{M}_t$ is the remainder $r = C - kM_t$, and in $\bar{\mathcal{M}}_t$ it is $\bar{r} = C - k\bar{M}_t$, where we note that $r \geq \bar{r}.$ Next let us consider any $i \in \mathcal{M}_t$. If $i$ receives $k+1$ effectors as part of $\bar{\mathcal{M}}_t$, it is because $i$ is among the $\bar{r}$- highest priority members of $\bar{\mathcal{M}}_t$, and hence $i$ must be among the $r$- highest priority members of $\mathcal{M}_t,$ guaranteeing $i$ also receives $k+1$ effectors as part of $\mathcal{M}_t$. \\
\noindent 
{\bf Case 2: $\pi_{\textsf{FA}}$ breaks ties uniformly at random.} We will couple the uniform tie-breaking such that $[\pi_{\textsf{FA}}(\mathcal{M}_t)]_i \geq [\pi_{\textsf{FA}}(\bar{\mathcal{M}}_t)]_i$ for all $i \in \mathcal{M}_t$. If $k > \bar{k}$, then all threats in $\mathcal{M}_t$ receive at least $k \geq \bar{k} + 1 = \lceil C/\bar M_t \rceil$, where $\lceil C/\bar M_t \rceil$ is the most any threat in $\bar{\mathcal{M}}_t$ can receive, satisfying the claim deterministically. 
If $k = \bar{k}$, the policy must select a subset of targets $A \subseteq \mathcal{M}_t$ of size $r = C - kM_t$ to receive $k+1$ effectors, and a subset $B \subseteq \bar{\mathcal{M}}_t$ of size $\bar{r} = C - k\bar{M}_t$ to receive $k+1$ effectors, where we note that $r \geq \bar r.$  We couple the uniform random selections of $A$ and $B$ as follows:
First, draw $B$ uniformly at random from all size-$\bar{r}$ subsets of $\bar{\mathcal{M}}_t$. 
Next, observe the intersection $B \cap \mathcal{M}_t$. Because $|B \cap \mathcal{M}_t| \leq |B| = \bar{r} \leq r$, we can deterministically place all elements of $B \cap \mathcal{M}_t$ into $A$. Finally, we complete $A$ by selecting the remaining $r - |B \cap \mathcal{M}_t|$ elements uniformly at random from $\mathcal{M}_t \setminus B$. 
By the symmetry of permutations on $\mathcal{M}_t$, the resulting marginal distribution of $A$ is uniform over all size-$r$ subsets of $\mathcal{M}_t$, matching the definition of the random policy. Crucially, under this joint distribution, $B \cap \mathcal{M}_t \subseteq A$ with probability 1. Therefore, if any target $i \in \mathcal{M}_t$ receives the extra effector in $\bar{\mathcal{M}}_t$ ($i \in B$), it is guaranteed to receive it in $\mathcal{M}_t$ ($i \in A$). 

In summary, since $\mathcal{M}_{\bar{t}} \subsetneq \bar{\mathcal{M}}_{\bar{t}}$, and both tie-breaking variants guarantee that for all $t \geq \bar{t}$, the policy $\pi_{\textsf{FA}}$ will assign no fewer effectors to any threat $i \in \mathcal{M}_t$ than it does to the exact same threat $i \in \bar{\mathcal{M}}_t$ (either deterministically in case 1, or along the coupled sample path in case 2). Consequently, the $\mathcal{M}_t \subseteq \bar{\mathcal{M}}_t$ is preserved for all $t,$ revealing $T^{\pi_{\textsf{FA}}} \leq \bar{T}^{\pi_{\textsf{FA}}}$ almost surely.

For the second inequality, we note that $\mathbb{E}[T^*(r)]$ is non-decreasing in $r \in [0,1)^n$, i.e., for $r, r' \in [0,1)^n,$
\[
r \leq r' \implies \mathbb{E}[T^*(r)] \leq \mathbb{E}[T^*(r')].
\]
Furthermore, since $(\min_i q_i) \cdot \mathbf{1}_n$ presents a homogeneous (Salvo) instance, \Cref{theorem:fair_is_optimal_when_qi=q}  establishes that the fair allocation policy is optimal for this instance; hence, 
\[
\mathbb{E}\left[T^{\pi_{\textsf{FA}}}\left((\min_i q_i) \cdot \mathbbm{1}_n \right)\right] = \mathbb{E}\left[T^*\left((\min_i q_i) \cdot \mathbbm{1}_n \right)\right] \leq \mathbb{E}\left[T^*(q)\right].
\]

\end{proof}

\subsubsection{Proof of \Cref{theorem:DTM_fair_ratio}.}
\begin{proof}{Proof of \Cref{theorem:DTM_fair_ratio}.}

We first prove the competitive ratio, which directly yields the first performance ratio bound. Then we prove the constant ratio bound when $q_i\in [\underline{q},\bar{q}]$.

\noindent\underline{Competitive ratio.}
We consider the order in which threats are neutralized. Let us rank threats by their death times: the \emph{$k$th-from-last} threat is the one that dies when exactly $k$ threats (including itself) remain alive just before the round begins (ties broken arbitrarily). Under $\pi_{\textsf{FA}}$, when $m$ threats are alive at the beginning of a round, each alive threat receives either $\lfloor C/m\rfloor$ or $\lceil C/m\rceil $ units in that round. Therefore, in the round that the $k$th-from-last threat is neutralized, the number of alive threats at the beginning of that round is at least $k$, and the number of effectors allocated to the threat in the round is at most $\lceil C/k \rceil$. 
Let $W_{i}^{\pi}:=\max(\sum_{t=1}^\tau x_i^t-L_{i},0)$ be the number of wasted effectors for threat $i$. Since $W_{\tau,i}^{\pi}$ is at most the amount of effectors allocated to threat $i$ minus 1 in its neutralized round, we obtain the following bound
\begin{align*}
  \sum_{i=1}^n W_i^{\pi_{\textsf{FA}}}
\;&\le\;\sum_{k=1}^n (\lceil \frac{C}{k}\rceil  -1)=\sum_{k=1}^n\left\lfloor\frac{C-1}{k}\right\rfloor=\sum_{k=1}^{n \wedge (C-1)}\left\lfloor\frac{C-1}{k}\right\rfloor \leq (C-1) \left(1 + \ln( n \wedge C)\right)
\end{align*}
Hence, 
\begin{align*}
    T^{\pi_{\textsf{FA}}}&\le \left\lceil\frac{1}{C}\left(\sum_{i=1}^n L_i \;+\; \sum_{i=1}^n W_i^{\pi_{\textsf{FA}}} \right)\right\rceil
    \leq \left\lceil\frac{1}{C}\sum_{i=1}^n L_i\right\rceil \;+\; \left\lceil\frac{1}{C}\sum_{i=1}^n W_i^{\pi_{\textsf{FA}}}\right\rceil\\
        &\leq \left\lceil\frac{1}{C}\sum_{i=1}^n L_i\right\rceil \; + \; \frac{(C-1) \left(1 + \ln( n \wedge C)\right)}{C} + \frac{C-1}{C},
\end{align*}
which, combined with $T^*\geq \left\lceil\frac{1}{C}\sum_{i=1}^n L_i\right\rceil$, yields both
\[
\frac{T^{\pi_{\textsf{FA}}}}{T^*} \leq 1 + \frac{C(\ln(n\wedge C) +2 )}{\sum_{i=1}^n L_i} \leq 1 + \frac{C(\ln(n\wedge C) +2 )}{n} \text{ and } \frac{\mathbb{E}\left[T^{\pi_{\textsf{FA}}} \right]}{\mathbb{E}\left[T^*\right]} \leq 1 + \frac{C(\ln(n\wedge C) +2 )}{\sum_{i=1}^n \frac{1}{1-q_i}}
\]

\noindent\underline{Constant ratio bound.}
Let $n, C$ and $q \in [\underline{q}, \bar{q}]^n$ be given, with $T^{\pi_{\textsf{FA}}} \coloneqq T^{\pi_{\textsf{FA}}}(q)$ denoting the (random) time to defuse under $\pi_{\textsf{FA}},$ and $T^* \coloneqq T^*(q)$ the optimal (random) time to defuse. Next, we define the shorthands $\bar{T}^{\pi_\textsf{FA}} \coloneqq T^{\pi_{\textsf{FA}}}(\bar{q}\cdot \mathbbm{1}^n)$, respectively $\underline{T}^{\pi_\textsf{FA}} \coloneqq T^{\pi_{\textsf{FA}}}(\underline{q}\cdot \mathbbm{1}^n)$.

By \Cref{lem:monotonicity} we have $\mathbb{E}[T^{\pi_\textsf{FA}}]/\mathbb{E}[T^*]\leq \mathbb{E}[\bar{T}^{\pi_\textsf{FA}}]/\mathbb{E}[\underline{T}^{\pi_\textsf{FA}}]$, and it remains to show $\mathbb{E}[\bar{T}^{\pi_\textsf{FA}}]/\mathbb{E}[\underline{T}^{\pi_\textsf{FA}}]\leq \lceil\ln \underline{q}/\ln \bar{q}\rceil$. Let $\pi^*_{\underline{q}}$ denote the optimal (fair) policy for the problem in which all $n$ threats have difficulty $\underline{q}$ and we get feedback every $C$ assignments. Then let $\pi$ denote a new policy in which feedback is considered at the slower rate of every $\hat C$ assignments and batches of $\hat C:= C\cdot \lceil \log_{\bar{q}}(\underline{q})\rceil$ assignments are fired, where given a set $\mathcal{M}$ of live threats, $\pi(\mathcal{M}) = \lceil \log_{\bar{q}}(\underline{q})\rceil \cdot \pi_{\textsf{FA}}(\mathcal{M})$. 

Let $\hat q:= \bar{q}^{\lceil \log_{\bar{q}}(\underline{q})\rceil} \leq \underline{q}$. Then there exists a coupling such that $T_{\bar{q}}^\pi$ and $T_{\hat q}^{\pi_{\textsf{FA}}}$ satisfy
\[
T_{\bar{q}}^\pi = \lceil \log_{\bar{q}}(\underline{q})\rceil \cdot T_{\hat q}^{\pi_{\textsf{FA}}}.
\]

Therefore,
\[
\mathbb{E}\left[T^{\pi_{\textsf{FA}}}_{\bar{q}}\right] \leq \mathbb{E}\left[T^\pi_{\bar{q}}\right] = \lceil \log_{\bar{q}}(\underline{q})\rceil \cdot \mathbb{E}\left[T_{\hat q}^{\pi_{\textsf{FA}}}\right] \leq \mathbb{E}\left[T^{\pi_{\textsf{FA}}}_{\underline{q}}\right] \cdot \lceil \log_{\bar{q}}(\underline{q})\rceil.
\]
\halmos

\end{proof}

\subsection{\Cref{thm:spread-upper}}
\begin{definition}[Defuse Recognition Time] \label{def::RecognitionTime}
    For any policy $\pi \in \Pi,$ and instance $(n,C,q)$ we define the (random) \textit{defuse recognition time} of policy $\pi$, along with accompanying filtration $\{\mathcal{F}_t\}_{t=1}^\infty$ via 
    \begin{equation}
        \mathbb{S}^{\pi} \coloneqq \min\{t \geq 1: \mathcal{M}_t = \emptyset\} 
        = T^{\pi} + 1, ~~~~~~~ \mathcal{F}_t \coloneqq \sigma\big(\mathcal{M}_1, \ldots, \mathcal{M}_t\big).
    \end{equation}
   $\mathbb{S}^{\pi} $ is the first period $t$ in which the decision maker recognizes that there are no longer any active threats. We note that $\{\mathbb{S}^\pi \leq t\} \in \mathcal{F}_t$ for every $t$, i.e., $\mathbb{S}^\pi$ is a stopping time.
\end{definition}
 
To prove \Cref{thm:spread-upper}, it suffices to prove all claims for $\mathbb{S}^{\pi_{\textsf{FA}}},$ as it is a fixed unit away from $T^{\pi_{\textsf{FA}}}.$ More precisely, throughout, the strategy is to compare this stochastic stopping time against a deterministic \textit{hitting time}, defined as follows. 
\begin{definition}[Hitting Time]
    Let $F: \mathbb{R} \rightarrow \mathbb{R}$ be an increasing function. Given any $a_1 \in\mathbb{R},$ there is an associated recursive sequence $\{a_k\}_{k=1}^\infty$ in which $a_{k+1} = F(a_k)$ for $k \geq 1.$ We define the $\textit{hitting time}$ of target $y$ by $F$ starting at $a_1$ via
    \begin{equation} \label{eq: DefHittingTime}
    s_F(a_1; y):=\min\{k \geq 1: a_k\ge y\}
\end{equation}
\end{definition}
The following lemma establishes the growth of this hitting time with respect to the target $y$ by a particularly fast-growing function.

\begin{lemma}[Hitting time under iterated sequences]\label{lem:det-correct}
Let $b>1$, $K>0$, and define $F(z):=Kzb^z$. Let $z_1>\max(0,-\log_b K)$, and define $z_{k+1}=F(z_k)$.
Then, as $y\to\infty$, $s_F(z_1; y)= \Theta(\ln^* y)$
\end{lemma}
\begin{proof}{Proof of Lemma \ref{lem:det-correct}.}
Because $K b^{z_1}>1$, we have $z_2 = z_1 K b^{z_1} > z_1$. Inductively, because the sequence is increasing ($z_k > z_1$), we have $b^{z_k} \geq b^{z_1}$, which yields $z_{k+1} = z_k K b^{z_k} \geq z_k K b^{z_1} \geq z_1 (K b^{z_1})^k$. Thus the sequence $z_k$ is strictly increasing and grows without bound.
Since $\ln F(z) = z \cdot\ln b + \ln(Kz),$ there exists an $A > 0$ such that for all $z$ sufficiently large,
\begin{equation}\label{eq:natural-exp-sandwich}
b^z \le F(z)\le \exp(Az)~~  \text{ and }  ~~ z \leq (\ln b) b^z;
\end{equation}
hence, for all sufficiently large $k$, by \eqref{eq:natural-exp-sandwich}, 
\[
\exp(z_k) \leq b^{b^{z_k}} \leq b^{F(z_k)} \leq F( F(z_k) ) = z_{k+2},
\]
so that for all sufficiently large $k,$ the sequence $z_k$ keeps pace with the $j$-fold tower of exponentials; more precisely, $\exp^{(j)}(z_k) \leq z_{k+2j}$, which means $s_F(z_1, y) = O(\ln^*y)$. On the other hand, to see that $s_F(z_1, y) = \Omega(\ln^*y)$, note that there exists a constant $B$ such that $\ln^*(Az) \leq \ln^*z + B$ for all sufficiently large $z$ so that by \eqref{eq:natural-exp-sandwich}, for all sufficiently large $k,$
\[
\ln^*z_{k+1}  = \ln^* F(z_k) \leq \ln^*(\exp(Az_k)) \leq 1 + \ln^*(Az_k) \leq 1 + \ln^*z_k + B, 
\]
revealing $\ln^* z_k = O(k).$ Then $\ln^*y \leq \ln^*z_{s_F(z_1, y)} = O(s_F(z_1, y))$, which implies $s_F(z_1, y) = \Omega(\ln^*y).$ \halmos

\end{proof}

\subsubsection{Proof of \Cref{thm:spread-upper}}
\begin{proof}{Proof of \Cref{thm:spread-upper}}
Statements 1 and 2 will follow from the two steps (a) and (b), proven below. Both will leverage \Cref{lem:det-correct}.

For step (a), we will use $G(z) \coloneqq (\bar{q}/2)z(1/\bar{q})^z$ and $z_1 \coloneqq -\log_{1/\bar{q}}(\bar{q}/2) + 1.$
\underline{{\bf Step (a): $C \geq n \implies \mathbb{E}\left[\mathbb{S}^{\pi_{\textsf{FA}}}\right] \leq 1 + z_1 / (1-\bar{q}) + 2s_G(z_1, n) + 1/(1-\bar{q})$:}} Let us define policy $\lfloor \pi_{\textsf{FA}}\rfloor$ as follows: at any period $t$, if $\mathcal{M}_t$ is the live-threat set, then commit $\lfloor n/M_t \rfloor$ to every member of $\mathcal{M}_t$. Since $C \geq n$, it holds that, given any live-threat set, $\pi_{\textsf{FA}}$ commits no fewer than $\lfloor \pi_{\textsf{FA}}\rfloor$ would to every live threat, meaning $\mathbb{S}^{\pi_{\textsf{FA}}} \preceq \mathbb{S}^{\lfloor \pi_{\textsf{FA}}\rfloor},$ so it suffices to upper bound $\mathbb{E}\left[\mathbb{S}^{\lfloor \pi_{\textsf{FA}}\rfloor}\right]$. In what follows, let $\{\mathcal{M}_t\}_{t=1}^{\mathbb{S}^{\lfloor \pi_{\textsf{FA}}\rfloor}}$ be the nested family of live-threat sets throughout the execution of $\lfloor \pi_{\textsf{FA}}\rfloor$, and let $Z_t \coloneqq (n/M_t)\mathbbm{1}_{\{M_t\ge1\}}+\infty\cdot\mathbbm{1}_{\{M_t=0\}}$.
Then let $\tau_n \coloneqq \inf\{t \geq 1: Z_t \geq n\}$ so that $\mathbb{S}^{\lfloor\pi_{\textsf{FA}}\rfloor} \preceq \tau_n + Geom(1 - \bar{q})$, and we proceed to upper bound the expected time $\mathbb{E}\left[\tau_n\right]$ that it takes for $Z_t$ to cross $n.$ 

$Z_t$ increases at each $t$ in which $M_t$ decreases; intuitively, the rate of this decrease must grow, as remaining threats get higher per-round allocations. Towards characterizing recursive growth, we can use the (recursive) upper bound on the expected number of live threats in the next round,
\[
\mathbb{E}\left[M_{t+1} \;| \; \mathcal{M}_t\right] \leq M_t \bar{q}^{\lfloor Z_t \rfloor} \leq M_t \bar{q}^{Z_t - 1},
\]
to fashion an increasing function $G$ such that the growth of $Z_t$ will, with high probability, match, recursively; more precisely, the Markov inequality translates this inequality to the probability bound
\begin{align} \label{eq: GrowthOfZ}
    \mathbb{P}\left(Z_{t+1} \geq G(Z_t)\; | \; \mathcal{M}_t\right) = 1 - \mathbb{P}\left(M_{t+1} > \frac{n}{G(Z_t)} \; | \; \mathcal{M}_t\right) \geq 1 - \frac{\mathbb{E}\left[M_{t+1} \; |\; \mathcal{M}_t\right]}{n} G(Z_t)\geq 1 - \frac{M_t \bar{q}^{Z_t - 1}}{n} G(Z_t) \geq 1/2,
\end{align}
when $G(z) \coloneqq (\bar{q}/2)z(1/\bar{q})^z$. The function $G$ belongs to the class of functions examined in \Cref{lem:det-correct} so that the recursion $G^{(k)}(z_1)$, at some suitably large $z_1$, grows rapidly (a $k$- tower of exponentials) with $k$; intuitively, if $Z_t$ were to keep pace with this (at some frequency), it would share its growth. 

Following \Cref{lem:det-correct}, let us define the recursive sequence $\{z_k\}_{k=1}^\infty$ via
\[
z_1 \coloneqq -\log_{1/\bar{q}}(\bar{q}/2) + 1 > 1, \text{ and } z_{k+1}\coloneqq G(z_k) ~~~~~ \forall k \geq 1
\]
which will be used as a reference for the growth of $Z_t,$ and abbreviate $s_G(z_1, n)$ with $s_G(n).$ Indeed, if $\tau_1 \coloneqq \inf\{t \geq 1: Z_t \geq z_1\}$ is the time it takes for $Z_t$ to get suitably large, then 
\[
\tau_n \leq \tau_1 + \inf\{k\geq 1: Z_{\tau_1 + k} \geq n\}.
\] 
The first term expectation $
\mathbb{E}[\tau_1]\leq 1+\frac{z_1}{1-\bar q};$ indeed, since every live threat receives at least
$\left\lfloor \frac{n}{M_t}\right\rfloor\geq 1$
effectors under $\lfloor\pi_{\textsf{FA}}\rfloor$, we have $\mathbb{E}[M_t]\leq n\bar q^{t-1}$ for $ t\geq 1.$ Therefore,
\[
\mathbb{E}[\tau_1 - 1] = \sum_{t=1}^{\infty}\mathbb{P}(\tau_1>t) \leq \sum_{t=1}^{\infty}\mathbb{P}(Z_t<z_1) = \sum_{t=1}^{\infty}\mathbb{P}\left(M_t>\frac{n}{z_1}\right) \leq \sum_{t=1}^{\infty}\frac{z_1}{n}\mathbb{E}[M_t] \leq \sum_{t=1}^{\infty}z_1\bar q^{t-1} = \frac{z_1}{1-\bar q}.
\]
The second term $\inf\{k\geq 1: Z_{\tau_1 + k} \geq n\} \preceq \operatorname{NegBinom}(s_G(n), 1/2)$, meaning its expectation is upper-bounded by $2s_G(n)$. To see this, let $1 \leq k_1 < k_2 < \dots < k_{s_G(n)}$ be the (random) subsequence that marks the first $s_G(n)$-many time stamps after $\tau_1$ in which $Z_t \geq G(Z_{t-1})$. 
Then 
\[
Z_{\tau_1 + k_1} \geq G(Z_{\tau_1 + k_1 - 1}) \geq G(z_1) \implies
Z_{\tau_1 + k_2} \geq G(Z_{\tau_1 + k_2 - 1}) \geq G(Z_{\tau_1 + k_1}) \geq G^{(2)}(z_1),
\]
which proceeds until we find $Z_{\tau_1 + k_{s_G(n)}} \geq G^{s_G(n)}(z_1) =  z_{s_G(n) + 1}$. Consequently, $\inf\{k\geq 1: Z_{\tau_1 + k} \geq n\}$ is upper-bounded by $k_{s_G(n)}$, which, by the construction of $G$ in \eqref{eq: GrowthOfZ}, is dominated by $ \operatorname{NegBinom}(s_G(n), 1/2)$, whose expectation is $2s_G(n) = \Theta\left(\ln^*n\right)$, as claimed. \\

For step (b) in the following, we will use $H(z) \coloneqq (2/\underline{q}) z (1/\underline{q})^{z}$ and $\nu_1 \coloneqq 1.$

\underline{{\bf Step (b): $\lim_{n\rightarrow \infty} C(n) = \infty, \;C(n) \leq n \implies \mathbb{E}\left[\mathbb{S}^{\pi_{\textsf{FA}}}\right] = \Omega(\ln^*C(n))$: }} Let us define a policy $\lceil \pi_{\textsf{FA}}\rceil$ as follows: at any period $t$, if $\mathcal{M}_t$ is the live-threat set, then commit $\lceil C(n)/M_t \rceil$ to every member of $\mathcal{M}_t$. Technically this is not a feasible policy, but because $\mathbb{S}^{\lceil \pi_{\textsf{FA}}\rceil} \preceq \mathbb{S}^{\pi_{\textsf{FA}}},$ it will suffice for lower-bounding $\mathbb{E}\left[\mathbb{S}^{\lceil \pi_{\textsf{FA}}\rceil}\right].$ In what follows, now let  $\{\mathcal{M}_t\}_{t=1}^{\mathbb{S}^{\lceil \pi_{\textsf{FA}}\rceil}}$ be the nested family of live-threat sets throughout the execution of $\lceil\pi_{\textsf{FA}}\rceil$. Our strategy here will be to show that there exists some (later defined) $\ell_n = \Theta(\ln C(n))$ for which the recursive sequence of a function $H$ from the class in \Cref{lem:det-correct} will, with high probability, outpace and reach before the sequence $Z_t \coloneqq C(n)/M_t$ does. By \Cref{lem:det-correct}, this then means it takes the $Z_t$ sequence at least $\Theta(\ln^* \ell_n)$ -time just to reach $\ell_n$ which establishes the $\Omega(\ln^*C(n))$ claim.

In analogous fashion to the design of the upper bound, we can leverage the (recursive) lower bound on the expected number of live threats in the next round,
\[
\mathbb{E}\left[M_{t+1} \;| \; \mathcal{M}_t\right] \geq M_t \underline{q}^{\lceil Z_t \rceil} \geq M_t \underline{q}^{Z_t + 1},
\]
to fashion an increasing function $H$ such that the growth of $Z_t$ will not exceed, for all periods up until $\ell_n$, with high probability; more precisely, the multiplicative Chernoff bound for sums of independent Bernoulli variables translates this inequality to the probability lower bound
\begin{equation} \label{eq: DecreasingZProbBound}
\mathbb{P}\left(Z_{t+1} \leq H(Z_t)\; |\; \mathcal{M}_t\right) \geq \mathbb{P}\left(M_{t+1} \geq \frac{1}{2} \mathbb{E}\left[M_{t+1}\;|\; \mathcal{M}_t\right]\; | \; \mathcal{M}_t\right) \geq 1 - \exp(-\frac{M_t \underline{q}^{Z_t + 1}}{8}), 
\end{equation}
upon defining $H(z)\coloneqq (2/\underline{q}) z (1/\underline{q})^{z}$, since 
\[
M_{t+1} \geq \frac{1}{2} \mathbb{E}\left[M_{t+1}\;|\; \mathcal{M}_t\right] \implies Z_{t+1} \leq \frac{2C(n)}{\mathbb{E}\left[M_{t+1}\;|\; \mathcal{M}_t\right]} \leq 
\frac{2C(n)}{M_t\underline{q}^{Z_t + 1}} 
= (2/\underline{q}) Z_t (1/\underline{q})^{Z_t} = H(Z_t).
\]

Following \Cref{lem:det-correct}, we define the recursive sequence $\{\nu_k\}_{k=1}^\infty$ via
\[
\nu_1 \coloneqq 1 > -\log_{1/\underline{q}}(2/\underline{q}),
\text{ and } \nu_{k+1} \coloneqq H(\nu_k) ~~~~\forall k \geq 1
\]
which will be used as a reference for the growth of $Z_t$, and abbreviate $s_H(\nu_1, \ell_n)$ with $s_H(\ell_n)$. Since $\mathbb{E}\left[\mathbb{S}^{\lceil \pi_{\textsf{FA}} \rceil}\right] \geq (s_H(\ell_n) - 1) \cdot \mathbb{P}\left(\mathbb{S}^{\lceil \pi_{\textsf{FA}} \rceil} \geq s_H(\ell_n) - 1\right)$,
it suffices to show $\mathbb{P}\left(\mathbb{S}^{\lceil \pi_{\textsf{FA}} \rceil} \geq s_H(\ell_n) - 1\right) \rightarrow 1$ for some choice of $\ell_n = \Theta(\ln C(n))$, which we proceed to do so using \eqref{eq: DecreasingZProbBound}.
Since $\ell_n = \Theta(\ln C(n))$ it follows that, for sufficiently large $n,$ the inequality $\ell_n < C(n)$ holds so that $\nu_{s_{H}(\ell_n) - 1} < \ell_n < C(n)$, providing the lower bound
\[
\mathbb{P}\left(\mathbb{S}^{\lceil \pi_{\textsf{FA}} \rceil} \geq s_H(\ell_n) - 1\right) \geq \mathbb{P}\left(Z_{s_{H}(\ell_n) - 1} \leq \nu_{s_{H}(\ell_n) - 1}\right).
\]
Note that, with the shorthand $A_t \coloneqq [Z_{t+1} \leq \nu_{t+1}]$ for $t\in \mathbb{Z}_{\geq 0},$ and the fact $Z_1 = C(n) / M_1 = C(n) / n \leq 1 = \nu_1,$ we can in turn lower bound the likelihood of this event via 
\begin{align*}
    &\mathbb{P}\left(Z_{s_{H}(\ell_n) - 1} \leq \nu_{s_{H}(\ell_n) - 1}\right) \geq \mathbb{P}\left(\bigcap_{t=1}^{s_{H}(\ell_n) - 2} A_t \right) = 1 - \mathbb{P}\left(\bigcup_{t=1}^{s_{H}(\ell_n) - 2} [A_t^c \cap \bigcap_{i=1}^{t-1} A_i]
    \right) \geq 1 - \underbrace{\mathbb{P}\left(\bigcup_{t=1}^{s_{H}(\ell_n) - 2} [A_t^c \cap A_{t-1}]
    \right)}_{(*)(n)},
\end{align*}
where $(*)(n)$ itself is upper bounded (via union bound, \eqref{eq: DecreasingZProbBound}, and $\nu_{s_{H}(\ell_n) - 1} < \ell_n$) via
\begin{align*}
    \sum_{t=1}^{s_{H}(\ell_n) - 2} &\mathbb{E}\left[\mathbb{P}\left(Z_{t+1} >  \nu_{t+1}, Z_t \leq \nu_t \; |\; \mathcal{M}_t \right)\right] \leq \sum_{t=1}^{s_{H}(\ell_n) - 2} \mathbb{E}\left[\mathbb{P}\left(Z_{t+1} > H(Z_t)\right); [Z_t \leq \nu_t]\right]\\
    &\leq \sum_{t=1}^{s_{H}(\ell_n) - 2} \mathbb{E}\left[\exp(-\frac{M_t \underline{q}^{Z_t + 1}}{8}); [Z_t \leq \nu_t]\right]\leq s_H(\ell_n) \cdot \exp(-\frac{C(n) \underline{q}^{\ell_n + 1}}{8 \ell_n}).
\end{align*}
Towards making $(*)(n) \rightarrow 0,$ and with it, $\mathbb{P}\left(\mathbb{S}^{\lceil \pi_{\textsf{FA}} \rceil} \geq s_H(\ell_n) - 1\right) \rightarrow 1,$ we can now explicitly define $\ell_n \coloneqq \lfloor \log_{1/\underline{q}^2}(C(n)) \rfloor$, for an example member of $\Theta(\ln C(n))$ that achieves this.
In summary, 
\[\mathbb{E}\left[\mathbb{S}^{ \pi_{\textsf{FA}}}\right] \geq \mathbb{E}\left[\mathbb{S}^{\lceil \pi_{\textsf{FA}} \rceil}\right] \geq (s_H(\ell_n) - 1) \cdot \mathbb{P}\left(\mathbb{S}^{\lceil \pi_{\textsf{FA}} \rceil} \geq s_H(\ell_n) - 1\right) = \Omega(s_H(\ell_n)) =  \Omega(\ln^*\ell_n) = \Omega(\ln^*C(n)).
\]  \\
\underline{{\bf Proving $C(n)<n \implies \mathbb{E}[\mathbb{S}^{\pi_{\textsf{FA}}}]
        =
        1/C(n) \cdot \sum_{i=1}^n 1/(1-q_i^{(n)})
        +
        O(\ln^* C(n))$:}}\\
Let $\tau_{init} \coloneqq \inf \{t \geq 1: M_t \leq C(n)\} = \inf\{t\geq 1: C(n)/M_t \geq 1\}$ Then 
$\mathbb{S}^{\pi_{\textsf{FA}}} \preceq \tau_{init} + \inf\{k \geq 1: M_{\tau_{init} + k} \leq 0\}$, and we need only bound the expectation of two summands.

Regarding the first summand, note that $\tau_{init} - 1 = (1/C) \sum_{i \in [n] \setminus \mathcal{M}_{\tau}} L_i \leq (1/C) \sum_{i =1}^n L_i $, a.s., since $\pi_{\textsf{FA}}$ incurs no waste from allocations made in periods 1 through $\tau_{init}-1$; hence, $\mathbb{E}\left[\tau_{init}\right] \leq (1/C) \sum_{i=1}^n 1/(1-q_i^{(n)}) + 1$. As for the second summand, since $C(n) \geq M_{\tau_{init}},$ then by memorylessness of the geometric distribution, we can appeal to step (a) above to find
\[
\mathbb{E}\left[\inf\{k \geq 1: M_{\tau_{init} + k} = 0\} \; | M_{\tau_{init}}\; \right] \leq 1 + z_1 / (1-\bar{q}) + 2s_G(M_{\tau_{init}}) + 1/(1-\bar{q}).
\]
With $M_{\tau_{init}} \leq C(n)$, it holds that $s_G(M_{\tau_{init}})\leq s_G(C(n)) = \Theta(\ln^*(C(n)),$ as desired.

\end{proof}

\medskip

\subsection{Theorem \ref{thm:dtm_fair_additive_gap}}
In the preceding proof of \Cref{thm:spread-upper}, the two functions $G$ and $H$ belonged to the family:
\begin{equation} \label{eq: DefExponentialFamily}
    \mathcal{L}\coloneqq \{ z \mapsto \kappa zb^z:  \frac{\bar{q}}{2} \leq \kappa \leq \frac{2}{\underline{q}}, \frac{1}{\bar{q}} \leq b \leq \frac{1}{\underline{q}}\}.
\end{equation}
According to \Cref{lem:det-correct}, the recursive sequences, or \textit{ladders}, $\{a_k\}_{k=1}^\infty$ yielded by members of $\mathcal{L}$ behave like an iterated-exponential. The following preliminary lemma establishes that the difference in hitting times \eqref{eq: DefHittingTime} between any two members of $\mathcal{L}$ is bounded.  %

\begin{lemma}[Uniform additive base change for iterated-exponential ladders]
\label{lem:logstar-additive}
Fix intervals $\mathcal B=[b_-,b_+]\subset(1,\infty)$ and $\mathcal K= [\kappa_-,\kappa_+]\subset(0,\infty)$.
Then there exists $1 \leq A<\infty$ such that:
if $F_1, F_2$ are functions written $F_i(z) \coloneqq \kappa_i z b_i^z$, for $\kappa_i \in \mathcal{K}, b_i \in \mathcal{B}$,
then 
\[
    |s_{F_1}(a;y)-s_{F_2}(a;y)|\le 1. ~~~~~~\forall a \geq A, ~~~~\forall y
\]
\end{lemma}

\begin{proof}{Proof of Lemma \ref{lem:logstar-additive}.}
Let $F_-, F_+$ denote, respectively, the smallest and largest functions of the class $\mathcal{L}$, i.e., 
\[
F_-(z)\coloneqq \kappa_{-}zb_{-}^z \leq \kappa_+ z b_+^z \eqqcolon F_+(z).
\]
Let $A \geq 1$ such that for all $z \geq A$: (1) $F_-(z) \geq z^2$; (2) $F_-(z^2) = \kappa_-z^2b_-^{z^2} \geq \kappa_+^2z^2b_+^{2z} = [F_+(z)]^2$. It follows that: for $z \geq A$, both
\begin{equation} 
F_i(z) \geq F_-(z) \geq z^2  \;\; \forall i \in \{1,2\} ~\text{ and } ~ F_2(z^2) \geq F_-(z^2) \geq [F_+(z)]^2 \geq [F_1(z)]^2.  \tag{*}
\end{equation}
Letting $a\geq A$ be arbitrary, we define $z^{(i)}_1 \coloneqq a$ and $z^{(i)}_{k+1} = F_i(z^{(i)}_k)$ for $i = 1,2,$ we can show by induction that, after one extra iteration, ladder 2 is at least the square of ladder 1, i.e., 
\begin{equation}
\label{eq:additive-base-induction-claim}
    z^{(2)}_{k+1}\ge \bigl(z^{(1)}_k\bigr)^2,\qquad k\ge 1.
\end{equation}
Starting with the base case that $z_2^{(2)} = F_2(a) \geq a^2 = \left(z_1^{(1)}\right)^2,$ by the induction hypothesis and $(*)$,
\[
    z^{(2)}_{k+2}=F_2\bigl(z^{(2)}_{k+1}\bigr) 
    \geq F_2\bigl(z^{(1)}_k)^2\bigr) \geq \bigl(F_1(z^{(1)}_k)\bigr)^2 =\bigl(z^{(1)}_{k+1}\bigr)^2.
\]
The comparison in \eqref{eq:additive-base-induction-claim} can be used at $k=s_1(y;a)$ to conclude. Indeed, for any $y \geq a \geq A \geq 1,$
\[
    z^{(2)}_{s_{F_1}(a;y)+1}\ge \bigl(z^{(1)}_{s_{F_1}(a;y)}\bigr)^2\ge y^2\ge y.
\]
Hence $s_{F_2}(a;y)\le s_{F_1}(a;y)+1$, and by symmetry, we conclude $|s_{F_1}(a;y)-s_{F_2}(a;y)|\le 1$. \halmos
\end{proof}

\medskip

\begin{lemma}[Homogeneous fair allocation is base-insensitive]
\label{lem:homogeneous-fa-base-insensitive}
Let $0<\underline q<\bar q<1$ and $\Gamma\in(0,\infty)$.
Then there exists a constant $B=B(\underline q,\bar q,\Gamma)<\infty$ such that, for all integers $n,C$ satisfying $1\le n\le \Gamma C$ and all $\underline q\le r_1\le r_2\le \bar q$,
\begin{equation} \label{eq:FairAllocationDefusingTimesAreClose}
    0\le \mathbb{E}\left[T^{\pi_{\textsf{FA}}}(r_2 \cdot \mathbbm 1_n; C)\right] - \mathbb{E}\left[T^{\pi_{\textsf{FA}}}(r_1 \cdot \mathbbm 1_n; C)\right]\le B.
\end{equation}
\end{lemma}
\begin{proof}{Proof of Lemma \ref{lem:homogeneous-fa-base-insensitive}.} 
The claimed lower bound on the difference holds via the relation \eqref{eq: monotonicity} established in the proof of \Cref{theorem:DTM_fair_ratio}.  
We proceed to prove the upper bound on the difference between the defusing times, and it suffices to instead do so for the difference between their corresponding recognition times; towards this, we define the shorthand (recall \Cref{def::RecognitionTime}):
\[
\mathbb{S}_{r,n, C}^{\pi_{\textsf{FA}}}\coloneqq T^{\pi_{\textsf{FA}}}(r \cdot \mathbbm 1_n; C) + 1 ~~~~~~~~\forall r \in [\underline{q}, \bar{q}]
\]
All constants in this proof may depend on $\underline q$, $\bar q$, and $\Gamma$, but not on $n$ or $C$, or on any choice of $r\in[\underline q,\bar q]$. For example, let $A\geq 1$ be as prescribed in \Cref{lem:logstar-additive}, when $\mathcal{B} = [1/\bar{q}, 1/ \underline{q}], \mathcal{K} = [\bar{q}/2, 2/\underline{q}]$. As well, all asymptotic order notation (e.g. $O, \Omega, \Theta, o$) will be with respect to $C\rightarrow \infty$. 

The plan is to  show that for some constant $z_1 > \max(A, -\log_{1/\bar{q}}(\bar{q}/2))$ and $\ell_C \in \Theta(\ln C)$, it holds that
\begin{equation} \label{eq:RecognitionTimeApproximation}
      s_{H_r}(z_1,\ell_C)-O(1)  \le \mathbb{E}\left[\mathbb{S}_{r,n,C}^{\pi_{\textsf{FA}}}\right]
    \le s_{G_r}(z_1,\ell_C)+O(1), ~~~~\forall r \in [\underline{q}, \bar{q}], \; \forall n \in [1, \Gamma C]_{\mathbb{Z}}
\end{equation}
where $G_r(z) \coloneqq (r/2)z(1/r)^z$ and $H_r(z) \coloneqq (2/r)z(1/r)^z$. Once this is shown, the desired inequality \eqref{eq:FairAllocationDefusingTimesAreClose} will follow by \Cref{lem:logstar-additive}, completing the proof.  

Let $r \in [\underline{q}, \bar{q}],$ and let $n,C$ satisfy $1\leq n \leq \Gamma C.$ Then let $\{\mathcal{M}_t\}_{t=1}^{\mathbb{S}^{\pi_\textsf{FA}}}$ be the (nonincreasing) nested family of live-threat sets throughout the execution of $\pi_\textsf{FA}$ on the instance $(n, C, r\cdot \mathbf 1_n)$. It is without loss of generality that we may proceed under the assumption that $M_1 \equiv n \leq C,$ equiv., $Z_1 \geq 1$. Indeed, if $n > C,$ then the time $\tau_{init} \coloneqq \min\{t\geq 1: M_t \leq C\}$ in expectation can be bounded by a constant as we did in \Cref{thm:spread-upper}'s proof,
\[
0 \leq \mathbb{E}\left[\tau_{init} - 1\right] \leq (1/C) \sum_{i=1}^n 1/(1-r_i)  \leq n/C \cdot 1/(1-r) \leq \Gamma/(1-\bar{q});
\]
so because $\mathbb{E}\left[\mathbb{S}_{r,n,C}^{\pi_{\textsf{FA}}} - \tau_{init}\right] = \mathbb{E}\left[\mathbb{E}\left[\mathbb{S}_{r,n,C}^{\pi_{\textsf{FA}}}(\mathcal{M}_{\tau_{init}}) \; |\; \mathcal{M}_{\tau_{init}}\right]\right] - 1$, where $\mathbb{S}_{r,n,C}^{\pi_{\textsf{FA}}}(\mathcal{M}_{\tau_{init}})$ denotes the recognition time starting when the initial set of threats is  $\mathcal{M}_{\tau_{init}}$ (in which $|\mathcal{M}_{\tau_{init}}| \leq C$), it suffices to establish \eqref{eq:RecognitionTimeApproximation} under $n \leq C.$ 

Let $Z_t \coloneqq (C/M_t)\mathbbm{1}_{\{M_t\ge1\}}+\infty\cdot\mathbbm{1}_{\{M_t=0\}}$, and note that 
for all $t$, $[\pi_{\textsf{FA}}(\mathcal{M}_t)]_i \in \{\lfloor Z_t\rfloor,
\lceil Z_t\rceil\}$; further, $\{Z_t\}_t$ is a nondecreasing sequence that starts at $Z_1 \geq 1$ (since $n \leq C$) and traverses upward through the interval $[1, \infty]$, which admits the partitioning
\[
[1, \infty] = [1, z_1) \cup [z_1, \ell_C) \cup [\ell_C, +\infty],
\]
where we leave constant $z_1$ and $\ell_C = \Theta(\ln C)$ undefined for the moment, but assume (without loss of generality) that $z_1 < \ell_C.$
As such, we will approximate $\mathbb{S}_{r,n,C}^{\pi_{\textsf{FA}}}$ by approximating the time that the process spends in each of these 3 intervals; more precisely, we define 
\[
\tau_1 \coloneqq \min\{t: Z_t \geq z_1\} ~~~\tau_C \coloneqq \min\{t: Z_t \geq \ell_C\},
\]
and derive separate approximations of: (a) $\mathbb{E}\left[\tau_1\right]$; (b) $\mathbb{E}\left[\tau_C - \tau_1\right]$; and (c) $\mathbb{E}\left[\mathbb{S}_{r,n,C}^{\pi_{\textsf{FA}}} - \tau_C\right]$. Upon adding them, we will attain \eqref{eq:RecognitionTimeApproximation}, as desired. \\
\noindent
\underline{\bf (a) Approximating $\mathbb{E}\left[\tau_1\right] $:} In similar fashion to step (a) of \Cref{thm:spread-upper}'s proof, we note that
\begin{equation} \label{eq:EnteringLadder}
\mathbb{E}[\tau_1 - 1] = \sum_{t=1}^{\infty}\mathbb{P}(\tau_1>t) \leq \sum_{t=1}^{\infty}\mathbb{P}(Z_t<z_1) = \sum_{t=1}^{\infty}\mathbb{P}\left(M_t>\frac{C}{z_1}\right) \leq \sum_{t=1}^{\infty}\frac{z_1}{C}\mathbb{E}[M_t] \leq \sum_{t=1}^{\infty}z_1\bar q^{t-1} = \frac{z_1}{1-\bar q} = O(1), 
\end{equation}
where $z_1$ is a constant to be defined shortly. \\
\noindent
\underline{\bf (b) Approximating $\mathbb{E}\left[\tau_C - \tau_1\right]$:} 
In similar fashion to the activity in \Cref{thm:spread-upper}'s proof, we take the following three inequalities that hold for all $t \in [\tau_1,\tau_C)_{\mathbb{Z}}$,
\[
M_t r^{Z_t + 1} \leq M_t r^{\lceil Z_t \rceil} \leq \mathbb{E}\left[M_{t+1} \;| \; \mathcal{M}_t\right] \leq M_t r^{\lfloor Z_t \rfloor} \leq M_t r^{Z_t - 1} ~ \text{ and } ~ Z_t \leq \ell_C ~ \text{ and } ~ M_t > C/\ell_C,
\]
and translate them into the two Chernoff bounds:
\begin{equation} \label{eq: ChernoffG}
\mathbb{P}\left(Z_{t+1} < G_r(Z_t)\; |\; \mathcal{M}_t\right) \leq \mathbb{P}\left(M_{t+1} > 2 \mathbbm{E}\left[M_{t+1}\;|\; \mathcal{M}_t\right]\; | \; \mathcal{M}_t\right) \leq \exp(-\frac{M_t \underline{q}^{Z_t + 1}}{3}) \leq \exp(-\frac{C\underline{q}^{\ell_C + 1}}{3\ell_C}), 
\end{equation}

\begin{equation} \label{eq: ChernoffH}
\mathbb{P}\left(Z_{t+1} > H_r(Z_t)\; |\; \mathcal{M}_t\right) \leq \mathbb{P}\left(M_{t+1} < \frac{1}{2} \mathbbm{E}\left[M_{t+1}\;|\; \mathcal{M}_t\right]\; | \; \mathcal{M}_t\right) \leq \exp(-\frac{M_t \underline{q}^{Z_t + 1}}{8}) \leq \exp(-\frac{C \underline{q}^{\ell_C + 1}}{8\ell_C}).
\end{equation}
As far as upper-bounding $\mathbb{E}\left[\tau_C - \tau_1\right]$, we can follow the analysis of step (a) in \Cref{thm:spread-upper}'s proof, and use \eqref{eq: ChernoffG} 
to find that 
\begin{equation}
\mathbb{E}\left[\tau_C - \tau_1\right] =
\mathbb{E}\left[\min\{k \geq 1: Z_{\tau_1 + k} \geq \ell_C\}\right] \leq \frac{s_{G_r}(z_1; \ell_C)}{1 - \exp(-C\underline{q}^{\ell_C + 1}/(3\ell_C))}, 
\tag{Upper}
\end{equation}
As far as lower-bounding $\mathbb{E}\left[\tau_C - \tau_1\right],$ we may proceed in a similar way to step (b) of \Cref{thm:spread-upper}'s proof, by noting 
\begin{equation}
\mathbb{E}\left[\tau_C - \tau_1\right] \geq (s_{H_r}(z_1; \ell_C) - 1) \cdot \underbrace{ \mathbb{P}\left(\tau_C - \tau_1 \geq s_{H_r}(z_1; \ell_C) - 1\right)}_{\eqqcolon (***)}, \tag{Lower}
\end{equation}
where,  
using the shorthand $A_t \coloneqq [Z_{t+1} \leq z_{t+1}]$ for $t\in \mathbb{Z}_{\geq 0},$ and \eqref{eq: ChernoffH}, 
\begin{align*}
    (***) &\geq \mathbb{P}\left(Z_{s_{H_r}(z_1;\ell_C) - 1} \leq z_{s_{H_r}(z_1;\ell_C) - 1}\right) \geq \mathbb{P}\left(\bigcap_{t=1}^{s_{H_r}(z_1;\ell_C) - 2} A_t \right) = 1 - \mathbb{P}\left(\bigcup_{t=1}^{s_{H_r}(z_1;\ell_C) - 2} A_t^c
    \right) \\
    &\geq 1 - s_{H_r}(z_1;\ell_C) \cdot \exp(-\frac{C \underline{q}^{\ell_C + 1}}{8\ell_C}).
\end{align*} 
Upon defining $\ell_C \coloneqq \lfloor \log_{1/\underline{q}^8}(C) \rfloor = \Theta(\ln C),$ the resulting forms of (Upper) and (Lower), along with the fact that $s_{H_r}(z_1; \ell_C) = \Theta(\ln^*C)$, provide 
\begin{equation} \label{eq: LadderRegion}
s_{H_r}(z_1; \ell_C) - O(1) \leq \mathbb{E}\left[\tau_C - \tau_1\right] \leq s_{G_r}(z_1; \ell_C) + O(1). 
\end{equation}
\\
\noindent
\underline{\bf (c) Approximating $\mathbb{E}\left[\mathbb{S}_{r,n,C}^{\pi_{\textsf{FA}}} - \tau_C\right]$:} Without loss of generality, suppose that $M_{\tau_{C}} > 0$; otherwise, $\mathbb{S}_{r,n,C}^{\pi_{\textsf{FA}}} - \tau_C = 0$. Hence, $Z_{\tau_C} \geq \ell_C$, equiv., $M_{\tau_C} \leq C/\ell_C$, and we will assume $C$ is sufficiently large enough that $C/\ell_C < C.$ Next, observe that for a function $f(C),$ undefined for now, we have
\begin{align*}
\mathbb{E}\left[\mathbb{S}_{r,n,C}^{\pi_{\textsf{FA}}} - \tau_C\; | \; \mathcal{M}_{\tau_C}\right] 
&\leq 2 + 
\underbrace{
\mathbb{E}\left[(\mathbb{S}_{r,n,C}^{\pi_{\textsf{FA}}} - (\tau_C+2))  \mathbf{1}_{M_{\tau_C +1} \leq f(C)}\; | \; \mathcal{M}_{\tau_C}\right]
}_{(*)}
+ 
\underbrace{
\mathbb{E}\left[(\mathbb{S}_{r,n,C}^{\pi_{\textsf{FA}}} - (\tau_C+2))  \mathbf{1}_{M_{\tau_C +1} > f(C)}\; | \; \mathcal{M}_{\tau_C}\right].
}_{(**)}
\end{align*}
With regards to (**), as $M_{\tau_C}<C,$ we may leverage step (a) from \Cref{thm:spread-upper}'s proof to find
\[
(**)
    \le O(\ln^*C) \cdot \mathbb{P}\left(M_{\tau_C + 1} > f(C) \;| \; \mathcal{M}_{\tau_C}\right) \leq O(\ln^*C) \frac{\mathbb{E}\left[M_{\tau_C + 1}\; |\; \mathcal{M}_{\tau_C}\right]}{f(C)} \leq  O(\ln^*C)  \frac{(C/\ell_C) \bar{q}^{\ell_C}}{f(C)}
\]
With regards to (*), since it is plain to see that
\[
\mathbb E\!\left[M_{\tau_C+2}
        {\bf 1}_{\{M_{\tau_C+1}\le f(C)\}}\mid\mathcal M_{\tau_C}\right]
    \le f(C) \bar{q}^{\lfloor C/f(C) \rfloor}, \;\; \text{ and } \;\; \mathbb{E}\left[\mathbb{S}_{r,n,C}^{\pi_{\textsf{FA}}} - (\tau_C+2) \; |\; M_{\tau_C + 2}\right] \leq \frac{M_{\tau_C + 2}}{1-\bar{q}},
\]
it follows that 
\[
(*) \leq \frac{1}{1-\bar{q}}f(C) \bar{q}^{\lfloor C/f(C) \rfloor}.
\]
Upon letting $f(C) \coloneqq C^{1-\log_{\underline{q}}(\bar{q})/32},$ both $(*)$ and $(**)$ are $O(1),$ so that we obtain, as desired,  
\begin{equation} \label{eq: TerminalRegion}
\mathbb{E}\left[\mathbb{S}_{r,n,C}^{\pi_{\textsf{FA}}} - \tau_C\right]= O(1). 
\end{equation} \\
\noindent
\underline{\bf Combining (a), (b), (c):} Altogether, \eqref{eq:EnteringLadder},\eqref{eq: LadderRegion}, and \eqref{eq: TerminalRegion} provide \eqref{eq:RecognitionTimeApproximation}. 
Consequently, for any $r_1, r_2 \in [\underline{q}, \bar{q}]$ in which $r_1 < r_2$,
\[
s_{H_{r_2}}(z_1; \ell_C) - s_{G_{r_1}}(z_1; \ell_C) - O(1) \leq \mathbb{E}\left[\mathbb{S}_{r_2,n,C}^{\pi_{\textsf{FA}}}\right] - \mathbb{E}\left[\mathbb{S}_{r_1, n, C}^{\pi_{\textsf{FA}}}\right] \leq s_{G_{r_2}}(z_1; \ell_C) - s_{H_{r_1}}(z_1; \ell_C) + O(1),
\]
so that by \Cref{lem:logstar-additive}, we conclude \eqref{eq:FairAllocationDefusingTimesAreClose}, completing the proof. \halmos
\end{proof}

\medskip

\subsubsection{Proof of \Cref{thm:dtm_fair_additive_gap}.}

\begin{proof}{Proof of \Cref{thm:dtm_fair_additive_gap}.} \label{sec:AppendixProofAdditiveGap}
Since \(C(n) \geq \underline \alpha n\), we have \(n\le C(n)/\underline \alpha\). 
Invoking \Cref{lem:homogeneous-fa-base-insensitive} with
\(\Gamma=1/\underline \alpha\), \(r_1=\underline q\), and \(r_2=\bar q\), there exists a $B = B(\underline{q}, \bar{q}, 1/\underline{\alpha})$ so that 
\[
0\leq \mathbb E\left[T^{\pi_{\textsf{FA}}}( \bar{q} \cdot \mathbbm{1}_{n}; C(n))\right] - \mathbb E\left[T^{\pi_{\textsf{FA}}}(\underline{q} \cdot \mathbbm{1}_{n}; C(n))\right] \leq B.  ~~~~~~\forall n \in \mathbb{Z}_{\geq 1}
\]
For sufficiently large $n,$ it holds that $ \underline{q} \cdot \mathbbm{1}_{n} \leq q^{(n)} \leq \bar{q} \cdot \mathbbm{1}_{n}$, so by \Cref{lem:monotonicity}'s \eqref{eq: monotonicity}, 
\[
\mathbb E\left[T^{\pi_{\textsf{FA}}}(\underline{q} \cdot \mathbbm{1}_{n} ; C(n))\right] \leq \mathbb E\left[T^{\pi_{\textsf{FA}}}(q^{(n)}  ; C(n))\right] \leq \mathbb E\left[T^{\pi_{\textsf{FA}}}( \bar{q} \cdot \mathbbm{1}_{n}  ; C(n))\right] 
\]
and 
\[
\mathbb E\left[T^{\pi_{\textsf{FA}}}(\underline{q} \cdot \mathbbm{1}_{n} ; C(n))\right] \leq \mathbb E\left[T^*(q^{(n)}; C(n))\right].
\]
Therefore, for sufficiently large $n,$
\[
0\leq \mathbb E\left[T^{\pi_{\textsf{FA}}}( q^{(n)}; C(n))\right] - \mathbb E\left[T^*(q^{(n)}; C(n))\right] \leq \mathbb E\left[T^{\pi_{\textsf{FA}}}( \bar{q} \cdot \mathbbm{1}_{n}; C(n))\right] - \mathbb E\left[T^{\pi_{\textsf{FA}}}(\underline{q} \cdot \mathbbm{1}_{n}; C(n))\right] \leq B .
\]
\halmos

\end{proof}

\medskip

\subsection{\Cref{thm:tau1_ratio_cfree}}

\begin{proof}{Proof of \Cref{thm:tau1_ratio_cfree}.}

\smallskip

\noindent\underline{Bound 1: $\exp \left(-\frac{C}{e\lfloor C / n\rfloor}\right)$.}
Let $S:=\{i\in[n]: x_i^*>\lfloor C/n \rfloor\}$ be the set of threats where the optimal policy allocates more than $\lfloor C/n \rfloor$ effectors. We denote $m:=|S|$.
For $i\notin S$ we have $x_i^*\le \lfloor C/n \rfloor$. Since $1-q_i^x$ is increasing in $x$, we know $1-q_i^{\lfloor C/n \rfloor}\ \ge\ 1-q_i^{x_i^*}$ holds for all $i\notin S$. For $i\in S$, since $x_i^*\geq 1$, and the mapping $x \mapsto 1-q_i^x$ is concave, we
obtain
\[
1-q_i^{\lfloor C/n \rfloor}\ \ge\ \frac{\lfloor C/n \rfloor}{x_i^*}\,(1-q_i^{x_i^*})  + \left(1 - \frac{\lfloor C/n \rfloor}{x_i^*} \right) (1-q_i^0) = \frac{\lfloor C/n \rfloor}{x_i^*}\,(1-q_i^{x_i^*})\qquad\text{for all }i\in S.
\]
Since $\pi_{\textsf{FA}}$ allocates at least $\lfloor C/n \rfloor$ effectors for each threat, multiplying these inequalities over $i=1,\dots,n$ yields
\begin{equation}
\label{eq:key_ratio_step}
V_{\textsf{SLM}_1}^{\pi_\textsf{FA}}
\geq \prod_{i=1}^n (1-q_i^{\lfloor C/n \rfloor})
\ \ge\ \Bigl(\prod_{i\in S}\frac{\lfloor C/n \rfloor}{x_i^*}\Bigr)\prod_{i=1}^n (1-q_i^{x_i^*})
\ =\ \frac{\lfloor C/n \rfloor^{|S|}}{\prod_{i\in S}x_i^*}\,V_{\textsf{SLM}_1}^*.
\end{equation}
It remains to upper bound $\prod_{i\in S}x_i^*$.
Since $x_i^*\ge 1$ for $i\notin S$, we know
\[
\sum_{i\in S}x_i^*
 = C-\sum_{i\notin S}x_i^*
\ \le\ C-(n-|S|)
= C-n+|S|.
\]
Then by AM-GM inequality, we can upper bound it as follows:
\begin{equation}\label{eq:thm1_fair_slm_prod_xi}
   \prod_{i\in S}x_i^*\ \le \left(\frac{\sum_{i\in S}x_i^*}{|S|}\right)^{|S|} \leq  \Bigl(\frac{C-n+|S|}{|S|}\Bigr)^{|S|}. 
\end{equation}
Finally, minimizing the ratio over $|S|\in\{1,\dots,n\}$ yields
\[
\frac{V_{\textsf{SLM}_1}^{\pi_{\textsf{FA}}}}{V_{\textsf{SLM}_1}^*}\ \ge\ \min_{1\le |S|\le n}  \frac{\lfloor C/n \rfloor^{|S|}}{\prod_{i\in S}x_i^*} \geq \min_{1\le |S|\le n}\Bigl(\frac{\lfloor C/n \rfloor |S|}{C-n+|S|}\Bigr)^{|S|},
\]
where the first inequality is from equation \eqref{eq:key_ratio_step} and the second inequality is from equation \eqref{eq:thm1_fair_slm_prod_xi}. For any $|S| \in\{1, \ldots, n\}$, we have
\begin{equation}\label{eq:thm_fair_slm_1}
\left(\frac{\lfloor C/n \rfloor |S|}{C-n+|S|}\right)^{|S|}=\exp \left(|S| \ln \left(\frac{\lfloor C/n \rfloor |S|}{C-n+|S|}\right)\right)=\exp \left(-|S| \ln \left(\frac{C-n+|S|}{\lfloor C/n \rfloor |S|}\right)\right) .
\end{equation}
Using the fact that $\ln x \leq x/e
$ holds for all $x>0$ and applying it to $x=\frac{C-n+|S|}{\lfloor C/n \rfloor |S|}>0$, we get
\begin{equation}\label{eq:thm_fair_slm_2}
    |S| \ln \left(\frac{C-n+|S|}{\lfloor C/n \rfloor |S|}\right) \leq |S| \cdot \frac{1}{e} \cdot \frac{C-n+|S|}{\lfloor C/n \rfloor |S|}=\frac{C-n+|S|}{e \lfloor C/n \rfloor}\leq \frac{C}{e \lfloor C/n \rfloor}.
\end{equation}
Then combining equations \eqref{eq:thm_fair_slm_1} and \eqref{eq:thm_fair_slm_2}, we have
\begin{equation}\label{eq:thm_fair_slm_3}
\left(\frac{\lfloor C/n \rfloor |S|}{C-n+|S|}\right)^{|S|}=\exp \left(-|S| \ln \left(\frac{C-n+|S|}{\lfloor C/n \rfloor |S|}\right)\right) \geq \exp \left(-\frac{C}{e \lfloor C/n \rfloor}\right),
\end{equation}
which holds for every $|S|=1, \ldots, n$. Finally, taking the minimum over $|S|$ towards equation \eqref{eq:thm_fair_slm_3} yields
$$
\frac{V_{\textsf{SLM}_1}^{\pi_{\textsf{FA}}}}{V_{\textsf{SLM}_1}^*}\ \ge\ \min _{1 \leq |S| \leq n}\left(\frac{\lfloor C/n \rfloor |S|}{C-n+|S|}\right)^{|S|} \geq \exp \left(-\frac{C}{e \lfloor C/n \rfloor}\right).
$$

\smallskip

\noindent\underline{Bound 2: $\left(\frac{\lfloor C/n\rfloor}{\,C-n+1\,}\right)^n$.} This bound is much looser and thus easier to derive. Since the optimal policy $\pi^*$ has to allocate at least one effector to each threat, each threat can receive at most $C-n+1$ effectors. That is, the policy allocates $C-n+1$ effectors to that threat and one effector to any other threats. On the other hand, under $\pi_{\textsf{FA}}$, each threat receives at least $\lfloor C/n \rfloor$ effectors. Therefore, we can bound $V_{\textsf{SLM},1}^*/V_{\textsf{SLM},1}^{\pi_{\textsf{FA}}}$ using the following inequality:
$$
\frac{V_{\textsf{SLM}_1}^{\pi_{\textsf{FA}}}}{V_{\textsf{SLM}_1}^*}\geq \prod_{i=1}^n \frac{1-q_i^{\lfloor C/n\rfloor}}{1-q_i^{C-n+1}}\geq \prod_{i=1}^n \frac{\lfloor C/n\rfloor}{C-n+1} \left(1-q_i^{C-n+1}\right) \frac{1}{1-q_i^{C-n+1}}=\left(\frac{\lfloor C/n\rfloor}{\,C-n+1\,}\right)^n,
$$
where the second inequality follows from the concavity of $x \mapsto 1-q_i^x$. 

\smallskip

\noindent\underline{Tightness Construction.}
To show the worst-case ratio indeed has an exponential decay, for each $n$, we construct $m\in\{1,\cdots,n-1\}$ hard threats and $n-m$ easy threats. Let $H\subseteq[n]$ be the set of hard threats with $|H|=m\in\{1,\ldots,n\}$. %
We set $q_i=\epsilon$ for $i\notin H$ and $q_i=e^{-\epsilon}$ for $i\in H$, where $\epsilon=1/n^4\rightarrow 0$ as $n\rightarrow\infty$. We set $C=1/\sqrt{\epsilon}=n^{2}$. Without loss of generality, we can restrict attention to policy $\pi$ such that $\pi$ allocates $x_i\ge 1$ for all $i\notin H$.
Under this restriction, we have $\sum_{i\in H} x_i = C-\sum_{i\notin H}x_i \le C-(n-m)$. As $\epsilon\rightarrow 0$, for each easy threat $i\notin H$ with $x_i\ge 1$ we have
$1-q_i^{x_i}=1-\epsilon^{x_i}=1-o(1)$. For each hard threat $i\in H$, $q_i^{x_i}=e^{-\epsilon x_i}$ and hence $1-q_i^{x_i}=1-e^{-\epsilon x_i}=\epsilon x_i + O(\epsilon^2 x_i^2)$.

We multiply the $m$ hard-threat factors and use the fact that the easy-threat
product equals $1+o(1)$. Then for any policy $\pi$ with $x_i\ge 1$ for all $i$, we have
$V_{\textsf{SLM}_1}^\pi=\epsilon^m\Big(\prod_{i\in H}x_i\Big)\,(1-o(1))$ when $\epsilon\rightarrow 0$. For $\pi^*$, one can show that it allocates exactly one effector to threats $i\notin H$ and allocates the remaining effectors as evenly as possible to threats $i\in H$. When $n\rightarrow\infty$, the integer-rounding errors are $o(1)$, so we can assume $(C-n)/m\in\mathbb{Z}$ without loss of generality. Thus, we have
\[
V_{\textsf{SLM}_1}^*
=\epsilon^m\Big(\frac{C-(n-m)}{m}\Big)^m\,(1-o(1)).
\]
On the other hand, under fair allocation, $x_i^{\textsf{FA}}=C/n$ for all $i$, so the performance of $\pi_{\textsf{FA}}$ follows
$
V_{\textsf{SLM}_1}^{\pi_{\textsf{FA}}}
=\epsilon^m(C/n)^m\,(1-o(1))$.
Consequently, the ratio can be written as
\[
\frac{V_{\textsf{SLM}_1}^{\pi_{\textsf{FA}}}}{V_{\textsf{SLM}_1}^*}
=\left(\frac{m}{n}\cdot \frac{C}{C-(n-m)}\right)^m(1-o(1)).
\]
When $n\to\infty$, $C/n\rightarrow\infty$, and $\epsilon C\rightarrow 0$, this becomes
\[
\frac{V_{\textsf{SLM}_1}^{\pi_{\textsf{FA}}}}{V_{\textsf{SLM}_1}^*}
=\Big(\frac{m}{n}\Big)^m(1-o(1)).
\]
Finally, minimizing $(m/n)^m$ over $m\in\{1,\ldots,n\}$ is achieved at $m/n=1/e$ (i.e., $m\approx n/e$), which yields
\[
\inf_{1\le m\le n}\Big(\frac{m}{n}\Big)^m
=\exp\bigl(-(1-o(1))\,n/e\bigr),
\qquad n\to\infty.
\]
This shows the exponential decay $\exp(-(1+o(1))n/e)$ is asymptotically tight under (\textsf{SLM}$_1$).

\end{proof}

\medskip

\subsection{\Cref{thm:fair_alloc_min_waste}}

\begin{proof}{Proof of \Cref{thm:fair_alloc_min_waste}.}

Let $V_{\textsf{EAM}_\tau}^{\pi}(t,L;M_t,M_{t+1})$ be the number of effective assignments in round $t$ under policy $\pi$ and realization $L$ given the number of alive threats at the beginning of $t$, $M_t$, and at the beginning of $t+1$, $M_{t+1}$, respectively. Then we have the following lemma that gives a lower bound for $V_{\textsf{EAM}_\tau}^{\pi_\textsf{FA}}(t,L;M_t,M_{t+1})$.

\begin{lemma}\label{lemma:num_effective_assignments_in_round_r}
For any realization $L$, the number of effective assignments in round $t$ satisfies
$$
V_{\textsf{EAM}_\tau}^{\pi_\textsf{FA}}(t,L;M_t,M_{t+1}) \geq M_{t+1}\frac{C}{M_t}.
$$
\end{lemma}

\begin{proof}{Proof of \Cref{lemma:num_effective_assignments_in_round_r}.}
    Recall that under $\pi_{\textsf{FA}}$, at round $t$, we first allocate $\lfloor \frac{C}{M_t}\rfloor$ units to each threat, then randomly select $C-M_t\lfloor \frac{C}{M_t}\rfloor$ threats to allocate 1 additional unit. Then we have
    \begin{equation}\label{eq:thm_fair_eam_lemma}
        V_{\textsf{EAM}_\tau}^{\pi_\textsf{FA}}(t;M_t,M_{t+1}) \geq M_{t+1}\lfloor \frac{C}{M_t}\rfloor+(M_t-M_{t+1})+\left[ C-M_t\lfloor \frac{C}{M_t}\rfloor-(M_t-M_{t+1}) \right]^+. 
    \end{equation}
    In RHS of equation \eqref{eq:thm_fair_eam_lemma}, the first term is from the fact that all effectors (excluding $C-M_t\lfloor \frac{C}{M_t}\rfloor$ surplus units under random allocation) allocated to threats not neutralized in round $t$ must be effective assignments. The second term is from the fact that each neutralized threat in round $t$ receives at least one effective assignment. The last term is from the fact that the $C-M_t\lfloor \frac{C}{M_t}\rfloor$ additional units allocated to threats not neutralized must be effective as well. Thus, by equation \eqref{eq:thm_fair_eam_lemma}, it is sufficient to show
\begin{equation}\label{eq:thm_fair_eam_lemma2}
M_{t+1}\left\lfloor\frac{C}{M_t}\right\rfloor+\max \left(C-M_t\left\lfloor\frac{C}{M_t}\right\rfloor, M_t-M_{t+1}\right) \geq \frac{M_{t+1} C}{M_t} .
\end{equation}
By moving the first term in the LHS of equation \eqref{eq:thm_fair_eam_lemma2} to RHS, it suffices to prove
\begin{equation}\label{eq:thm_fair_eam_lemma3}
   \max \left(C-M_t \lfloor\frac{C}{M_t}\rfloor, M_t-M_{t+1}\right) \geq \frac{M_{t+1}}{M_t}(C-M_t \lfloor\frac{C}{M_t}\rfloor). 
\end{equation}
Now we split our proof into following two cases.

\smallskip

\noindent\underline{Case 1: $C-M_t \lfloor\frac{C}{M_t}\rfloor \geq M_t-M_{t+1}$.}
Then $\max \left(C-M_t \lfloor\frac{C}{M_t}\rfloor, M_t-M_{t+1}\right)=C-M_t \lfloor\frac{C}{M_t}\rfloor$, and since $M_{t+1} \leq M_t$, the LHS in equation \eqref{eq:thm_fair_eam_lemma3} can be lower bounded by
$$
\max \left(C-M_t \lfloor\frac{C}{M_t}\rfloor, M_t-M_{t+1}\right)=C-M_t \lfloor\frac{C}{M_t}\rfloor \geq \frac{M_{t+1}}{M_t} (C-M_t \lfloor\frac{C}{M_t}\rfloor).
$$

\smallskip

\noindent\underline{Case 2: $C-M_t \lfloor\frac{C}{M_t}\rfloor<M_t-M_{t+1}$.}
Then $\max \left(C-M_t \lfloor\frac{C}{M_t}\rfloor, M_t-M_{t+1}\right)=M_t-M_{t+1}$. We have
$$
\frac{M_{t+1}}{M_t}(C-M_t \lfloor\frac{C}{M_t}\rfloor) <\frac{M_{t+1}\left(M_t-M_{t+1}\right)}{M_t} \leq M_t-M_{t+1},
$$
where the last inequality uses $M_{t+1} \leq M_t$.
Thus, in both cases, we have
$$
\max \left(C-M_t \lfloor\frac{C}{M_t}\rfloor, M_t-M_{t+1}\right) \geq \frac{M_{t+1}}{M_t}(C-M_t \lfloor\frac{C}{M_t}\rfloor),
$$
which completes the proof. \halmos
\end{proof}

\medskip

Now we move back to the proof of \Cref{thm:fair_alloc_min_waste}. Under $\pi_{\textsf{FA}}$, if $M_{\tau+1}=0$, then all threats are neutralized at or before $\tau$, so $V_{\textsf{EAM}_\tau}^{\pi_\textsf{FA}}/V_{\textsf{EAM}_\tau}^*=1$ in this case. Thus, we only need to consider the case when $M_{\tau+1}\geq 1$. By \Cref{lemma:num_effective_assignments_in_round_r}, summing over rounds and taking the expectation over $M_2,M_3,\cdots,M_{\tau+1}$ (in the offline setting, $L_1,\cdots,L_n$ are given), $V_{\textsf{EAM}_\tau}^{\pi_\textsf{FA}}$ can be lower bounded by
\begin{equation}\label{eq:thm_fair_eaM_1}
    V_{\textsf{EAM}_\tau}^{\pi_\textsf{FA}}(L)=\sum_{t=1}^\tau V_{\textsf{EAM}_\tau}^{\pi_\textsf{FA}}(t,L;M_t,M_{t+1}) \geq C\ \sum_{t=1}^\tau \frac{M_{t+1}}{M_t}.
\end{equation}
Finally, by AM-GM inequality and using the facts $M_1=n$ and $M_{\tau+1}\geq 1$ with probability 1, we have
\begin{equation}\label{eq:thm_fair_eam_2}
\frac{1}{\tau} \sum_{t=1}^\tau \frac{M_{t+1}}{M_t} \geq\left(\prod_{r=1}^\tau \frac{M_{t+1}}{M_t}\right)^{1 / \tau}=\left(\frac{M_{\tau+1}}{M_1}\right)^{1/\tau}\geq \left(\frac{1}{M_1}\right)^{1/\tau} = n^{-1 / \tau}
\end{equation}
holds for any $M_2,\cdots,M_{\tau+1}$. Combine equations \eqref{eq:thm_fair_eaM_1} and \eqref{eq:thm_fair_eam_2} and we get $V_{\textsf{EAM}_\tau}^{\pi_\textsf{FA}}(L) \geq C \tau n^{-1 / \tau}$. Since $\min\{C\tau ,\sum_{i=1}^n L_i\} \leq \tau C$ holds trivially, we finally get the worst-case ratio bound
$V_{\textsf{EAM}_\tau}^{\pi_\textsf{FA}}(L)/\min\{C\tau ,\sum_{i=1}^n L_i\} \geq n^{-1 / \tau}$, which proves the theorem.

\smallskip

\noindent\underline{Tightness construction under the offline setting.}
To show the $n^{-1/\tau}$ worst-case rate is tight when $C\rightarrow\infty$, let us consider the following example. Assume $n=a^\tau$ for some integer $a\ge2$ and pick $C=kn$ for some integer $k\ge1$. 
We separate the last threat $n$ and partition the first $n-1$ threats into $\tau$ layers. Let $S_t, t=1,2,\cdots,\tau$ be the set of threats for salvo $t$. Under the following construction, all threats in salvo $t$ are neutralized exactly in round $t$. That is, $|S_t|=M_t-M_{t+1}$ is the number of threats neutralized in round $t$. In specific, we construct $|S_t|=a^{\tau-t+1}-a^{\tau-t}$, $L_i= 1+\sum_{j=1}^{t-1} x_j $ for each $i\in S_t$, where we set $x_t=C/M_t=ka^{t-1}$ for $t=1,2,\cdots,\tau$. We set the last threat $L_n := \tau C$. We also set $M_1=n$, $M_{\tau+1}=1$, and $M_t=a^{\tau-t+1}$ for $t=2,\dots,\tau+1$.

Then under $\pi_{\textsf{FA}}$, one can show that every threat in $S_t$ survives rounds $1,\dots,t-1$ and has \emph{exactly one} remaining unit at the start of round $r$, hence it dies in round $t$ under $\pi_{\textsf{FA}}$. Therefore, in round $t$, the $M_{t+1}$ survivors contribute $M_{t+1}x_t$ effective assignments, while the $|S_t|=M_t-M_{t+1}$ dying threats contribute exactly $1$ effective assignment each. Hence
\[
V_{\textsf{EAM}_\tau}^{\pi_\textsf{FA}}(t;M_t,M_{t+1}) 
= M_{t+1}x_t + (M_t-M_{t+1})
= \frac{M_t}{a}\cdot\frac{C}{M_t} + (M_t-M_{t+1})
= \frac{C}{a} + (M_t-M_{t+1}),
\]
and summing over $t=1,\dots,\tau$ gives
$ V_{\textsf{EAM}_\tau}^{\pi_\textsf{FA}}
= \tau C/a + (M_1-M_{\tau+1})
= \tau C/a + (n-1)$. Since the offline optimum can allocate all $\tau C$ assignments to threat $n$, the number of effective assignments under the offline optimal policy is
$\tau C$. Thus the competitive ratio on this instance is
\[
\frac{V_{\textsf{EAM}_\tau}^{\pi_\textsf{FA}}(L)}{\min\{C\tau ,\sum_{i=1}^n L_i\}} 
=
\frac{\tau(C/a)+(n-1)}{\tau C}
=
\frac{1}{a}+\frac{n-1}{\tau C}
=
n^{-1/\tau}+\frac{n-1}{\tau C},
\]
which approaches $n^{-1/\tau}$ as $C\to\infty$.

\smallskip

\noindent\underline{Tightness construction under the online setting when $\tau=1$.}
Now we show that when $\tau=1$, the worst case ratio is also asymptotically tight under the online setting where the decision maker does not know $L$ in prior. Let $C=nk$ for some $k\in\mathbb{Z}$ and consider the instance $q_1=1-\epsilon$ and $q_2=\cdots=q_n=\epsilon$ where $\epsilon=1/C^2$. Under fair allocation, each threat receives $k=C/n$ units, the expected number of effective assignments follows
\begin{equation}\label{eq:v_eam1_fair} V_{\textsf{EAM}_1}^{\pi_{\textsf{FA}}}
=
\frac{1-(1-\epsilon)^k}{\epsilon}
+
(n-1)\frac{1-\epsilon^k}{1-\epsilon}.
\end{equation}
Since $k\epsilon=1/(nC)\to 0$ as $C\to\infty$, a Taylor expansion gives
$(1-\epsilon)^k = 1-k\epsilon+O(k^2\epsilon^2)$. Therefore, the first term in equation \eqref{eq:v_eam1_fair} can be written as
\[
\frac{1-(1-\epsilon)^k}{\epsilon}
=
k+O(k^2\epsilon)
=
k+O(1).
\]
Also, when $\epsilon\to 0$, we have $(1-\epsilon^k)/(1-\epsilon)\to 1$. Therefore, from \eqref{eq:v_eam1_fair}, it follows that
$
V_{\textsf{EAM}_1}^{\pi_{\textsf{FA}}}
=
C/n+O(n)
$.
On the other hand, from $\pi_{\textsf{EA}}$, the optimal allocation follows $x^*=(C-n+1,1,\dots,1)$ for large enough $C$,
which yields
\[
V_{\textsf{EAM}_1}^*=\frac{1-(1-\epsilon)^{C-n+1}}{\epsilon}+(n-1).
\]
Similarly, we obtain $V_{\textsf{EAM}_1}^*=C+o(C)$. Therefore, when $C\rightarrow \infty$, we know that
\[
\frac{V_{\textsf{EAM}_1}^{\pi_{\textsf{FA}}}}{V_{\textsf{EAM}_1}^*}
=
\frac{\frac{C}{n}+O(n)}{C+o(C)}
\to \frac{1}{n}
\]
holds for any fixed $n$. This proves asymptotic tightness.

\end{proof}

\medskip

\begin{proof}{Proof of Lemma \ref{lem:separable_concave_greedy}.}
We prove the result by induction on $C$. The claim is immediate when $C=0$. Fix $C\ge 1$ and suppose it holds for budget $C-1$. Let $i^*\in \argmax_{i\in[n]} \Delta_{i}(0)$ where $\Delta_{i}(x):=f_i(x+1)-f_i(x)$ is the marginal gain for threat $i$. We first show that there exists an optimal solution $x^*$ to the budget-$C$ problem with $x_{i^*}^*\ge 1$. Let $x^*$ be any optimal solution. If $x_{i^*}^*\ge 1$, there is nothing to prove. Otherwise, since $\sum_{i=1}^n x_i^*=C$, there exists some $j\neq i^*$ with $x_j^*\ge 1$. Define $\tilde x$ by moving one unit from $j$ to $i^*$:
\[
\tilde x_{i^*}=1,\qquad \tilde x_j=x_j^*-1,\qquad \tilde x_\ell=x_\ell^* \ \text{for }\ell\notin\{i^*,j\}.
\]
Then
\[
\sum_{i=1}^n f_i(\tilde x_i)-\sum_{i=1}^n f_i(x_i^*)
=\Delta_{i^*}(0)-\Delta_j(x_j^*-1)\ge \Delta_{i^*}(0)-\Delta_j(0)\ge 0,
\]
where the first inequality uses the discrete concavity of $f_j$, and the second follows from the choice of $i^*$. Hence $\tilde x$ is also optimal, proving that some optimal solution allocates one unit to $i^*$. Then conditional on assigning one unit to $i^*$, the remaining problem is
\[
\max_{y\in\mathbb Z_{\ge 0}^n:\ \sum_{i=1}^n y_i=C-1}
\Bigl(f_{i^*}(y_{i^*}+1)+\sum_{i\neq i^*} f_i(y_i)\Bigr).
\]
Now we define $
g_{i^*}(k):=f_{i^*}(k+1)$ and $g_i(k):=f_i(k)$ for $i\neq i^*$. Each $g_i$ remains non-decreasing and discretely concave. Therefore, by the induction hypothesis, the above problem is optimally solved by sequentially assigning the remaining $C-1$ units to the coordinate with the largest current marginal gain. Since the first unit is assigned to $i^*$ by maximizing $\Delta_i(0)$, the entire budget-$C$ solution is obtained by the stated greedy rule. This proves the optimality of the greedy construction for every $C$.
\halmos
\end{proof}

\section{\Cref{sec:greedy_policies} Proofs}

\subsection{\Cref{example:ml_greedy_no_const_ratio}}

\begin{proof}{Proof of \Cref{example:ml_greedy_no_const_ratio}.}
We divide our threats into two groups, and each threat within each group has the same difficulty level. We construct an instance with deadline $\tau=2$. There are $m$ hard threats and $2m^3-m$ easy threats, so that $n(m)=2m^3$. The per-round capacity is $C(m)=5m^3-m$. Each hard threat has difficulty $q_H:=1-\epsilon$, where $\epsilon:=1/(4m^4)$, and each easy threat has difficulty $q_E:=1/(m^2+1)$. Let $\Delta_E(i)$ be the exponentiated marginal improvement from allocating one more effector to an easy threat given that we have already assigned $i$ effectors to that threat. Let $\Delta_H(i)$ be the corresponding quantity for a hard threat. Then we have
\begin{equation}
\label{eq:easy_marginal}
\Delta_E(1)
=
\frac{1-q_E^2}{1-q_E}
=
1+q_E
=
\frac{m^2+2}{m^2+1}.
\end{equation}
Now we compare $\Delta_H(\cdot)$ with $\Delta_E(1)$. For any integer $x\ge 1$,
\begin{equation}\label{eq:prop1_qh}
\begin{aligned}
\frac{1-q_H^{x+1}}{1-q_H^{x}}= &1+\frac{q_H^{x}(1-q_H)}{1-q_H^x}=1+\frac{q_H^{x}}{1+q_H+\cdots+q_H^{x-1}}\geq 1+\frac{q_H^{x}}{x},
\end{aligned}
\end{equation}
where the inequality is due to $1+q_H+\cdots+q_H^{x-1}\le x$. By Bernoulli inequality, we have $q_H^{m^2}=(1-\epsilon)^{m^2}\ge 1-m^2\epsilon$, and with $\epsilon=1/(4m^4)$, we know
$$
1-m^2\epsilon=1-\frac{1}{4m^2}>\frac{m^2}{m^2+1}=1-\frac{1}{m^2+1}.
$$
Thus, by equation \eqref{eq:prop1_qh}, the marginal comparison between $\Delta_H(m^2)$ and $\Delta_E(1)$ follows
\begin{equation}
\label{eq:hard_beats_easy_m2}
\Delta_H(m^2)=\frac{1-q_H^{m^2+1}}{1-q_H^{m^2}}
\ \ge\ 1+\frac{q_H^{m^2}}{m^2}
\ >\ 1+\frac{1}{m^2+1}
=\frac{m^2+2}{m^2+1}=\Delta_E(1).
\end{equation}
On the other hand, since $1+q_H+\cdots+q_H^{m^2}\ge (m^2+1)q_H^{m^2}$, we have
$$
\begin{aligned}
\frac{1-q_H^{m^2+2}}{1-q_H^{m^2+1}}
&=1+\frac{q_H^{m^2+1}}{1+q_H+\cdots+q_H^{m^2}}\le 1+\frac{q_H^{m^2+1}}{(m^2+1)q_H^{m^2}}=1+\frac{q_H}{m^2+1}
<1+\frac{1}{m^2+1}
=\frac{m^2+2}{m^2+1},
\end{aligned}
$$
which implies that the marginal comparison $\Delta_H(m^2+1)<\Delta_E(1)$%

Since unassigned threats have the largest first marginal value under $\pi_{\textsf{SL}}$, the first $n(m)$ effectors give one effector to every threat. Since $\Delta_H(m^2+1)<\Delta_E(1)$, equation \eqref{eq:hard_beats_easy_m2}  implies that in round $1$, $\pi_{\textsf{SL}}$ allocates $m^2+1$ effectors to each hard threat and $2$ effectors to each easy threat. Indeed, this exhausts the capacity because $
m(m^2+1)+2(2m^3-m)=5m^3-m=C(m)$.

Let $A_H$ be the event that all \emph{hard threats} are neutralized by round $2$ under $\pi_{\textsf{SL}}$. Since neutralizing all threats implies neutralizing all hard threats, we have $V_{\textsf{SLM}_2}^{\pi_{\textsf{SL}}}\le \Pr(A_H)$. After round $1$, each hard threat has received $a$ effectors. Let $S\subseteq\{1,\dots,m\}$ be the random set of \emph{hard threats} that survive round $1$. Conditional on the round-$1$ feedback, $\pi_{\textsf{SL}}$ chooses its round-$2$ allocation. Recall that $x_i^2$ denote the number of round-$2$ effectors assigned to threat $i$. By the memoryless property of the geometric demand distribution, conditional on the round-$1$ feedback $\mathcal{M}_2$,
$$
\Pr(A_H\mid \mathcal{M}_2)
=
\prod_{i\in S}\left(1-q_H^{x_i^2}\right).
$$
Taking expectations gives $\Pr(A_H)
=
\mathbb{E}\left[
\prod_{i\in S}\left(1-q_H^{x_i^2}\right)
\right]$. Now we condition only on the set $S$ of hard survivors and upper bound over all possible second-round allocations. When $|S|=s$ for some $s$, the probability of this hard-survivor set is $q_H^{(m^2+1)s}(1-q_H^{m^2+1})^{m-s}$. This implies that
\begin{equation}
\label{eq:prop1_all_hard_neutralized}
\begin{aligned}
\Pr(A_H)
&\le
\sum_{s=0}^m
\binom{m}{s}
q_H^{(m^2+1)s}(1-q_H^{m^2+1})^{m-s}
\max_{\substack{x_i\ge 0\\ \sum_{i=1}^s x_i\le C(m)}}
\prod_{i=1}^s \left(1-q_H^{x_i}\right).
\end{aligned}
\end{equation}
Using $q_H^{(m^2+1)s}\le 1$, $1-q_H^{m^2+1}\le \epsilon (m^2+1)$, and $1-q_H^x=1-(1-\epsilon)^x\le \epsilon x$ for all $x\ge 0$, we obtain
$$
\Pr(A_H)
\le
\sum_{s=0}^m
\binom{m}{s}
(\epsilon (m^2+1))^{m-s}
\max_{\substack{x_i\ge 0\\ \sum_{i=1}^s x_i\le C(m)}}
\prod_{i=1}^s \epsilon x_i .
$$
For $s\ge 1$, by AM-GM inequality, we can further bound the product of $x_i$ as $\prod_{i=1}^s x_i
\le
\left(C(m)/s\right)^s$. Thus, we have
\begin{equation}
\label{eq:greedy_upper_sum_m2}
\Pr(A_H)
\le
\sum_{s=0}^m
\binom{m}{s}
(\epsilon (m^2+1))^{m-s}
\left(\frac{\epsilon C(m)}{s}\right)^s,
\end{equation}
where the $s=0$ term is interpreted as $(\epsilon a)^m$. Let $\theta:= s/m\in[0,1]$ and
$H(\theta):=-\theta\ln\theta-(1-\theta)\ln(1-\theta)$. Using the entropy bound $\binom{m}{s} \leq \exp \{m H(\theta)\}$ and $C(m)=m^3(5+o(1))$,  the $s$-th term in equation \eqref{eq:greedy_upper_sum_m2} is at most
$$
\begin{aligned}
&\binom{m}{s}(\epsilon (m^2+1))^{m-s}\left(\frac{\epsilon C(m)}{s}\right)^s
\leq  \exp \{m H(\theta)\}\left(\epsilon m^2(1+o(1))\right)^{m-s}\left(\epsilon m^2\left(\frac{5}{\theta}+o(1)\right)\right)^s \\
= &\ (\epsilon m^2)^m \exp \{m H(\theta)\}\left(\frac{5}{\theta}\right)^s \exp \{o(m)\} 
= \left[\epsilon m^2 \exp \left\{H(\theta)+\theta \ln \frac{5}{\theta}+o(1)\right\}\right]^m
.
\end{aligned}
$$

Since there are only $m+1$ terms in the sum, this polynomial prefactor is absorbed into the $o(1)$ term. Therefore, we have
\begin{equation}
\label{eq:greedy_upper_rho_m2}
\Pr(A_H)
\le
\left(\left(\max_{\theta\in[0,1]}
\exp\left\{
H(\theta)+\theta\ln (5/\theta)
\right\}+o(1)\right)\epsilon m^2\right)^m,
\end{equation}
The maximizer $\theta^*$ satisfies $5(1-\theta^*)=e(\theta^*)^2$. One can show that $\theta^*<0.72$, and therefore we have
$$\max_{\theta\in[0,1]}
\exp\left\{
H(\theta)+\theta\ln (5/\theta)
\right\}=e^{\theta^*}/(1-\theta^*)
<e^{0.72}/0.28<7.34.$$ 
Then by equations \eqref{eq:greedy_upper_rho_m2}, we have
\begin{equation}
\label{eq:greedy_upper_m2}
V_{\textsf{SLM}_2}^{\pi_{\textsf{SL}}}
\le
\Pr(A_H)
\le
\left((7.34+o(1))\epsilon m^2\right)^m.
\end{equation}

Now we compare $\pi_{\textsf{SL}}$ to another policy $\pi_1$, which is constructed as follows. In round $1$, it allocates $3m^2$ effectors to each hard threat and $1$ effector to each easy threat. This is feasible because $3m^2\cdot m+(2m^3-m)=5m^3-m=C(m)$. In round $2$, it allocates one unit to each easy threat that survived round $1$ and allocates the remaining capacity among the hard threats evenly.

Let $S_E:=\#\{i:q_i=q_E,\ L_i>1\}$ be the number of easy threats that require more than one effector to neutralize. Then $S_E\sim\mathrm{Bin}(2m^3-m,q_E)$ and $\mathbb E[S_E]=(2m^3-m)/(m^2+1)\le 2m$. By Chernoff's bound,
\begin{equation}
\label{eq:chernoff_m2}
\Pr(S_E\le 4m)\ge 1-e^{-\mathbb E[S_E]/3}\to 1.
\end{equation}
When $S_E\le 4m$ occurs, round $2$ allocates at most $4m$ effectors to easy threats, so at least $C(m)-4m$ effectors are allocated to hard threats. Allocating the remaining capacity evenly implies each surviving hard threat receives at least $\lfloor(C(m)-4 m) / m\rfloor=5 m^2-5$. Thus, if every hard threat has total demand at most $3m^2+(5m^2-5)=8m^2-5$, then all hard threats are neutralized by round $2$. Conditional on $S_E\le 4m$, after assigning one additional effector to each surviving easy threat, all surviving easy threats are neutralized with probability at least $(1-q_E)^{4m}$. Therefore, $V_{\textsf{SLM}_2}^{\pi_1}$ can be lower bounded by
\begin{equation}
\begin{aligned}
\label{eq:alt_lower_m2}
V_{\textsf{SLM}_2}^{\pi_1}
&\ge
\Pr(S_E\le 4m)\,
(1-q_E)^{4m}\,
\left(1-q_H^{8m^2-5}\right)^m .
\end{aligned}
\end{equation}
Next, with $\epsilon=1/(4m^4)$, we can bound the last term in equation \eqref{eq:alt_lower_m2} as
\begin{equation}\label{eq:alt_lower_m3}
\begin{aligned}
1-q_H^{8m^2-5}
&=1-(1-\epsilon)^{8m^2-5} \ge 1-e^{-\epsilon(8m^2-5)} \ge \epsilon(8m^2-5)\left(1-\frac{\epsilon(8m^2-5)}{2}\right) =(8-o(1))\epsilon m^2.
\end{aligned}
\end{equation}
In equation \eqref{eq:alt_lower_m3}, the first inequality uses the standard bound $(1-\epsilon)^x \leq e^{-\epsilon x}$ for $\epsilon \in(0,1)$, applied with $x=8m^2-5$. The second inequality uses the Taylor lower bound $1-e^{-x} \geq x-\frac{x^2}{2}=x(1-x/2)$ for $x \geq 0$, applied with $x=\epsilon(8m^2-5)$. Since $(1-q_E)^{4m}\to 1$ and $\Pr(S_E\le 4m)\to 1$, equations \eqref{eq:alt_lower_m2} and \eqref{eq:alt_lower_m3} imply
\begin{equation}
\label{eq:alt_lower_final_m2}
V_{\textsf{SLM}_2}^{\pi_1}
\ge
\left((8-o(1))\epsilon m^2\right)^m.
\end{equation}
Finally, for large enough $m$, combining \eqref{eq:greedy_upper_m2} and \eqref{eq:alt_lower_final_m2} gives
$$
\frac{V_{\textsf{SLM}_2}^{\pi_{\textsf{SL}}}}{V_{\textsf{SLM}_2}^*}
\le
\frac{V_{\textsf{SLM}_2}^{\pi_{\textsf{SL}}}}{V_{\textsf{SLM}_2}^{\pi_1}}
\le
\left(
\frac{7.34+o(1)}{8-o(1)}
\right)^m
\to 0,
$$
which completes the proof.
\end{proof}

\subsection{\Cref{thm:sl_vs_fair_n2} and \Cref{thm:ea_vs_fair_n2}}

\begin{proof}{Proofs of \Cref{thm:sl_vs_fair_n2} and \Cref{thm:ea_vs_fair_n2}.}
    To show the optimality under \cref{eq::MinDefuse} objective, it is sufficient to show that $\pi_{\textsf{SL}}$ is better than $\pi_\textsf{FA}$ under \cref{eq::SurvivalObjective} for any $\tau$. Now we split the proof into the following two cases.

\smallskip

\noindent\underline{Case 1: $V_{\textsf{SLM}_\tau}^{\pi_\textsf{SL}}\;\ge\;V_{\textsf{SLM}_\tau}^{\pi_\textsf{FA}}$.}
The core idea of the proof is based on mathematical induction and the use of lemma \ref{lemma:switch_dominate}. We use an induction on $\tau$. Before moving forward, we assume $q_1>q_2$ without loss of generality and consider the following two policies:
\begin{itemize}
    \item $\pi_1$: use fair allocation $\pi_{\textsf{FA}}$ in the first round and greedy policy $\pi_{\textsf{SL}}$ starting from the second round.
    \item $\pi_2$: use greedy policy $\pi_{\textsf{SL}}$ in the first round, fair allocation $\pi_{\textsf{FA}}$ in the second round, and again greedy policy $\pi_{\textsf{SL}}$ starting from the third round.
\end{itemize}
For induction hypothesis, we assume that $V_{\textsf{SLM}_\tau}^{\pi_\textsf{SL}}\geq V_{\textsf{SLM}_\tau}^{\pi_2}\ge V_{\textsf{SLM}_\tau}^{\pi_1} \ge V_{\textsf{SLM}_\tau}^{\pi_\textsf{FA}}$ holds for $\tau=1,\cdots,\tau_0$. Note that it holds trivially when $\tau=1$. We want to show it holds for $\tau=\tau_0+1$ as well. 
Let $(x_1^{\textsf{SL}},x_2^{\textsf{SL}})$ be the first round allocation under $\pi_{\textsf{SL}}$ and $(x_1^{\textsf{FA}},x_2^{\textsf{FA}})$ be the first round allocation under $\pi_{\textsf{FA}}$. It is clear that $x_1^{\textsf{SL}}\geq x_1^{\textsf{FA}}$. Without loss of generality, we can assume $x_1^{\textsf{SL}}>x_1^{\textsf{FA}}$.
By Lemma \ref{lemma:switch_dominate}, using $(u_1,v_1)=(x_1^{\textsf{FA}},x_2^{\textsf{FA}})$ and $(u_2,v_2)=(x_1^{\textsf{SL}},x_2^{\textsf{SL}})$, we have $V_{\textsf{SLM}_\tau}^{\pi_1}\leq V_{\textsf{SLM}_\tau}^{\pi_2}$. On the other hand, by induction, we know $\pi_{\textsf{SL}}$ is better than $\pi_{\textsf{FA}}$ for $\tau=1,\cdots,\tau_0$. Moreover, notice that $\pi_1$ and $\pi_{\textsf{FA}}$ has the same allocation in the first round, so we can couple the two policies together and assume they have the same outcome in the first round. If any threats are neutralized in the first round, then we will put all effectors in the active threat, and this will not yield any difference of outcomes towards the two policies. Thus, we only need to consider the case when no threats are neutralized. As a result, $\pi_1$ is better than $\pi_{\textsf{FA}}$ under deadline $\tau_0+1$ as well, which yields $V_{\textsf{SLM}_{\tau_0+1}}^{\pi_1}\geq V_{\textsf{SLM}_{\tau_0+1}}^{\pi_\textsf{FA}}$. Similarly, notice that $\pi_2$ and $\pi_{\textsf{SL}}$ has the same allocation in the first round, by induction, we have $V_{\textsf{SLM}_{\tau_0+1}}^{\pi_\textsf{SL}}\geq V_{\textsf{SLM}_{\tau_0+1}}^{\pi_2}$. Therefore, for $\tau=\tau_0+1$, we also have $V_{\textsf{SLM}_\tau}^{\pi_\textsf{SL}}\geq V_{\textsf{SLM}_\tau}^{\pi_2}\ge V_{\textsf{SLM}_\tau}^{\pi_1} \ge V_{\textsf{SLM}_\tau}^{\pi_\textsf{FA}}$. The induction holds and we know $V_{\textsf{SLM}_\tau}^{\pi_\textsf{SL}} \ge V_{\textsf{SLM}_\tau}^{\pi_\textsf{FA}}$ for any $\tau$ when $n=2$.

\smallskip

\noindent\underline{Case 2: $V_{\textsf{EAM}_\tau}^{\pi_\textsf{EA}}\;\ge\;V_{\textsf{EAM}_\tau}^{\pi_\textsf{FA}}$.} Since lemma \ref{lemma:switch_dominate} also holds for \cref{eq::EffectiveObjective}, the proof under \cref{eq::EffectiveObjective} is exactly the same as one under \cref{eq::SurvivalObjective}. We omit the details here.

\end{proof}

\subsection{\Cref{thm:adaptive_greedy}}
\begin{proof}{Proof of \Cref{thm:adaptive_greedy}.}
Let $V_{\textsf{EAM}_\tau}^{\pi}(t)$ be the expected number of effective assignments in round $t\leq \tau$ under policy $\pi$. Let $\pi_{j}$ be the policy that uses $\pi_{\textsf{EA}}$ for first $j$ rounds, and uses optimal policy $\pi^*$ starting from round $j+1$. Let $V_{\textsf{EAM}_\tau}^{\pi}(\mathcal{M})$ be the corresponding expected number of effective assignments with a set of threats $\mathcal{M}$ at the beginning of the process. Let $\mathcal{M}_t$ be the set of threats alive under policy $\pi_j$ at the beginning of round $t$.

When $\pi_j$ is used, the first $j$ rounds coincide with $\pi_{\textsf{EA}}$. Hence, at round $j+1$, the current threat set is $\mathcal{M}_{j+1}$, and the contribution of $\pi_{\textsf{EA }}$ in that round can be written as $V_{\textsf{EAM}_{\tau+j}}^{\pi_{\textsf{EA}}}(j+1)=
\mathbb{E}[V_{\textsf{EAM}_1}^{\pi_{\textsf{EA}}}\left(\mathcal{M}_{j+1}\right)]$, where the expectation is taken over the randomness of $\mathcal{M}_{j+1}$. Now fix any later round $t \geq j+1$. Since threats can only be removed over time, we have $\mathcal{M}_t \subseteq \mathcal{M}_{j+1}$. Therefore every action feasible for $\mathcal{M}_t$ can be viewed as a feasible action for    $\mathcal{M}_{j+1}$ by assigning zero units to threats in $\mathcal{M}_{j+1} \backslash \mathcal{M}_t$. Then it follows that the optimal one-round effective-assignment value is monotone, that is, $\mathbb{E}[V_{\textsf{EAM}_1}^{\pi_ \textsf{EA}}\left(\mathcal{M}_{j+1}\right)] \geq \mathbb{E}[V_{\textsf{EAM}_1}^{\pi_\textsf{EA}}\left(\mathcal{M}_t\right)]$. Moreover, $\pi_{\textsf{EA}}$ is optimal for the one-round $(\tau=1)$ effective-assignment problem. Hence, on the state $\mathcal{M}_t$, its one-round value is at least that obtained by any other policy, in particular $\pi_j$: $\mathbb{E}[V_{\textsf{EAM}_1}^{\pi_{\textsf{EA}}}\left(\mathcal{M}_t\right)] \geq V_{\textsf{EAM}_{\tau+j}}^{\pi_j}(t)$. Combining the above equations, we have   \begin{equation}\label{eq:AdaptiveMonotonicity}
    \begin{aligned}
     V_{\textsf{EAM}_{j+1}}^{\pi_{\textsf{EA}}}-V_{\textsf{EAM}_j}^{\pi_{\textsf{EA}}}&=V_{\textsf{EAM}_{j+1}}^{\pi_{\textsf{EA}}}(j+1)
            = \frac{\tau \mathbb{E}[V_{\textsf{EAM}_1}^{\pi_{\textsf{EA}}}(\mathcal{M}_{j+1})]}{\tau} \geq \frac{\sum_{t=j+1}^{\tau+j} \mathbb{E}[V_{\textsf{EAM}_1}^{\pi_{\textsf{EA}}}(\mathcal{M}_t)]}{\tau}\\
            &\geq \frac{\sum_{t=j+1}^{\tau+j} V_{\textsf{EAM}_{\tau+j}}^{\pi_j}(t)}{\tau} =\frac{V_{\textsf{EAM}_{\tau+j}}^{\pi_j}-V_{\textsf{EAM}_j}^{\pi_j}}{\tau}.
        \end{aligned}
    \end{equation}

Now we claim that $V_{\textsf{EAM}_\tau}^* \leq V_{\textsf{EAM}_{\tau+j}}^{\pi_j}$. To prove the claim, let $\pi^*$ be an optimal policy for deadline $\tau$, so that $V_{\textsf{EAM}_\tau}^*=V_{\textsf{EAM}_\tau}^{\pi^*}$. Towards comparing $\pi^*$ and $\pi_j$, we introduce an auxiliary policy $\pi_j^\prime$ operating over $\tau+j$ rounds. Like $\pi_j$, the auxiliary policy $\pi_j'$ also executes $\pi_{\textsf{EA}}$ for its first $j$ rounds, and then starting from round $j+1$, it also executes $\pi^*$. However, $\pi_j'$ will critically depart from $\pi_j$ starting at time step $j+1$, whereupon instead of taking $\mathcal{M}_{j+1}$ as the set of live threats (as $\pi_j$ would), $\pi_j'$ will view $[n]$ as the set of live threats; more precisely, we append ``pseudo threats" to turn $\mathcal{M}_{j+1}$ into $[n]$ by simply drawing from $\text{Geom}(1-q_i)$ for each $i \in [n] \setminus \mathcal{M}_{j+1}$ (independently). During the time steps $j+1$ through $j + \tau$, the auxiliary policy $\pi_j'$ then executes $\pi^*$ on this inflated set $[n]$ but is of course only credited with effective assignments toward the threats among $\mathcal{M}_{j+1}$. We can show via a coupling argument that, for any threat, the expected number of assignments made by $\pi_j'$ during the $\tau$ many steps from $j+1$ to $j+\tau$ is equal to the expected number of assignments under the execution of $\pi^*$ from $1$ to $\tau,$ which suffices to establish the claim, since then $\pi_j^\prime$ allocates no fewer effectors than does $\pi^*$ to any threat during their $\tau + j$ and $\tau$ many periods respectively, and hence the number of effective assignments obtained by $\pi_j^\prime$ over $\tau+j$ rounds is at least that obtained by $\pi^*$ over $\tau$ rounds. Let $L=(L_i)_{i \in [n]}$ denote a vector of independently drawn demands, and let $L'$ denote the random vector of demands that $\pi'_j$ faces at time $j+1$ after contending with $L$ for $j$ periods (which includes newly drawn demands for the ``pseudo" threats). If $L$ and $L'$ are stochastically equivalent, then upon letting a single, random vector $\bar{L}$ be simultaneously the vector of demands that $\pi^*$ contends with during periods $1$ to $\tau$ as well as the vector of demands that $\pi_j'$ contends with during periods $j$ to $j + \tau,$ the claim follows immediately. The stochastic equivalence holds because $L_i' \sim \text{Geom}(1-q_i)$ for all $i \in [n]$, independently; indeed, in the event that $i \notin \mathcal{M}_{j+1}$, by construction, $L_i' \sim \text{Geom}(1-q_i)$, and in the event that $i \in \mathcal{M}_{j+1},$ again $L_i' \sim \text{Geom}(1-q_i)$  by the memoryless property.

This yields $V_{\textsf{EAM}_\tau}^*\leq V_{\textsf{EAM}_{\tau+j}}^{\pi_j^\prime} \leq V_{\textsf{EAM}_{\tau+j}}^{\pi_j}$.
Consequently, by equation \eqref{eq:AdaptiveMonotonicity}, we have
\begin{equation}\label{eq:theorem7_V_ea_tau_ineq}
        V_{\textsf{EAM}_\tau}^* \leq V_{\textsf{EAM}_{\tau+j}}^{\pi_j} \leq V_{\textsf{EAM}_j}^{\pi_{\textsf{EA}}} + \tau\cdot \left(V_{\textsf{EAM}_{j+1}}^{\pi_{\textsf{EA}}}-V_{\textsf{EAM}_j}^{\pi_{\textsf{EA}}}\right).
\end{equation}
    Equation \eqref{eq:theorem7_V_ea_tau_ineq} is equivalent to
    \[
    V_{\textsf{EAM}_\tau}^* - V_{\textsf{EAM}_{j+1}}^{\pi_{\textsf{EA}}} \leq (1 - \frac{1}{\tau}) \left(V_{\textsf{EAM}_\tau}^* - V_{\textsf{EAM}_j}^{\pi_{\textsf{EA}}}\right).
    \]
    Hence, upon applying this inductively on $j=\tau-1,\tau-2,\cdots,1$, we get 
     \[
    V_{\textsf{EAM}_\tau}^* - V_{\textsf{EAM}_\tau}^{\pi_{\textsf{EA}}} \leq (1 - \frac{1}{\tau})^\tau \cdot V_{\textsf{EAM}_\tau}^*
    \]
    as desired. \halmos
\end{proof}

\medskip

\section{Bellman equations and the two-threat case}\label{sec:bellman_optimality}
In this section, we discuss how to derive an optimal policy using the Bellman equations and show that the problem can be solved efficiently when $n = 2$. Recall that each threat $i\in\mathcal M_t$ survives independently with probability $q_i^{x_i}$ under allocation $x$, so the next state is a random subset $\mathcal S\subseteq\mathcal M_t$ with transition probability
\[
\operatorname{Pr}(\mathcal M_{t+1}=\mathcal S\mid \mathcal M_t, x)
=\prod_{i\in\mathcal S} q_i^{x_i}\prod_{i\in\mathcal M_t\setminus\mathcal S}(1-q_i^{x_i}).
\]
We start with \cref{eq::MinDefuse} objective. Under this objective, the expected number of rounds needed to neutralize a given set of threats does not depend on the round index $t$. Hence, the state can be simplified to the current set of active threats alone, namely $\mathcal M\subseteq [n]$. Let $V_{\textsf{DTM}}(\mathcal M)$ denote the optimal expected number of remaining rounds needed to neutralize all threats starting from state $\mathcal M$. Then $V_{\textsf{DTM}}(\emptyset)=0$, and for any nonempty $\mathcal M\neq\emptyset$ the Bellman optimality equation under \cref{eq::MinDefuse} is
\begin{equation}\label{eq:Bellman}
V_{\textsf{DTM}}(\mathcal M)
=
1+\min_{x\in\mathcal X(\mathcal M)}
\sum_{\mathcal S\subseteq\mathcal M}
V_{\textsf{DTM}}(\mathcal S)\,
\prod_{i\in\mathcal S} q_i^{x_i}\,
\prod_{i\in\mathcal M\setminus\mathcal S}(1-q_i^{x_i}),
\end{equation}
where $\mathcal{X}(\mathcal M_t):=\{x\in\mathbb Z_{\ge0}^{M_t}:\sum_{i\in\mathcal M_t}x_i= C\}$ is the action space under state $(t,\mathcal{M}_t)$. Equation \eqref{eq:Bellman} can be derived from a one-step analysis. After one round, the system transitions from some active threat set $\mathcal{M}$ to $\mathcal{S} \subseteq \mathcal{M}$, where $\mathcal{S}$ is the set of threats that survive the round. Then the total expected number of remaining rounds can be computed as the one round used plus the minimum expected number of additional rounds needed to neutralize all active threats.

\paragraph{Solving the Bellman equations.}
A direct method to solve equation \eqref{eq:Bellman} is via dynamic programming. In our setting, starting from any set of threats $\mathcal M$, the process can only move to subsets of $\mathcal M$ in the next round. Separating the term $\mathcal S=\mathcal M$ in equation \eqref{eq:Bellman} yields the equivalent form:
\begin{equation} \label{eq:Bellman_min}
V_{\textsf{DTM}}(\mathcal M)
=
\min_{x\in\mathcal X(\mathcal M)}
\frac{
1 + \sum_{\mathcal S\subset \mathcal M} V_{\textsf{DTM}}(\mathcal S)\,
\prod_{i\in\mathcal S} q_i^{x_i}\,
\prod_{i\in\mathcal M\setminus\mathcal S}(1-q_i^{x_i})
}{
1-\prod_{i\in\mathcal M} q_i^{x_i}
}.
\end{equation}
For any allocation $x$, the right-hand side of equation \eqref{eq:Bellman_min} depends only on values $V_{\textsf{DTM}}(\mathcal S)$ for strict subsets $\mathcal S\subset\mathcal M$. This implies that $V_{\textsf{DTM}}(\mathcal M)$ can be computed recursively in increasing order of $|\mathcal M|$: starting from $V_{\textsf{DTM}}(\emptyset)=0$, we compute $V_{\textsf{DTM}}(\mathcal M)$ for all singleton sets, then all pairs, and so on up to $| \mathcal M |=n$. Once $V_{\textsf{DTM}}([n])$ is computed, an optimal policy is obtained by choosing an allocation $x$ attaining the minimum in equation \eqref{eq:Bellman_min} for each $\mathcal{M}_t$ in round $t$. The recursion provides an efficient exact solution for small $n$.

We can formulate the Bellman equations under \cref{eq::SurvivalObjective} and \cref{eq::EffectiveObjective} in a similar manner. For each $\tau$, the value function can be written as the immediate reward plus the expected value-to-go with deadline $\tau-1$. Under \cref{eq::SurvivalObjective}, the reward is $1$ if all threats have been neutralized before the next round (and $0$ otherwise), whereas under \cref{eq::EffectiveObjective} the reward is the number of effective assignments in the round. These dynamic programs can be formulated in a similar way and we omit the full recursions for brevity.

For the structural results that follow, it is useful to isolate the special case $n=2$. In this case, we substitute $x_2 = C - x_1$ into equation \eqref{eq:Bellman_min}, which reduces the allocation decision in each round to a one-dimensional problem. On the other hand, when only a single threat $i$ remains, the \cref{eq::MinDefuse} value function has the simple closed-form expression $V_{\textsf{DTM}}(\{i\}) = 1/(1 - q_i^C)$. Thus, under \cref{eq::MinDefuse}, an optimal policy can be obtained by performing a line search towards equation \eqref{eq:Bellman} over $x_1 \in \{0,\ldots,C\}$.  Moreover, as illustrated by the following lemma, the resulting one-dimensional optimization problem is concave under both \cref{eq::SurvivalObjective} and \cref{eq::EffectiveObjective}, given the optimal objective values under $\tau-1$. Therefore, under both \cref{eq::SurvivalObjective} and \cref{eq::EffectiveObjective}, an optimal policy can be computed by dynamically solving the concave program associated with the \cref{eq::SurvivalObjective} or \cref{eq::EffectiveObjective} at each stage $\tau=1,\ldots,\tau$.

\begin{proposition}[Two-threat concavity]
\label{lem:ctm_n2_concavity}
Let there be $n=2$ threats. Given $V_{\textsf{SLM}_{\tau-1}}^*$ and $V_{\textsf{EAM}_{\tau-1}}^*$, the maximization for \cref{eq::SurvivalObjective} and \cref{eq::EffectiveObjective} can be written as a one-dimensional concave optimization over $x\in\{0,\dots,C\}$ with $(x_1,x_2)=(x,C-x)$.
\end{proposition}

\begin{proof}{Proof of Proposition \ref{lem:ctm_n2_concavity}.}
We prove the \cref{eq::SurvivalObjective} case and \cref{eq::EffectiveObjective} case as follows.

\smallskip

\noindent \underline{\cref{eq::SurvivalObjective} case.}
    Let the allocation in the first round $(x_1,x_2)=(x,C-x)$. Let $V_{\textsf{SLM}_\tau}(x)$ be the survival probability that uses $(x_1,x_2)=(x,C-x)$ in the first round and optimal policy starting from round $2$. Then the survival probability can be written as
    \begin{align*}
        V_{\textsf{SLM}_\tau}(\{1,2\})=& \, \max_{x\in\{0,1,\cdots,C\}} V_{\textsf{SLM}_\tau}(x) = \max_{x\in\{0,1,\cdots,C\}} \left(1-q_1^x\right)\left(1-q_2^{C-x}\right)+q_1^x\left(1-q_2^{C-x}\right)\left(1-q_1^{(\tau_1-1)C}\right)\\
        &+q_2^{C-x}\left(1-q_1^x\right)\left(1-q_2^{(\tau_2-1)C}\right)+q_1^x q_2^{C-x}V_{\textsf{SLM}_{\tau-1}}(\{1,2\})\\
        =&\max_{x\in\{0,1,\cdots,C\}} \, 1+K q_1^x q_2^{C-x}-q_1^{(\tau_1-1)C+x}-q_2^{\tau_2 C-x},
    \end{align*}
    where $K:=V_{\textsf{SLM}_{\tau-1}}(\{1,2\})+q_1^{C(\tau-1)}+q_2^{C(\tau-1)}-1$.
    Therefore, the second order difference is
    \begin{align*}
        \Delta^2 V_{\textsf{SLM}_\tau}(x)&:=V_{\textsf{SLM}_\tau}(x+1)-2V_{\textsf{SLM}_\tau}(x)+V_{\textsf{SLM}_\tau}(x-1)\\
        &=K q_1^{x-1} q_2^{C-x-1}\left(q_1-q_2\right)^2-q_1^{(\tau-1)C+x-1}\left(1-q_1\right)^2-q_2^{\tau C-x-1}\left(1-q_2\right)^2.
    \end{align*}
    On the other hand, $V_{\textsf{SLM}_\tau}(x)$ will be smaller than the survival probability when threat 2 is killed and only threat 1 is left to kill, so $V_{\textsf{SLM}_{\tau-1}}(x)$ satisfies
    $V_{\textsf{SLM}_{\tau-1}}(\{1,2\})\leq 1-q_1^{C(\tau-1)}$. Thus, we have $K\leq q_2^{C(\tau-1)}$. Then the concavity can be established by the following inequality:
    \begin{align*}
        \Delta^2 V_{\textsf{SLM}_\tau}(x)&\leq q_1^{x-1} q_2^{\tau C-x-1}\left(q_1-q_2\right)^2-q_1^{(\tau-1)C+x-1}\left(1-q_1\right)^2-q_2^{\tau C-x-1}\left(1-q_2\right)^2\\
        &\leq q_2^{\tau C-x-1}\left(q_1-q_2\right)^2-q_2^{\tau C-x-1}\left(1-q_2\right)^2\leq 0.
    \end{align*}

\smallskip

\noindent \underline{\cref{eq::EffectiveObjective} case.}
Let the allocation in the first round be $(x_1,x_2)=(x,C-x)$. Let $V_{\textsf{EAM}_\tau}(x)$ be the expected number of effective assignments obtained by using $(x,C-x)$ in the first round and then applying an optimal policy from round $2$ onward. Then we have
\begin{align*}
    V_{\textsf{EAM}_\tau}(\{1,2\})
    =&\max_{x\in\{0,1,\cdots,C\}} V_{\textsf{EAM}_\tau}(x)\\
    =&\max_{x\in\{0,1,\cdots,C\}}
    \Bigg[
    \frac{1-q_1^x}{1-q_1}
    +\frac{1-q_2^{C-x}}{1-q_2}
    +q_1^x\left(1-q_2^{C-x}\right)\frac{1-q_1^{(\tau-1)C}}{1-q_1}\\
    &\hspace{3.8cm}
    +q_2^{C-x}\left(1-q_1^x\right)\frac{1-q_2^{(\tau-1)C}}{1-q_2}
    +q_1^xq_2^{C-x}V_{\textsf{EAM}_{\tau-1}}(\{1,2\})
    \Bigg].
\end{align*}
Indeed, the first two terms are the expected number of effective assignments in round $1$, while the remaining three terms correspond respectively to the cases in which only threat $1$ survives, only threat $2$ survives, and both threats survive after round $1$. Rearranging terms yields
\begin{align*}
    V_{\textsf{EAM}_\tau}(x)
    =
    \frac{1}{1-q_1}+\frac{1}{1-q_2}
    +Kq_1^xq_2^{C-x}
    -\frac{q_1^{(\tau-1)C+x}}{1-q_1}
    -\frac{q_2^{\tau C-x}}{1-q_2},
\end{align*}
where
$$
K:=V_{\textsf{EAM}_{\tau-1}}(\{1,2\})
-\frac{1-q_1^{(\tau-1)C}}{1-q_1}
-\frac{1-q_2^{(\tau-1)C}}{1-q_2}.
$$
Therefore, the second-order difference is
\begin{align*}
    \Delta^2 V_{\textsf{EAM}_\tau}(x)
    &:=V_{\textsf{EAM}_\tau}(x+1)-2V_{\textsf{EAM}_\tau}(x)+V_{\textsf{EAM}_\tau}(x-1)\\
    &=Kq_1^{x-1}q_2^{C-x-1}(q_1-q_2)^2
    -q_1^{(\tau-1)C+x-1}(1-q_1)
    -q_2^{\tau C-x-1}(1-q_2).
\end{align*}

On the other hand, over the remaining $\tau-1$ rounds, the number of effective assignments contributed by threat $i$ is pathwise at most $\min\{L_i,C(\tau-1)\}$. Hence
\[
V_{\textsf{EAM}_{\tau-1}}(\{1,2\})
\le
\mathbb E[\min\{L_1,C(\tau-1)\}]+\mathbb E[\min\{L_2,C(\tau-1)\}]=\frac{1-q_1^{(\tau-1)C}}{1-q_1}+\frac{1-q_2^{(\tau-1)C}}{1-q_2}.
\]
Thus $K\le 0$, and we have
$\Delta^2 V_{\textsf{EAM}_\tau}(x)
\le
-q_1^{(\tau-1)C+x-1}(1-q_1)
-q_2^{\tau C-x-1}(1-q_2)
\le 0$. Therefore, $V_{\textsf{EAM}_\tau}(x)$ is concave in $x\in\{0,1,\ldots,C\}$, and the maximization for $V_{\textsf{EAM}_\tau}(\{1,2\})$ is a one-dimensional concave optimization problem. 
    
\end{proof}

\section{Reoptimized fluid policy under \texorpdfstring{\textsf{EAM}$_\tau$}{EAM tau}}\label{appendix:additional_experiments}
In this section, we compare $\pi_{\textsf{EA}}$ with the \emph{reoptimized feasible fluid policy} $\pi_R$ introduced in \cite{brown2025fluid} under \cref{eq::EffectiveObjective}. The policy $\pi_\textsf{R}$ is implemented as follows. In each round, they first solve a fluid-relaxation linear program. The resulting optimal solution induces a randomized allocation in each round that is optimal under the relaxed problem. However, this randomized policy may not satisfy the capacity constraint. To obtain the feasible policy $\pi_{\textsf{R}}$, they sample the number of effectors allocated to each threat prescribed by this randomized policy. The feasibility is then enforced by setting the number of effectors allocated to a threat to zero whenever including it would cause the total assigned effectors to exceed the available capacity.

Now we show how we adapt the reoptimized fluid policy to our model. The fluid relaxation replaces the binary status (alive or neutralized) of each threat by a
continuous mass. Recall that $\mathcal{M}_t\subseteq[n]$ refers to the current active set at the beginning of round $t$. Let $h:=\tau-t+1$ be the remaining rounds. Given $\mathcal{M}_t$, the fluid LP can be written as
\begin{align}
\max_{u\ge 0}\quad
& \sum_{\ell=0}^{h-1}\sum_{i=1}^n\sum_{a=0}^C \frac{1-q_i^a}{1-q_i}u_{\ell i a} \tag{\textsf{FR-EAM}}\label{FR-EAM} \\
\text{s.t.}\quad
& \sum_{a=0}^C u_{0ia}=\mathbf 1\{i\in \mathcal{M}_t\},
    && i=1,\ldots,n, \notag \\
& \sum_{a=0}^C u_{\ell i a}
    =
    \sum_{a=0}^C q_i^a u_{\ell-1,i,a},
    && i=1,\ldots,n,\quad \ell=1,\ldots,h-1,  \notag\\
& \sum_{i=1}^n\sum_{a=0}^C a\,u_{\ell i a}\le C.
    && \ell=0,\ldots,h-1 .  \notag
\end{align}

Here, the decision variable $u_{\ell i a}$ denote the alive fluid mass of assigning threat
$i$ with $a$ effectors in relative round $\ell=0,\ldots,h-1$.
In \cref{FR-EAM}, The first constraint initializes the fluid state from the realized active set. The second constraint refers to the deterministic fluid flow equation. That is, if threat $i$ is alive and receives $a$ effectors, then a fraction $q_i^a$ remains alive in the next round. The last constraint refers to the capacity constraint in each round.

After solving the \cref{FR-EAM}, we then convert the solution into a feasible action for our model. Let $u^*$ be the optimal solution to \cref{FR-EAM}. For each active threat $i\in \mathcal{M}_t$, we sample an allocation $A_i$ according to
$\mathbb P(A_i=a)=u^*_{0ia}$ for $a=0,1,\ldots,C$. Starting from $x_i^t=0$ for all $i$, we then consider each threat in index order
and accept the proposed allocation $A_i$ if it fits within the remaining capacity:
\[
    x_i^t=
    \begin{cases}
    A_i, & \text{if } i\in \mathcal{M}_t \text{ and } \sum_{j<i}x_j+A_i\le C,\\
    0, & \text{otherwise}.
    \end{cases}
\]
This ensures that the implemented allocation satisfies the capacity constraint $\sum_{i=1}^n x_i^t\le C$. We implement this heuristic in every round after we update active set.

Now we compare $\pi_\textsf{R}$ to $\pi_{\textsf{EA}}$ with the following experimental setup. We consider the number of salvos $s \in \{2,\ldots,10\}$. For each value of $s$, we generate 10,000 independent instances. For each instance, salvo $j$ contains a random number of threats drawn independently from $\operatorname{Unif}\{1,2,\ldots,10\}$, and all threats within the same salvo have the same difficulty level, which is sampled independently from $\text{Beta}(4,3)$. Thus, we have $\mathbb{E}[L_i]=3$. The capacity is drawn as $C \sim \operatorname{Unif}\{\lfloor n/2\rfloor,\ldots,2n\}$, and the deadline is drawn as $\tau \sim \operatorname{Unif}\{2,3,4,5\}$. For $\pi_R$, we re-solve the fluid relaxation at each round and sample a randomized allocation from the resulting first-round marginals, followed by a feasibility projection. We report both the average number of effective assignments and the average computation time.

\begin{table}[t]
\small
\centering
\caption{Performance and average runtime comparison between $\pi_{\textsf{EA}}$ and $\pi_{\textsf{R}}$ under varying numbers of salvos. Here, $\hat{V}_{\textsf{EAM}_\tau}^{\pi_{\textsf{EA}}}$ and $\hat{V}_{\textsf{EAM}_\tau}^{\pi_{\textsf{R}}}$ refer to the number of effective assignments averaged over 10,000 instances under $\pi_{\textsf{EA}}$ and $\pi_{\textsf{R}}$, respectively.}
\label{tab:eam_fluid_salvos}
\begin{tabular}{ccccc}
\toprule
$s$ & $\hat{V}_{\textsf{EAM}_\tau}^{\pi_{\textsf{EA}}}$ & $\hat{V}_{\textsf{EAM}_\tau}^{\pi_{\textsf{R}}}$ & $\pi_{\textsf{EA}}$ time (sec) & $\pi_R$ time (sec) \\
\midrule
2  & 26.2951  & 25.8867  & 0.000592 & 0.010030 \\
3  & 40.1836  & 39.7473  & 0.000915 & 0.012359 \\
4  & 54.9293  & 54.4584  & 0.001328 & 0.015560 \\
5  & 69.3305  & 68.8731  & 0.001795 & 0.020514 \\
6  & 83.4329  & 82.9820  & 0.002298 & 0.024452 \\
7  & 98.5022  & 98.0570  & 0.002845 & 0.028498 \\
8  & 113.0328 & 112.5757 & 0.003447 & 0.034184 \\
9  & 127.6908 & 127.2542 & 0.004081 & 0.039775 \\
10 & 141.6469 & 141.2037 & 0.004760 & 0.046163 \\
\bottomrule
\end{tabular}
\end{table}

We have some observations from Table \ref{tab:eam_fluid_salvos}. First, the performance comparison shows that $\pi_{\textsf{EA}}$ consistently outperforms $\pi_{\textsf{R}}$ for all values of $s$, and there is no indication that $\pi_{\textsf{R}}$ improves relative to $\pi_{\textsf{EA}}$ as $s$ increases. Second, from average runtime comparison, $\pi_{\textsf{R}}$ is approximately 10 times slower than $\pi_{\textsf{EA}}$ on the tested instances. This is not surprising, as $\pi_{\textsf{R}}$ requires solving a linear program at each round, while $\pi_{\textsf{EA}}$ is stationary and only involves a simple greedy allocation. Overall, the results indicate that $\pi_{\textsf{EA}}$ is not only easier to implement but also achieves better performance than $\pi_{\textsf{R}}$ in our model.

\section{Extending to \eqref{eq::SurvivalObjective} with Distinct Deadlines}\label{sec:distinct_deadlines}

So far, the formulation of \textsf{SLM}$_\tau$ assumes that all threats share a common deadline $\tau$. In some applications, however, different threats may have different times to impact, so the defender must neutralize different threats by different deadlines. To capture this, let each threat $i\in[n]$ have its own deadline $\tau_i\in\mathbb Z_{>0}$. A natural extension of the survival-likelihood objective is then
\[
V^{*}_{\textsf{SLM}}(\tau)
=
\max_{\pi\in\Pi}
\left(V^{\pi}_{\textsf{SLM}}(\tau)
:=
\Pr\!\left(
\sum_{t=1}^{\tau_i} x_i^t \ge L_i,\ \forall i\in[n]
\right)\right).
\]
In this section, we will specifically analyze the case when all threat difficulties are the same, i.e., $q = q_i$ for all $i.$ As such, we will
find the following equivalent, recursive form to be more useful:

        \begin{itemize}
            \item $V^{*}_{\textsf{SLM}}(\tau) = \max_{x\in \mathcal{A}} V_x(\tau)$
            \item $V_x(\tau):=
                \begin{cases}
                    1 , & \text{if } \tau_1 = \ldots = \tau_n = \infty\\
                    0 , & \text{if any } \tau_i = 0\\
                    \sum_{S \subseteq [n]}V^{*}_{\textsf{SLM}}\big(\tau + \infty\cdot \mathbf{1}_S - \mathbf{1}_{[n] \setminus S}\big) \cdot \prod_{i \in S} (1- q_i^{x_i}) \cdot \prod_{i \notin S} q_i^{x_i}, & \text{otherwise}
        \end{cases}$
\end{itemize}

\subsection{Optimal Policy for 2 Staggered Threats}
The following proposition formalizes the intuition that, when two threats are equally difficult, the one with the earlier deadline should receive absolute priority until its resolution.
\begin{proposition}[First Come, First Serve] \label{prop::Greedy2Deadline} Let there be two threats with equal difficulty levels, i.e., $q_1= q_2=q$, and their deadlines following $0 < \tau_1 < \tau_2$. Then $(x_1^*,x_2^*)=(C,0)$ is optimal. 
\end{proposition}
\begin{proof}{Proof.} After arithmetic simplifications, the Bellman equation can be written
 \begin{align*}
    V_{\textsf{SLM}}^*(\tau_1, \tau_2) = \max_{x \in \mathbb{Z}^2_{\geq0}: x_1 + x_2 \leq C}
    & [1 - q^{x_1} - q^{x_2} + q^C] V_{\textsf{SLM}}^*(\infty, \infty) + [q^{x_2} - q^C] V_{\textsf{SLM}}^*(\infty, \tau_2 - 1)\\
    &+ [q^{x_1} - q^C] V_{\textsf{SLM}}^*(\tau_1 - 1, \infty) + q^C V_{\textsf{SLM}}^*(\tau_1 - 1, \tau_2 - 1).
\end{align*}
Upon substituting $x_2 = C - x_1$ and dropping all terms independent of $x_1,$
we reduce to the equivalent one-dimensional problem \begin{align*}
     \max_{x_1 \in [0, C]_{\mathbb{Z}}} q^{x_1} \left[V_{\textsf{SLM}}^*(\tau_1 -1, \infty) - V_{\textsf{SLM}}^*(\infty, \infty)\right] + \frac{q^C \left[V_{\textsf{SLM}}^*(\infty, \tau_2 -1) - V_{\textsf{SLM}}^*(\infty, \infty)
    \right]}{q^{x_1}},
\end{align*}
which has a unique unconstrained maximizer $\bar{x}$ 
    such that $q^{\bar{x}_1} = \sqrt{\frac{q^C \left[V_{\textsf{SLM}}^*(\infty, \tau_2 -1) - V_{\textsf{SLM}}^*(\infty, \infty)
    \right]}{\left[V_{\textsf{SLM}}^*(\tau_1 -1, \infty) - V_{\textsf{SLM}}^*(\infty, \infty)\right]}} 
    \in [0,1]$. Since 
    \begin{align*}
        1 = q^0 > q > q^2 > \ldots > q^C &\geq q^{C/2}\cdot \sqrt{q^{C(\tau_2 - \tau_1)}} = q^{C/2}\cdot \sqrt{\frac{1 - q^{C(\tau_2 - 1)} - 1}{1 - q^{C(\tau_1 - 1)} - 1}} = q^{\bar{x}_1},
    \end{align*}
    we can conclude that $x_1^* = C$ is the constrained, integer decision that gets $q^{x_1^*}$ closest to $q^{\bar{x}_1}$, which verifies $(x_1^*,x_2^*) = (C,0)$ is the constrained maximizer. \halmos
\end{proof}
\begin{corollary} \label{corollary: V2Deadline} Let there be two deadlines $0 < \tau_1 < \tau_2$, and arbitrary integer capacity $C\geq 1$. Then
        \begin{align*}
        V_{\textsf{SLM}}^*(\tau_1, \tau_2) &= 1 - q^{C \tau_1} - \tau_1 (1 - q^C) q^{C(\tau_2 - 1)}\\
        &= V_{\textsf{SLM}}^*(\tau_1) - \tau_1 [V_{\textsf{SLM}}^*(\tau_2) - V_{\textsf{SLM}}^*(\tau_2 - 1)].
        \end{align*}
    \end{corollary}
    \begin{proof}{Proof.}
    By Proposition \ref{prop::Greedy2Deadline}, 
        \begin{align*}
        V_{\textsf{SLM}}^*(\tau_1, \tau_2) &= V_{\textsf{SLM}}^*(\tau_1-1, \tau_2 -1)q^C + V_{\textsf{SLM}}^*(\infty, \tau_2 - 1)(1 - q^C)\\
        &= V_{\textsf{SLM}}^*(\tau_1-2, \tau_2 -2)q^{2C} + V_{\textsf{SLM}}^*(\infty, \tau_2 - 2)(1-q^C)q^C + V_{\textsf{SLM}}^*(\infty, \tau_2 - 1)(1 - q^C)\\
        &= V_{\textsf{SLM}}^*(\tau_1 - k, \tau_2 - k)q^{kC} + \sum_{i=1}^k V_{\textsf{SLM}}^*(\infty, \tau_2 - i) (1 - q^C)q^{C(i-1)}\\
        &= \sum_{i=1}^{\tau_1} V_{\textsf{SLM}}^*(\infty, \tau_2 - i) (1 - q^C)q^{C(i-1)} \\
        &= (1-q^C) \sum_{i = 1}^{\tau_1} \left(q^{C(i-1)} - q^{C(\tau_2 - 1)}\right) \\
        &= (1-q^C)\left(-\tau_1q^{C(\tau_2 - 1)} + \frac{1 - q^{C\tau_1}}{1 - q^C}\right).
    \end{align*} 

    \halmos
    \end{proof}

\subsection{Optimal Policy for 3 Staggered Threats}
In the case of 3 staggered threats, we show that it is optimal to divide effectors between only the 2 most imminent threats. In contrast to \Cref{prop::Greedy2Deadline}, not all effectors are given to the most imminent threat; rather, a reservation is implemented.
\begin{proposition}[Reservation Policy] 
\label{prop::Greedy3Deadline}
    Let there be three threats with equal difficulty levels, i.e, $q_1= q_2=q_3$ and their deadlines following $0 < \tau_1 < \tau_2<\tau_3$. 
    Define the quantity 
    \begin{align*}
        r^* = q^{C/2} \sqrt{\frac{ 1 - q^{C (\tau_2 -1)} - (\tau_2 -1) (1 - q^C) q^{C(\tau_3 - 2)} - 1 + q^{C(\tau_3 - 1)}}{1 - q^{C (\tau_1 -1)} - (\tau_1 -1) (1 - q^C) q^{C(\tau_3 - 2)} - 1 + q^{C(\tau_3 - 1)}}}, 
    \end{align*}
    and the function \\
    $h(x_1):= q^{x_1} \left[- q^{C (\tau_1 -1)} - (\tau_1 -1) (1 - q^C) q^{C(\tau_3 - 2)}+ q^{C(\tau_3 - 1)}\right] + \frac{q^C \left[- q^{C (\tau_2 -1)} - (\tau_2 -1) (1 - q^C) q^{C(\tau_3 - 2)} + q^{C(\tau_3 - 1)}
    \right]}{q^{x_1}}.$
    
    \noindent 
    Then:
    \begin{enumerate}
        \item If $\log_q(r^*) \geq C$, then $(x_1^*, x_2^*, x_3^*) = (C, 0, 0)$ is optimal. 
        \item 
        Otherwise, 
        $x_1^* \in \argmax_{x_1 \in \{\lfloor \log_{q}(r^*)\rfloor, \lceil \log_{q}(r^*) \rceil\}} h(x_1)$,  $x_2^* = C - x_1^*, $ and $x_3^* = 0$ is optimal.
    \end{enumerate}
\end{proposition}
\begin{proof}{Proof.}
    We begin by writing
    \begin{align*}
        V_{\textsf{SLM}}^*(\tau_1, \tau_2, \tau_3) = \max_{x\in\mathbb{Z}^3_{\geq0}: \sum_i x_i \leq C}(1 - q^{x_1}) f(x_2, x_3) + q^{x_1} g(x_2, x_3),
    \end{align*}
    where 
    \begin{align*}
        f(x_2, x_3) = &(1 - q^{x_2}) \left[(1- q^{x_3}) V_{\textsf{SLM}}^*(\infty, \infty, \infty) + q^{x_3}V_{\textsf{SLM}}^*(\infty, \infty, \tau_3 - 1)\right]\\
        &+ q^{x_2} \left[(1-q^{x_3}) V_{\textsf{SLM}}^*(\infty, \tau_2 - 1, \infty) + q^{x_3}V_{\textsf{SLM}}^*(\infty, \tau_2 - 1, \tau_3 - 1)\right],
    \end{align*}
    \begin{align*}
        g(x_2, x_3) = &(1 - q^{x_2}) \left[(1- q^{x_3}) V_{\textsf{SLM}}^*(\tau_1 - 1, \infty, \infty) + q^{x_3}V_{\textsf{SLM}}^*(\tau_1 - 1, \infty, \tau_3 - 1)\right]\\
        &+ q^{x_2} \left[(1-q^{x_3}) V_{\textsf{SLM}}^*(\tau_1 - 1, \tau_2 - 1, \infty) + q^{x_3}V_{\textsf{SLM}}^*(\tau_1 - 1, \tau_2 - 1, \tau_3 - 1)\right].
    \end{align*}
    Next, we show that
    \[
    x_2^* \in \argmax_{x_2 \leq C - x_1^*} (1-q)^{x_1^*} f(x_2, C - x_1^* - x_2) + q^{x_1^*} g(x_2, C - x_1^* - x_2) \implies x_2^* = C - x_1^*,
    \]
    which holds if simultaneously 
    \[
    \left[\argmax_{x_2 + x_3 \leq C - x_1^*} f(x_2, x_3)\right] = (C - x_1^*, 0), \;\; \left[\argmax_{x_2 + x_3 \leq C - x_1^*} g(x_2, x_3)\right] = (C - x_1^*, 0).
    \]
    The latter can be argued similarly as in Proposition \ref{prop::Greedy2Deadline}, because we know 
    \begin{align*}
    \sqrt{\frac{V_{\textsf{SLM}}^*(\tau_1 - 1, \infty, \tau_3 - 1) - V_{\textsf{SLM}}^*(\tau_1 - 1, \infty, \infty)}{V_{\textsf{SLM}}^*(\tau_1 - 1, \tau_2 - 1, \infty) - V_{\textsf{SLM}}^*(\tau_1 - 1, \infty, \infty)}}q^{\frac{C-x_1^*}{2}} &= \sqrt{\frac{q^{C(\tau_3 - 2)}}{q^{C(\tau_2 - 2)}}}q^{\frac{C-x_1^*}{2}} = q^{\frac{C(\tau_3 - \tau_2)}{2}} q^{\frac{C-x_1^*}{2}} \leq q^{C - \frac{x_1^*}{2}} \\
    &\leq q^{C - x_1^*} < \ldots < q^{0} = 1 
    \end{align*}
    and similarly, 
    \begin{align*}
    \sqrt{\frac{V_{\textsf{SLM}}^*(\infty, \infty, \tau_3 - 1) - V_{\textsf{SLM}}^*(\infty, \infty, \infty)}{V_{\textsf{SLM}}^*(\infty, \tau_2 - 1, \infty) - V_{\textsf{SLM}}^*(\infty, \infty, \infty)}}q^{\frac{C - x_1^*}{2}} &= \sqrt{\frac{1 - q^{C(\tau_3 - 1)} - 1}{1 - q^{C(\tau_2 - 1)} - 1}}q^{\frac{C-x_1^*}{2}} = q^{\frac{C(\tau_3 - \tau_2)}{2}}q^{\frac{C-x_1^*}{2}} \leq q^{C - \frac{x_1^*}{2}} \\
    &\leq q^{C - x_1^*} < \ldots < q^{0} = 1. 
    \end{align*}
To conclude, 
    \begin{align*}
        \max_{x_2 + x_3 \leq C - x_1^*} &(1 - q^{x_1^*}) f(x_2, x_3) + q^{x_1^*} g(x_2, x_3) \leq (1 - q^{x_1^*}) \left[\max_{x_2 + x_3 \leq C - x_1^*} f(x_2, x_3)\right] + q^{x_1^*} \left[\max_{x_2 + x_3 \leq C - x_1^*} g(x_2, x_3)\right]\\
        &= (1-q^{x_1^*}) f(C-x_1^*, 0) + q^{x_1^*} g(C-x_1^*, 0) \leq  \max_{x_2 + x_3 \leq C - x_1^*} (1 - q^{x_1^*}) f(x_2, x_3) + q^{x_1^*} g(x_2, x_3), 
    \end{align*}
    which indicates $(x_2^*, x_3^*) = (C - x_1^*, 0),$ as desired.

Consequently, we may set $x_3 = 0$ and $x_2 = C - x_1$ to find 
    \begin{align*}
        V_{\textsf{SLM}}^*(\tau_1, \tau_2, \tau_3) &= \max_{x\in\mathbb{Z}^3_{\geq0}: \sum_i x_i \leq C}(1 - q^{x_1}) f(x_2, x_3) + q^{x_1} g(x_2, x_3) \\
        &=  \max_{x_1 \in [0,C]_{\mathbb{Z}}} (1-q^{x_1}) \left[(1 - q^{C-x_1})V_{\textsf{SLM}}^*(\infty, \infty, \tau_3 - 1) + q^{C-x_1}V_{\textsf{SLM}}^*(\infty, \tau_2 - 1, \tau_3 - 1)\right]\\
        & ~~~~~~~~~~~+ q^{x_1} \left[(1 - q^{C-x_1})V_{\textsf{SLM}}^*(\tau_1 -1, \infty, \tau_3 - 1) + q^{C-x_1}V_{\textsf{SLM}}^*(\tau_1 - 1, \tau_2 - 1, \tau_3 - 1)\right],
    \end{align*}
    which is equivalent to 
    \[
    \max_{x_1 \in [0,C]_{\mathbb{Z}}} q^{x_1} \left[V_{\textsf{SLM}}^*(\tau_1 -1, \infty, \tau_3 -1) - V_{\textsf{SLM}}^*(\infty, \infty, \tau_3 - 1)\right] + \frac{q^C \left[V_{\textsf{SLM}}^*(\infty, \tau_2 -1, \tau_3 - 1) - V_{\textsf{SLM}}^*(\infty, \infty, \tau_3 - 1)
    \right]}{q^{x_1}}, 
    \]
    solved by an unconstrained $\bar{x}_1$ satisfying $q^{\bar{x}_1} = r^*.$ \halmos
    \end{proof}

\subsection{A Greedy Policy for \eqref{eq::SurvivalObjective} with Distinct Deadlines}
The greedy policy $\pi_{\textsf{SL}}$ can be adapted to distinct deadlines with a similar marginal-priority logic discussed in section \ref{sec:greedy_policies}. The difference is that now the extended policy $\tilde{\pi}_{\textsf{SL}}$ is constructed backward in time. Let $\mathcal I_\tau:=\{i:\tau_i\le r\}$ denote the index set of threats whose deadlines are no greater than $\tau$, so that $\mathcal I_1\subseteq \mathcal I_2\subseteq \cdots \subseteq \mathcal I_{\max_i \tau_i}$. Starting from the last round $\mathcal{I}_1$,  $\tilde{\pi}_{\textsf{SL}}$ applies the greedy allocation to the threats in $\mathcal I_\tau$. Specifically, $\tilde{\pi}_{\textsf{SL}}$ assigns effectors one at a time according to the current cumulative score $s_i=\sum_{k=1}^{\tilde{x}_i} q_i^{-k}$ where $\tilde{x}_i$ is the number of effectors already assigned to threat $i$ since the last round. After all $C$ effectors have been assigned for a given $\tau$, the algorithm proceeds to the previous round by updating $r \leftarrow r - 1$ and repeating the same procedure until $r=\max_i \tau_i$ (i.e., the current round). The allocation $\tilde{x}$ is updated cumulatively across rounds. Finally, we output the allocation constructed under $\mathcal{I}_{\max_i \tau_i}$, which refers to the first round allocation. 

Algorithm \ref{alg:: ML_Estim_adj} has several advantages. On one hand, because $\tilde{x}$ is constructed cumulatively through a backward iteration, threats with later deadlines are processed earlier in the algorithmic iteration and therefore receive effectors earlier in this construction. As a result, they are given lower priority when the allocation is traced back to earlier rounds. On the other hand, the score updates can still capture the effect of heterogeneous difficulty levels. Under $\tilde{\pi}_{\textsf{SL}}$, the allocation used in the actual first round is the one obtained in the final backward step. Thus, the extension preserves the idea of $\pi_{\textsf{SL}}$ while adapting it to the presence of distinct deadlines. We formalize this policy in algorithm \ref{alg:: ML_Estim_adj}.

\begin{algorithm} 
\caption{$\Tilde{\pi}_{\textsf{SL}}$} \label{alg:: ML_Estim_adj}
\begin{algorithmic} 
\small
\Require $n$ alive threats with difficulty $q_1,\cdots,q_n$ and deadlines $\tau_1, \tau_2, \cdots, \tau_n$. Let $\mathcal{I}_r=\{i: \tau_i\leq r\}$.
\State Initialize $\Tilde{x}_i\leftarrow 0 ,x_i\leftarrow 0$ and $s_i\leftarrow 0$ for $i\in [n]$.
\For{$\tau=1,\cdots,\max_{i\in[n]}\tau_i$}
    \For{$e=1,\cdots,C$}
        \State $i^* \gets \argmin_{i\in \mathcal{I}_r}\ s_i$ \Comment{ties broken arbitrarily}
    \State $\Tilde{x}_{i^*} \gets \Tilde{x}_{i^*} + 1$
    \State $s_{i^*} \gets s_{i^*} + q_{i^*}^{-x_{i^*}}$
    \If{$\tau=\max_{i\in[n]}\tau_i$}
        \State $x_{i^*} \gets x_{i^*} + 1$
    \EndIf
    \EndFor
\EndFor\\
\Return Current round allocation $x$
\end{algorithmic}
\end{algorithm}

\section{Additional Plots and Tables for Section \ref{sec:numerical_experiment} %
} \label{appendix:additional_plots}
In this section, we present several additional plots from the experiments described in Sections \ref{sec:exp_hypercube}, \ref{sec:diff_var_Li}. \Cref{fig:hypercube_avg_obj_cn,fig:two_slices_avg_obj} show the average performance under each objective under various experimental setups across $n=2,3,\cdots,7$. Table \ref{tab:simulation} reports the average performance for each objective based on the simulation results.

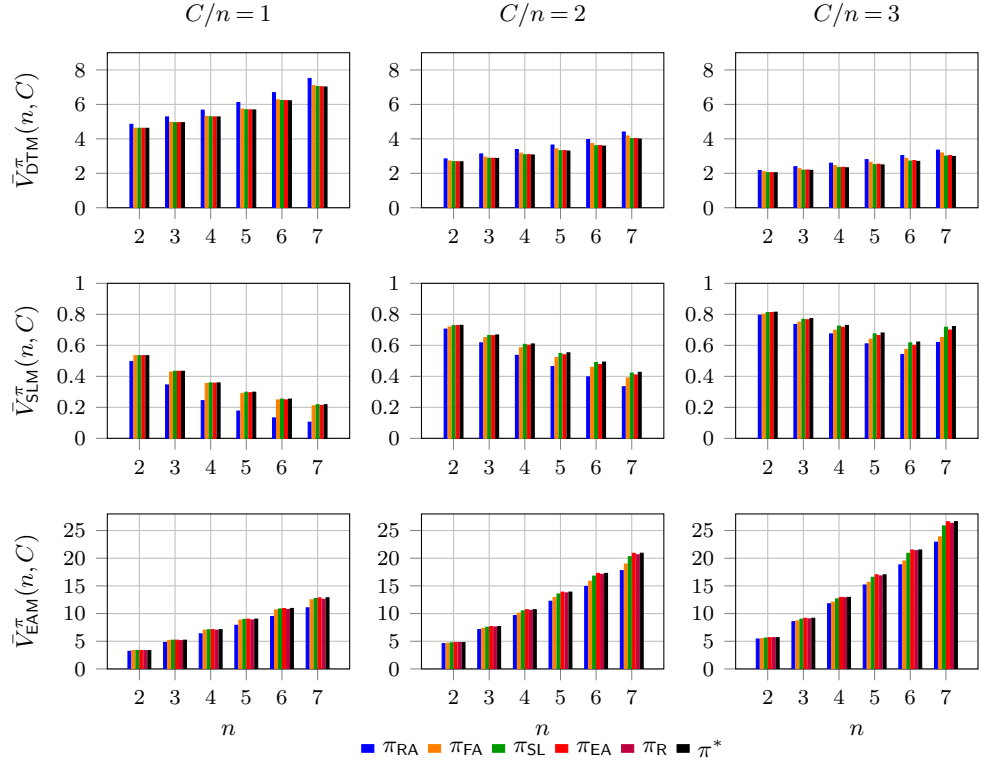
\begin{figure}[htbp]
\centering

\begin{tikzpicture}

\def\DTMymax{9}
\def\SLMymax{1.0}
\def\EAMymax{28}

\begin{groupplot}[
    group style={
        group size=3 by 3,
        horizontal sep=0.95cm,
        vertical sep=1.00cm
    },
    width=0.29\textwidth,
    height=0.22\textwidth,
    grid=both,
    tick align=outside,
    xtick pos=bottom,
    ytick pos=left,
    tick label style={font=\scriptsize},
    label style={font=\footnotesize},
    title style={font=\footnotesize},
    ylabel style={
        align=center,
        at={(axis description cs:-0.24,0.5)},
        anchor=south
    },
    scaled y ticks=false,
    xtick=data,
    xticklabel style={/pgf/number format/fixed},
    ybar,
    /pgf/bar width=1.1pt,
    enlarge x limits=0.18,
]

\nextgroupplot[
    title={$C/n=1$},
    ylabel={$\bar V_{\textsf{DTM}}^{\pi}(n,C)$},
    ymin=0,
    ymax=\DTMymax,
    ytick={0,2,4,6,8},
    yticklabel style={
        text width=2.0em,
        align=right,
        /pgf/number format/fixed,
        /pgf/number format/precision=1
    }
]
\addplot+[bar shift=-3pt, fill=blue, draw=blue]
table[x=n,y=uniform_random,col sep=comma] {tikzdata/hypercube_new_cn/cn1_avg_obj_min_diffuse.csv};
\addplot+[bar shift=-1.5pt, fill=orange, draw=orange]
table[x=n,y=fair,col sep=comma] {tikzdata/hypercube_new_cn/cn1_avg_obj_min_diffuse.csv};
\addplot+[bar shift=0pt, fill=green!60!black, draw=green!60!black]
table[x=n,y=ml_greedy,col sep=comma] {tikzdata/hypercube_new_cn/cn1_avg_obj_min_diffuse.csv};
\addplot+[bar shift=1.5pt, fill=red, draw=red]
table[x=n,y=ea_greedy,col sep=comma] {tikzdata/hypercube_new_cn/cn1_avg_obj_min_diffuse.csv};
\addplot+[bar shift=3pt, fill=black, draw=black]
table[x=n,y=optimal,col sep=comma] {tikzdata/hypercube_new_cn/cn1_avg_obj_min_diffuse.csv};

\nextgroupplot[
    title={$C/n=2$},
    ymin=0,
    ymax=\DTMymax,
    ytick={0,2,4,6,8},
    yticklabel style={
        text width=2.0em,
        align=right,
        /pgf/number format/fixed,
        /pgf/number format/precision=1
    }
]
\addplot+[bar shift=-3pt, fill=blue, draw=blue]
table[x=n,y=uniform_random,col sep=comma] {tikzdata/hypercube_new_cn/cn2_avg_obj_min_diffuse.csv};
\addplot+[bar shift=-1.5pt, fill=orange, draw=orange]
table[x=n,y=fair,col sep=comma] {tikzdata/hypercube_new_cn/cn2_avg_obj_min_diffuse.csv};
\addplot+[bar shift=0pt, fill=green!60!black, draw=green!60!black]
table[x=n,y=ml_greedy,col sep=comma] {tikzdata/hypercube_new_cn/cn2_avg_obj_min_diffuse.csv};
\addplot+[bar shift=1.5pt, fill=red, draw=red]
table[x=n,y=ea_greedy,col sep=comma] {tikzdata/hypercube_new_cn/cn2_avg_obj_min_diffuse.csv};
\addplot+[bar shift=3pt, fill=black, draw=black]
table[x=n,y=optimal,col sep=comma] {tikzdata/hypercube_new_cn/cn2_avg_obj_min_diffuse.csv};

\nextgroupplot[
    title={$C/n=3$},
    ymin=0,
    ymax=\DTMymax,
    ytick={0,2,4,6,8},
    yticklabel style={
        text width=2.0em,
        align=right,
        /pgf/number format/fixed,
        /pgf/number format/precision=1
    }
]
\addplot+[bar shift=-3pt, fill=blue, draw=blue]
table[x=n,y=uniform_random,col sep=comma] {tikzdata/hypercube_new_cn/cn3_avg_obj_min_diffuse.csv};
\addplot+[bar shift=-1.5pt, fill=orange, draw=orange]
table[x=n,y=fair,col sep=comma] {tikzdata/hypercube_new_cn/cn3_avg_obj_min_diffuse.csv};
\addplot+[bar shift=0pt, fill=green!60!black, draw=green!60!black]
table[x=n,y=ml_greedy,col sep=comma] {tikzdata/hypercube_new_cn/cn3_avg_obj_min_diffuse.csv};
\addplot+[bar shift=1.5pt, fill=red, draw=red]
table[x=n,y=ea_greedy,col sep=comma] {tikzdata/hypercube_new_cn/cn3_avg_obj_min_diffuse.csv};
\addplot+[bar shift=3pt, fill=black, draw=black]
table[x=n,y=optimal,col sep=comma] {tikzdata/hypercube_new_cn/cn3_avg_obj_min_diffuse.csv};

\nextgroupplot[
    ylabel={$\bar V_{\textsf{SLM}}^{\pi}(n,C)$},
    ymin=0,
    ymax=\SLMymax,
    ytick={0,0.2,0.4,0.6,0.8,1.0},
    yticklabel style={
        text width=2.0em,
        align=right,
        /pgf/number format/fixed,
        /pgf/number format/precision=1
    }
]
\addplot+[bar shift=-3pt, fill=blue, draw=blue]
table[x=n,y=uniform_random,col sep=comma] {tikzdata/hypercube_new_cn/cn1_avg_obj_max_ml.csv};
\addplot+[bar shift=-1.5pt, fill=orange, draw=orange]
table[x=n,y=fair,col sep=comma] {tikzdata/hypercube_new_cn/cn1_avg_obj_max_ml.csv};
\addplot+[bar shift=0pt, fill=green!60!black, draw=green!60!black]
table[x=n,y=ml_greedy,col sep=comma] {tikzdata/hypercube_new_cn/cn1_avg_obj_max_ml.csv};
\addplot+[bar shift=1.5pt, fill=red, draw=red]
table[x=n,y=ea_greedy,col sep=comma] {tikzdata/hypercube_new_cn/cn1_avg_obj_max_ml.csv};
\addplot+[bar shift=3pt, fill=black, draw=black]
table[x=n,y=optimal,col sep=comma] {tikzdata/hypercube_new_cn/cn1_avg_obj_max_ml.csv};

\nextgroupplot[
    ymin=0,
    ymax=\SLMymax,
    ytick={0,0.2,0.4,0.6,0.8,1.0},
    yticklabel style={
        text width=2.0em,
        align=right,
        /pgf/number format/fixed,
        /pgf/number format/precision=1
    }
]
\addplot+[bar shift=-3pt, fill=blue, draw=blue]
table[x=n,y=uniform_random,col sep=comma] {tikzdata/hypercube_new_cn/cn2_avg_obj_max_ml.csv};
\addplot+[bar shift=-1.5pt, fill=orange, draw=orange]
table[x=n,y=fair,col sep=comma] {tikzdata/hypercube_new_cn/cn2_avg_obj_max_ml.csv};
\addplot+[bar shift=0pt, fill=green!60!black, draw=green!60!black]
table[x=n,y=ml_greedy,col sep=comma] {tikzdata/hypercube_new_cn/cn2_avg_obj_max_ml.csv};
\addplot+[bar shift=1.5pt, fill=red, draw=red]
table[x=n,y=ea_greedy,col sep=comma] {tikzdata/hypercube_new_cn/cn2_avg_obj_max_ml.csv};
\addplot+[bar shift=3pt, fill=black, draw=black]
table[x=n,y=optimal,col sep=comma] {tikzdata/hypercube_new_cn/cn2_avg_obj_max_ml.csv};

\nextgroupplot[
    ymin=0,
    ymax=\SLMymax,
    ytick={0,0.2,0.4,0.6,0.8,1.0},
    yticklabel style={
        text width=2.0em,
        align=right,
        /pgf/number format/fixed,
        /pgf/number format/precision=1
    }
]
\addplot+[bar shift=-3pt, fill=blue, draw=blue]
table[x=n,y=uniform_random,col sep=comma] {tikzdata/hypercube_new_cn/cn3_avg_obj_max_ml.csv};
\addplot+[bar shift=-1.5pt, fill=orange, draw=orange]
table[x=n,y=fair,col sep=comma] {tikzdata/hypercube_new_cn/cn3_avg_obj_max_ml.csv};
\addplot+[bar shift=0pt, fill=green!60!black, draw=green!60!black]
table[x=n,y=ml_greedy,col sep=comma] {tikzdata/hypercube_new_cn/cn3_avg_obj_max_ml.csv};
\addplot+[bar shift=1.5pt, fill=red, draw=red]
table[x=n,y=ea_greedy,col sep=comma] {tikzdata/hypercube_new_cn/cn3_avg_obj_max_ml.csv};
\addplot+[bar shift=3pt, fill=black, draw=black]
table[x=n,y=optimal,col sep=comma] {tikzdata/hypercube_new_cn/cn3_avg_obj_max_ml.csv};

\nextgroupplot[
    xlabel={$n$},
    xlabel style={font=\footnotesize},
    ylabel={$\bar V_{\textsf{EAM}}^{\pi}(n,C)$},
    ymin=0,
    ymax=\EAMymax,
    ytick={0,5,10,15,20,25},
    yticklabel style={
        text width=2.0em,
        align=right,
        /pgf/number format/fixed,
        /pgf/number format/precision=0
    }
]
\addplot+[bar shift=-3.75pt, fill=blue, draw=blue]
table[x=n,y=uniform_random,col sep=comma] {tikzdata/hypercube_new_cn/cn1_avg_obj_max_effective.csv};
\addplot+[bar shift=-2.25pt, fill=orange, draw=orange]
table[x=n,y=fair,col sep=comma] {tikzdata/hypercube_new_cn/cn1_avg_obj_max_effective.csv};
\addplot+[bar shift=-0.75pt, fill=green!60!black, draw=green!60!black]
table[x=n,y=ml_greedy,col sep=comma] {tikzdata/hypercube_new_cn/cn1_avg_obj_max_effective.csv};
\addplot+[bar shift=0.75pt, fill=red, draw=red]
table[x=n,y=ea_greedy,col sep=comma] {tikzdata/hypercube_new_cn/cn1_avg_obj_max_effective.csv};
\addplot+[bar shift=2.25pt, fill=purple, draw=purple]
table[x=n,y=reoptimized_fluid,col sep=comma] {tikzdata/hypercube_new_cn/cn1_avg_obj_max_effective.csv};
\addplot+[bar shift=3.75pt, fill=black, draw=black]
table[x=n,y=optimal,col sep=comma] {tikzdata/hypercube_new_cn/cn1_avg_obj_max_effective.csv};

\nextgroupplot[
    xlabel={$n$},
    xlabel style={font=\footnotesize},
    ymin=0,
    ymax=\EAMymax,
    ytick={0,5,10,15,20,25},
    yticklabel style={
        text width=2.0em,
        align=right,
        /pgf/number format/fixed,
        /pgf/number format/precision=0
    }
]
\addplot+[bar shift=-3.75pt, fill=blue, draw=blue]
table[x=n,y=uniform_random,col sep=comma] {tikzdata/hypercube_new_cn/cn2_avg_obj_max_effective.csv};
\addplot+[bar shift=-2.25pt, fill=orange, draw=orange]
table[x=n,y=fair,col sep=comma] {tikzdata/hypercube_new_cn/cn2_avg_obj_max_effective.csv};
\addplot+[bar shift=-0.75pt, fill=green!60!black, draw=green!60!black]
table[x=n,y=ml_greedy,col sep=comma] {tikzdata/hypercube_new_cn/cn2_avg_obj_max_effective.csv};
\addplot+[bar shift=0.75pt, fill=red, draw=red]
table[x=n,y=ea_greedy,col sep=comma] {tikzdata/hypercube_new_cn/cn2_avg_obj_max_effective.csv};
\addplot+[bar shift=2.25pt, fill=purple, draw=purple]
table[x=n,y=reoptimized_fluid,col sep=comma] {tikzdata/hypercube_new_cn/cn2_avg_obj_max_effective.csv};
\addplot+[bar shift=3.75pt, fill=black, draw=black]
table[x=n,y=optimal,col sep=comma] {tikzdata/hypercube_new_cn/cn2_avg_obj_max_effective.csv};

\nextgroupplot[
    xlabel={$n$},
    xlabel style={font=\footnotesize},
    ymin=0,
    ymax=\EAMymax,
    ytick={0,5,10,15,20,25},
    yticklabel style={
        text width=2.0em,
        align=right,
        /pgf/number format/fixed,
        /pgf/number format/precision=0
    }
]
\addplot+[bar shift=-3.75pt, fill=blue, draw=blue]
table[x=n,y=uniform_random,col sep=comma] {tikzdata/hypercube_new_cn/cn3_avg_obj_max_effective.csv};
\addplot+[bar shift=-2.25pt, fill=orange, draw=orange]
table[x=n,y=fair,col sep=comma] {tikzdata/hypercube_new_cn/cn3_avg_obj_max_effective.csv};
\addplot+[bar shift=-0.75pt, fill=green!60!black, draw=green!60!black]
table[x=n,y=ml_greedy,col sep=comma] {tikzdata/hypercube_new_cn/cn3_avg_obj_max_effective.csv};
\addplot+[bar shift=0.75pt, fill=red, draw=red]
table[x=n,y=ea_greedy,col sep=comma] {tikzdata/hypercube_new_cn/cn3_avg_obj_max_effective.csv};
\addplot+[bar shift=2.25pt, fill=purple, draw=purple]
table[x=n,y=reoptimized_fluid,col sep=comma] {tikzdata/hypercube_new_cn/cn3_avg_obj_max_effective.csv};
\addplot+[bar shift=3.75pt, fill=black, draw=black]
table[x=n,y=optimal,col sep=comma] {tikzdata/hypercube_new_cn/cn3_avg_obj_max_effective.csv};

\end{groupplot}

\path (group c1r3.south west) -- (group c3r3.south east)
    coordinate[midway] (legendmid);

\matrix[
    matrix of nodes,
    anchor=north,
    row sep=0pt,
    column sep=3pt,
    nodes={anchor=west, inner sep=0pt, outer sep=0pt, font=\footnotesize}
] at ($(legendmid)+(0,-0.8cm)$) {
\tikz{\draw[fill=blue,draw=blue] (0,0) rectangle (0.16,0.10);} &
$\pi_{\textsf{RA}}$ &
\tikz{\draw[fill=orange,draw=orange] (0,0) rectangle (0.16,0.10);} &
$\pi_{\textsf{FA}}$ &
\tikz{\draw[fill=green!60!black,draw=green!60!black] (0,0) rectangle (0.16,0.10);} &
$\pi_{\textsf{SL}}$ &
\tikz{\draw[fill=red,draw=red] (0,0) rectangle (0.16,0.10);} &
$\pi_{\textsf{EA}}$ &
\tikz{\draw[fill=purple,draw=purple] (0,0) rectangle (0.16,0.10);} &
$\pi_{\textsf{R}}$ &
\tikz{\draw[fill=black,draw=black] (0,0) rectangle (0.16,0.10);} &
$\pi^*$ \\
};

\end{tikzpicture}
\caption{Average objective values under \cref{eq::MinDefuse}, (\textsf{SLM}), and (\textsf{EAM}) for different capacity ratios $C/n\in\{1,2,3\}$. Each row corresponds to one objective, and each column corresponds to one value of $C/n$.}
\label{fig:hypercube_avg_obj_cn}
\end{figure}

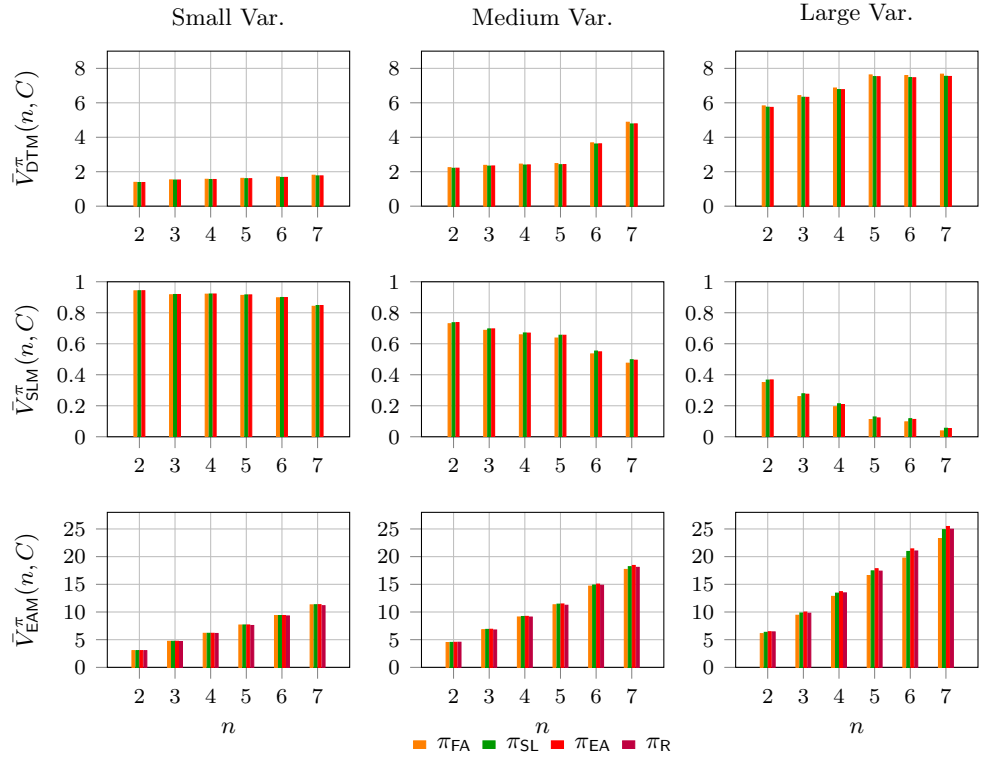
\begin{figure}[htbp]
\centering
\begin{tikzpicture}

\def\DTMymax{9}
\def\SLMymax{1.0}
\def\EAMymax{28}

\begin{groupplot}[
    group style={
        group size=3 by 3,
        horizontal sep=0.95cm,
        vertical sep=1.00cm
    },
    width=0.29\textwidth,
    height=0.22\textwidth,
    grid=both,
    tick align=outside,
    xtick pos=bottom,
    ytick pos=left,
    tick label style={font=\scriptsize},
    label style={font=\footnotesize},
    title style={font=\footnotesize},
    ylabel style={
        align=center,
        at={(axis description cs:-0.24,0.5)},
        anchor=south
    },
    scaled y ticks=false,
    scaled x ticks=false,
    xtick=data,
    xticklabel style={/pgf/number format/fixed},
    ybar,
    /pgf/bar width=1.1pt,
    enlarge x limits=0.18,
]

\nextgroupplot[
    title={Small Var.},
    ylabel={$\bar V_{\textsf{DTM}}^{\pi}(n,C)$},
    ymin=0,
    ymax=\DTMymax,
    ytick={0,2,4,6,8},
    yticklabel style={
        text width=2.0em,
        align=right,
        /pgf/number format/fixed,
        /pgf/number format/precision=1
    }
]
\addplot+[bar shift=-1.5pt, fill=orange, draw=orange]
table[x=n,y=piFA] {tikzdata/salvo_variance/dtm_small_obj.dat};
\addplot+[bar shift=0pt, fill=green!60!black, draw=green!60!black]
table[x=n,y=piSL] {tikzdata/salvo_variance/dtm_small_obj.dat};
\addplot+[bar shift=1.5pt, fill=red, draw=red]
table[x=n,y=piEA] {tikzdata/salvo_variance/dtm_small_obj.dat};

\nextgroupplot[
    title={Medium Var.},
    ymin=0,
    ymax=\DTMymax,
    ytick={0,2,4,6,8},
    yticklabel style={
        text width=2.0em,
        align=right,
        /pgf/number format/fixed,
        /pgf/number format/precision=1
    }
]
\addplot+[bar shift=-1.5pt, fill=orange, draw=orange]
table[x=n,y=piFA] {tikzdata/salvo_variance/dtm_medium_obj.dat};
\addplot+[bar shift=0pt, fill=green!60!black, draw=green!60!black]
table[x=n,y=piSL] {tikzdata/salvo_variance/dtm_medium_obj.dat};
\addplot+[bar shift=1.5pt, fill=red, draw=red]
table[x=n,y=piEA] {tikzdata/salvo_variance/dtm_medium_obj.dat};

\nextgroupplot[
    title={Large Var.},
    ymin=0,
    ymax=\DTMymax,
    ytick={0,2,4,6,8},
    yticklabel style={
        text width=2.0em,
        align=right,
        /pgf/number format/fixed,
        /pgf/number format/precision=1
    }
]
\addplot+[bar shift=-1.5pt, fill=orange, draw=orange]
table[x=n,y=piFA] {tikzdata/salvo_variance/dtm_large_obj.dat};
\addplot+[bar shift=0pt, fill=green!60!black, draw=green!60!black]
table[x=n,y=piSL] {tikzdata/salvo_variance/dtm_large_obj.dat};
\addplot+[bar shift=1.5pt, fill=red, draw=red]
table[x=n,y=piEA] {tikzdata/salvo_variance/dtm_large_obj.dat};

\nextgroupplot[
    ylabel={$\bar V_{\textsf{SLM}}^{\pi}(n,C)$},
    ymin=0,
    ymax=\SLMymax,
    ytick={0,0.2,0.4,0.6,0.8,1.0},
    yticklabel style={
        text width=2.0em,
        align=right,
        /pgf/number format/fixed,
        /pgf/number format/precision=1
    }
]
\addplot+[bar shift=-1.5pt, fill=orange, draw=orange]
table[x=n,y=piFA] {tikzdata/salvo_variance/slm_small_obj.dat};
\addplot+[bar shift=0pt, fill=green!60!black, draw=green!60!black]
table[x=n,y=piSL] {tikzdata/salvo_variance/slm_small_obj.dat};
\addplot+[bar shift=1.5pt, fill=red, draw=red]
table[x=n,y=piEA] {tikzdata/salvo_variance/slm_small_obj.dat};

\nextgroupplot[
    ymin=0,
    ymax=\SLMymax,
    ytick={0,0.2,0.4,0.6,0.8,1.0},
    yticklabel style={
        text width=2.0em,
        align=right,
        /pgf/number format/fixed,
        /pgf/number format/precision=1
    }
]
\addplot+[bar shift=-1.5pt, fill=orange, draw=orange]
table[x=n,y=piFA] {tikzdata/salvo_variance/slm_medium_obj.dat};
\addplot+[bar shift=0pt, fill=green!60!black, draw=green!60!black]
table[x=n,y=piSL] {tikzdata/salvo_variance/slm_medium_obj.dat};
\addplot+[bar shift=1.5pt, fill=red, draw=red]
table[x=n,y=piEA] {tikzdata/salvo_variance/slm_medium_obj.dat};

\nextgroupplot[
    ymin=0,
    ymax=\SLMymax,
    ytick={0,0.2,0.4,0.6,0.8,1.0},
    yticklabel style={
        text width=2.0em,
        align=right,
        /pgf/number format/fixed,
        /pgf/number format/precision=1
    }
]
\addplot+[bar shift=-1.5pt, fill=orange, draw=orange]
table[x=n,y=piFA] {tikzdata/salvo_variance/slm_large_obj.dat};
\addplot+[bar shift=0pt, fill=green!60!black, draw=green!60!black]
table[x=n,y=piSL] {tikzdata/salvo_variance/slm_large_obj.dat};
\addplot+[bar shift=1.5pt, fill=red, draw=red]
table[x=n,y=piEA] {tikzdata/salvo_variance/slm_large_obj.dat};

\nextgroupplot[
    xlabel={$n$},
    xlabel style={font=\footnotesize},
    ylabel={$\bar V_{\textsf{EAM}}^{\pi}(n,C)$},
    ymin=0,
    ymax=\EAMymax,
    ytick={0,5,10,15,20,25},
    yticklabel style={
        text width=2.0em,
        align=right,
        /pgf/number format/fixed,
        /pgf/number format/precision=0
    }
]
\addplot+[bar shift=-2.25pt, fill=orange, draw=orange]
table[x=n,y=piFA] {tikzdata/salvo_variance/eam_small_obj.dat};
\addplot+[bar shift=-0.75pt, fill=green!60!black, draw=green!60!black]
table[x=n,y=piSL] {tikzdata/salvo_variance/eam_small_obj.dat};
\addplot+[bar shift=0.75pt, fill=red, draw=red]
table[x=n,y=piEA] {tikzdata/salvo_variance/eam_small_obj.dat};
\addplot+[bar shift=2.25pt, fill=purple, draw=purple]
table[x=n,y=piR] {tikzdata/salvo_variance/eam_small_obj.dat};

\nextgroupplot[
    xlabel={$n$},
    xlabel style={font=\footnotesize},
    ymin=0,
    ymax=\EAMymax,
    ytick={0,5,10,15,20,25},
    yticklabel style={
        text width=2.0em,
        align=right,
        /pgf/number format/fixed,
        /pgf/number format/precision=0
    }
]
\addplot+[bar shift=-2.25pt, fill=orange, draw=orange]
table[x=n,y=piFA] {tikzdata/salvo_variance/eam_medium_obj.dat};
\addplot+[bar shift=-0.75pt, fill=green!60!black, draw=green!60!black]
table[x=n,y=piSL] {tikzdata/salvo_variance/eam_medium_obj.dat};
\addplot+[bar shift=0.75pt, fill=red, draw=red]
table[x=n,y=piEA] {tikzdata/salvo_variance/eam_medium_obj.dat};
\addplot+[bar shift=2.25pt, fill=purple, draw=purple]
table[x=n,y=piR] {tikzdata/salvo_variance/eam_medium_obj.dat};

\nextgroupplot[
    xlabel={$n$},
    xlabel style={font=\footnotesize},
    ymin=0,
    ymax=\EAMymax,
    ytick={0,5,10,15,20,25},
    yticklabel style={
        text width=2.0em,
        align=right,
        /pgf/number format/fixed,
        /pgf/number format/precision=0
    }
]
\addplot+[bar shift=-2.25pt, fill=orange, draw=orange]
table[x=n,y=piFA] {tikzdata/salvo_variance/eam_large_obj.dat};
\addplot+[bar shift=-0.75pt, fill=green!60!black, draw=green!60!black]
table[x=n,y=piSL] {tikzdata/salvo_variance/eam_large_obj.dat};
\addplot+[bar shift=0.75pt, fill=red, draw=red]
table[x=n,y=piEA] {tikzdata/salvo_variance/eam_large_obj.dat};
\addplot+[bar shift=2.25pt, fill=purple, draw=purple]
table[x=n,y=piR] {tikzdata/salvo_variance/eam_large_obj.dat};

\end{groupplot}

\path (group c1r3.south west) -- (group c3r3.south east)
    coordinate[midway] (legendmid);

\matrix[
    matrix of nodes,
    anchor=north,
    row sep=0pt,
    column sep=4pt,
    nodes={
        anchor=west,
        inner sep=0pt,
        outer sep=0pt,
        font=\footnotesize
    }
] at ($(legendmid)+(0,-0.8cm)$) {
\tikz{\draw[fill=orange,draw=orange] (0,0) rectangle (0.16,0.10);} &
$\pi_{\textsf{FA}}$ &
\tikz{\draw[fill=green!60!black,draw=green!60!black] (0,0) rectangle (0.16,0.10);} &
$\pi_{\textsf{SL}}$ &
\tikz{\draw[fill=red,draw=red] (0,0) rectangle (0.16,0.10);} &
$\pi_{\textsf{EA}}$ &
\tikz{\draw[fill=purple,draw=purple] (0,0) rectangle (0.16,0.10);} &
$\pi_{\textsf{R}}$ \\
};

\end{tikzpicture}
\caption{Average objective values under different salvo-variance regimes. Rows correspond to DTM, SLM, and EAM, while columns correspond to small, medium, and large variance in $\mathbb{E}[L_i]$.}
\label{fig:two_slices_avg_obj}
\end{figure}

\newcommand{\DataDir}{simulation_uniform_q_005_095_tau3_winfreq}

\pgfplotstableread[col sep=comma]{\DataDir/dtm_summary_avg_obj_win_pct.csv}\dtmsummary
\pgfplotstableread[col sep=comma]{\DataDir/slm_summary_avg_obj_win_pct.csv}\slmsummary
\pgfplotstableread[col sep=comma]{\DataDir/eam_summary_avg_obj_win_pct.csv}\eamsummary

\newcommand{\FormatObj}[1]{%
    \pgfmathprintnumber[
        fixed,
        fixed zerofill,
        precision=2
    ]{#1}%
}

\newcommand{\FormatPct}[1]{%
    \pgfmathprintnumber[
        fixed,
        fixed zerofill,
        precision=1
    ]{#1}\%%
}

\newcommand{\AvgWinThree}[3]{%
    \pgfplotstablegetelem{#2}{#3_avg_obj}\of#1%
    \edef\avgobj{\pgfplotsretval}%
    \pgfplotstablegetelem{#2}{#3_win_pct}\of#1%
    \edef\curpct{\pgfplotsretval}%
    \pgfplotstablegetelem{#2}{piFA_win_pct}\of#1%
    \edef\pctFA{\pgfplotsretval}%
    \pgfplotstablegetelem{#2}{piSL_win_pct}\of#1%
    \edef\pctSL{\pgfplotsretval}%
    \pgfplotstablegetelem{#2}{piEA_win_pct}\of#1%
    \edef\pctEA{\pgfplotsretval}%
    \pgfmathsetmacro{\maxpct}{max(\pctFA,max(\pctSL,\pctEA))}%
    \pgfmathtruncatemacro{\isbest}{(\curpct >= \maxpct - 1e-8) ? 1 : 0}%
    \FormatObj{\avgobj}~(\FormatPct{\curpct})%
    \ifnum\isbest=1 $^{*}$\fi%
}

\newcommand{\AvgWinFour}[3]{%
    \pgfplotstablegetelem{#2}{#3_avg_obj}\of#1%
    \edef\avgobj{\pgfplotsretval}%
    \pgfplotstablegetelem{#2}{#3_win_pct}\of#1%
    \edef\curpct{\pgfplotsretval}%
    \pgfplotstablegetelem{#2}{piFA_win_pct}\of#1%
    \edef\pctFA{\pgfplotsretval}%
    \pgfplotstablegetelem{#2}{piSL_win_pct}\of#1%
    \edef\pctSL{\pgfplotsretval}%
    \pgfplotstablegetelem{#2}{piEA_win_pct}\of#1%
    \edef\pctEA{\pgfplotsretval}%
    \pgfplotstablegetelem{#2}{piR_win_pct}\of#1%
    \edef\pctR{\pgfplotsretval}%
    \pgfmathsetmacro{\maxpct}{max(max(\pctFA,\pctSL),max(\pctEA,\pctR))}%
    \pgfmathtruncatemacro{\isbest}{(\curpct >= \maxpct - 1e-8) ? 1 : 0}%
    \FormatObj{\avgobj}~(\FormatPct{\curpct})%
    \ifnum\isbest=1 $^{*}$\fi%
}

\begin{table}[htbp]
\centering
\caption{Simulation performance under uniformly sampled difficulty levels with fixed $\tau=3$.}
\label{tab:uniform_q_tau3_winfreq_summary}
\label{tab:simulation}
\begin{threeparttable}
\scriptsize
\setlength{\tabcolsep}{2pt}
\renewcommand{\arraystretch}{1.05}
\begin{tabular*}{\textwidth}{@{}l@{\extracolsep{\fill}}lllll@{}}
\toprule
$(C,n)$ & Objective
& $\pi_{\textsf{FA}}$
& $\pi_{\textsf{SL}}$
& $\pi_{\textsf{EA}}$
& $\pi_{\textsf{R}}$ \\
\midrule

\multirow[t]{3}{*}{$(10,10)$}
& \Cref{eq::MinDefuse} & \AvgWinThree{\dtmsummary}{0}{piFA} & \AvgWinThree{\dtmsummary}{0}{piSL} & \AvgWinThree{\dtmsummary}{0}{piEA} & -- \\
& (\textsf{SLM$_3$}) & \AvgWinThree{\slmsummary}{0}{piFA} & \AvgWinThree{\slmsummary}{0}{piSL} & \AvgWinThree{\slmsummary}{0}{piEA} & -- \\
& (\textsf{EAM$_3$}) & \AvgWinFour{\eamsummary}{0}{piFA} & \AvgWinFour{\eamsummary}{0}{piSL} & \AvgWinFour{\eamsummary}{0}{piEA} & \AvgWinFour{\eamsummary}{0}{piR} \\
\addlinespace[1pt]

\multirow[t]{3}{*}{$(20,20)$}
& \Cref{eq::MinDefuse} & \AvgWinThree{\dtmsummary}{1}{piFA} & \AvgWinThree{\dtmsummary}{1}{piSL} & \AvgWinThree{\dtmsummary}{1}{piEA} & -- \\
& (\textsf{SLM$_3$}) & \AvgWinThree{\slmsummary}{1}{piFA} & \AvgWinThree{\slmsummary}{1}{piSL} & \AvgWinThree{\slmsummary}{1}{piEA} & -- \\
& (\textsf{EAM$_3$}) & \AvgWinFour{\eamsummary}{1}{piFA} & \AvgWinFour{\eamsummary}{1}{piSL} & \AvgWinFour{\eamsummary}{1}{piEA} & \AvgWinFour{\eamsummary}{1}{piR} \\
\addlinespace[1pt]

\multirow[t]{3}{*}{$(30,30)$}
& \Cref{eq::MinDefuse} & \AvgWinThree{\dtmsummary}{2}{piFA} & \AvgWinThree{\dtmsummary}{2}{piSL} & \AvgWinThree{\dtmsummary}{2}{piEA} & -- \\
& (\textsf{SLM$_3$}) & \AvgWinThree{\slmsummary}{2}{piFA} & \AvgWinThree{\slmsummary}{2}{piSL} & \AvgWinThree{\slmsummary}{2}{piEA} & -- \\
& (\textsf{EAM$_3$}) & \AvgWinFour{\eamsummary}{2}{piFA} & \AvgWinFour{\eamsummary}{2}{piSL} & \AvgWinFour{\eamsummary}{2}{piEA} & \AvgWinFour{\eamsummary}{2}{piR} \\

\midrule

\multirow[t]{3}{*}{$(20,10)$}
& \Cref{eq::MinDefuse} & \AvgWinThree{\dtmsummary}{3}{piFA} & \AvgWinThree{\dtmsummary}{3}{piSL} & \AvgWinThree{\dtmsummary}{3}{piEA} & -- \\
& (\textsf{SLM$_3$}) & \AvgWinThree{\slmsummary}{3}{piFA} & \AvgWinThree{\slmsummary}{3}{piSL} & \AvgWinThree{\slmsummary}{3}{piEA} & -- \\
& (\textsf{EAM$_3$}) & \AvgWinFour{\eamsummary}{3}{piFA} & \AvgWinFour{\eamsummary}{3}{piSL} & \AvgWinFour{\eamsummary}{3}{piEA} & \AvgWinFour{\eamsummary}{3}{piR} \\
\addlinespace[1pt]

\multirow[t]{3}{*}{$(40,20)$}
& \Cref{eq::MinDefuse} & \AvgWinThree{\dtmsummary}{4}{piFA} & \AvgWinThree{\dtmsummary}{4}{piSL} & \AvgWinThree{\dtmsummary}{4}{piEA} & -- \\
& (\textsf{SLM$_3$}) & \AvgWinThree{\slmsummary}{4}{piFA} & \AvgWinThree{\slmsummary}{4}{piSL} & \AvgWinThree{\slmsummary}{4}{piEA} & -- \\
& (\textsf{EAM$_3$}) & \AvgWinFour{\eamsummary}{4}{piFA} & \AvgWinFour{\eamsummary}{4}{piSL} & \AvgWinFour{\eamsummary}{4}{piEA} & \AvgWinFour{\eamsummary}{4}{piR} \\
\addlinespace[1pt]

\multirow[t]{3}{*}{$(60,30)$}
& \Cref{eq::MinDefuse} & \AvgWinThree{\dtmsummary}{5}{piFA} & \AvgWinThree{\dtmsummary}{5}{piSL} & \AvgWinThree{\dtmsummary}{5}{piEA} & -- \\
& (\textsf{SLM$_3$}) & \AvgWinThree{\slmsummary}{5}{piFA} & \AvgWinThree{\slmsummary}{5}{piSL} & \AvgWinThree{\slmsummary}{5}{piEA} & -- \\
& (\textsf{EAM$_3$}) & \AvgWinFour{\eamsummary}{5}{piFA} & \AvgWinFour{\eamsummary}{5}{piSL} & \AvgWinFour{\eamsummary}{5}{piEA} & \AvgWinFour{\eamsummary}{5}{piR} \\

\midrule

\multirow[t]{3}{*}{$(30,10)$}
& \Cref{eq::MinDefuse} & \AvgWinThree{\dtmsummary}{6}{piFA} & \AvgWinThree{\dtmsummary}{6}{piSL} & \AvgWinThree{\dtmsummary}{6}{piEA} & -- \\
& (\textsf{SLM$_3$}) & \AvgWinThree{\slmsummary}{6}{piFA} & \AvgWinThree{\slmsummary}{6}{piSL} & \AvgWinThree{\slmsummary}{6}{piEA} & -- \\
& (\textsf{EAM$_3$}) & \AvgWinFour{\eamsummary}{6}{piFA} & \AvgWinFour{\eamsummary}{6}{piSL} & \AvgWinFour{\eamsummary}{6}{piEA} & \AvgWinFour{\eamsummary}{6}{piR} \\
\addlinespace[1pt]

\multirow[t]{3}{*}{$(60,20)$}
& \Cref{eq::MinDefuse} & \AvgWinThree{\dtmsummary}{7}{piFA} & \AvgWinThree{\dtmsummary}{7}{piSL} & \AvgWinThree{\dtmsummary}{7}{piEA} & -- \\
& (\textsf{SLM$_3$}) & \AvgWinThree{\slmsummary}{7}{piFA} & \AvgWinThree{\slmsummary}{7}{piSL} & \AvgWinThree{\slmsummary}{7}{piEA} & -- \\
& (\textsf{EAM$_3$}) & \AvgWinFour{\eamsummary}{7}{piFA} & \AvgWinFour{\eamsummary}{7}{piSL} & \AvgWinFour{\eamsummary}{7}{piEA} & \AvgWinFour{\eamsummary}{7}{piR} \\
\addlinespace[1pt]

\multirow[t]{3}{*}{$(90,30)$}
& \Cref{eq::MinDefuse} & \AvgWinThree{\dtmsummary}{8}{piFA} & \AvgWinThree{\dtmsummary}{8}{piSL} & \AvgWinThree{\dtmsummary}{8}{piEA} & -- \\
& (\textsf{SLM$_3$}) & \AvgWinThree{\slmsummary}{8}{piFA} & \AvgWinThree{\slmsummary}{8}{piSL} & \AvgWinThree{\slmsummary}{8}{piEA} & -- \\
& (\textsf{EAM$_3$}) & \AvgWinFour{\eamsummary}{8}{piFA} & \AvgWinFour{\eamsummary}{8}{piSL} & \AvgWinFour{\eamsummary}{8}{piEA} & \AvgWinFour{\eamsummary}{8}{piR} \\

\bottomrule
\end{tabular*}

\begin{tablenotes}
\scriptsize
\item Each entry reports the average objective value, followed by the percentage of simulation replications in which the policy attains the best objective value. If multiple policies tie for the best objective value in a simulation replication, each tied policy is counted as attaining the best value in that replication. For DTM, smaller values are better; for SLM and EAM, larger values are better. The policy $\pi_{\textsf{R}}$ is reported only for EAM. A superscript $*$ marks the policy with the largest reported best-performance percentage in each row.
\end{tablenotes}
\end{threeparttable}
\end{table}

\end{document}